\documentclass[english,11pt,twoside,a4paper]{article}

\usepackage{amsmath,amsthm}
\usepackage{amssymb}
\usepackage{mathabx}
\usepackage{amsfonts}
\usepackage{amscd}
\usepackage[all]{xy}
 \usepackage{color}
\usepackage{accents}
\font \smallrm=cmr10 at 9pt

\newcounter{amoi}
\numberwithin{equation}{section}

\newtheorem{thm}{Theorem}[subsection]
\newtheorem{theorem}[thm]{Theorem}

\newtheorem{lemma}[thm]{Lemma}
\newtheorem{proposition}[thm]{Proposition}

\theoremstyle{definition}
\newtheorem{definition}[thm]{Definition}
\newtheorem{example}[thm]{Example}
\newtheorem{examples}[thm]{Examples}
\newtheorem{free text}[thm]{}
\newtheorem{remark}[thm]{Remark}

\newcommand{\nil}[1]{\accentset{\circ}#1}

\newcommand{\N} {\mathbb{N}}
\newcommand{\Z} {\mathbb{Z}}
\newcommand{\R} {\mathbb{R}}
\newcommand{\C} {\mathbb{C}}
\newcommand{\K} {\mathbb{K}}
\newcommand{\bk} {\Bbbk}

\newcommand{\bG} {\mathbf{G}}

\newcommand{\cA}{\mathcal{A}}
\newcommand{\cB}{\mathcal{B}}
\newcommand{\cC}{\mathcal{C}}
\newcommand{\cF}{\mathcal{F}}
\newcommand{\cH}{\mathcal{C}}
\newcommand{\cI} {\mathcal{I}}

\newcommand{\cL} {\mathcal{L}}
\newcommand{\cO}{\mathcal{O}}
\newcommand{\cP} {\mathcal{P}}
\newcommand{\cR} {\mathcal{R}}
\newcommand{\cS} {\mathcal{S}}
\newcommand{\cW} {\mathcal{W}}
\newcommand{\cX} {\mathcal{X}}
\newcommand{\cY} {\mathcal{Y}}

\newcommand{\fa} {\mathfrak{a}}
\newcommand{\fg} {\mathfrak{g}}

\newcommand{\fb} {\mathfrak{b}}
\newcommand{\fk} {\mathfrak{k}}

\newcommand{\fn} {\mathfrak{n}}

\newcommand{\GP} {G_{{}_{\!\cP}}}
\newcommand{\GPt} {G_{{}_{\!\cP}}^{\hskip0,5pt \circ}}
\newcommand{\GPtm} {G_{{}_{\!\cP}}^{\hskip0,5pt \circ, {\scriptscriptstyle -}}}
\newcommand{\GPe} {G_{{}_{\!\cP}}^{\hskip0,5pt{}e}}
\newcommand{\GPem} {G_{{}_{\!\cP}}^{\hskip0,5pt{}e, {\scriptscriptstyle -}}}

\newcommand{\zero} {{\mathbf{0}}}
\newcommand{\uno} {{\mathbf{1}}}

\newcommand{\Uuno} {\mathbf{1\hskip-3,4pt{}l}}

\newcommand{\cat}[1]{{\boldsymbol{\mathsf{#1}}}}  % categoria
\newcommand{\alg} {\text{\rm ($\cat{alg}$)}}

\newcommand{\salg} {\text{\rm ($\cat{salg}$)}}
\newcommand{\Wsalg} {\text{\rm ($\cat{Wsalg}$)}}
\newcommand{\Grass} {\text{\rm ($\cat{Grass}$)}}

\newcommand{\set}{\text{\rm ($\cat{set}$)}}
\newcommand{\grp} {\text{\rm ($\cat{group}$)}}
\newcommand{\Azsmfd} {\text{\rm ($\mathcal{A}_\zero\,\text{--}\,\cat{smfd}$)}}
\newcommand{\Azamfd} {\text{\rm ($\mathcal{A}_\zero\,\text{--}\,\cat{amfd}$)}}
\newcommand{\Azhmfd} {\text{\rm ($\mathcal{A}_\zero\,\text{--}\,\cat{hmfd}$)}}

\newcommand{\ssmfd} {\text{\rm ($\cat{ssmfd}$)}}
\newcommand{\asmfd} {\text{\rm ($\cat{asmfd}$)}}
\newcommand{\hsmfd} {\text{\rm ($\cat{hsmfd}$)}}
\newcommand{\Lsgrp} {\text{\rm ($\cat{Lsgrp}$)}}
\newcommand{\lie} {\text{\rm ($\cat{Lie}$)}}
\newcommand{\sLie} {\text{\rm ($\cat{sLie}$)}}
\newcommand{\sHCp} {\text{\rm ($\cat{sHCp}$)}}

\newcommand{\Ad}{\hbox{\sl Ad}}
\newcommand{\ad}{\hbox{\sl ad}}
\newcommand{\Hom}{\hbox{\sl Hom}}
\newcommand{\End}{\hbox{\sl End}}

\newcommand{\Lie}{\hbox{\sl Lie}}
\newcommand{\Ker}{\hbox{\sl Ker}}

\newcommand{\rGL}{\mathrm{GL}}

\newcommand{\rgl}{\mathfrak{gl}}

\begin{document}

%
% {\ }
%
% \vskip-11pt
%
%%
%  {\ }
%%%
%    \centerline{\smallrm  {\smallsl Transactions of the American Mathematical Society\/}  (to appear)}
%%%
%    \centerline{\smallrm preprint  {\smallsl arXiv:1308.0462 [math.RA] (2013)}}
%   {\smallbf ??}  (????), ??--???.   \ \ \ --- \ \ \  {\smallbf DOI:}
% 10.1007/978-3-8348-9831-9$\underline{\ }$4}
% %
%  \vskip1pt
% %
%    \centerline{\smallrm {\smallsl The original publication is available at\/}
% \  www.springerlink.com/content/t64184631908j166/}
% %
% \vskip39pt   {\ }
%%

\centerline{\Large \bf LIE SUPERGROUPS vs.}
 \vskip11pt
\centerline{\Large \bf SUPER HARISH-CHANDRA PAIRS:}
 \vskip11pt
\centerline{\Large \bf A NEW EQUIVALENCE}

\vskip25pt

\centerline{ Fabio GAVARINI }

\vskip7pt

\centerline{\it Dipartimento di Matematica, Universit\`a di Roma ``Tor Vergata'' } \centerline{\it via della ricerca scientifica 1  --- I-00133 Roma, Italy}

\centerline{{\footnotesize e-mail: gavarini@mat.uniroma2.it}}

\vskip55pt

\begin{abstract}
 \vskip7pt
 {\smallrm
%%%
 \let\thefootnote\relax\footnote{\hskip-7pt 2010 {\it MSC}\;: \, Primary 14M30, 14A22 / 58A50, 58C50; Secondary 17B20 / 81Q60  \\
%%%
% \;  ---  \;   {\bf October 30th, 2018}}
%%%
   {\ } \indent   {\sl Keywords:}  Lie supergroups; Lie superalgebras; super Harish-Chandra pairs.  \\
   {\ } \indent   Partially supported by a MIUR grant PRIN 2012, n.\ 2012KNL88Y\_\!\_\,002, and by the MIUR  {\sl Excellence Department Project\/}  awarded to the Department of Mathematics, University of Rome ``Tor Vergata'', CUP E83C18000100006.}
%%%
%%%
   It is known that there exists a natural functor  $ \Phi $  from Lie supergroups to super Harish-Chandra pairs.  A functor going backwards, that associates a Lie supergroup with each super Harish-Chandra pair, yielding an equivalence of categories, was found by Koszul  \cite{koszul};  this result was later extended by other authors, to different levels of generality, but always elaborating on Koszul's original idea.
                                                       \par
   In this paper, I provide two new backwards equivalences, i.e.\ two different functors  $ \Psi^\circ $  and  $ \Psi^e $  that construct a Lie supergroup (thought of as a special group-valued functor) out of a given super Harish-Chandra pair, so that any Lie supergroup is recovered from its naturally associated super Harish-Chandra pair; more precisely, both  $ \Psi^\circ $  and  $ \Psi^e $  are quasi-inverse to the functor  $ \Phi \, $.}
\end{abstract}

\vskip45pt

\tableofcontents

\vskip41pt

\section{Introduction}

\smallskip

   {\ } \;\;  The study of supergroups is a chapter of supergeometry, i.e.\ geometry in a  $ \Z_2 $--graded  sense.  In particular, the relevant structure sheaves of (commutative) algebras sitting on top of the topological spaces one works with are replaced with sheaves of (commutative)  {\sl superalgebras}.
                                                    \par
   When dealing with  {\sl differential\/}  supergeometry, our ``superspaces'' are supermanifolds, that are real smooth, real analytic or complex holomorphic (depending on the context): any such supermanifold can be considered as a  {\sl classical\/}  (i.e.\ non-super) manifold   -- of the appropriate type ---   endowed with a suitable sheaf of commutative superalgebras.  The supergroups in this context are then  {\sl Lie\/}  supergroups, of smooth, analytic or holomorphic type according to the chosen setup.

\smallskip

   For every Lie supergroup  $ G $  there exists a special pair of objects, say  $ \, \big(G_\zero,\fg\big) \, $,  that is naturally associated with it:  $ G_\zero $  is the  {\sl classical Lie group underlying}  $ G $   --- roughly given by ``killing the odd part'' of the structure sheaf on  $ G $  ---   while  $ \, \fg = \Lie(G) \, $  is the  {\sl tangent Lie superalgebra of}  $ G \, $,  and these two objects are ``compatible'' in a natural sense.  More in general, any similar pair  $ \, \big(K_+\,,\fk\,\big) \, $  made by a Lie group  $ K $  and a Lie superalgebra  $ \fk $  obeying the same compatibility constraints is called ``super Harish-Chandra pair'' (a terminology first found in  \cite{dm}),  or just ''sHCp'' for short: in fact, this notion is tailored in such a way that mapping  $ \, G \mapsto \big(G_\zero,\fg\big) \, $  yields a functor, call it  $ \Phi \, $,  from the category of Lie supergroups   --- either smooth, analytic or holomorphic ---   to the category of super Harish-Chandra pairs --- of smooth, analytic or holomorphic type respectively.
                                                    \par
%
%    The key question then is: can one recover a Lie supergroup out of its associated sHCp\,?
% Or, even more precisely: does there exist any functor  $ \Psi $  from sHCp's to Lie supergroups
% which be a quasi-inverse for  $ \Phi \, $,  so that the two categories be equivalent?  And if
% the answer is positive, how much explicit such a functor is\,?
%
   The key question is: can one recover a Lie supergroup out of its associated sHCp\,?  More precisely: is there any functor  $ \Psi $  from sHCp's to Lie supergroups which be a quasi-inverse for  $ \Phi \, $,  so that the two categories be equivalent?  And how much explicit such a functor (if any) is\,?
                                                     \par
    A first answer to this question was given by Kostant and by Koszul in the real smooth case (see  \cite{kostant}  and  \cite{koszul}),  with equivalent methods, providing an explicit quasi-inverse for  $ \Phi \, $.  Later on, Vishnyakova  (see  \cite{vis})  fixed the complex holomorphic case, and her proof works for the real analytic case as well.
%
%   More recently Carmeli and Fioresi (see  \cite{cf})  raised and solved the same problem
% in the setup of  {\sl algebraic supergeometry},  i.e.\ for algebraic supergroups (and corresponding
% sHCp's), over a ground ring  $ \bk $  that is an algebraically closed field of characteristic zero.
% This was improved by Masuoka (in  \cite{mas}),  who only required that  $ \bk $  be a field whose
% characteristic is not 2; and later on (see \cite{mas-shi}),  Masuoka and Shibata further extended
% this result up to work on every commutative ring.
%
 More recently, this result was increasingly extended to the setup of  {\sl algebraic supergeometry},  i.e.\ for algebraic supergroups (and corresponding sHCp's), over fields and then over rings (see  \cite{cf},  \cite{mas},  \cite{mas-shi}).  It is worth remarking, though, that all these
%
% (increasingly deeper)
%
 subsequent results were, in the end, further improvements of the original idea by Koszul
%
% in  \cite{koszul}
%
 (while Kostant's method was a slight variation of that): indeed, Koszul defines a Lie supergroup out of a sHCp  $ \, \big(K_+\,,\fk\,\big) \, $  as a super-ringed space, just defining the ``proper'' sheaf of (commutative) superalgebras onto  $ K_+ $  by means of
%
% the Lie superalgebra
%
 $ \fk \, $;  the authors of the successive, above mentioned papers just re-worked this same recipe.

\smallskip

   In this paper I present a new method to solve that problem, i.e.\ I provide a different, more concrete functor  $ \Psi $  from
%
% super Harish-Chandra pairs
%
 sHCp's
 to Lie supergroups that is quasi-inverse to  $ \Phi \, $.  The starting idea is to follow a different approach to supergeometry, \`a la Grothendieck: namely instead of thinking of supermanifolds as being super-ringed manifolds (i.e.\ classical manifolds endowed with a sheaf of commutative superalgebras), one studies (or directly defines) them through their ``functor of points''.  Thus, if  $ M $  is a supermanifold, then for each commutative superalgebra  $ A $  one has the manifold  $ M(A) $  of  $ A $--points  of  $ M \, $;  in fact, in order to recover the full supermanifold  $ M $  one can restrict this functor to a smaller category, namely that of  {\sl Weil superalgebras}   --- roughly, those
%%
% whose even part is made of
%%
 which are direct sum of a copy of our ground field
%
% (namely  $ \R $  or  $ \C \, $)
%
 plus a finite-dimensional nilpotent ideal.  Conversely, functors from Weil superalgebras to manifolds enjoying some additional properties do correspond to Lie supergroups (i.e., they are the functor of points of some Lie supergroup): so one can directly call ``Lie supergroup'' any such special functor.  This functorial point of view
%
% has proved being definitely fit
%
 allows
 to unify several different approaches to super\-geometry (see  \cite{bcf})  and also
%
% to be ``the'' right one
%
 to treat infinite-dimensional supermanifolds
%
% too
%
 (see  \cite{all-lau}).  For a broader discussion of the interplay between different approaches to supergeometry
%
% the reader may
%
 we
 refer to classical sources as  \cite{be},  \cite{dm},  \cite{leites},  \cite{vsv}  or more recent ones like  \cite{bcf}, \cite{be-so},  \cite{ccf},  \cite{mol}.
%
% ,  \cite{sac}.
%
                                                     \par
   Now, if we want a functor  $ \Psi $  from sHCp's to Lie supergroups, we need a Lie supergroup  $ \GP $  for each sHCp  $ \cP \, $;  using the functorial point of view, in order to have  $ \GP $  as a functor we need a Lie group  $ \GP(A) $  for each Weil superalgebra  $ A \, $,  and their definition must be natural in  $ A \, $:  moreover, one still has to show that the resulting functor have those additional properties that make it into a Lie supergroup.  Finally, all this should aim to find a  $ \Psi $  that is quasi-inverse to  $ \Phi $   --- and this fixes ultimate bounds to the construction we aim to.
                                                     \par
   Bearing all this in mind, the construction that I present goes as follows.  Given a super Harish-Chandra pair  $ \, \cP = (G_+,\fg) \, $,  for each Weil superalgebra, say  $ A \, $,  I define a group  $ \GP(A) $  abstractly, by generators and relations: this definition is uniform
%
% with respect to  $ A \, $,  and natural (in  $ A \, $),
%
 and natural with respect to  $ A \, $,
 hence it yields a functor from Weil algebras to (abstract) groups, call it  $ \GP \, $.  As key step in the work, one proves that  $ \GP $
%
% has a special structure   --- called  {\sl ``global splitting''}  ---   namely
%
 admits a ``global splitting'', i.e.\
 it is the direct product of  $ G_+ $  times a totally odd affine superspace (isomorphic to  $ \fg_\uno \, $,  the odd part of  $ \fg \, $):  as both these are supermanifolds, it turns out that  $ \GP $  itself is a supermanifold as well, and in fact it is a Lie supergroup because (as a functor) it is group-valued too.
%
% With yet another step, one shows
%
 One more step proves
 that the construction of  $ \GP $  is natural in  $ \cP \, $,
%
% hence
%
 so
 it yields a functor  $ \Psi $  from sHCp's to Lie supergroups: this is our candidate to be a quasi-inverse to $ \Phi \, $.
                                                     \par
   It is immediate to check that  $ \, \Phi \circ \Psi \, $  is isomorphic to the identity functor onto sHCp's;
%
% --- that is,  $ \, \Phi\big(\GP\big) \cong \cP \, $  for every sHCp  $ \cP \, $.
%
 on the other hand, proving that  $ \, \Psi \circ \Phi \, $  is isomorphic to the identity functor onto Lie supergroups
%
% --- that is,  $ \, \Psi\big(\Phi(G)\big) \cong G \, $  for every Lie supergroup  $ G $  ---
%
 is much more demanding.  In fact, to get the latter we need to know that every Lie supergroup  $ G $  has a ``global splitting'' on its own:
%
% , just like  $ \Psi\big(\Phi(G)\big) \, $:
%
 this implies that  $ G $  and  $ \Psi\big(\Phi(G)\big) $  share the same structure, in that both are the direct product of  $ G_\zero $  and
%
% of a copy of
%
  $ \, \fg_\uno = {\big( \Lie(G) \big)}_\uno \, $.  Now, the fact that a ``Global Splitting Theorem'' for Lie supergroups does hold true is (more or less) known among specialists; however, we need it stated in a genuine geometrical form, while it is usually given in sheaf-theoretic terms, so in the end we work it out explicitly.
%
% ; in turn, this requires quite a bit of work.
%
 In fact, we find  {\sl two\/}  different formulations of such a
%
% ``Global Splitting Theorem'':
%
 result: this is why, building upon them, we can provide  {\sl two\/}  versions,  $ \Psi^\circ $  and  $ \Psi^e \, $,  of a functor  $ \Psi $  as required.

\smallskip

   Two last words are still in order:
  \vskip1pt
%
%    \quad  {\it (a)}\,  The same recipe given here for the functor  $ \Psi $  was first
% presented in  \cite{ga4}   --- and the key idea was also in  \cite{fg1},  \cite{fg2},
% \cite{ga1},  \cite{ga2}  and  \cite{ga3}  ---   to solve the same problem in the context
% of  {\sl algebraic\/}  supergeometry; in that paper, I loosely mentioned, that the given
% recipe would apply to the differential setup, i.e.\ to Lie supergroups and their sHCp's,
% if one treated it with the functorial language.  Nevertheless, I later realized that such
% a switch from the algebraic to the differential context is not that trivial, also because
% the functorial approach in differential supergeometry is not so widely known (or applied):
% this is why, eventually, I decided to write down the present paper.
%
   \quad  {\it (a)}\,  The recipe given here   --- for  $ \Psi^\circ $  ---   was originally presented in  \cite{ga4}   to solve the same problem in the context of algebraic supergeometry; adapting this idea to the  {\sl differential\/}  setup (i.e.\ to Lie supergroups and their sHCp's), however, is definitely  {\sl not\/}  straightforward.
%
% , as it requires to keep the functorial approach, which in differential
% supergeometry has started being used only quite recently.
%
 The  {\sl second\/}  recipe instead   --- introducing the functor  $ \Psi^e $  ---   is entirely original; with some work, it can be adapted to the setup of  {\sl algebraic\/}  supergroups and sHCp's too.
                                            \par
%
%    {\it (b)}\,  Although in the present work we deal with Lie supergroups (and sHCp's)
% of finite dimension, our construction of the functor  $ \Psi $  is perfectly fit for the
% {\sl infinite dimensional\/}  case as well   --- still following the most updated approach
% via the functorial language, as in  \cite{all-lau}.  Clearly, this requires some fine-tuning
% technicalities, which goes beyond the scopes of the present work, so we do not fulfill that
% task; however, the core idea (and strategy) to follow is already displayed hereafter.
%
   {\it (b)}\,  We deal here with Lie supergroups (and sHCp's) of finite dimension; nevertheless, our construction of the functor  $ \Psi $  is perfectly fit for the  {\sl infinite dimensional\/}  case too   --- still following the functorial approach, as in  \cite{all-lau}.  This requires extra technicalities which go beyond our scopes, so we do not fulfill that task; however, the core strategy to follow is already displayed hereafter.

\medskip

%
%    Finally, the paper is organized as follows.  In section  \ref{preliminaries}  we
% establish language and notations, in particular all what is needed to deal with the
% functorial approach to Lie supergroups.  Then section  \ref{sgroups-to-sHCp's}
% presents the notion of super Harish-Chandra pair and the natural functor  $ \Phi $
% from Lie supergroups to super Harish-Chandra pairs.  Section  \ref{interlude}
% presents structure results for Lie supergroups, in particular about ``global
% splittings'': more or less, these results are known (or should be known), but
% we could not find them in literature so we write them down ourselves in the
% form we need them to be.
%
   Finally, the paper is organized as follows.  In section  \ref{preliminaries},  I fix language and notations.  Section  \ref{sgroups-to-sHCp's}  introduces the notion of super Harish-Chandra pair and the natural functor  $ \Phi $  from Lie supergroups to super Harish-Chandra pairs.  Section  \ref{interlude}  presents structure results for Lie supergroups, in particular about ``global splittings'': more or less, these results are known (or should be known), but I could not find them in literature (in the form I need them), so I wrote them down myself.
                                                         \par
   The core of the paper is in sections  \ref{sHCp's->Lsgroups}  and  \ref{sec-equivalences}\,.  In section  \ref{sHCp's->Lsgroups},  I introduce two definitions of functor  $ \Psi \, $,  namely  $ \Psi^\circ $  and  $ \Psi^e \, $,  and I prove key structure results for the Lie supergroups  $ \, \GP \! := \Psi(\cP) \, $   --- with  $ \, \Psi \in \big\{\, \Psi^\circ , \Psi^e \,\big\} \, $;  in fact, in both cases the very definition of  $ \GP $  and the results about its structure are ``prescribed'' by the structure results of section  \ref{interlude}  for Lie supergroups in general.  In section  \ref{sec-equivalences}  then I prove, using the structure results of sections  \ref{interlude}  and  \ref{sHCp's->Lsgroups}  (mainly the ``Global Splitting Theorems''), that both versions of functor  $ \Psi $  are indeed quasi-inverse to  $ \Phi \, $,  as expected.
                                                         \par
   Finally, section  \ref{sec_lin-cae/repr's}  treats special cases and applications.

\vskip19pt

   \centerline{\bf Acknowledgements}
 \vskip5pt
   \centerline{ \smallrm The author thanks Alexander Alldridge, Claudio Carmeli and Rita Fioresi for many valuable conversations. }

\bigskip
 \bigskip

\section{Preliminaries}  \label{preliminaries}

\smallskip

   {\ } \;\;   In all this work,  $ \K $  will denote the field  $ \R $  or  $ \C $  of real or complex numbers, respectively, following the context.  All modules (i.e., vector spaces), algebras etc.\ will be considered over  $ \K \, $.

\medskip

\subsection{Supermodules and superalgebras}  \label{supermodules/algebras}

\smallskip

\begin{free text}  \label{bas-alg-sobjcts}
 {\bf Basic algebraic superobjects.}  We call  {\it  $ \K $--supermodule},  or  {\it  $ \K $--super vector space},  any  $ \K $--module  $ V $  endowed with a  $ \Z_2 $--grading
%
% ,  say
%
 $ \, V = V_\zero \oplus V_\uno \, $,  where  $ \, \Z_2 = \{\zero\,,\uno\} \, $  is the group with two elements.  Then
%
% $ \K $--submodule
%
 $ V_\zero $  and its elements are called  {\it even},  while  $ V_\uno $  and its elements  {\it odd}.  By  $ \, |x| $  or  $ p(x) $  $(\in \Z_2) \, $  we denote the  {\sl parity\/}  of any homogeneous element,
%
% which is
%
 defined by the condition  $ \, x \in V_{|x|} \, $.

\vskip4pt

%
%    We call  {\it  $ \K $--superalgebra\/}  any associative, unital  $ \K $--algebra
% $ A $  which is  $ \Z_2 $--graded  (as a  $ \K $--algebra):  so  $ A $  has a
% $ \Z_2 $--splitting  $ \, A = A_\zero \oplus A_\uno \, $,  and  $ \, A_{\mathbf{a}}
% \, A_{\mathbf{b}} \subseteq A_{\mathbf{a}+\mathbf{b}} \; $.  All  $ \K $--superalgebras
% form a category, whose morphisms are those of unital  $ \K $--algebras  that preserve
% (also) the  $ \Z_2 $--grading.
%                                                                 \par
%    A superalgebra  $ A $  is said to be  {\it commutative}   --- in the ``super sense'' ---
% iff  $ \; x \, y = (-1)^{|x|\,|y|} y \, x \; $  for all homogeneous  $ \, x $,  $ y \in A \, $
% and  $ \; z^2 = 0 \; $  for all odd  $ \, z \in A_\uno \, $.
%
% \smallskip
%
%
%    We denote  by $ \salg $  the category of commutative  $ \K $--superalgebras;  if necessary,
% we shall stress the r{\^o}le of  $ \K $  by writing  $ \salg_\K \, $.  Moreover, we shall
% denote by  $ \alg $   --- or  $ \alg_\K \, $,  sometimes  ---   the category of (associative,
% unital) commutative  $ \K $--algebras,  and by  $ \text{\bf (mod)}_\K $  that of  $ \K $--modules.
%
   We call  {\it  $ \K $--superalgebra\/}  any associative, unital  $ \K $--algebra  $ A $  which is  $ \Z_2 $--graded  (as a  $ \K $--algebra):  so  $ A $  has a  $ \Z_2 $--splitting  $ \, A = A_\zero \oplus A_\uno \, $,  and  $ \, A_{\mathbf{a}} \, A_{\mathbf{b}} \subseteq A_{\mathbf{a}+\mathbf{b}} \; $;
  any such  $ A $  is said to be  {\it commutative}   --- in ``super sense'' ---   iff  $ \; x \, y = (-1)^{|x|\,|y|} y \, x \; $  for all homogeneous  $ \, x $,  $ y \in A \, $  and  $ \; z^2 = 0 \; $  for all odd  $ \, z \in A_\uno \, $.
  All  $ \K $--superalgebras  form a category, whose morphisms are those of unital  $ \K $--algebras preserving the  $ \Z_2 $--grading;  inside it, commutative  $ \K $--superalgebras  form a subcategory,
that we denote  by $ \salg \, $,  or also by  $ \salg_\K \, $.  Moreover, we shall denote by  $ \alg $   --- or  $ \alg_\K \, $,  sometimes  ---   the category of (associative, unital) commutative  $ \K $--algebras,  and by  $ \text{\bf (mod)}_\K $  that of  $ \K $--modules.

\smallskip

   For  $ \, A \in \salg \, $,  $ \, n \in \N \, $,  we call  $ A_\uno^{\,[n]} $  the  $ A_\zero \, $--submodule  of  $ A $  spanned by all products  $ \, \vartheta_1 \cdots \vartheta_n \, $  with  $ \, \vartheta_i \in A_\uno \, $  for all  $ i \, $.  We need also the following constructions: if  $ \, J_A := (A_\uno) \, $  is the ideal of  $ A $  generated by  $ A_\uno \, $,  then  $ \, J_A = A_\uno^{[2]} \oplus A_\uno \, $,  and  $ \, \overline{A} := A \big/ J_A \, $  is a commutative superalgebra which is  {\sl totally even},  i.e.~$ \, \overline{A} \in \alg \, $;  also, there is an obvious isomorphism  $ \, \overline{A} := A \big/ (A_\uno) \cong A_\zero \big/ A_\uno^{[2]} \; $.
                                              \par
   Finally, the constructions of  $ A_\zero \, $,  of  $ A_\uno^{(n)} $  and of  $ \, \overline{A} \, $  all are functorial in  $ A \, $.
\end{free text}

\smallskip

\begin{free text}  \label{Weil-salg}
 {\bf Weil superalgebras.}  We now introduce a special class of commutative superalgebras, the  {\it Weil superalgebras},  or  {\it ``super Weil algebras''},  to be used later on  (cf.\  \cite{bcf}  and references therein).

\smallskip

\begin{definition}  \label{def:Weil-salg}
 We call  {\it Weil superalgebra\/}  any finite-dimensional commutative  $ \K  $--superalgebra  $ A $  such that   $ \; A = \K \oplus \nil{A} \; $  where  $ \K $  is even and  $ \, \nil{A} = \nil{A}_\zero \oplus \nil{A}_\uno \, $  is a  $ \Z_2 $--graded  nilpotent ideal (the  {\sl nilradical\/}  of  $ A \, $).  By construction, every Weil superalgebra  $ A $  is automatically endowed with the canonical (super)algebra morphisms  $ \, p_A : A \longrightarrow \K \, $  and  $ \, u_A : \K \longrightarrow A \, $  associated with the direct sum splitting  $ \, A = \K \oplus \nil{A} \, $;  in particular  $ \, p_A \circ u_A = \text{\sl id}_\K \, $,  so that  $ \, p_A \, $  is surjective and  $ \, u_A \, $  is injective.  Weil superalgebras over  $ \K $  form a full subcategory of  $ \salg_\K \, $,  denoted  $ \Wsalg_\K $  or  $ \Wsalg \, $.
 $ \diamondsuit $
\end{definition}

\smallskip

   A special class of Weil superalgebras is given by  {\sl Grassmann algebras\/}:  namely, for any  $ \, n \in \N \, $  by  {\it Grassmann algebra in  $ n $  variables over  $ \, \K $}  we mean the polynomial  $ \K $--algebra
%
% $ \, \Lambda_n := \K\big[\xi_1, \dots, \xi_n\big] \, $
%
 $ \, \Lambda_\K(\xi_1,\dots,\xi_n) := \K\big[\xi_1, \dots, \xi_n\big] \, $
 in  $ n $  mutually anticommuting indeterminates  $ \xi_1 $,  $ \dots $,  $ \xi_n \, $.  Giving degree  $ \uno $  to every  $ \xi_i $  all these Grassmann algebras turn into commutative  $ \K $--superalgebras;  we denote by  $ \Grass_\K \, $,  or just  $ \Grass \, $,  the full subcategory of  $ \Wsalg_\K $  whose objects are isomorphic to some  $ \Lambda_n $  ($ \, n \in \N \, $).
\end{free text}

\medskip

\subsection{Lie superalgebras}  \label{Lie-superalgebras}

\smallskip

   {\ } \;\;   The infinitesimal counterpart of Lie supergroups is given by the notion of Lie superalgebras.  We shall see their link later on, while now we fix that notion, as well as its functorial formulation.

\smallskip

\begin{definition}  \label{def:Lie-superalgebras}
 Let  $ \, \fg = \fg_\zero \oplus \fg_\uno \, $  be a  $ \K $--supermodule.  We say that  $ \fg $  is a  {\sl Lie superalgebra\/}  if we have a  {\it (Lie super)bracket\/}  $ \; [\,\cdot\, , \cdot\, ] : \fg \times \fg \longrightarrow \fg \, $,  $ \; (x,y) \mapsto [x,y] \, $,  \, which is  $ \K $--bilinear,  $ \Z_2 $--graded  and satisfies the following properties (for all  $ \, x , y, z \in \fg_\zero \cup \fg_\uno \, $):
 \vskip3pt
 \hskip-9pt
   {\it (a)}  \hskip43pt   $ \qquad  [x,y] \, + \, {(-1)}^{|x| \, |y|}[y,x] \; = \; 0  \;\, $   \hskip43pt  {\sl (anti-symmetry)}\,;
 \vskip5pt
 \hskip-9pt
   {\it (b)}
 $ \qquad  {(-\!1)}^{|x| \, |z|} [x,  [y,z]] + {(-\!1)}^{|y| \, |x|} [y , [z,x]] \, + \, {(-\!1)}^{|z| \, |y|} [z , [x,y]] \, = \, 0 \;\, $   \hskip23pt  {\sl (Jacobi identity).}
 \vskip9pt
%
%    In this situation, as a matter of notation, we shall write
% %
%  \vskip7pt
% %
%  \hskip-9pt
%    {\it (c)}
% %
%  $ \hskip55pt  Y^{\langle 2 \rangle} := \, 2^{-1} \, [Y,Y]  \qquad  \big( \in \fg_\zero \big) \, $
% \qquad \qquad  for all  $ \; Y \in \fg_\uno \; $.
%
 \vskip1pt
 \hskip-9pt
   {\it (c)}\;   In this situation, we shall write
 $ \;\;\; Y^{\langle 2 \rangle} := \, 2^{-1} \, [Y,Y] \;\; \big(\! \in \fg_\zero \big) \;\;\; $  for all  $ \; Y \in \fg_\uno \; $.

 \vskip9pt

   All Lie  $ \K $--superalgebras  form a category, denoted  $ \sLie_\K $  or just  $ \sLie \, $,  whose morphisms are the  $ \K $--linear,  graded maps that preserve the bracket.
 \hfill   $ \diamondsuit $
\end{definition}

\smallskip

   {\sl Note\/}  that if  $ \fg $  is a Lie  $ \K $--superalgebra,  then its even part  $ \fg_\zero $  is automatically a Lie  $ \K $--algebra.

\smallskip

\begin{example}  \label{def-End(V)}
  Let  $ \, V = V_\zero \oplus V_\uno \, $  be a  $ \K $--supermodule,  and consider  $ \End(V) \, $,  the endomorphisms of  $ V $  as an ordinary  $ \K $--module.  This is in turn a  $ \K $--supermodule,  $ \; \End(V) = {\End(V)}_\zero \oplus {\End(V)}_\uno \, $,  \, where  $ {\End(V)}_\zero $  are the morphisms which preserve the parity, and  $ {\End(V)}_\uno $  those which reverse it.  If  $ V $  has finite dimension and we choose a basis for  $ V $  of homogeneous elements, writing first the even ones, then  $ \End(V)_\zero $  is the set of all diagonal block matrices, while  $ {\End(V)}_\uno $  is the set of all anti-diagonal block matrices.
 Thus  $ \End(V) $  is a Lie  $ \K $--superalgebra  with bracket
 $ \,\; [A,B] \, := \, A B - {(-1)}^{|A||B|} \, B A \;\, $  for homogeneous  $ \, A, B \in \End(V) \, $;
 and then  $ \, Y^{\langle 2 \rangle} = Y^2 \, $  for odd  $ Y \, $.
                                                           \par
   The standard example is $ \, V := \K^p \oplus \K^q \, $,  with  $ \, V_\zero := \K^p \, $  and  $ \, V_\uno := \K^q \, $.  In this case we also write  $ \; \End\big(\K^{p|q}\big) \! := \End(V) \; $  or  $ \; \rgl_{\,p\,|q} := \End(V) \; $.   \hfill  $ \blacklozenge $
\end{example}

\smallskip

\begin{free text}  \label{Lie-salg_funct}
 {\bf Functorial presentation of Lie superalgebras.}  \ Let  $ \Wsalg_\K $  be the category of Weil  $ \K $--superalgebras  (see  \S \ref{supermodules/algebras})  and denote by  $ \lie_\K $  the category of Lie  $ \K $--algebras  and by  $ \mod_\K $  the category of  $ \K $--modules.  Every Lie  $ \K $--superalgebra  $ \, \fg \in \sLie_\K \, $  yields a functor
 \vskip3pt
   \centerline{ $ \; \cL_\fg : \Wsalg_\K \relbar\joinrel\relbar\joinrel\longrightarrow \lie_\K \; $,  \quad  $ A \, \mapsto \, \cL_\fg(A) := \big( A \otimes \fg \,\big)_\zero = (A_\zero \otimes \fg_\zero)
\oplus (A_\uno \otimes \fg_\uno) \; $ }
 \vskip3pt
\noindent
 Indeed,  $ \, A \otimes \fg \, $  is a Lie superalgebra (in a suitable, more general sense, on the Weil  $ \K $--superalgebra  $ A \, $)  on its own, with Lie bracket
  $ \; \big[\, a \otimes  X \, , \, a' \otimes X' \,\big] \, := \, {(-1)}^{|X|\,|a'|} \, a\,a' \otimes \big[X,X'\big] \; $,  given by the so-called ``sign rules'';
 now  $ \cL_\fg(A) $  is the even part of the Lie superalgebra  $ \, A \otimes \fg \,$,  hence it is a Lie algebra on its own (see  \cite{ccf}  for details).  In particular, all this applies to  $ \, \fg := \End(V) \, $.
                                              \par
   More in general, the following holds.  Every  $ \K $--supermodule  $ \, \mathfrak{m} = \mathfrak{m}_\zero \oplus \mathfrak{m}_\uno \, $  defines a functor
 \vskip3pt
   \centerline{ $ \mathcal{L}_{\mathfrak{m}} : \Wsalg_\K \relbar\joinrel\relbar\joinrel\longrightarrow \text{\bf (mod)}_\K \; $,  \quad  $ A \, \mapsto \, \cL_{\mathfrak{m}}(A) := {\big( A \otimes_\K \mathfrak{m} \big)}_\zero = (A_\zero \otimes \mathfrak{m}_\zero) \bigoplus (A_\uno \otimes \mathfrak{m}_\uno) $ }
 \vskip3pt
\noindent
 and then  \quad  $ \mathfrak{m} $  {\it is a Lie  $ \K $-superalgebra  $ \,\; \Longleftrightarrow \;\, $  $ \mathcal{L}_{\mathfrak{m}} $  takes values in  $ \lie_\K \;\; $}.
 \vskip3pt
   In fact, this holds true also if we replace  $ \Wsalg_\K $  with its (smaller) subcategory  $ \Grass_\K \, $.

\vskip5pt

   This ``functorial presentation'' of Lie superalgebras can be adapted to representations too.  Indeed, let  $ V $  be a  $ \fg $--module,  with representation map  $ \; \phi : \fg \longrightarrow \End(V) \; $   --- a Lie superalgebra morphism.  Scalar extension induces a morphism  $ \; \text{\sl id}_A \otimes \phi : A \otimes \fg \longrightarrow A \otimes \End(V) \; $  for each  $ \, A \in \Wsalg_\K \, $,  whose restriction to the even part gives a morphism  $ \; \big( A \otimes \fg \,\big)_\zero \longrightarrow \big( A \otimes \End(V) \big)_\zero \; $,  that is a morphism  $ \; \cL_\fg(A) \longrightarrow \cL_{\text{\sl End}(V)}(A) \; $  in  $ \lie_\K \, $.  The whole construction is natural in  $ A \, $,  hence it induces a natural transformation of functors  $ \, \cL_\fg \longrightarrow \cL_{\text{\sl End}(V)} \, $.
                                          \par
   The same considerations apply as well if  $ \Wsalg_\K $  is replaced with  $ \Grass_\K \, $.

\vskip5pt

   In the sequel, we shall call  {\it quasi-representable\/}  any functor  $ \; \cL : \Wsalg_\K \! \longrightarrow \lie_\K \; $  for which there exists  $ \, \fg \in \sLie_\K \, $  such that  $ \, \cL \cong \cL_\fg \; $;  the same applies with  $ \Grass_\K $  instead of  $ \Wsalg_\K \, $.
%
%%%
%    Finally, note that all this has a natural, non-super counterpart which is obtained
% by letting ``Lie algebras'' replace ``Lie superalgebras'' all over the place.
%%%
%
\end{free text}

\medskip

\subsection{Supermanifolds and supergroups}  \label{super-manfds/groups}

\smallskip

   {\ } \;\;   In this subsection we introduce the ``supermanifolds''   --- real smooth, real analytic or complex holomorphic ones ---   as well as the corresponding group objects.  This material is more or less standard; we follow two equivalent approaches: we refer to  \cite {bcf}   --- possibly with different terminology and notation ---   where all needed details can be found, as well as the original references.

\smallskip

\begin{definition}  \label{def-superspaces}
 A  {\it superspace\/}  is a pair  $ \, S = \big( |S|, \cO_S \big) \, $  of a topological space  $ |S| $  and a sheaf of commutative superalgebras  $ \cO_S $  on it such that the stalk of  $ \cO_S $  at each point  $ \, x \in |S| \, $,  denoted by  $ \cO_{S,x} \, $,  is a local superalgebra.
 %%%
%%                                                                             \par
 %%%
   If  $ S $  and  $ T $  are two superspaces, a  {\it morphism\/}  $ \, \phi : S \relbar\joinrel\longrightarrow T \, $  between them is a pair  $ \big( |\phi| \, , \phi^* \big) $  where  $ \, |\phi| : |S| \relbar\joinrel\longrightarrow |T| \, $  is a continuous map of topological spaces and  $ \, \phi^* : \cO_T \relbar\joinrel\longrightarrow |\phi|_*\big(\cO_S\big) \, $  is a morphism of sheaves on  $ |T| $  such that  $ \, \phi_x^*(\mathfrak{m}_{|\phi|(x)}) \subseteq \mathfrak{m}_x \, $,  where  $ \mathfrak{m}_{|\phi|(x)} $  and  $ \mathfrak{m}_{x} $  denote the maximal ideals in the stalks  $ \cO_{T,|\phi|(x)} $  and  $ \cO_{S,x} $  respectively.
 \hfill   $ \diamondsuit $
\end{definition}

\vskip-9pt

\begin{examples}  \label{superspace-pq}
  Fix any  $ \, p, q \in \N_+ \, $.
 \vskip3pt
%
%
%    {\it (a)}  {\sl The real smooth local model}  ---  The superspace  $ \R^{p|q} $
% is the topological space  $ \R^p $  endowed with the following sheaf of commutative
% superalgebras:  $ \; \cO_{\R^{p|q}}(U) := \cC^\infty_{\R^p}(U) \otimes_\R
% \Lambda_\R(\vartheta_1,\dots,\vartheta_q) \; $  for any open set $ \, U \subseteq
% \R^p \, $,  where  $ \Lambda_\R(\vartheta_1,\dots,\vartheta_q) $  is the real Grassmann
% algebra on  $ q $  variables  $ \vartheta_1 $,  $ \dots $,  $ \vartheta_q \, $  and
% $ \cC^\infty_{\R^p} $  is the sheaf of smooth real functions on  $ \R^p \, $.
%
   {\it (a)}  {\sl The real smooth local model}  ---  The superspace  $ \R^{p|q} $  is the topological space  $ \R^p $  endowed with the following sheaf of commutative superalgebras:  $ \; \cO_{\R^{p|q}}(U) := \cC^\infty_{\R^p}(U) \otimes_\R \Lambda_\R(\xi_1,\dots,\xi_q) \; $  for any open set $ \, U \subseteq \R^p \, $,  where  $ \cC^\infty_{\R^p} $  is the sheaf of smooth functions on  $ \R^p \, $.
%
% and  $ \Lambda_\R(\xi_1,\dots,\xi_q) $  is the real Grassmann algebra on  $ q $  variables
% $ \xi_1 $,  $ \dots $,  $ \xi_q \, $.
%
 \vskip1pt
%
%    {\it (b)}  {\sl The real analytic local model}  ---  The superspace
% $ \R^{p|q}_{\,\omega} $  is the topological space  $ \R^p $  endowed with the
% following sheaf of commutative superalgebras:  $ \; \cO_{\R^{p|q}_{\,\omega}}(U)
% := \cC^\omega_{\R^p}(U) \otimes_\R \Lambda_\R(\vartheta_1,\dots,\vartheta_q) \; $
% for any open set $ \, U \subseteq \R^p \, $,  where  $ \Lambda_\R(\vartheta_1, \dots,
% \vartheta_q) $  is the real Grassmann algebra on  $ q $  variables  $ \vartheta_1 $,
% $ \dots $,  $ \vartheta_q \, $  and  $ \cC^\omega_{\R^p} $  is the sheaf of analytic
% real functions on  $ \R^p \, $.
%
   {\it (b)}  {\sl The real analytic local model}  ---  The superspace  $ \R^{p|q}_{\,\omega} $  is the topological space  $ \R^p $  endowed with the following sheaf of commutative superalgebras:  $ \; \cO_{\R^{p|q}_{\,\omega}}(U) := \cC^\omega_{\R^p}(U) \otimes_\R \Lambda_\R(\xi_1,\dots,\xi_q) \; $  for any open set $ \, U \subseteq \R^p \, $,  where  $ \cC^\omega_{\R^p} $  is the sheaf of analytic functions on  $ \R^p $
%
% and  $ \Lambda_\R(\xi_1,\dots,\xi_q) $  is the real Grassmann algebra on  $ q $
% variables  $ \xi_1 $,  $ \dots $,  $ \xi_q \, $.
%
 \vskip1pt
%
%    {\it (c)}  {\sl The holomorphic local model}  ---  The superspace  $ \C^{p|q} $
% is the topological space  $ \C^p $  endowed with the following sheaf of commutative
% superalgebras:  $ \; \cO_{\C^{p|q}}(U) := \cH_{\C^p}(U) \otimes_\C \Lambda_\C(\vartheta_1,
% \dots, \vartheta_q) \; $  for any open set $ \, U \subseteq \C^p \, $,  where
% $ \cH_{\C^p} $  is the sheaf of complex holomorphic functions on  $ \C^p $  and
% $ \Lambda_\C(\vartheta_1,\dots,\vartheta_q) $  is the complex Grassmann algebra
% on  $ q $  variables  $ \vartheta_1 $,  $ \dots $,  $ \vartheta_q \, $.
%
   {\it (c)}  {\sl The holomorphic local model}  ---  The superspace  $ \C^{p|q} $  is the topological space  $ \C^p $  endowed with the following sheaf of commutative superalgebras:  $ \; \cO_{\C^{p|q}}(U) := \cH_{\C^p}(U) \otimes_\C \Lambda_\C(\vartheta_1,\dots,\vartheta_q) \; $  for any open set $ \, U \subseteq \C^p \, $,  where  $ \cH_{\C^p} $  is the sheaf of holomorphic functions on  $ \C^p \, $.
%
% and  $ \Lambda_\C(\xi_1,\dots,\xi_q) $  is the complex Grassmann algebra on  $ q $
% variables  $ \xi_1 $,  $ \dots $,  $ \xi_q \, $.
%
\end{examples}

%
%%%
% \vskip5pt
%
% \begin{remark}
%  In cases  {\it (a)\/}  and  {\it (c)\/}  above we denoted  $ \K^{p|q} $  a specific
% superspace; by a standard abuse of notation, the same symbols also denote the super
% vector space  $ \, \K^p \oplus \K^q \, $,  whose first summand is the even part and
% the second summand the odd one.
% %
% \end{remark}
%%%
%

\smallskip

   Patching together these local models one makes up ``supermanifolds'', defined as follows:

\vskip11pt

\begin{definition}  \label{def-supermanifold}  {\ }
 \vskip3pt
%
%    {\it (a)}\,  A  {\it (real) smooth supermanifold\/}  of (super)dimension  $ p|q $
% is a superspace  $ \, M = \big( |M| \, , \cO_M \big) \, $  that is locally isomorphic
% to  $ \R^{p|q} \, $,  i.~e.\ for every  $ \, x \in |M| \, $  there exists an open set
% $ \, V_x \subseteq |M| \, $  with  $ \, x \in V_x \, $  and  $ \, U \subseteq \R^p \, $
% such that  $ \; \cO_M{\big|}_{V_x} \cong \, \cO_{\R^{p|q}}{\big|}_U \; $,  and such that
% $ |M| $  is Hausdorff and second-countable.  A  {\sl morphism\/}  of smooth supermanifolds
% is a morphism of the underlying superspaces.
%
   {\it (a)}\,  A  {\it (real) smooth supermanifold\/}  of (super)dimension  $ p|q $  is a superspace  $ \, M = \big( |M| \, , \cO_M \big) \, $ such that  $ |M| $  is Hausdorff and second-countable and  $ M $  is locally isomorphic to  $ \R^{p|q} \, $,  i.~e.\ for each  $ \, x \in |M| \, $  there is an open set  $ \, V_x \subseteq |M| \, $  with  $ \, x \in V_x \, $  and  $ \, U \subseteq \R^p \, $  such that  $ \; \cO_M{\big|}_{V_x} \cong \, \cO_{\R^{p|q}}{\big|}_U \; $.
%
% A  {\sl morphism\/}  of smooth supermanifolds is a morphism of the underlying superspaces.
%
 \vskip3pt
   {\it (b)}\,  A  {\it (real) analytic supermanifold\/}  of (super)dimension  $ p|q $  is defined like in  {\it (a)\/}  but with  $ \R^{p|q}_{\,\omega} $  replacing  $ \R^{p|q} $  everywhere: in particular, it is locally isomorphic to  $ \R^{p|q}_{\,\omega} \, $.
%
% A  {\sl morphism\/}  of analytic supermanifolds is just a morphism of the underlying superspaces.
%
 \vskip3pt
   {\it (c)}\,  A  {\it (complex) holomorphic supermanifold\/}  of (super)dimension  $ p|q $  is defined like in  {\it (a)\/}  but with  $ \C^{p|q} $  replacing  $ \R^{p|q} $  everywhere: in particular, it is locally isomorphic to  $ \C^{p|q} \, $.
%
% A  {\sl morphism\/}  of holomorphic supermanifolds is a morphism of the underlying superspaces.
%
 \vskip3pt
   {\it (d)}\,  A  {\sl morphism\/}  of smooth, analytic or holomorphic supermanifolds is, by definition, a morphism of the underlying superspaces.
 \vskip3pt
   In either case above, the sheaf  $ \cO_M $  is called the {\it structure sheaf\/}  of  $ M \, $:  to simplify notation, we shall write  $ \cO(M) $  instead of  $ \cO_M\big(|M|\big) $  to denote the superalgebra of global sections of  $ \cO_M \, $.
                                                                 \par
   We denote the category of (real) smooth, (real) analytic, or (complex) holomorphic supermanifolds by  $ \ssmfd \, $,  $ \asmfd \, $,  or  $ \hsmfd \, $,  respectively.
 \hfill   $ \diamondsuit $
\end{definition}

\smallskip

   In most cases later on the distinction between the smooth, the analytic or the holomorphic case is immaterial: therefore, in order to minimize repetitions, I shall often refer only to ``supermanifolds''.

\medskip

\begin{free text}  \label{red-submfd}
 {\bf The reduced submanifold of a supermanifold.}  Let  $ M $  be a smooth supermanifold and  $ U $  an open subset in $ |M| \, $.  Let  $ \cI_M(U) $  be the ideal of  $ \cO_M(U) $  generated by the odd part of the latter: then $ \, \cO_M \big/ \cI_M \, $  is a presheaf, whose sheafification  $ \widetilde{ \cO_M \big/ \cI_M} $  is a sheaf of purely even superalgebras over  $ |M| \, $,  locally isomorphic to  $ \cC^\infty\big(\R^p\big) \, $.  Then  $ \, M_\zero := \Big( |M| \, , \widetilde{ \cO_M \big/ \cI_M} \Big) \, $  is a  {\sl classical\/}  smooth manifold, called the  {\it reduced smooth manifold\/}  associated with  $ M \, $;  the standard, built-in projection  $ \, s \mapsto \tilde{s} := s + \cI_M(U) \, $  (for all  $ \, s \in \cO_M(U) \, $)  at the sheaf level corresponds to a natural embedding  $ \, M_\zero \longrightarrow M \, $,  so that  $ M_\zero $  can be thought of as an
%%%%%
  \hbox{embedded sub(super)manifold of  $ M $  itself.}
%%%%%
%%%%%
                                                                        \par
   A similar construction applies when the supermanifold  $ M $  is analytic, resp.\ holomorphic,
%
%%%
% : just consider the relevant ideal of nilpotent elements in the structure sheaf,
% consider the corresponding quotient, etc.  In each case, we eventually find a
%
 yielding the notion
of  {\it ``reduced analytic manifold''},  resp.\  {\it ``reduced holomorphic manifold''},  $ M_\zero $  %
% that is naturally embedded into
%
 of  $ M \, $.
                                                                        \par
   A key feature of this construction is that it is natural, i.e.\ it provides a  {\sl functor\/}  from the category of supermanifolds (of either type: smooth, etc.) to the category of manifolds (of the corresponding type), defined on objects by  $ \, M \mapsto M_\zero \, $  and on morphisms in a natural way   --- cf.\  \cite{bcf}  for details.
 \vskip3pt
   Finally, by the very definition of supermanifolds (smooth or analytic or holomorphic), one sees at once that also ``classical'' manifolds (of either type) can be seen as ``supermanifolds'', simply observing that their structure sheaf is one of superalgebras that are actually {\sl totally even}, i.e.\ with trivial odd part.  Conversely, any supermanifold enjoying the latter property is actually a ``classical'' manifold, nothing more.  In other words, classical manifolds identify with those supermanifolds  $ M $  that actually coincide with their reduced (sub)manifolds  $ M_\zero \, $.
\end{free text}

\smallskip

   We finish this subsection introducing the notion of ``Lie supergroup'':

\smallskip

\begin{definition}  \label{def-supergroups}
 A  {\sl group object\/}  in the category  $ \ssmfd \, $,  or  $ \asmfd \, $,  or  $ \hsmfd \, $,  is called  {\it (real smooth) Lie supergroup},  resp.\  {\it (real) analytic Lie supergroup},  resp.\  {\it (complex) holomorphic Lie supergroup}.
   These objects, together with the obvious morphisms, form a subcategory among supermanifolds, denoted  $ \Lsgrp_\R^\infty \, $,  resp.\  $ \Lsgrp_\R^\omega \, $,  resp.\  $ \Lsgrp_\C^\omega \, $.
 \hfill   $ \diamondsuit $
\end{definition}

\medskip

\subsection{The functorial point of view}  \label{funct-pt-view}

\smallskip

   {\ } \;\;   In this subsection we introduce the language of ``functor(s) of points'' for supermanifolds and Lie supergroups, whose basic idea goes back to Weil's and Grothendieck's approach to algebraic geometry.  It will be just a quick reminder, for more details the interested reader can refer to  \cite{bcf}.

\smallskip

   We begin with some notation.  For any two categories  $ \cat{A} $  and  $ \cat{B} \, $,  with  $ \, [\cat{A},\cat{B}] \, $  we denote the category of all functors between  $ \cat{A} $  and  $ \cat{B} \, $,  the morphisms in  $ [\cat{A},\cat{B}] $  being the natural transformations.  As usual  $ {\cat{A}}^{\text{op}} $  will denote the  {\it opposite category\/}  to  $ \cat{A} \, $,  so that  $ \, \big[ {\cat{A}}^{\text{op}} , \cat{B} \big] \, $  is nothing but the category of contravariant functors from  $ \cat{A} $  to  $ \cat{B} \, $.
%
% \vskip5pt
%%%%%
  \eject
%%%%%
%

\begin{free text}  \label{functor_pts}
 {\bf The functor of points of a supermanifold.}  Our first kind of ``functor of points'' is the following.  Given a (real) smooth supermanifold  $ \, M \in \ssmfd \, $,  its associated  {\it functor of points\/}  $ \; \cF_M : \ssmfd^{\text{op}} \relbar\joinrel\longrightarrow \set \; $  is defined on objects by  $ \;\cF_M(S) := \Hom(S,M) \; $  and on morphisms by  $ \; \cF_M(\phi): \cF_M(S) \relbar\joinrel\relbar\joinrel\longrightarrow \cF_M(T) \, , \, f \mapsto \big(\cF_M(\phi)\big)(f) := f \circ \phi \; $,  \, for all  $ \, S, T \in \ssmfd \, $  and  $ \, \phi \in \Hom(S,T) \, $.  The elements in  $ \cF_M(S) $  are called the  {\it $ S $--points\/}  of  $ M \, $.  A similar definition holds for analytic, resp.\ holomorphic, supermanifolds, with  $ \asmfd \, $,  resp.\  $ \hsmfd \, $,  replacing  $ \ssmfd $  wherever it occurs in the previous definition.

\smallskip

   Now consider two supermanifolds  $ M $  and  $ N \, $:  in order to write precise formulas we assume them to be both smooth, but the discussion hereafter makes sense the same if  $ M $  and  $ N $  are both analytic or both holomorphic.
%
%                                                                    \par
%
   By definition of ``functor of points'' Yoneda's lemma yields a bijection
  $$  \Hom_\ssmfd(M,N) \, \leftarrow\joinrel\relbar\joinrel\relbar\joinrel\relbar\joinrel\rightarrow \, \Hom_{\,[\ssmfd^{\text{op}},\set]}\big(\cF_M , \cF_N\big)  $$
between morphisms  $ \, M \relbar\joinrel\longrightarrow N \, $  and natural transformations  $ \, \cF_M \relbar\joinrel\longrightarrow \cF_N \, $  (cf.\  \cite{maclane},  ch.~3, or  \cite{eh},  ch.~6).  Thus we have an immersion
 $ \; \mathcal{Y} : \ssmfd \relbar\joinrel\relbar\joinrel\longrightarrow \big[ \ssmfd^{\text{op}} , \set \big] \; $
 of  $ \ssmfd $  into  $ \, \big[ \ssmfd^{\text{op}} , \set \big] \, $  that is full and faithful.  The objects in  $ \big[ \ssmfd^{\text{op}} , \set \big] $  that lie in the image of this immersion --- i.e., that arise as functors of points of some supermanifold ---   are exactly those which are representable; indeed, not all objects in  $ \big[ \ssmfd^{\text{op}} , \set \big] $  are representable, but important  {\sl representability criteria\/}  exist   --- e.g., see  \cite{bcf},  Theorem 2.13.
 \vskip3pt
   Finally, in this functorial approach the Lie supergroups are characterized as follows: any supermanifold  $ M $  is actually a Lie supergroup if and only if its functor of points  $ \cF_M $  is actually group-valued, i.e.\ its target category is  $ \grp $  rather than  $ \set \, $.
\end{free text}

\vskip5pt

\begin{free text}  \label{Weil-Berezin_funct_A-pts}
 {\bf The Weil-Berezin functor of  $ \cA $--points.}  We introduce now the notion of  $ \cA $--points  of a supermanifold, following the presentation in  \cite{bcf},  \S 3.2.  This looks similar to the functor of points considered above, but in the end it is quite different indeed: in particular, in order to ``recover'' all the information corresponding to a given supermanifold  $ M $  one is forced to endow each set of  $ A $--points  of  $ M $  with an extra structure, namely that of an  ``{\sl  $ A_\zero $--manifold\/}\,'',  that we shall see later.
 \vskip2pt
   In the constructions we mainly refer to smooth supermanifolds but, with minimal changes, they adapt to analytic and holomorphic supermanifolds too  (cf.\ \cite{bcf}).  We begin with the key definitions:
\end{free text}

\vskip3pt

\begin{definition}  \label{def_A-pts}  {\ }
 Let  $ M $  be a supermanifold (either smooth, etc.).
 \vskip2pt
   {\it (a)}\,    For every point  $ \, x \in |M| \, $  and every Weil superalgebra  $ \, A \in \Wsalg \, $  we define the set of  {\it  $ A $--points  near  $ x \, $}  as given by
 $ \; M_{A,x} \, := \, \Hom_\salg\big( \cO_{M,x} \, , A \big) \; $
and the  {\it set of (all)  $ A $--points\/}  as given by  $ \; M_A \, := \, \bigsqcup_{\,x \in |M|} M_{A,x} \; $. If  $ \, x_{\scriptscriptstyle A} \in M_{A,x} \, $  we call  $ \tilde{x}_A := p_A \circ x_{\scriptscriptstyle A} \, $  the  {\it base point\/}  of  $ x_{\scriptscriptstyle A} \, $,  where  $ \, p_A : A \longrightarrow \K \, $  is the built-in projection of  $ A $  onto  $ \K \, $;  in fact,  $ \tilde{x}_{\scriptscriptstyle A} $  canonically identifies with  $ x \, $.
 \vskip2pt
   {\it (b)}\,  We denote by  $ \; \cW_M : \Wsalg \longrightarrow \set \; $  the functor defined by  $ \, A \mapsto \cW_M(A) := M_A \, $  on objects and on morphisms by  $ \; \rho \mapsto \cW_M(\rho) := \rho^{\scriptscriptstyle (M)} \; $  for every  $ \, A, B \in \Wsalg \, $  and every  $ \, \rho \in \Hom_\salg(A,B) \, $  with  $ \; \rho^{\scriptscriptstyle (M)} : M_A \longrightarrow M_B \, $,  $ \, x_{\scriptscriptstyle A} \mapsto \rho \circ x_{\scriptscriptstyle A} \; $.
 \hfill   $ \diamondsuit $
\end{definition}

\smallskip

\begin{free text}  \label{Shvarts_embedding}
 {\bf Weil-Berezin functors and Shvarts embedding.}  The above mentioned construction of  $ \cW_M $  for a supermanifold  $ M $  is clearly natural in  $ M \, $:  in other words, it gives rise to a functor  $ \; \cB : \ssmfd \longrightarrow [\Wsalg,\set] \; $  given on objects by  $ \, M \mapsto \cB(M) := \cW_M \, $.  As it is explained in  \cite{bcf},  \S 3.3,  this functor is an embedding which is faithful but definitely not full: thus it is unfit to describe  $ \ssmfd \, $,  as it does not identify the latter with a full subcategory of  $ \, [\Wsalg,\set] \, $.  One instead has to understand what is exactly the (non-full) subcategory of  $ \, [\Wsalg,\set] \, $  which is the image of  $ \cB \, $,  and then consequently adapt  $ \cB $  itself to a ``nicer'' functor.  I shall now shortly sketch the construction that is needed to solve this problem, referring to  \cite{bcf},  \S 4, for further details.
 \vskip4pt
   One starts by introducing the notion of  $ A_\zero $--manifold.  Roughly speaking, given a finite dimensional commutative algebra  $ \, A_\zero \in \alg_\K \, $,  an  {\it  $ A_\zero $--manifold\/}  is a (smooth, etc.) manifold  $ M $  endowed with an  $ L $--atlas,  i.e.\ an atlas of local charts each of which is diffeomorphic (or bianalytic, or biholomorphic) with some open subset of a given finite dimensional  $ A_\zero $--module  $ L $  in such a way that the differential of every change of charts is an  $ A_\zero $--module  isomorphism (between local copies of  $ L \, $).
%
%  Given two such  $ A_\zero $--modules  --- for possibly different  $ L $'s  ---
% and two  $ A_\zero $--modules  $ M $  and  $ N $  modelled on them,
%
 Given two  $ A_\zero $--manifolds  $ M $  and  $ N \, $,  possibly modelled on different  $ A_\zero $--modules,
 an  {\it  $ A_\zero $--smooth  (or analytic, or holomorphic) morphism\/}  $ \, \phi : M \longrightarrow N \, $  is any morphism from  $ M $  to  $ N $  in the standard sense (smooth,
%
% analytic,
%
 etc.) such that its differential at each point is  $ A_\zero $--linear.  The resulting category then is denoted  ({\bf  $ A_\zero\,\text{--}\,\cat{smfd}$}),  resp.\  ({\bf  $ A_\zero\,\text{--}\,\cat{amfd}$}),  resp.\  ({\bf  $ A_\zero\,\text{--}\,\cat{hmfd}$}),  possibly with a subscript  $ \K \, $.
 \vskip3pt
   For the second step, we gather together all possible  $ A_\zero $--manifolds  {\sl for all finite dimensional}  $ \, A \in \alg_\K \, $.  Now consider  $ \, A', A'' \in \alg_\K \, $,  a morphism  $ \, \rho : A' \!\longrightarrow\! A'' \, $,
%
% between them, an  $ A'_\zero $--module  $ M' $  and an  $ A''_\zero $--module  $ M'' \, $.
%
 an  $ A'_\zero $--manifold  $ M' $  and an  $ A''_\zero $--manifold  $ M'' \, $.
 By scalar restriction (through  $ \rho \, $),  $ M'' $  is also an  $ A'_\zero $--manifold,  thus we define a  {\it morphism\/}  from  $ M' $  to  $ M'' $  as being a morphism of  $ A'_\zero $--manifolds   
%%%%%
   \hbox{--- in particular, its differential}
%%%%%
%%%%%
 is  $ A'_\zero $--linear  (through  $ \rho \, $).  With this notion of ``morphism'', the various  $ A_\zero $--manifolds  (for all  $ A_\zero $'s)  altogether form a new category, denoted by  $ \Azsmfd \, $,  resp.\ $ \Azamfd \, $,  resp.\  $ \Azhmfd \, $,  in the smooth, resp.\ analytic, resp.\ holomorphic case.  We can now introduce our next definition:

\smallskip

\begin{definition}
 We denote  $ \, \big[\big[ \Wsalg \, , \Azsmfd \big]\big] \, $  the subcategory of  $ \big[ \Wsalg \, , \Azsmfd \big] $  whose objects are those in  $ \, \big[ \Wsalg \, , \Azsmfd \big] \, $  and whose morphisms are all natural transformations  $ \; \phi : \mathcal{G} \relbar\joinrel\longrightarrow \mathcal{H} \; $  such that for every  $ \, A \in \Wsalg \, $  the induced  $ \; \phi_A : \mathcal{G}(A) \relbar\joinrel\longrightarrow \mathcal{H}(A) \; $  is  $ A_\zero $--smooth.  Similarly we define the categories  $ \, \big[\big[ \Wsalg , \Azamfd \big]\big] \, $  and  $ \, \big[\big[ \Wsalg , \Azhmfd \big]\big] \, $  respectively in the analytic and in the holomorphic case.
 \hfill   $ \diamondsuit $
\end{definition}

\smallskip

%
%%%
%   The motivation for introducing the notion of  $ A_\zero $--manifolds  and the categories
% $ \Azsmfd \, $,  $ \Azamfd \, $  and  $ \Azhmfd \, $  lies in the following three results.
% First, let  $ M $  be a supermanifold, say real smooth (the other cases are just the same):
% then for each Weil superalgebra  $ \, A \in \Wsalg \, $  the set  $ \cW_M(A) $  can be naturally
% endowed with a unique structure of  $ A_\zero $--manifold.  Second, if  $ \, \phi : M \longrightarrow
% N \, $  is a morphism of (say smooth) supermanifolds, then every corresponding map  $ \; \phi_A :
% \cW_M(A) \longrightarrow \cW_N(A) \, \big( x_A \mapsto x_A \circ \phi^* \big) \; $  is an
% $ A_\zero $--smooth  morphism, for all  $ \, A \in \Wsalg \, $.  Third, if  $ \, \rho :
% A \longrightarrow B \, $  is a morphism in  $ \Wsalg $  then  $ \; \rho^{\scriptscriptstyle (M)} :
% \cW_M(A) \longrightarrow \cW_M(B) \, \big( x_A \mapsto \rho \circ x_A \big) \; $  is an
% $ A_\zero $--smooth  morphism.  The situation in the analytic and in the holomorphic case
% is the same.
%%%
%
   The motivation for introducing the notion of  $ A_\zero $--manifold
%
% and the categories  $ \Azsmfd \, $,  $ \Azamfd \, $  and  $ \Azhmfd \, $
%
 lies in the following three results:
 \vskip3pt
   {\it (1)} \;  if  $ M $  is any supermanifold, for each  $ \, A \in \Wsalg \, $  the set  $ \cW_M(A) $  can be naturally endowed with a canonical structure of  $ A_\zero $--manifold;
 \vskip3pt
   {\it (2)} \;  if  $ M $  is a supermanifold and  $ \, \rho : A \relbar\joinrel\relbar\joinrel\relbar\joinrel\longrightarrow B \, $  a morphism in  $ \Wsalg \, $,  then the map  $ \, \cW_{\!M\!}(A) \,{\buildrel {\rho^{\scriptscriptstyle (M)}} \over {\relbar\joinrel\longrightarrow}} \cW_{\!M\!}(B) \; \big( x_A \! \mapsto \! \rho \circ x_A \big) \, $  is a morphism in  ({\bf  $ A_\zero\text{--}\cat{smfd}$}),  resp.\  ({\bf $ A_\zero\text{--}\cat{amfd}$}),  resp.\  ({\bf  $ A_\zero\text{--}\cat{hmfd}$}).
 \vskip3pt
   {\it (3)} \;  if  $ \, \phi : M \relbar\joinrel\longrightarrow N \, $  is a morphism of supermanifolds, then for all  $ \, A \in \Wsalg \, $  the map  $ \; \phi_A : \cW_M(A) \relbar\joinrel\longrightarrow \cW_N(A) \; \big( x_A \mapsto x_A \circ \phi^* \big) \; $  is
%
%%%
% $ A_\zero $--smooth,  resp.\  $ A_\zero $--analytic,  resp.\  $ A_\zero $--holomorphic:
% in other words,  $ \phi_A $  is
%%%
%
 a morphism in  $ \, \big[\big[ \Wsalg \, , \Azsmfd \big]\big] \, $,  resp.\ in  $ \, \big[\big[ \Wsalg \, , \Azamfd \big]\big] \, $,  resp.\ in  $ \, \big[\big[ \Wsalg \, , \Azhmfd \big]\big] \, $.

\vskip13pt

   Thanks to the above, we can correctly introduce Weil-Berezin functors and Shvarts embedding:

\smallskip

\begin{definition}  {\ }   \label{def_Weil-Berezin-funct_Shvarts-embed}
 \vskip2pt
 \hskip-13pt   {\it (a)}\;  For every smooth supermanifold  $ \, M \in \ssmfd \, $,  we call  {\it Weil-Berezin (local) ``functor of  $ A $--points''\/}  of  $ M $  the functor  $ \; \cW_M : \Wsalg \relbar\joinrel\longrightarrow \Azsmfd \; $  defined as in  Definition \ref{def_A-pts}{\it (b)}   --- which makes sense thanks to the previous remarks.  The same terminology applies,  {\sl mutatis mutandis},  in the case of any analytic or holomorphic supermanifold.
 \vskip2pt
 \hskip-13pt   {\it (b)}\;  We call  {\it Shvarts embedding\/}  the functor  $ \, \cS \! : \! \ssmfd \!\longrightarrow\! \big[\big[ \Wsalg , \Azsmfd \big]\big] \, $,  \hbox{in the smooth}  case, defined on objects by  $ \, M \mapsto \cW_M \; $;  and similarly in the analytic and the holomorphic case.
   \hfill   $ \diamondsuit $
\end{definition}

\vskip5pt

   The key point here is that  {\it the Shvarts embedding is a full and faithful functor},  so that for any two supermanifolds, say smooth,  $ M $  and  $ N $  one has
 \vskip5pt
    \centerline{ $ \Hom_{\ssmfd}(M,N) \, \cong \, \Hom_{[[\Wsalg\,,\,\Azsmfd]]}\big( \cS(M) , \cS(N) \big) \, = \, \Hom_{[[\Wsalg\,,\,\Azsmfd]]}\big( \cW_M , \cW_N \big) $ }
 \vskip5pt
\noindent
 hence in particular  {\it  $ \, M \cong N \, $  if and only if  $ \; \cS(M) \cong \cS(N) \; $,  \, that is  $ \; \cW_M \cong \cW_N \; $}.
 \vskip5pt
   Therefore one can correctly study supermanifolds via their Weil-Berezin functors.  However, to do that one still has to be able to characterize those objects in  $ \big[\big[ \Wsalg \, , \Azsmfd \big]\big] $   --- in the smooth case, and similarly in the other cases ---   that actually are (isomorphic to) the Weil-Berezin functors of some supermanifolds; in other words, one needs a characterization of the image of Shvarts embedding, which is actually  {\sl not\/}  all of its target category, but a proper subcategory of it.  This is the ``representability problem'', which we do not really care so much for the present work.
                                                         \par
   What is still relevant to us, is that the Shvarts embedding  $ \cS $  preserves products, hence also group objects.  This means, in the end, that the following holds true (cf.\  \cite{bcf},  \S 4):

\vskip9pt

%
% \begin{proposition}  {\ }   \label{funct-char_Lie-supergrps}
% %
%  \vskip1pt
% %
%  A supermanifold  $ M $  is a  {\sl Lie supergroup}  if and only if  $ \, \cS(M) := \cW_M \, $
% takes values in the subcategory   --- among  $ \cA_\zero $--manifolds  ---   of group objects
% (what we might call  {\sl ``Lie  $ \cA_\zero $--groups''\/}).
% %
% \end{proposition}
%

\begin{proposition}   \label{funct-char_Lie-supergrps}
 A supermanifold  $ M $  is a  {\sl Lie supergroup}  if and only if  $ \, \cS(M) := \cW_M \, $  takes values in the subcategory (among  $ \cA_\zero $--manifolds)  of group objects, that we call  {\sl ``Lie  $ \cA_\zero $--groups''}.
\end{proposition}

\end{free text}

\vskip11pt

\begin{free text}  \label{WB_red-mfd}
 {\bf The Weil-Berezin approach for reduced submanifolds and the classical case.}  As we saw in  \S \ref{red-submfd},
%
% by the very definition of supermanifolds,
%
 ``classical'' manifolds can obviously be seen as ``supermanifolds'', simply observing that their structure sheaf is one of superalgebras which are actually {\sl totally even},  i.e.\ with trivial odd part; conversely, any supermanifold with this property is actually ``classical''.  In other words, these are all those supermanifolds  $ M $  which coincide with their (classical) reduced subsupermanifold  $ M_\zero \, $.
                                                                   \par
   From the functorial point of view, it is clear that a supermanifold  $ M $  is classical if and only if the associated Weil-Berezin functor  $ \, \cS(M) := \cW_M \in \big[\big[ \Wsalg , \set \big]\big] \, $  coincides with its restriction to the subcategory  $ \, \alg \bigcap \Wsalg \, $;  in short, these are those  $ M $  such that  $ \, \cW_M(A) = \cW_M(A_\zero) \, $  for all  $ \, A \in \Wsalg_\K \, $.  If instead one deals with a (general) supermanifold  $ M \, $,  then the restriction to  $ \, \alg \bigcap \Wsalg \, $  of its Weil-Berezin functor coincides with the Weil-Berezin functor (for classical manifolds) of its associated reduced submanifold  $ M_\zero \, $:  in a nutshell,  $ \, \cS(M){\big|}_{\alg \bigcap\, \Wsalg} = \cS(M_\zero) \; $.
\end{free text}

\vskip-3pt

\begin{free text}  \label{functor_Lambda-pts}
 {\bf The functor of  $ \Lambda $--points.}  As it is explained in  \cite{bcf},  \S 4.3,  one can repeat the construction of the Weil-Berezin functor of any supermanifold and of the consequent Shvarts embedding functor in a somewhat simpler manner, namely replacing systematically the category  $ \Wsalg_\K $  with its full subcategory  $ \Grass_\K \, $.  Thus for each smooth supermanifold  $ \, M \in \ssmfd_\R \, $  the restriction to  $ \Grass_\R $  of its Weil-Berezin functor yields a new functor  $ \, \cW_M^{\,\scriptscriptstyle (\Lambda)} \! : \Grass_\R \longrightarrow \Azsmfd \, $,  which we call ``(Weil-Berezin)  {\sl functors of  $ \Lambda $--points  of  $ M $''}.  As  $ M $  ranges in  $ \ssmfd_\R \, $  all these  $ \cW_M^{\,\scriptscriptstyle (\Lambda)} $'s  give a new functor  $ \; \cS^{\,\scriptscriptstyle (\Lambda)} \! : \ssmfd_\R \relbar\joinrel\relbar\joinrel\longrightarrow \big[\big[ \Grass_\R \, , \, \Azsmfd \big]\big] \; $  whose main feature is that it is again a full and faithful embedding, which we call again ``Shvarts embedding''.
                                                 \par
   Similarly one does with analytic or holomorphic supermanifolds.
 \vskip3pt
   The outcome is that to study supermanifolds it is enough to consider their functors of  $ \Lambda $--points; in particular, two supermanifolds (of either type) are isomorphic if and only if their corresponding functors of points are.
%                                                  \par
   Moreover, the new Shvarts embedding  $ \cS^{\,\scriptscriptstyle (\Lambda)} $  again preserves products, so it takes group objects to group objects: therefore, the characterization of Lie supergroups stated in  Proposition \ref{funct-char_Lie-supergrps}  still makes sense reading  ``$ \, \cS^{\,\scriptscriptstyle (\Lambda)}(M) := \cW^{\,\scriptscriptstyle (\Lambda)}_M \, $''  instead of  ``$ \, \cS(M) := \cW_M \, $''.
 \vskip3pt
   A direct consequence of all this is the following.  In the rest of the paper, we shall work with Lie supergroups considered as special functors, i.e.\ we study  $ M $  using its Weil-Berezin functor of points  $ \cW_M \, $;  or even, conversely, we consider special functors  $ \, \cW : \Wsalg_\R \longrightarrow \Azsmfd \, $   --- in the smooth case, say ---   and then prove that there exists some (smooth) Lie supergroup  $ M $  such that  $ \, \cW_M = \cW \, $.  Now, due to the above discussion it makes sense to (try to) follow the same strategy using  $ \Grass $  instead of  $ \Wsalg $  and the functor of  $ \Lambda $--points  $ \cW^{\,\scriptscriptstyle (\Lambda)}_M $  instead of  $ \cW \, $.  The good news are, indeed, that this is actually feasible, in that all our discussion in the sequel will perfectly makes sense and will be equally correct in both approaches.  Thus one can choose to work in the larger framework of Weil superalgebras (as we do) or in the simpler setup of Grassmann algebras, and in both cases our procedure and results will apply and hold true exactly the same.
\end{free text}

\bigskip

\section{From Lie supergroups to super Harish-Chandra pairs}  \label{sgroups-to-sHCp's}

\smallskip

   {\ } \;\;   In this section we present the notion of super Harish-Chandra pair, showing how it naturally arises from that of Lie supergroup.  Indeed, here ``naturally'' means that one has a functorial construction that, starting from any Lie supergroup, leads to a special ``pair'', whose properties are then singled out to set down the very definition of ``super Harish-Chandra pairs''.
%
% Overall, this provides a remarkable functor from Lie supergroups to super Harish-Chandra pairs.
%

\medskip

\subsection{The notion of super Harish-Chandra pair}  \label{sec-sHCp's}

\smallskip

   {\ } \;\;   We present now the notion of  {\it super Harish-Chandra pair},  introduced by Koszul in  \cite{koszul}   --- but this terminology is first found in  \cite{dm};
%
% .  When given out of the blue, it might look somewhat artificial, yet we shall see
% in the next subsection that it naturally arises from the study of Lie supergroups.
%
 in the next subsection we shall also see its (natural) motivation.

\vskip9pt

\begin{definition}  \label{def-sHCp}
 We call  {\it (smooth, analytic or holomorphic) super Harish-Chandra pair}   --- or just  {\it ``sHCp''},  in short ---   over  $ \K $  any pair  $ \, (G_+ \, , \, \fg) \, $  such that
 \vskip4pt
   {\it (a)}  $ \!\quad G_+ $  is a (smooth, analytic or holomorphic) Lie group over  $ \K \, $,  and  $ \; \fg \in \sLie_\K \; $;
 \vskip3pt
   {\it (b)}  $ \!\quad \Lie(G_+) = \fg_\zero \; $;
 \vskip3pt
   {\it (c)}  \!\quad  there is a (smooth, analytic or holomorphic)  $ G_+ $--action  on  $ \fg $  by Lie  $ \K $--superalgebra  automorphisms, hereafter denoted by  $ \, \Ad : G_+ \relbar\joinrel\relbar\joinrel\longrightarrow \text{\sl Aut}(\fg) \, $,  such that its restriction to  $ \fg_\zero $  is the adjoint action of  $ G_+ $  on  $ \Lie(G_+) = \fg_\zero \, $  and the differential of this action is the restriction to  $ \; \Lie(G_+) \times \fg \, = \, \fg_\zero \times \fg \; $  of the adjoint action of  $ \fg $
 on itself.
 \vskip5pt
   If  $ \, \big( G'_+ \, , \, \fg' \big) \, $  and  $ \, \big( G''_+ \, , \, \fg'' \big) \, $  are two super Harish-Chandra pairs over  $ \K \, $,  a  {\it morphism\/}  among them is any pair  $ \; (\phi_+,\varphi) : \big( G'_+ \, , \, \fg' \big) \longrightarrow \big( G''_+ \, , \, \fg'' \big) \; $  where  $ \, \phi_+ : G'_+ \longrightarrow G''_+ \, $  is a morphism of Lie groups (in the smooth, analytic or holomorphic sense),  $ \, \varphi : \fg' \longrightarrow \fg'' \, $  is a morphism of Lie superalgebras, and the two are compatible with the additional structure, that is to say
 \vskip3pt
  {\it (d)}  $ \quad  \varphi{\big|}_{\fg_\zero} \, = \, d\phi_+  \quad ,  \qquad  \Ad\big(\phi_+(g)\big) \circ \varphi \, = \, \varphi \circ \Ad(g)  \quad \;\forall \;\, g \in G_+ \;\; $.
 \vskip5pt
   All super Harish-Chandra pairs over  $ \K $  along with their morphisms form a category, denoted  $ \sHCp_\K \, $.  When we have to specify its type, we write  $ \sHCp_\R^\infty $  if this type is real smooth,  $ \sHCp_\R^\omega $  if it is real analytic, and  $  \sHCp_\C^\omega $  if it is complex holomorphic.
     \hfill  $ \diamondsuit $
\end{definition}

\medskip

\subsection{Super Harish-Chandra pairs from Lie supergroups}  \label{sgroups->sHCp's}

\smallskip

   {\ } \;\;   In this subsection we show how one can naturally associate a significant super Harish-Chandra pair with any Lie supergroup; indeed, this is the reason why the very notion of super Harish-Chandra pair was introduced.  What follows is well-known, for details and proofs we refer to  \cite{ccf}.

\vskip11pt

\begin{free text}  \label{reduced-subgrp}
 {\bf The reduced subgroup of a Lie supergroup.}  Let  $ G $  be a Lie supergroup (of either type: smooth, etc.).  As it is a supermanifold, from  \S \ref{red-submfd}  we know that there exists also a reduced submanifold  $ G_\zero $  of  $ G \, $.  Taking the functorial point of view, we know that  $ \; \cS_{G_\zero} = \cS_G{\big|}_{\alg \bigcap\, \Wsalg} \; $  (cf.\  \S \ref{WB_red-mfd})  and  $ \cS_G $  takes values in the subcategory of Lie  $ \cA_\zero $--groups,  (by  Proposition \ref{funct-char_Lie-supergrps});  but then the latter is true for  $ \cS_{G_\zero} $  as well, hence   --- by  Proposition \ref{funct-char_Lie-supergrps}  again ---   we argue that  {\it  $ {G_\zero} $  itself is indeed a Lie group\/}  (either smooth, etc., like  $ G $  is).
                                                        \par
   Moreover, when  $ \, \phi : G' \longrightarrow G'' \, $  is a morphism of Lie supergroups, the morphism of manifolds  $ \, \phi_0 : G'_\zero \longrightarrow G''_\zero \, $   --- induced by the functoriality of the construction  $ \, G \mapsto G_\zero \, $  ---   is in addition a Lie group morphism.  Therefore, we conclude that  $ \, G \mapsto G_\zero \, $  and  $ \, \phi \mapsto \phi_\zero \, $  define a functor from Lie supergroups (of either type) to Lie groups (of the same type).
\end{free text}

\smallskip

\begin{free text}  \label{tangent-Lie-superalg}
 {\bf The tangent Lie superalgebra of a Lie supergroup.}  We now quickly recall how to associate with a Lie supergroup its ``tangent Lie superalgebra'', referring to  \cite{ccf}  for further details.

\smallskip

   For any  $ \, A \in \Wsalg_\K \, $,  we let  $ \, A[\varepsilon] := A[x]\big/\big(x^2\big) \, $  be the so-called  {\it superalgebra  of dual numbers\/}  over  $ A \, $,  in which  $ \, \varepsilon := x \! \mod \! \big(x^2\big) \, $  is taken to be  {\it even}.  Then  $ \, A[\varepsilon] = A \oplus A \varepsilon \, $,  and there are two natural morphisms
 $ \; i_{{}_A} : A \longrightarrow A[\varepsilon] \, $,  $ \, a \;{\buildrel {\,i_{{}_{A_{\,}}}} \over \mapsto}\; a \, $,  \, and
 $ \; p_{{}_A} : A[\varepsilon] \longrightarrow A \, $,  $ \, \big( a + a'\varepsilon \big) \;{\buildrel {\,p_{{}_{A_{\,}}}} \over \mapsto}\; a \; $,
 \, such that  $ \; p_{{}_A} \! \circ i_{{}_A} = \, {\mathrm{id}}_A \; $.  Note also that it follows by construction that  $ \, A[\varepsilon] \in \Wsalg_\K \, $  again.
\end{free text}

\vskip3pt

\begin{definition} \label{tangent_Lie_superalgebra}
   Given a functor  $ \, G : \Wsalg_\K \! \longrightarrow \grp \, $,  let  $ \; G(p_A) : G (A(\varepsilon)) \longrightarrow G(A) \; $  be the morphism associated with the morphism  $ \; p_A : A[\varepsilon] \relbar\joinrel\longrightarrow A \; $  in  $ \Wsalg_\K \, $.  Then there exists a unique functor  $ \; \Lie(G) : \Wsalg_\K \longrightarrow \set \; $  given on objects by  $ \; \Lie(G)(A) := \Ker\,\big(G(p_A)\big) \; $.
   \hfill  $ \diamondsuit $
\end{definition}

\vskip5pt

   The key fact is that when the functor  $ G $  as above is a (smooth, analytic or holomorphic) Lie supergroup,  then  $ \Lie(G) $  is Lie algebra valued, i.e.\ it is a functor  $ \; \Lie(G) : \Wsalg_\K \longrightarrow \lie_\K \; $.  This requires a non-trivial proof (like in the classical case), for which we refer to  \cite{ccf},  Ch.\ 11 (with the few adaptations needed for the present setup), and only quickly sketch here the main steps.
 \vskip3pt
   The Lie structure on any object  $ \, \Lie(G)(A) \, $  is introduced as follows.  First, define the  {\it adjoint action\/}  of  $ G $  on  $ \Lie(G) $  as given, for every  $ \, A \in \Wsalg_\K \, $,  by
 \vskip4pt
   \centerline{ $ \Ad : G(A)  \longrightarrow  \rGL\big(\Lie(G)(A)\big) \quad ,  \qquad  \Ad(g)(x) \, := \, G(i_A)(g) \cdot x \cdot {\big(G(i_A)(g)\big)}^{-1} $ }
 \vskip5pt
\noindent
 for all  $ \, g \in G(A) \, $,  $ \, x \in \Lie(G)(A) \, $.  Second, define the  {\it adjoint morphism\/}  $ \ad $  as
 \vskip4pt
   \centerline{ $ \ad \, := \, \Lie(\Ad) : \Lie(G) \longrightarrow \Lie(\rGL(\Lie(G))) := \End(\Lie(G)) $ }
\vskip5pt
\noindent
 and finally define  $ \; [x,y] := \ad(x)(y) \; $  for all  $ \, x,y \in \Lie(G)(A) \, $.  Then we have the following:

\medskip

\begin{proposition}  \label{Lie-funct_Lie(G)}
 Given a (smooth, analytic or holomorphic) Lie supergroup  $ G \, $,  let  $ \, \fg := T_e(G) \, $  be the  {\sl tangent}  $ \K $--supermodule  to  $ G $  at the unit point  $ \, e \in G \, $.
 \vskip5pt
   \hskip-7pt   (a)  $ \; \Lie(G) $  with the bracket  $ \, [\,\cdot\, , \cdot\, ] $  above is Lie algebra valued, i.e.\  $ \, \Lie(G) : \Wsalg_\K \! \longrightarrow \lie_\K \; $;
 \vskip3pt
   \hskip-7pt   (b)  $ \; \Lie(G) $  is quasi-representable (see  \S \ref{Lie-salg_funct}),  namely  $ \, \Lie(G) = \cL_\fg \, $,  where  $ \fg $  is endowed with a canonical structure of Lie  $ \K $--superalgebra,  and it is also representable, namely represented by  $ \fg^* \, $;
 \vskip1pt
   \hskip-7pt   (c)  for every  $ \; A \in \Wsalg_\K \, $  one has  $ \, \Lie(G)(A) = \Lie\big(G(A)\big) \, $,  the latter being the tangent Lie algebra to the Lie group  $ G(A) \, $.
\end{proposition}

\vskip5pt

   {\sl Note\/}  that the previous proposition also collects all that was shortly explained in  \S \ref{Lie-salg_funct}  about Lie superalgebras and their functorial presentation/characterization; again, see  \cite{ccf}  for details.

\vskip5pt

   {\sl  $ \underline{\text{N.B.}} $:}\,  due to the previous proposition, in the sequel we shall freely identify the functor  $ \, \Lie(G) = \cL_\fg \, $  with the tangent superspace  $ \fg $   --- now thought of as a Lie superalgebra ---   calling this common object ``the  {\sl tangent Lie superalgebra\/}  of (or ``to'') the Lie supergroup  $ G \, $''.
                                                          \par
   There exist several other realizations of the tangent Lie superalgebra to  $ G \, $,  and canonical identifications among all of them.  We will only occasionally need some of them, so we do not go into further details, but refer instead to the literature, in particular  \cite{cf}  (especially \S 5.2 therein).

\vskip5pt

   Finally, the construction  $ \, G \mapsto \Lie(G) \, $  for Lie supergroups is actually natural, in that any morphism  $ \, \phi : G' \relbar\joinrel\longrightarrow G'' \, $  of Lie supergroups induces a morphism  $ \, \Lie(\phi) : \Lie\big(G'\big) \relbar\joinrel\longrightarrow \Lie(G'') \, $  of Lie superalgebras.  Eventually, all this together provides functors  $ \; \Lie : \Lsgrp_\R^\infty \!\relbar\joinrel\relbar\joinrel\longrightarrow \sLie_\R \; $,  $ \; \Lie : \Lsgrp_\R^\omega \!\relbar\joinrel\relbar\joinrel\longrightarrow \sLie_\R \; $  and  $ \; \Lie : \Lsgrp_\C^\omega \!\relbar\joinrel\relbar\joinrel\longrightarrow \sLie_\C \; $;  see  \cite{ccf}  and  \cite{bcf}  for details.

\vskip15pt

\begin{free text}  \label{Liesgrps-->sHCp's}
 {\bf The super Harish-Chandra pair of a Lie supergroup.}  Gathering together the previous results
we get the core of the present section.  Namely, if  $ G $  is any Lie supergroup then  $ \big( G_\zero \, , \Lie(G) \big) $  is a super Harish-Chandra pair, and this construction is functorial, as the following claims:

\vskip9pt

\begin{theorem}   \label{thm_Lsgrps-->sHCp's}
 {\sl (see for instance  \cite{cf})}  There exist functors
 \vskip3pt
   \centerline{ $ \; \Phi : \Lsgrp_\R^\infty \!\relbar\joinrel\longrightarrow \sHCp_\R^\infty \;\;\; $,  \quad
      $ \; \Phi : \Lsgrp_\R^\omega \relbar\joinrel\longrightarrow \sHCp_\R^\omega \;\;\; $,  \quad
      $ \; \Phi : \Lsgrp_\C^\omega \relbar\joinrel\longrightarrow \sHCp_\C^\omega $ }
 \vskip3pt
\noindent
 that are given on objects by  $ \; G \mapsto \big( G_\zero \, , \Lie(G) \big) \; $  and on morphisms by  $ \; \phi \mapsto \big( \phi_\zero \, , \Lie(\phi) \big) \; $.
\end{theorem}
\end{free text}

\bigskip

\section{Interlude: special splittings for Lie supergroups}  \label{interlude}

\smallskip

%
%    {\ } \;\;   In this section we present some results concerning the possibility
% to split Lie supergroups in some (more or less canonical) ``special'' ways: Boseck's
% splittings and global splittings.  These results are essentially well-known, yet
% possibly formulated in different ways; however, we provide independent proofs for
% them, so to fill any possible gap in literature and mostly to have a self-contained
% presentation, dealing with Lie supergroups of all types (smooth, analytic and holomorphic).
%
   {\ } \;\;   In this section we present some results on the possibility to split Lie supergroups in two special ways: ``Boseck's splittings'' and ``global splittings''.  These are essentially well-known, but usually stated in different ways; so we provide independent proofs for them, so to fill any possible gap in literature and to have a self-contained presentation
  \hbox{(for all cases: smooth, analytic and holomorphic).}

\medskip

\subsection{Boseck's splitting for Lie supergroups}  \label{Boseck-split}

\smallskip

   {\ } \;\;   This subsection is devoted to a first kind of splitting that we refer to as ``Boseck's splitting'', as it was first mentioned in Boseck's work  \cite{bos}. The starting point is the following easy result:

\vskip11pt

\begin{lemma}  \label{pre_Boseck-split}
 Let  $ \, p : A' \relbar\joinrel\longrightarrow A'' \, $  and  $ \, u : A'' \relbar\joinrel\longrightarrow A' \, $  be morphisms in  $ \Wsalg_\K $  such that  $ \, p \circ u = \text{\sl id}_{A''} \, $  (hence  $ p_A $  is surjective and  $ u_A $  injective), and let  $ \, G : \Wsalg_\K \longrightarrow \grp \, $  be any functor.  Then  $ G(A) $  canonically splits into a semi-direct product
  $$  G(A)  \,\; = \;\,  \text{\sl Im}\big(G(u)\big) \ltimes \text{\sl Ker}\big(G(p)\big)  \,\; \cong \;\,  G\big(A''\big) \ltimes \text{\sl Ker}\big(G(p)\big)  $$
\end{lemma}

\begin{proof}
 From  $ \; p \circ u = \text{\sl id}_{A''} \; $  we get  $ \; G(p) \circ G(u) = G(p \circ u) = G\big(\text{\sl id}_{A''}\big) = \text{\sl id}_{G(A'')} \; $;  thus  $ G(u) \, $,  resp.\  $ G(p) \, $,  is a section of  $ G(p) \, $,  resp.\  a retraction of  $ G(u) \, $,  in the category of groups: in particular,  $ G(u) $ is injective and  $ G(p) $  surjective.  The claim then follows by standard group-theoretic arguments.
\end{proof}

\medskip

%
%    When the group-valued functor  $ G $  is in fact a  {\sl Lie supergroup},  the previous
% result yields an interesting outcome, as follows:
%
   When the functor  $ G $  is in fact a  {\sl Lie supergroup},  we have the following, interesting outcome:

\medskip

\begin{proposition}  \label{prop:Boseck-split}
 {\sl (cf.\ \cite{bos}, \S 2, Proposition 7)}
 \vskip1pt
   Let  $ G $  be any (smooth, analytic or holomorphic) Lie supergroup over\/  $ \K \, $.  Then for every Weil superalgebra  $ \, A \in \Wsalg_\K \, $  there exists a canonical splitting of Lie groups
\begin{equation}  \label{Boseck-split_Lie-sgrp}
  G(A)  \,\; \cong \;\,  G_\zero(\K) \ltimes N_G(A)
\end{equation}
where  $ \, G(\K) = G_\zero(\K) \, $  is nothing but the classical, ordinary Lie group underlying  $ G $  (i.e., the Lie group of\/  $ \K $--points  of\/  $ G_\zero $)  and  $ \; N_G(A) := \text{\sl Ker}\big(G\big(p_{\scriptscriptstyle A\!}\big)\big) \; $
%
% where  $ \, p_{\scriptscriptstyle A} : A \relbar\joinrel\longrightarrow \K \, $
% is the built-in projection of  $ A $  onto  $ \K $  (cf.\ Definition
%
 with  $ \, p_{\scriptscriptstyle A} : A \relbar\joinrel\longrightarrow \K \, $  as in  Definition \ref{def:Weil-salg}.
\end{proposition}

\begin{proof}
 By assumption, for the given  $ \, A \in \Wsalg_\K \, $  we have morphisms  $ \, p_{\scriptscriptstyle A} : A \relbar\joinrel\longrightarrow \K \, $  and  $ \, u_{\scriptscriptstyle A} : \K \relbar\joinrel\longrightarrow A \, $  in  $ \Wsalg_\K $  such that  $ \, p_{\scriptscriptstyle A} \circ u_{\scriptscriptstyle A} = \text{\sl id}_\K \, $  (cf.\ Definition \ref{def:Weil-salg}).  Then  Lemma \ref{pre_Boseck-split}
%
% applies, thus yielding
%
 yields
a group-theoretic splitting  $ \; G(A) \, = \, \text{\sl Im}\big(G\big(u_{\scriptscriptstyle A\!}\big)\big) \ltimes \text{\sl Ker}\big(G\big(p_{\scriptscriptstyle A\!}\big)\big) \; $.
%
% Now, as
%
 As
 $ G $  takes values into the category of Lie groups, this is actually a splitting
%
% inside that category   --- i.e., it is a splitting
%
 of Lie groups (actually, even one of  $ \cA_\zero $--Lie  groups indeed).  Moreover, we clearly have  $ \; \text{\sl Im}\big(G\big(u_{\scriptscriptstyle A\!}\big)\big) \cong G(\K) = G_\zero(\K) \; $,  whence  \eqref{Boseck-split_Lie-sgrp}  is proved.
\end{proof}

\medskip

%
%    It is worth stressing that this result also has a specific application for the case of
% {\sl totally even\/}  supergroups   --- in other words, classical Lie groups---   as follows:
%
   This result also applies to  {\sl totally even\/}  supergroups (i.e., classical Lie groups) as follows:

\medskip

\begin{proposition}  \label{prop:Boseck-split_class}
 \vskip1pt
   Let  $ G_+ $  be any (smooth, analytic or holomorphic) Lie group over\/  $ \K \, $.  Then for every  $ \, A_+ \in \Wsalg_\K \,\bigcap \alg_\K \, $  there exists a canonical splitting of Lie groups
\begin{equation}  \label{Boseck-split_Lie-grp}
  G_+\big(A_+\big)  \,\; \cong \;\,  G_+(\K) \ltimes N_{G_+}\big(A_+\big)
\end{equation}
where  $ \, G_+(\K) \, $  is the ordinary Lie group underlying  $ G $  (i.e., the Lie group of\/  $ \K $--points  of\/  $ G_+ \, $)  and  $ \; N_{G_+}\big(A_+\big) := \text{\sl Ker}\big(G\big(p_{\scriptscriptstyle {A_+}\!}\big)\big) \; $ with  $ \, p_{\scriptscriptstyle {A_+}} \! : A_+ \!\relbar\joinrel\longrightarrow \K \, $  as in  Definition \ref{def:Weil-salg}.
\end{proposition}

\begin{proof}
 The same arguments as for the proof of  Proposition \ref{prop:Boseck-split}  apply again.
\end{proof}

\medskip

\begin{remark}
 To the best of the author's knowledge, the (canonical) splitting  \eqref{Boseck-split_Lie-sgrp}  of  $ G(A) $  was first mentioned by Boseck (dealing with Lie supergroups defined over  $ \Grass_\K \, $,  but the idea is the same):  cf.\ \cite{bos}, \S 2, Proposition 7; thus we shall refer to  \eqref{Boseck-split_Lie-sgrp}  or  \eqref{Boseck-split_Lie-grp}  as to ``Boseck's splitting(s)''.  The same result was considered by other authors too, e.g.\ Molotkov: see (7.4.1) in \S 7.4 of  \cite{mol}.
\end{remark}

\smallskip

\begin{free text}
 {\bf Boseck's splitting for Lie superalgebras.}  The notion of ``Boseck's splitting''
% introduced above
 for Lie supergroups has a natural counterpart for Lie superalgebras, when thought of as
% Lie algebra valued
 functors.

  Indeed, consider a Lie  $ \K $--superalgebra  $ \, \fg = \fg_\zero \oplus \fg_\uno \, $  and the functor
 $ \; \cL_\fg : \salg_\K \longrightarrow \lie_\K \; $
given by
 $ \; \cL_\fg(A) := \big( A \otimes \fg \big)_\zero = (A_\zero \otimes \fg_\zero) \oplus (A_\uno \otimes \fg_\uno) \; $,  for all  $ \, A \in \salg_\K \, $,  as in  \S \ref{Lie-salg_funct}.
 Every  $ \, A \in \Wsalg_\K \, $  has built-in morphisms  $ \, p_{\scriptscriptstyle A} : A \relbar\joinrel\longrightarrow \K \, $  and  $ \, u_{\scriptscriptstyle A} : \K \relbar\joinrel\longrightarrow A \, $  such that  $ \, p_{\scriptscriptstyle A} \circ u_{\scriptscriptstyle A} = \text{\sl id}_\K \, $.  Then applying  $ \cL_\fg $  we get  $ \; \cL_\fg\big(p_{\scriptscriptstyle A}\big) \circ \cL_\fg\big(u_{\scriptscriptstyle A}\big) =  \cL_\fg\big( p_{\scriptscriptstyle A} \circ u_{\scriptscriptstyle A} \big) = \cL_\fg\big( \text{\sl id}_\K \big) = \text{\sl id}_{\cL_\fg(\K)} \; $,  a relation regarding morphisms of Lie algebras.  By standard arguments this yields a Lie algebra splitting
\begin{equation}  \label{pre-Boseck-split_Lie-salg}
  \cL_\fg(A)  \; = \;  \text{\sl Im}\big( \cL_\fg\big(u_{\scriptscriptstyle A\!}\big) \big) \,\oright\, \text{\sl Ker}\big( \cL_\fg\big(p_{\scriptscriptstyle A\!}\big) \big)
\end{equation}
where the symbol  ``$ \, \oright \, $''  denotes the (internal) semi-direct sum of  $ \text{\sl Im}\big( \cL_\fg\big(u_{\scriptscriptstyle A\!}\big) \big) $   --- a Lie subalgebra inside  $ \cL_\fg(A) $  ---   with $ \text{\sl Ker}\big( \cL_\fg\big(p_{\scriptscriptstyle A\!}\big) \big) \, $   --- a Lie ideal in  $ \cL_\fg(A) \, $.  Now, on the one hand definitions give
 $ \; \text{\sl Im}\big( \cL_\fg\big(u_{\scriptscriptstyle A\!}\big) \big)  \, \cong \,
 \cL_\fg(\K) := {\big( \K \otimes_\K \fg \big)}_\zero = \fg_\zero \; $;
 on the other hand, to simplify a bit the notation we write  $ \; \fn_\fg(A) \, := \, \text{\sl Ker}\big( \cL_\fg\big(p_{\scriptscriptstyle A\!}\big) \big) \; $.  Then  \eqref{pre-Boseck-split_Lie-salg}  reads also
\begin{equation}  \label{Boseck-split_Lie-salg}
  \qquad \qquad \qquad   \cL_\fg(A)  \; = \;  \fg_\zero \,\oright\, \fn_\fg(A)   \qquad \qquad \qquad  \forall \;\; A \in \Wsalg_\K
\end{equation}
In the following, we shall refer to  \eqref{Boseck-split_Lie-salg}  as to ``Boseck's splitting for  $ \cL_\fg $''   --- or simply ``for  $ \fg \, $''  itself.
 \vskip5pt
   It is still worth remarking that one has a non-trivial Boseck's splitting also when  $ \, \fg = \fg_\zero \, $,  i.e.\  $ \fg $  is a classical (=totally even) Lie algebra.  Indeed, if  $ \fg_+ $  is just a Lie  $ \K $--algebra then  \eqref{Boseck-split_Lie-salg}  reads
\begin{equation}  \label{Boseck-split_Lie-alg}
  \qquad \qquad \qquad   \cL_{\fg_+\!}\big(A_+\big)  \; = \;  \fg_+ \oright\, \fn_{\fg_+}\!\big(A_+\big)   \qquad \qquad \qquad  \forall \;\; A_+ \in \Wsalg_\K \;{\textstyle \bigcap}\, \alg_\K
\end{equation}
and will again be called ``Boseck's splitting for  $ \cL_{\fg_+} $''   --- or ``for  $ \fg_+ $''.
 \vskip5pt
   We will now give an explicit description of  $ \fn_\fg(A) \, $.  By definition,  $ \; \fn_\fg(A) := \text{\sl Ker}\big( \cL_\fg\big(p_{\scriptscriptstyle A\!}\big) \big) \; $  where  $ \; p_{\scriptscriptstyle A\!} : A = \K \oplus \nil{A} \relbar\joinrel\relbar\joinrel\twoheadrightarrow \K \; $  is the canonical projection of  $ \, A = \K \oplus \nil{A} \, $  onto its left-hand side summand.  Now,  $ \, A = A_\zero \oplus A_\uno \, $  with  $ \, A_\zero = \K \oplus \nil{A}_\zero \, $  and  $ \, A_\uno = \nil{A}_\uno \, $, hence
  $$  \cL_\fg(A)  \, := \,  {\big( A \otimes_\K \fg \big)}_\zero  \, = \,  \big( A_\zero \otimes_\K \fg_\zero \big) \oplus \big( A_\uno \otimes_\K \fg_\uno \big)  \, = \;  \fg_\zero \oplus \big( \nil{A}_\zero \otimes_\K \fg_\zero \big) \oplus \big( A_\uno \otimes_\K \fg_\uno \big)  $$
\noindent
 from which it clearly follows that
\begin{equation}  \label{descr_ng(A)}
 \fn_\fg(A)  \; := \;  \text{\sl Ker}\big( \cL_\fg\big(p_{\scriptscriptstyle A\!}\big) \big)  \; = \;  \big( \nil{A}_\zero \otimes_\K \fg_\zero \big) \oplus \big( \nil{A}_\uno \otimes_\K \fg_\uno \big)   \qquad  \forall \;\; A \in \Wsalg_\K
\end{equation}
%
% (by comparison with  \eqref{Boseck-split_Lie-salg},  or even directly by definitions).
%
 In the case when  $ \, \fg = \fg_\zero \, $  is just a ``classical'' Lie algebra, say  $ \fg_+ \, $,  this reads slightly simpler, namely
\begin{equation}  \label{descr_ng(A)_class}
 \fn_{\fg_+}\!\big(A_+\big)  \; = \;  \nil{A}_+ \otimes_\K \fg_+   \qquad  \forall \;\; A_+ \in \Wsalg_\K \,{\textstyle \bigcap}\, \alg_\K
\end{equation}

This entails the following:

\vskip13pt

\begin{proposition}  \label{ng(A)_nilpotent}  {\ }
 \vskip3pt
   (a)\,  Let  $ \fg $  be a Lie  $ \K $--superalgebra  and  $ \, A \in \Wsalg_\K \, $.  Then the Lie algebra\/  $ \fn_\fg(A) $  is nilpotent.
 \vskip3pt
   (b)\,  Let  $ \fg_+ $  be a Lie  $ \K $--algebra  and let  $ \, A_+ \in \Wsalg_\K \bigcap \alg_\K \, $.  Then\/  $ \fn_{\fg_+}\!\big(A_+\big) $  is nilpotent.
\end{proposition}

\begin{proof}
 Claim  {\it (a)\/}  follows at once from  \eqref{descr_ng(A)}  and the fact that  $ \nil{A} $  is nilpotent, and likewise for  {\it (b)}.
\end{proof}

\end{free text}

\medskip

\begin{free text}  \label{interplay_Boseck-splitt.'s}
 {\bf The interplay of Boseck's splittings for supergroups and superalgebras.}  Let again  $ G $  be a Lie supergroup over  $ \K \, $,  and  $ \, \fg := \Lie(G) \, $  be its tangent Lie superalgebra.  For any  $ \, A \in \Wsalg_\K \, $,  the Lie group  $ G(A) $  and the Lie algebra  $ \, \fg(A) := \cL_\fg(A) \, $   --- also equal to  $ \, \Lie(G)(A) = \Lie\big(G(A)\big) \, $,  cf.\  Proposition \ref{Lie-funct_Lie(G)}{\it (c)}  ---   are linked by the exponential map
%
%   $$  \exp : \fg(A) \relbar\joinrel\relbar\joinrel\longrightarrow G(A)  $$
 $ \; \exp : \fg(A) \relbar\joinrel\longrightarrow G(A) \; $
 which is a local isomorphism (either in the smooth, analytic or holomorphic sense, as usual).  Similarly for the ``even counterparts'' we have also the local isomorphism
%
%   $$  \exp_\zero : \fg_\zero \relbar\joinrel\relbar\joinrel\longrightarrow G_\zero(\K)  $$
%
 $ \; \exp_\zero : \fg_\zero \relbar\joinrel\longrightarrow G_\zero(\K) \; $
 with  $ \, \exp_\zero = \exp{\big|}_{\fg_\zero} \, $  if we think at  $ \fg_\zero $  as embedded into  $ \; \fg(A) := \cL_\fg(A) = \fg_\zero \oright \fn_\fg(A) \; $   --- cf.\  \eqref{Boseck-split_Lie-salg}.
                                                      \par
   Now, since the Lie algebra  $ \fn_\fg(A) $  is nilpotent   --- cf.\  Proposition \ref{ng(A)_nilpotent}  ---   its image  $ \, \exp\!\big(\fn_\fg(A)\big) \, $  for the exponential map is a (closed, connected) nilpotent Lie subgroup of  $ G(A) \, $.  Furthermore, let us use notation  $ \, \fg\big(p_{\scriptscriptstyle A}\big) := \cL_\fg\big(p_{\scriptscriptstyle A}\big) \, $  and  $ \, \fg\big(u_{\scriptscriptstyle A}\big) := \cL_\fg\big(u_{\scriptscriptstyle A}\big) \, $,  and consider the diagram
%
%%%
%
%   $$  \xymatrix{
%    \fg(A) \;  \ar[d]_\exp \ar@/^/[rrr]^{\fg(p_{\scriptscriptstyle A})}  &  &  &
% \ar[lll]^{\fg(u_{\scriptscriptstyle A})} \ar[d]^{\exp_\zero}  \; \fg_\zero  \\
%    G(A) \,  \ar[rrr]^{G(p_{\scriptscriptstyle A})}  &  &  &
% \ar@/^/[lll]^{G(u_{\scriptscriptstyle A})}  \, G_\zero(\K)  }  $$
%
%%%
%
% \begin{equation}  \label{CD_x_Bos.splitt.'s}
%
  $$
 \xymatrix{
   \fg(A) \;
 \ar[rrr]^\exp \ar[dd]^{\fg(p_{\scriptscriptstyle A\!})}
  &  &  &  \ar[dd]_{G(p_{\scriptscriptstyle A})}  \; G(A)  \\
   &  &  &  \\
   \fg_\zero \,
 \ar[rrr]_{\exp_\zero^{\phantom{\circ}}} \ar@/^/[uu]^{\fg(u_{\scriptscriptstyle A\!})}
  &  &  &  \ar@/_/[uu]_{G(u_{\scriptscriptstyle A})}  \, G_\zero(\K)  }
   $$
%
% \end{equation}
%
This diagram is commutative, hence in particular  $ \; G\big(p_{\scriptscriptstyle A}\big) \,\circ\, \exp \, = \, \exp_\zero \,\circ\, \fg\big(p_{\scriptscriptstyle A}\big) \, $,  which in turn implies at once
 $ \,\; G\big(p_{\scriptscriptstyle A}\big)\Big( \exp\!\big(\fn_\fg(A)\big) \!\Big) \, = \, \exp_\zero\!\Big( \fg\big(p_{\scriptscriptstyle A}\big)\big(\fn_\fg(A)\big) \!\Big) \, = \, \exp_\zero\!\big( \big\{0_{\fg_\zero}\big\} \big) \, = \, \big\{ 1_{G_\zero(\K)}\big\} \; $
 because  $ \, \fn_\fg(A) := \Ker\big(\fg\big(p_{\scriptscriptstyle A}\big)\big) \, $;  so in the end  $ \; \exp\!\big(\fn_\fg(A)\big) \, \subseteq \, \Ker\big(G\big(p_{\scriptscriptstyle A}\big)\big) \, =: N_G(A) \; $   --- cf.\  Proposition \ref{Boseck-split}.
 \vskip5pt
  The fact that  $ \; \exp : \fg(A) \relbar\joinrel\longrightarrow G(A) \; $  is a local isomorphism, together with Boseck's splittings   --- namely,  $ \, \fg(A) = \fg_\zero \oright \fn_\fg(A) \, $  and  $ \, G(A) = G_\zero(\K) \ltimes N_G(A) \, $  ---   and  $ \, \dim\!\big(\fg_\zero\big) = \dim\big(G_\zero(\K)\big) \, $,  jointly imply  $ \, \dim\!\big(\fn_\fg(A)\big) = \dim\!\big(N_G(A)\big) \, $.  On the other hand, as  $ \fn_\fg(A) $  is nilpotent, its exponential map   --- i.e., just the restriction to  $ \fn_\fg(A) $  of  $ \, \exp : \fg(A) \longrightarrow G(A) \, $  ---   is actually a {\sl global isomorphism\/}  of  $ \K $--manifolds  from  $ \fn_\fg(A) $  to  $ \exp\!\big(\fn_\fg(A)\big) \, $.  It then follows that  $ \, \dim\!\big(\exp\!\big(\fn_\fg(A)\big)\big) = \dim\!\big(N_G(A)\big) \, $.
                                                            \par
   Let now  $ {N_G(A)}^\circ $  be the connected component of  $ N_G(A) \, $,  so  $ \, \dim\!\big({N_G(A)}^\circ\big) = \dim\!\big(N_G(A)\big) \, $;  our previous analysis yields  $ \, \exp\!\big(\fn_\fg(A)\big) \subseteq {N_G(A)}^\circ \, $,  the former being a closed Lie subgroup of
the latter.  But  $ \, \dim\!\big(\exp\!\big(\fn_\fg(A)\big)\big) = \dim\!\big({N_G(A)}^\circ\big) \, $  too, so eventually we get  $ \; \exp\!\big(\fn_\fg(A)\big) = {N_G(A)}^\circ \; $.
                                                            \par
   We shall now analyze  $ \; \exp\!\big(\fn_\fg(A)\big) = {N_G(A)}^\circ \, $,  eventually proving that it coincides with  $ N_G(A) \, $.
\end{free text}

\smallskip

\begin{free text}  \label{1st-descr_N_G(A)}
 {\bf The Lie subgroup  $ N_G(A) \, $.}  Let  $ G $  be a Lie supergroup, as before.  As we saw in  \S \ref{funct-pt-view},  for any  $ \, A \in \Wsalg_\K \, $  the group  $ G(A) $  is defined as
%
%   $$  G(A)  \, := \,  G_A  \, = \,  {\textstyle \bigsqcup\limits_{g \in |G|}} G_{A,g}  $$
 $ \; G(A) := G_A = \bigsqcup_{g \in |G|} G_{A,g} \;$
where  $ \, |G| \, $  is the underlying topological space of  $ G $  and  $ \; G_{A,g} := \Hom_{\salg_\K}\big(\, \cO_{G,g} \, , A \,\big) \; $,  with  $ \, \cO_{G,g} \, $  being the stalk (a local superalgebra) of the structure sheaf of  $ G $  at the point  $ \, g \in |G| \, $.  We adopt the canonical identification  $ \, |G| = G(\K) \, $  via  $ \, g \mapsto \text{\sl ev}_g \, $  with  $ \, \text{\sl ev}_g : \cO_{G,g} \!\relbar\joinrel\longrightarrow \K \, $  given by  $ \, f \mapsto \text{\sl ev}_g(f\,) := f(g) \, $.
                                                            \par
   For every  $ \, g_{\scriptscriptstyle A} \in G_{A,g} \, $  we have  $ \; \widetilde{g}_{\scriptscriptstyle A} := p_A \circ g_{\scriptscriptstyle A} \; $  (cf.\  Definition \ref{def_A-pts})  which coincides with  $ \, \text{\sl ev}_g \; $;  moreover, the very definition gives also  $ \; \widetilde{g}_{\scriptscriptstyle A} := p_A \circ g_{\scriptscriptstyle A} = G\big(p_{\scriptscriptstyle A}\big)(g_{\scriptscriptstyle A}) \; $,  in short  $ \; \widetilde{g}_{\scriptscriptstyle A} = \, G\big(p_{\scriptscriptstyle A}\big)(g_{\scriptscriptstyle A}) \; $.  Finally, due to the splitting  $ \, A = \K \oplus \nil{A} \, $,  for every  $ \, g_{\scriptscriptstyle A} \in G(A) \, $,  say  $ \, g_{\scriptscriptstyle A} \in G_{A,g} \, $,  there exists also a unique map  $ \, \widehat{g}_{\scriptscriptstyle A} : \cO_{G,g} \relbar\joinrel\longrightarrow \nil{A} \, $  such that  $ \; g_{\scriptscriptstyle A} = \widetilde{g}_{\scriptscriptstyle A} + \widehat{g}_{\scriptscriptstyle A} \; $.

\vskip5pt

   Now assume  $ \; g_{\scriptscriptstyle A} \in N_G(A) := \Ker\big( G\big(p_{\scriptscriptstyle A}\big) \big) \; $.  Then  $ \; G\big(p_{\scriptscriptstyle A}\big)(g_{\scriptscriptstyle A}) = 1_{\scriptscriptstyle G_A} \, \in \, G(A) \; $;  therefore   --- by the previous analysis ---   we have  $ \, \widetilde{g}_{\scriptscriptstyle A} = 1 \, $,  whence  $ \; g_{\scriptscriptstyle A} = 1 + \widehat{g}_{\scriptscriptstyle A} \; $   --- which can be read as the sum, in the natural sense, of maps from  $ \cO_{G,1} $ to  $ A \, $.  We can re-write our  $ g_{\scriptscriptstyle A} $  as
\begin{equation}  \label{g_a as exp(X_g)}
 g_{\scriptscriptstyle A}  = \, 1 + \widehat{g}_{\scriptscriptstyle A}  = \, \exp\!\big( X_{g_{\!{}_A}} \big)   \qquad  \text{with}  \qquad   X_{g_{\!{}_A}}  := \, \log\!\big( g_{\scriptscriptstyle A} \big)  = \,  {\textstyle \sum\limits_{n=1}^{+\infty}} {(-1)}^{n+1} \, {{\;\widehat{g}_{\scriptscriptstyle A}^{\;n}\,} \over {\,n\,}}
\end{equation}
%
% where  $ \; \exp\!\big( X_{g_{\!{}_A}} \big) := \sum\limits_{n=0}^{+\infty}
% X_{g_{\!{}_A}}^{\,n} \Big/ n! \; $ and all powers (and products) involved in
% these formulas are simply given by  $ \, X_{g_{\!{}_A}}^{\,n}(f) := {\big(
% X_{g_{\!{}_A}}(f) \big)}^n \, $,  $ \; \widehat{g}_{\scriptscriptstyle A}^{\;n}(f)
% := {\big( \widehat{g}_{\scriptscriptstyle A}(f) \big)}^n \, $,  etc.  Note that all this
% does make sense because  $ \, \text{\sl Im}\big(\widehat{g}_{\scriptscriptstyle A}\big)
% \in \nil{A} \, $,  by construction; thus  $ \widehat{g}_{\scriptscriptstyle A} $  itself
% is nilpotent, hence  $ X_{g_{\!{}_A}} $  is given by a finite sum and it is in turn
% nilpotent, so the formal series defining  $ \exp\!\big(X_{g_{\!{}_A}}\big) $  is a
% finite sum as well.
%
 where  $ \, \exp\!\big( X_{g_{\!{}_A}} \big) := \! \sum\limits_{n=0}^{+\infty} X_{g_{\!{}_A}}^{\,n} \Big/ n! \, $ and all powers in these formulas are given by  $ \, X_{g_{\!{}_A}}^{\,n}(f) := {\big( X_{g_{\!{}_A}}\!(f) \big)}^n $,  $ \; \widehat{g}_{\scriptscriptstyle A}^{\;n}(f) := {\big( \widehat{g}_{\scriptscriptstyle A}(f) \big)}^n \, $,  etc.  All this makes sense because  $ \, \text{\sl Im}\big(\widehat{g}_{\scriptscriptstyle A}\big) \in \nil{A} \, $,  by construction; thus  $ \widehat{g}_{\scriptscriptstyle A} $  is nilpotent, hence  $ X_{g_{\!{}_A}} $  is given by a finite sum and it is nilpotent, so  $ \exp\!\big(X_{g_{\!{}_A}}\big) $  is a finite sum too.
 \vskip3pt
   By formal properties of exponential and logarithm, since  $ \, g_{\scriptscriptstyle A} : \cO_{G,1} \longrightarrow A \, $  is a (superalgebra) morphism it follows from  \eqref{g_a as exp(X_g)}  that  $ \, X_{g_{\!{}_A}} \! : \cO_{G,1} \longrightarrow A \, $  is in turn a (superalgebra) derivation; thus   --- cf.\ \S Proposition \ref{Lie-funct_Lie(G)}{\it (c)}  ---   we have  $ \; X_{g_{\!{}_A}} \! \in \Lie\big(G(A)\big) = \big(\Lie(G)\big)(A) = \cL_\fg(A) =: \fg(A) \; $.  Finally, by construction we have also  $ \, \text{\sl Im}\big( X_{g_{\!{}_A}} \big) \in \nil{A} \, $.  Along with Boseck's splitting  $ \, \fg(A) = \fg_\zero \oright \fn_\fg(A) \, $   --- see  \eqref{Boseck-split_Lie-salg}  for  $ \, \fg(A) := \cL_\fg(A) \, $  ---   and with  $ \; \fn_\fg(A) = \big( \nil{A}_\zero \otimes_\K \fg_\zero \big) \oplus \big( \nil{A}_\uno \otimes_\K \fg_\uno \big) \; $   --- as in  \eqref{descr_ng(A)}  ---   all this together eventually leads to  $ \; X_{g_{\!{}_A}} \!\in \fn_\fg(A) \; $.  Tiding everything up, we come now to the end:
\end{free text}

\vskip7pt

%
% \begin{proposition}
%  For every Lie supergroup  $ G $  and every  $ \, A \in \Wsalg_\K \, $  we have
%   $$  N_G(A)  \, = \,  \exp\!\big(\fn_\fg(A)\big)  $$
% %
% In particular,  $ N_G(A) $  is connected nilpotent, and (globally) isomorphic,
% as a manifold, to\/  $ \fn_\fg(A) \, $.
% %
% \end{proposition}
%

\begin{proposition}
 For any Lie supergroup  $ G $  and  $ \, A \in \Wsalg_\K \, $  we have  $ \, N_G(A) = \exp\!\big(\fn_\fg(A)\big) \, $.  In particular,  $ N_G(A) $  is connected nilpotent, and (globally) isomorphic, as a manifold, to\/  $ \fn_\fg(A) \, $.
\end{proposition}

\begin{proof}
%
%  Indeed, our analysis above shows that each  $ \, g_{\scriptscriptstyle A} \in N_G(A) \, $
% can be realized as  $ \; g_{\scriptscriptstyle A} = \exp\!\big( X_{g_{\!{}_A}} \big) \; $
% with  $ \, X_{g_{\!{}_A}} \!\in \fn_\fg(A) \, $;  hence  $ \; N_G(A) \subseteq \exp\!\big(
% \fn_\fg(A) \big) \; $.  On the other hand, from  \S \ref{interplay_Boseck-splitt.'s}
% we have also  $ \; \exp\!\big(\fn_\fg(A)\big) = {N_G(A)}^\circ \subseteq N_G(A) \; $.
% Thus  $ \; N_G(A) = \exp\!\big(\fn_\fg(A)\big) \; $  as claimed.
%                                                         \par
%    The last part of the claim then is clear.
%
 Our analysis above shows that each  $ \, g_{\scriptscriptstyle A} \in N_G(A) \, $  can be realized as  $ \, g_{\scriptscriptstyle A} = \exp\!\big( X_{g_{\!{}_A}} \big) \, $  with  $ \, X_{g_{\!{}_A}} \!\in \fn_\fg(A) \, $;  hence  $ \; N_G(A) \subseteq \exp\!\big(\fn_\fg(A)\big) \; $;  \, conversely,  \S \ref{interplay_Boseck-splitt.'s}  yields  $ \, \exp\!\big(\fn_\fg(A)\big) = {N_G(A)}^\circ \subseteq N_G(A) \; $.  Thus  $ \; N_G(A) = \exp\!\big(\fn_\fg(A)\big) \; $  as claimed.  The last part of the claim then is clear.
\end{proof}

\medskip

\subsection{Global splittings for Lie supergroups}  \label{global-split's}

\smallskip

%
%    {\ } \;\;   This subsection is devoted to find yet another remarkable splitting
% for the groups  $ G(A) $  of  $ A $--points  of any Lie supergroup  $ G \, $;  this
% is no longer canonical   --- as it depends on some choice to be made ---   but on the
% other hand it is natural in  $ A \, $,  hence overall means that one has a noteworthy
% splitting for  $ G $  itself as a functor, usually known as ``global splitting'' of
% $ G \, $.  However, such a result is often stated in a form which is not as ``geometric''
% as we wished   --- typically, as a splitting of the structure sheaf  (cf.\ for instance:
% \cite{be}, Ch.\ 2 \S 2; \cite{mol}, \S 7.4; \cite{vis}, \S 2) ---   so we prefer to
% recover such a result once more, stating it in an explicitly geometrical form.
%
   {\ } \;\;   This subsection is devoted to finding two remarkable splittings for the groups  $ G(A) $  of  $ A $--points  of any Lie supergroup  $ G \, $;  these are natural in  $ A \, $,  hence overall they give noteworthy splittings for  $ G $  as a functor, known as ``global splittings'' of  $ G \, $.  Such a result is often stated in a form which is not as ``geometric'' as we wished (typically, as a splitting of the structure sheaf   ---   cf.\ for instance:  \cite{be}, Ch.\ 2 \S 2; \cite{mol}, \S 7.4; \cite{vis}, \S 2) so we provide an independent proof, with a geometrical statement.
 \vskip5pt
%
%    The inspiring idea underneath the whole construction is that we look for a splitting
% of the form  $ \; G(A) \, = \, G_\zero(A) \times G_-^<(A) \; $  which has to be, somehow,
% a group-theoretic counterpart of the splitting of  $ \, \fg(A) := \cL_\fg(A) \, $  into
% $ \; \fg(A) \, = \, \big( A_\zero \otimes_\K \fg_\zero \big) \oplus \big( A_\uno \otimes_\K
% \fg_\uno \big) \; $.  Indeed, we shall achieve such a goal relying upon Boseck's splitting
% considered in  \S \ref{Boseck-split}  above.
%
   The inspiring idea is that we look for a splitting of the form  ``$ \; G(A) \, = \, G_\zero(A) \times G_\uno(A) \; $''  which has to be, somehow, a group-theoretic counterpart of the splitting  $ \; \fg(A) \, = \, \big( A_\zero \otimes_\K \fg_\zero \big) \oplus \big( A_\uno \otimes_\K \fg_\uno \big) \; $.  Indeed, we will achieve such a goal, in two versions, relying upon Boseck's splitting of  \S \ref{Boseck-split}  above.

\medskip

\begin{free text}  \label{twds_glob-split}
 {\bf Structure theorem and global splittings for Lie supergroups.}  As above let  $ G $  be a (smooth, analytic or holomorphic) Lie supergroup over  $ \K \, $,  whose tangent Lie superalgebra is  $ \, \fg := \Lie(G) \, $,  and let  $ \, A \in \Wsalg_\K \, $  be any Weil superalgebra.
%
%  If we consider the powers  $ \nil{A}^d $  of the nilradical  $ \nil{A} $  of  $ A \, $,
% they form a descending sequence that, by assumption (see  Definition \ref{def:Weil-salg}),
% one has $ \, \nil{A}^N = 0 \, $  for  $ \, N \gg 0 \, $.  Associated with this, we can
% consider the sequence
%
 The powers  $ \nil{A}^d $  of the nilradical  $ \nil{A} $  of  $ A $  form a descending sequence such that $ \, \nil{A}^N \! = 0 \, $  for  $ \, N \gg 0 \, $  (cf.\ Definition \ref{def:Weil-salg}).
%%%%%
   \hbox{Then we consider}
%%%%%
%%%%%
%
  $$  \fn_\fg^{\!(d)}\!(A)  \, := \,  {\Big( \nil{A}^d \otimes_\K \fg \Big)}_\zero  \, = \,  \Big(\! {\big( \nil{A}^d \,\big)}_\zero \otimes_\K \fg_\zero \Big) \oplus \Big(\! {\big( \nil{A}^d \,\big)}_\uno \otimes_\K \fg_\uno \Big)   \eqno  \forall \;\; d \in \N_+  \qquad  $$
this in turn yields a decreasing filtration of Lie subalgebras of  $ \fn_\fg(A) \, $,  with  $ \, \fn_\fg^{\!(N)}\!(A) = 0 \, $  for  $ \, N \gg 0 \, $.
 \vskip5pt
   As a matter of notation, let us consider the case of an element  $ \, \eta \, Y := \eta \otimes Y \in \fn_\fg(A) \, $  with  $ \, \eta \in A_\uno \, $,  $ \, Y \in \fg_\uno \, $.  By definition  $ \, \eta^2 = 0 \, $,  hence if we express  $ \, \exp(\eta\,Y) \, $  as a formal series we actually have  $ \; \exp(\eta\,Y) = 1 + \eta\,Y \; $.  Similarly, for every  $ \, c\,X = c \otimes X \in \fn_\fg(A) \, $  with  $ \, c \in A_\zero \, $,  $ \, X \in \fg_\zero \, $  if  $ \, c^2 = 0 \, $  then also the formal series expression of  $ \, \exp(c\,X) \, $  reads  $ \; \exp(c\,X) = 1 + c\,X \; $.
%%%%%
 %%%%%
   \eject
 %%%%%
%%%%%
%
 \vskip5pt
   For later use, we fix a  $ \K $--basis  $ B $  of $ \fg $  of the form  $ \, B := B_\zero \bigsqcup B_\uno \, $  with  $ \, B_\zero = {\big\{ X_j \big\}}_{j \in J} \, $,  resp.\  $ \, B_\uno = {\big\{ Y_i \big\}}_{i \in I} \, $,  being a  $ \K $--basis  of  $ \fg_\zero \, $,  resp.\ of  $ \fg_\uno \, $.  Moreover, we fix any total order  $ \, \preceq \, $ on both  $ I $  and  $ J \, $,  so that both  $ B_\zero $  and  $ B_\uno $  are totally ordered and, declaring elements from  $ B_\zero $  to be less than those of  $ B_\uno \, $   --- in a nutshell, setting  $ \, B_\zero \preceq B_\uno \, $  ---   overall the whole basis  $ B $  is totally ordered too.
 \vskip3pt
   Now consider the  $ \K $--algebra  $ \, \K\langle\langle Z_1 , Z_2 \rangle\rangle \, $  of formal power series in the non-commutative variables  $ Z_1 $  and  $ Z_2 \, $.  The well-known Campbell-Baker-Hausdorff formula  (see, e.g.,  \cite{jac})  in  $ \, \K\langle\langle Z_1 , Z_2 \rangle\rangle \, $  is  $ \; \exp(Z_1) \cdot \exp(Z_2) \, = \, \exp(Z_1 * Z_2) \; $  with  $ \; Z_1 * Z_2 := \log\big( \exp(Z_1) \cdot \exp(Z_2) \big) \, \in \, \K\langle\langle Z_1 , Z_2  \rangle\rangle \; $.  More precisely, the formal series expansion of  $ \, Z_1 * Z_2 \, $  can be re-arranged in the shape of a formal series  $ \; Z_1 * Z_2 \, = \, \sum_{n=1}^{+\infty} L_n(Z_1,Z_2) \; $  where each  $ L_n(Z_1,Z_2) $  is a homogeneous Lie monomial of degree  $ n $  in the free Lie  $ \K $--algebra  $ {\big\langle Z_1 \, , Z_2 \big\rangle}_{\text{\it Lie}}^\K $  generated by  $ Z_1 $ and  $ Z_2 \, $.  In particular, if one replaces  $ Z_1 $  and  $ Z_2 $  with elements  $ z_1 $  and  $ z_2 $  sitting in some nilpotent Lie algebra, then all but finitely many of the  $ L_n(z_1,z_2) $'s  do vanish, hence  $ \, z_1 * z_2 \, $  can be written as a finite sum.
 \vskip5pt
   Our next goal is another description of  $ \, N_G(A) = \exp\!\big( \fn_\fg(A) \big) \, $.  We need an auxiliary result:
\end{free text}

\smallskip

\begin{lemma}  \label{expans_exp(sum_S)}
 Let  $ \, S_1 , \dots , S_\ell \in \fn_\fg(A) \, $  with  $ \, S_i \in \fn_\fg^{\!(d_i)}\!(A) \, $  for some  $ \, d_i \in \N_+ \, $  ($ \, i = 1 , \dots , \ell \, $).  Then there exist  $ \, T_1 , \dots , T_k \in \fn_\fg(A) \, $  such that  $ \, T_j \in \fn_\fg^{\!(\partial_j)}\!(A) \, $  with  $ \, \partial_j \!\geq d_{a_j} + \, d_{b_j} \, $  for some  $ \, a_j , b_j \in \{1,\dots,\ell\} \, $  (for all  $ \, j = 1 , \dots , k $),  and
  $$  \exp\big( S_1 + \cdots + S_\ell \big)  \, = \,  \exp(S_1) \cdots \exp(S_\ell) \, \exp(T_1) \cdots \exp(T_k)  $$
\end{lemma}

\begin{proof}
 Writing all exponentials as formal series (actually  {\sl finite sum},  because of the nilpotency of all elements in  $ \fn_\fg(A) \, $,  cf.\ \S \ref{1st-descr_N_G(A)}),  the claim follows at once from
definitions by induction on  $ \ell $  via a straightforward application of Baker-Campbell-Hausdorff formula.
\end{proof}

\medskip

   We can now provide our new description of the subgroup  $ \; N_G(A) = \exp\big(\mathfrak{n}_\fg(A)\big) \; $:

\medskip

\begin{proposition}  \label{prop:2nd-descr_N_G(A)}
 The subgroup  $ \; N_G(A) = \exp\big(\mathfrak{n}_\fg(A)\big) \; $  of  $ G(A) $  is generated by the set
  $$  \varGamma_{\!{}_B}  \;\; := \;\;  \Big\{\, \exp\big(t_j X_j\big) \, , \, \exp\big(\eta_i Y_i\big) \;\Big|\;\, t_j \in \nil{A}_\zero \, , \, \eta_i \in \nil{A}_\uno = A_\uno \, , \; \forall \, j \in J, \, i \in I \,\Big\}  $$
where  $ \, {\big\{ X_j \big\}}_{j \in J} \bigsqcup {\big\{ Y_i \big\}}_{i \in I} = B_\zero \bigsqcup B_\uno = B \, $  is the  $ \K $--basis  of  $ \fg $  chosen in  \S \ref{twds_glob-split}  above.
\end{proposition}

\begin{proof}
 Let  $ \, n \in N_G(A) = \exp\!\big(\fn_\fg(A)\big) \, $,  say  $ \, n = \exp(Z) \, $  with  $ \, Z \in \fn_\fg(A) \, $;  clearly we can assume  $ \, Z \not= 0 \, $.  Using our fixed, ordered,  $ \K $--basis  $ B $  of  $ \fg $  our  $ Z $  expands into  $ \; Z = \sum_{j \in J} t'_j \, X_j \, + \sum_{i \in I} \eta'_i \, Y_i \; $  for some  $ \, t'_j \in \nil{A}_\zero \, $  and  $ \, \eta'_i \in \nil{A}_\uno \, $,  by the very definition of  $ \fn_\fg(A) \, $.  By  Lemma \ref{expans_exp(sum_S)},  this implies that
  $$  \exp(Z) \, = \, \exp\!\Big( {\textstyle \sum_{j \in J}} \, t'_j X_j \, + {\textstyle \sum_{i \in I}} \, \eta'_i Y_i \Big) \, = \, {\textstyle \overrightarrow{\textstyle \prod\limits_{j \in J}}} \exp\!\big(t'_j\,X_j\big) \, {\textstyle \overrightarrow{\textstyle \prod\limits_{i \in I}}} \exp\!\big(\eta'_i\,Y_i\big) \cdot \exp\!\Big(Z^{(1)}_1\Big) \cdots \exp\!\Big(Z^{(1)}_{k_1}\Big)  $$
 \vskip-5pt
\noindent
 for some  $ \, Z^{(1)}_1 , \dots , Z^{(1)}_{k_1} \in \fn_\fg(A) \, $,  where the symbols  $ \overrightarrow{\textstyle \prod}_{j \in J} $  and  $ \overrightarrow{\textstyle \prod}_{i \in I} $  denotes ordered products.  Even more, the Lemma ensures that the new terms  $ Z^{(1)}_h $'s  actually ``lie deeper'', in the decreasing filtration of  $ \fn_\fg(A) $  given by the  $ \fn_\fg^{\!(d)}\!(A) $'s,  than the initial  $ Z $  we started with, and this allows us to get to the end iterating this argument, in finitely many steps.
                                              \par
   Indeed, formally speaking we define  $ \, d(T) \in \N_+ \, $  by the condition  $ \, T \in \fn_\fg^{\!(d(T))}\!(A) \,\setminus\, \fn_\fg^{\!(d(T)+1)}\!(A) \, $.  Then in our construction  Lemma \ref{expans_exp(sum_S)}  ensures that for the newly occurring elements  $ Z^{(1)}_1 \, , \dots , Z^{(1)}_{k_1} $  we have  $ \; d\big(Z^{(1)}_h\big) > d(Z) \; $  for all  $ \, h = 1 , \dots , k \, $.  Now we can repeat our argument, with the  $ Z^{(1)}_h $'s  playing the role of  $ Z \, $,  and find similar results, i.e.\ a new expression for  $ \, \exp(Z) \, $  of the form
  $$  \exp(Z)  \,\; = \;\,  {\textstyle \overrightarrow{\textstyle \prod\limits_{j \in J}}} \exp\!\big(t'_j X_j\big) \, {\textstyle \overrightarrow{\textstyle \prod\limits_{i \in I}}} \exp\!\big(\eta'_i Y_i\big)
\cdot {\textstyle \mathop{\overrightarrow{\textstyle \prod}}\limits_{s=1}^{k_1}} \bigg(\;
{\textstyle \overrightarrow{\textstyle \prod\limits_{j \in J}}} \exp\!\big(t''_{s,j} X_j\big) \, {\textstyle \overrightarrow{\textstyle \prod\limits_{i \in I}}} \exp\!\big(\eta''_{s,i} Y_i\big) \!\bigg)
\cdot {\textstyle \mathop{\overrightarrow{\textstyle \prod}}\limits_{r=1}^{k_2}} \exp\big(Z^{(2)}_r\big)  $$
%
%%%
%%%%%
  \eject
%%%%%
%%%
%
\noindent
 for some  $ \, Z^{(2)}_1 , \dots , Z^{(2)}_{k_2} \in \fn_\fg(A) \, $  such that  $ \, d(Z^{(2)}_r) > \min{\big\{ d\big(Z^{(1)}_h\big) \big\}}_{h=1,\dots,k_1} \, $  for all  $ \, r = 1 , \dots , k_2 \, $.  Iterating this process we find at each step new factors belonging to  $ \varGamma_{\!{}_B} $  and possibly new factors of the form  $ \, \exp\!\big(Z^{(c)}_q\big) \, $,  $ \, q = 1 , \dots , k_c \, $, such that the sequence  $ \, n_c := \min{\big\{ d\big(Z^{(c)}_q\big) \big\}}_{q=1,\dots,k_c} \, $  is strictly increasing.  Then, since we have  $ \, \fn_\fg^{\!(N)}\!(A) = \{0\} \, $  for  $ \, N \gg 0 \, $,  after finitely many steps we find  $ \, Z^{(c)}_q \!\equiv 0 \, $,  i.e.\  no further terms actually occur, and  $ \, n = \exp(Z) \, $  is eventually written as a product of elements in  $ \varGamma_{\!{}_B\,} $;  thus the latter is indeed a generating set for  $ \, N_G(A) = \exp\!\big(\fn_\fg(A)\big) \, $,  \, as claimed.
\end{proof}

\smallskip

\begin{free text}  \label{spec-exp}
 {\bf Special exponentials in  $ G(A) \, $.}
 Before going on, let us consider elements in  $ G(A) $  of the form  $ \, \exp\!\big(t\,X\big) \, $  or  $ \, \exp\!\big(\eta\,Y\big) \, $   --- with  $ \, t \in A_\zero \, $  such that  $ \, t^2 = 0 \, $  (so that  $ \, t \in \nil{A}_\zero \, $  indeed),  $ \, X \in \fg_\zero \, $,  $ \, \eta \in \nil{A}_\uno \, $,  $ \, Y \in \fg_\uno \, $  ---   like those (see above) that generate  $ \, \exp\!\big(\fn_\fg(A)\big) = N_G(A) \, $.  Since both  $ t $  and  $ \eta $  have square zero, the formal power series expansion of both  $ \, \exp\!\big(t\,X\big) \, $  and  $ \, \exp\!\big(\eta\,Y\big) \, $  is actually truncated at first order, i.e.\ it reads  $ \, \exp\!\big(t\,X\big) = \big( 1 + t\,X \big) \, $  and  $ \, \exp\!\big(\eta\,Y\big) = \big( 1 + \eta\,Y \big) \, $
%
% Using this more inspiring notation, we now consider some interesting relations
% inside  $ G(A) $  involving group elements of this type.
%
 respectively.  More in general, we consider elements of the form  $ \, \exp(\cX) , \exp(\cY) \in \exp\!\big(\fn_\fg(A)\big) = N_G(A) \, $  with  $ \, \cX \in \nil{A}_\zero \!\otimes \fg_\zero \, $  and  $ \, \cY \in \nil{A}_\uno \otimes \fg_\uno = A_\uno \otimes \fg_\uno \, $.  As both  $ \nil{A}_\zero $ and  $ A_\uno $  are nilpotent, the formal power series expansion of both  $ \, \exp(\cX) $  and  $ \exp(\cY) $  can be seen again as a (finite) polynomial.
                                                                \par
   In the next Lemma we collect some identities in  $ G(A) $  involving ``special exponentials'' as those mentioned above.  This will be crucial all over the sequel.
\end{free text}

\smallskip

\begin{lemma}  \label{lemma:relations-in-G(A)}
 Let  $ \, A \in \Wsalg_\K \, $,  $ \, \eta, \eta', \eta'' \in A_\uno \, $,  $ \, \eta_i \in A_\uno \, $  (for all  $ \, i \in I \, $),  $ \, Y, Y' \! \in \fg_\uno \, $,  $ \, X \! \in \fg_\zero \, $  and  $ \, g_{{}_0} \in G_\zero(A) \, $.  Then inside  $ G(A) $  we have (notation of  Definition \ref{def:Lie-superalgebras})
 \vskip7pt
   (a) \;  $ \qquad \big( 1 + \, \eta \, \eta' \, [Y,Y'\,] \big) = \exp\big( \eta \, \eta' \, [Y,Y'\,] \big) \, \in \, G_\zero(A) $
 \vskip5pt
   (b)  $ \quad (1 + \eta \, Y) \, g_{{}_0} = \, g_{{}_0} \big( 1 + \eta \, \Ad\big(g_{{}_0}^{-1}\big)(Y) \big) \;\, $,
 $ \;\; \exp\!\big(\, {\textstyle \sum_{i \in I\,}} \eta_i Y_i \big) \, g_{{}_0} = \, g_{{}_0} \exp\!\big(\, {\textstyle \sum_{i \in I\,}} \eta_i \Ad\big(g_{{}_0}^{-1}\big)(Y_i) \big) $
 \vskip4pt
   (c) \;  $ \;\;\quad  \big( 1 + \eta' \, Y' \big) \, \big( 1 + \eta'' \, Y'' \big) \; = \; \big( 1 + \eta'' \, \eta' \, [Y',Y''\,] \big) \, \big( 1 + \eta'' \, Y'' \big) \, \big( 1 + \eta' \, Y' \big) $
 \vskip5pt
   (d) \;  $ \;\;\quad \big( 1 + \eta \, Y' \big) \, \big( 1 + \eta \, Y'' \big)  \; = \;  \big( 1 + \eta \, (Y'+Y'') \big)  \; = \;  \big( 1 + \eta \, Y'' \big) \, \big( 1 + \eta \, Y' \big) $
 \vskip5pt
   (e) \;  $ \;\;\quad \big( 1 + \eta' \, Y \big) \, \big( 1 + \eta'' \, Y \big)  \; = \;  \big( 1 + \eta'' \, \eta' \, Y^{\langle 2 \rangle} \big) \, \big( 1 + (\eta'+\eta'') \, Y \big) $
 \vskip5pt
   (f) \;  $ \;\;\quad (1 \, + \, \eta \, Y) \, (1 \, + \, \eta' \eta'' X)  \; = \;  (1 \, + \, \eta' \eta'' X) \, \big( 1 \, + \, \eta \, \eta' \eta'' \, [Y,X] \big) \, (1 \, + \, \eta \, Y)  \; = $
                                                              \hfill{\ }\break
   \indent \hskip-1pt   \phantom{(f) \;  $ \,\;\; (1 \, + \, \eta \, Y) \, (1 \, + \, \eta' \eta'' X) $}  $ \; = \;  (1 \, + \, \eta' \eta'' X) \, (1 \, + \, \eta \, Y) \, \big( 1 \, + \, \eta \, \eta' \eta'' \, [Y,X] \big) $
 \vskip5pt
   (g) \,  Let  $ \, (h,k) := h\,k\,h^{-1}k^{-1} $  be the commutator of elements  $ h $  and  $ k $  in a group.  Then
 \vskip4pt
   \centerline{ $ \big( \big( 1 + \eta \, Y \big) , \big( 1 + \eta' \, Y' \big) \big)  =  \big( 1 + \eta' \, \eta \, [Y,Y'\,] \big) \; ,  \;\;  \big( \big( 1 + \eta \, Y \big) , \big( 1 + \eta \, Y' \big) \big)  =  \big( 1 + \eta \, (Y+Y') \big) $ }
 \vskip3pt
   \centerline{ $ \big( \big( 1 + \eta' \, Y \big) \, , \big( 1 + \eta'' \, Y \big) \big) \, = \, {\big( 1 + \eta'' \, \eta' \, Y^{\langle 2 \rangle} \big)}^2 \, = \, \big( 1 + \eta'' \, \eta' \, 2 \, Y^{\langle 2 \rangle} \big) \, = \, \big( 1 + \eta'' \, \eta' \, [Y,Y] \big) $ }
 \vskip5pt
%
%%%
%    (g) \,  Let  $ \, (h,k) := h\,k\,h^{-1}k^{-1} $  be the commutator of elements  $ h $  and  $ k $
% in a group.  Then
%   $$  \displaylines{
%    \big( \big( 1 + \eta \, Y \big) , \big( 1 + \eta' \, Y' \big) \big)  =  \big( 1 + \eta' \, \eta \,
% [Y,Y'\,] \big) \; ,  \;\;  \big( \big( 1 + \eta \, Y \big) , \big( 1 + \eta \, Y' \big) \big)  =
% \big( 1 + \eta \, (Y+Y') \big)  \cr
% %
%    \big( \big( 1 + \eta' \, Y \big) \, , \big( 1 + \eta'' \, Y \big) \big) \, = \, {\big( 1 +
% \eta'' \, \eta' \, Y^{\langle 2 \rangle} \big)}^2 \, = \, \big( 1 + \eta'' \, \eta' \, 2 \,
% Y^{\langle 2 \rangle} \big) \, = \, \big( 1 + \eta'' \, \eta' \, [Y,Y] \big)  }  $$
% %
%%%
%
 (N.B.: taking the rightmost term in the last identity, the latter is a special case of the first).
 \vskip5pt
   (h) \,  For any  $ \, n \in \N_+ \, $,  there exist unique
%
% Lie polynomials
%
 $ \, T_\zero^{(n)} , T_\uno^{(n)} \in {\big\langle Z_1 \, , Z_2 \big\rangle}_{\text{\it Lie}}^\K \, $,  independent of  $ A \, $,  such that:
 \vskip2pt
   \quad  ---  $ \; T_\zero^{(n)} $  is a  $ \K $--linear  combination of Lie monomials of  {\sl even}  degree greater than  $ n \, $,
 \vskip2pt
   \quad  ---  $ \; T_\uno^{(n)} $  is a  $ \K $--linear  combination of Lie monomials of  {\sl odd}  degree greater than  $ n \, $,
 \vskip2pt
   \quad  ---  setting  $ \, d_1 := \text{\sl dim}(\fg_\uno) \, $,  for any  $ \, \cY' , \cY'' \in A_\uno \!\otimes_\K \fg_\uno \, $  we have
 \vskip4pt
   \centerline{ $ \exp\!\big( \cY' \big) \, \exp\!\big( \cY'' \big)  \,\; = \;\,  \exp\!\Big( P_\zero^{(d_1)}\big( \cY' , \cY'' \big) \Big) \, \exp\!\Big(\, \cY' + \cY'' + P_\uno^{(d_1)}\big( \cY' , \cY'' \big) \Big) $ }
\noindent
 with  $ \,\; P_\zero^{(d_1)} := \, T_\zero^{(1)} \!* T_\zero^{(2)} \!* \cdots * T_\zero^{(d_1 - 1)} \;\, $  and  $ \,\; P_\uno^{(d_1)} := \, T_\uno^{(d_1 - 1)} \!+ \cdots + T_\uno^{(2)} \!+ T_\uno^{(1)} \;\, $.
\end{lemma}

\begin{proof}
 Writing all exponentials as formal power series
%
%%%
%  (actually  {\sl finite sum},  because of the nilpotency of all
% elements in  $ \fn_\fg(A) \, $,  cf.\ \S \ref{1st-descr_N_G(A)}),
%%%
%
 (actually  {\sl finite sum},  as noticed above), claims  {\it (a)\/}  through  {\it (g)\/}  follow at once from definitions, via straightforward applications of the Baker-Campbell-Hausdorff formula.  Claim  {\it (a)\/}  is even simpler, since  $ \, \big( 1 + \eta \, \eta' \, [Y,Y'\,] \big) \, $  is just the formal power series expansion of  $ \, \exp\big( \eta \, \eta' \, [Y,Y'\,] \big) \, $,  and the latter belong to  $ \; \exp\big( \nil{A}_\zero \, \fg_\zero \big) \, \subseteq \, G_\zero(A) \; $.
%%%%%
  \eject
%%%%%
%%%%%
%
 \vskip5pt
   Claim  {\it (h)\/}  requires some more work.  An equivalent formulation of it is that the identity
\begin{equation}  \label{eq: expansion_Y'*Y''}
 \cY' \,*\, \cY'' \,\; = \;\,  P_\zero^{(d_1)}\big( \cY' , \cY'' \big) \,*\, \Big(\, \cY' + \cY'' + P_\uno^{(d_1)}\big( \cY' , \cY'' \big) \Big)
\end{equation}
holds true for some uniquely determined Lie polynomials  $ \,\; P_\zero^{(d_1)} := \, T_\zero^{(1)} * T_\zero^{(2)} * \cdots * T_\zero^{(d_1 - 1)} \;\, $  and  $ \,\; P_\uno^{(d_1)} := \, T_\uno^{(d_1 - 1)} + \cdots + T_\uno^{(2)} + T_\uno^{(1)} \;\, $  with the  $ T_{\zero/\uno}^{\,(i)} $'s  having the properties mentioned above.
                                                            \par
   We start working with the product  ``$ \, * \, $''  in  $ \, {\big\langle Z_1 \, , Z_2 \big\rangle}_{\text{\it Lie}}^\K \, $.  As a matter of terminology, we call  {\sl order\/}  of any non-zero Lie polynomial  $ P $  in two variables the least degree of a homogeneous monomial occurring with non-zero coefficient in the standard  $ \K $--linear expansion of  $ P $  (accordingly, the order of the zero polynomial will be  $ -\infty \, $).

   First of all, we need three technical results.  For formal symbols  $ \, F , G \, $  there exist unique  $ \; R \, , S \in {\big\langle F , G \big\rangle}_{\text{\it Lie}}^\K \; $  such that
\begin{equation}  \label{eq: factor_F*G}
  F \,*\, G  \,\; = \;\,  R \,*\, \big(\, F + G \big)
\end{equation}
\begin{equation}  \label{eq: factor_F+G}
  F \,+\, G  \,\; = \;\,  F \,*\, S \,*\, G
\end{equation}
and  $ \, R \, , S \, $  are Lie polynomials (in  $ F $  and  $ G \, $)  of order greater than 1.  In fact,  $ R $  is the unique solution of the equation
  $ \,\; \exp(F) \, \exp(G) \, = \, \exp(R) \, \exp(F+G) \; $,
 \, given by
  $$  R  \,\; = \;\,  \log\!\Big( \exp(F) \, \exp(G) \, {\exp(F+G)}^{-1} \Big)  \,\; = \;\,  F \,*\, G \,*\, (-F-G)  $$
while  $ S $  is the unique solution of the equation
  $ \,\; \exp(F+G) \, = \, \exp(F) \, \exp(S) \, \exp(G) \; $,
 \, given by
  $$  S  \,\; = \;\,  \log\!\Big( {\exp(F)}^{-1} \, \exp(F+G) \, {\exp(G)}^{-1} \Big)  \,\; = \;\,  (-F) \,*\, (F+G) \,*\, (-G)  $$
Then the explicit expression of the product  ``$ \, * \, $''  implies that both  $ R $  and  $ S $  have order greater than 1, as claimed; moreover, both are independent of  $ A $  and  $ \fg $  whatsoever.  Finally, for any Lie polynomial  $ T $  in two variables there exist unique  $ T_\zero $  and  $ T_\uno $  such that
\begin{equation}  \label{eq: split T = Te + To}
  T  \,\; = \;\,  T_\zero + \, T_\uno
\end{equation}
where  $ \, T_\zero \, $,  resp.\  $ \, T_\uno \, $,  is a  $ \K $--linear  combination of Lie monomials of  {\sl even},  resp.\  {\sl odd},  degree.
                                                                \par
   With repeated applications of  \eqref{eq: factor_F*G},  \eqref{eq: factor_F+G}  and  \eqref{eq: split T = Te + To}  we find
  $$  \displaylines{
   Z_1 \,*\, Z_2  \,\; {\buildrel {\eqref{eq: factor_F*G}} \over =} \;\,  R_1 \,*\, \big(\, Z_1 + Z_2 \big)  \,\; {\buildrel {(\,T^{(1)} := R_1)} \over =} \;\,  T^{(1)} \,*\, \big(\, Z_1 + Z_2 \big)  \,\; {\buildrel {\eqref{eq: split T = Te + To}} \over =} \;\,  \Big( T_\zero^{(1)} \! + T_\uno^{(1)} \Big) \,*\, \big(\, Z_1 + Z_2 \big)  \,\; {\buildrel {\eqref{eq: factor_F+G}} \over =}   \hfill  \cr
   {\buildrel {\eqref{eq: factor_F+G}} \over =} \;\,  T_\zero^{(1)} *\, S_1 \,*\, T_\uno^{(1)} *\, \big(\, Z_1 + Z_2 \big)  \,\; {\buildrel {\eqref{eq: factor_F*G}} \over =} \;\,  T_\zero^{(1)} *\, S_1 \,*\, R_2 \,*\, \Big( T_\uno^{(1)} \! + Z_1 + Z_2 \Big)  \,\; {\buildrel {(\,T^{(2)} := S_1 * R_2)} \over =}   \qquad \qquad  \cr
   {\buildrel {(\,T^{(2)} := S_1 * R_2)} \over =} \;\,  T_\zero^{(1)} *\, T^{(2)} \,*\, \Big( T_\uno^{(1)} \! + Z_1 + Z_2 \Big)  \,\; {\buildrel \eqref{eq: split T = Te + To} \over =} \;\,  T_\zero^{(1)} *\, \Big( T_\zero^{(2)} \! + T_\uno^{(2)} \Big) \,*\, \Big( T_\uno^{(1)} \! + Z_1 + Z_2 \Big)  \,\; {\buildrel {\eqref{eq: factor_F+G}} \over =}  \cr
   \hfill   {\buildrel {\eqref{eq: factor_F+G}} \over =} \;\,  T_\zero^{(1)} *\, T_\zero^{(2)} *\, S_2 \,*\, T_\uno^{(2)} *\, \Big( T_\uno^{(1)} \! + Z_1 + Z_2 \Big)  \,\; {\buildrel {\eqref{eq: factor_F*G}} \over =}   \qquad \qquad \qquad \qquad \qquad \qquad \qquad \qquad  \cr
   \hfill   {\buildrel {\eqref{eq: factor_F*G}} \over =} \;\,  T_\zero^{(1)} *\, T_\zero^{(2)} *\, S_2 \,*\, R_3 \,*\, \Big( T_\uno^{(2)} \! + T_\uno^{(1)} \! + Z_1 + Z_2 \Big)  \,\; {\buildrel {(\,T^{(3)} := S_2 * R_3)} \over =}   \qquad \qquad \qquad \qquad  \cr
   \hfill   {\buildrel {(\,T^{(3)} := S_2 * R_3)} \over =} \;\,  T_\zero^{(1)} *\, T_\zero^{(2)} *\, T^{(3)} \,*\, \Big( T_\uno^{(2)} \! + T_\uno^{(1)} \! + Z_1 + Z_2 \Big)  \,\; = \;\,  \cdots   \qquad \qquad \qquad \qquad  \cr
   \hfill   \cdots  \,\; = \;\,  T_\zero^{(1)} *\, T_\zero^{(2)} *\, \cdots \,*\, T_\zero^{(n-1)} *\, T^{(n)} \,*\, \Big( T_\uno^{(n-1)} \! + \cdots + T_\uno^{(2)} \! + T_\uno^{(1)} \! + Z_1 + Z_2 \Big)
  }  $$
so that in the end we get
  $$  Z_1 \,*\, Z_2  \,\; = \;\,  T_\zero^{(1)} *\, T_\zero^{(2)} *\, \cdots \,*\, T_\zero^{(n-1)} *\, T^{(n)} \,*\, \Big( T_\uno^{(n-1)} \! + \cdots + T_\uno^{(2)} \! + T_\uno^{(1)} \! + Z_1 + Z_2 \Big)  $$
for arbitrarily big  $ \, n \in \N \, $,  where each Lie polynomial  $ T_s $   --- by its very construction ---   has order greater than  $ s \, $.  Finally, we can re-write this last formula as
\begin{equation}  \label{eq: long-fact Z'*Z''}
  Z_1 \,*\, Z_2  \,\; = \;\,  P_\zero^{(n)}\!\big( Z_1 \,, Z_2 \big) \,*\, T^{(n)} \,*\, \Big( Z_1 + Z_2 + P_\uno^{(n)}\!\big( Z_1 \,, Z_2 \big) \Big)
\end{equation}
for all  $ \, n \in \N \, $,  with  $ \; P_\zero^{(n)} := T_\zero^{(1)} *\, T_\zero^{(2)} *\, \cdots \,*\, T_\zero^{(n-1)} \; $  and  $ \; P_\uno^{(n)} := T_\uno^{(n-1)} \! + \cdots + T_\uno^{(2)} \! + T_\uno^{(1)} \; $.
 \vskip5pt
   Now observe that every Lie monomial of degree  $ \; m > d_1 := \text{\sl dim}(\fg_\uno) \; $  vanishes when computed on  $ \, A_\uno \!\otimes_\K \fg_\uno \, $;  therefore  $ T^{(m-1)} $  vanishes as well.  It then follows that for  $ \, n = d_1 \, $  replacing  $ \cY' $  for  $ Z_1 $  and  $ \cY'' $  for  $ Z_2 $  in  \eqref{eq: long-fact Z'*Z''}  we eventually get  \eqref{eq: expansion_Y'*Y''},  q.e.d.
\end{proof}

\medskip

   We still need to introduce some auxiliary objects associated with  $ G \, $:

\medskip

\begin{definition}  \label{def:spec_subobj's_G}
 Let  $ G $  be a Lie supergroup, as above.  For any  $ \, A \in \Wsalg_\K \, $,  we define:
 \vskip7pt
   {\it (a)}  \;\;  $ \; G^-(A)  \; := \;  \Big\{\, {\textstyle \prod_{s=1}^n} \big( 1 + \eta_s Y_s \big) \;\Big|\; n \in \N \, , \; (\eta_s , Y_s) \in A_\uno \times \fg_\uno \;\, \forall \; s \in \{1,\dots,n\} \,\Big\}  \;\quad  \big(\, \subseteq G(A) \,\big) \; $
 \vskip7pt
   {\it (b)}  \quad  $ \; \exp\!\big( A_\uno \!\otimes_\K \fg_\uno \big)  \; := \;  \big\{\, \exp\!\big( \cY \big) \;\big|\; \cY \in A_\uno \!\otimes_\K \fg_\uno \,\big\}  \;\qquad  \Big( \subseteq \, G^-(A) \,\Big) \; $
 \vskip7pt
   {\it (c)}  \quad  $ \; N_G^{[2]}(A)  \; := \;  \exp\!\Big( A_\uno^{[2]} \!\otimes_\K \big[\fg_\uno,\fg_\uno\big] \Big)  \;\qquad  \big(\; \subseteq \, N_{G_\zero}\big(A_\zero\big) \,=\, N_G(A) \,\bigcap\, G_\zero(A) \,\big) \; $
 \vskip7pt
   {\it (d)}\;  for any fixed  $ \K $--basis  $ \, {\big\{ Y_i \big\}}_{i \in I} \, $  of  $ \, \fg_\uno \, $  (for some index set  $ I \, $)  and any fixed total order in  $ I $,
 \vskip-7pt
  $$  G_-^{\scriptscriptstyle \,<}(A)  \; := \;  \bigg\{\, {\textstyle \prod\limits_{i \in I}^\rightarrow} \big( 1 + \eta_i Y_i \big) \;\bigg|\,\; \eta_i \in A_\uno \;\, \forall \; i \in I \,\bigg\}  \qquad  \big(\; \subseteq\, G^-(A) \;\big)  $$
 \vskip-8pt
\noindent
 where  $ \, \prod\limits_{i \in I}^\rightarrow \, $  denotes an  {\sl ordered product}   --- with respect to the fixed total order in  $ I \, $.
 \hfill   $ \diamondsuit $
\end{definition}

\vskip3pt

\begin{remark}  \label{expA1g1-in-G-}
 By definition,  $ \, \exp\!\big( A_\uno \!\otimes_\K \fg_\uno \big) \, $  contains the set of generators of  $ \; G^-(A) \; $;  therefore, the former generates a subgroup  $ \, \Big\langle \exp\!\big( A_\uno \!\otimes_\K \fg_\uno \big) \Big\rangle \, $  of  $ G(A) $  that contains  $ G^-(A) \, $.  On the other hand, for any  $ \; \sum_{s=1}^n \eta_s \, Y_s \in A_\uno \!\otimes_\K \fg_\uno \, $,  \, the formal series expansion of  $ \, \exp\big( \sum_{s=1}^n \eta_s \, Y_s \big) \, $  yields
 \vskip5pt
   \centerline{ $ \,\;\; \exp\!\Big( \sum_{s=1}^n \eta_s \, Y_s \Big) \; = \; \prod_{\sigma \in \mathcal{S}_n} \!\! \big( 1 + \eta_{\sigma(1)} Y_{\sigma(1)} \big/ n! \,\big) \cdot \big( 1 + \eta_{\sigma(2)} Y_{\sigma(2)} \big/ n! \,\big) \cdots \big( 1 + \eta_{\sigma(n)} Y_{\sigma(n)} \big/ n! \,\big) $ }
 \vskip5pt
\noindent
  that implies  $ \,\; \Big\langle \exp\!\big( A_\uno \!\otimes_\K \fg_\uno \big) \Big\rangle \, \subseteq \, G^-(A) \;\, $.  The outcome is that  $ \,\; G^-(A) \, = \, \Big\langle \exp\!\big( A_\uno \!\otimes_\K \fg_\uno \big) \Big\rangle \;\, $.
\end{remark}

\vskip7pt

   From now on, we fix a  $ \K $--basis  $ \, {\big\{ Y_i \big\}}_{i \in I} \, $  of  $ \, \fg_\uno \, $  (for some index set  $ I \, $)  and we fix in  $ I $  a total order, as in  Definition \ref{def:spec_subobj's_G}{\it (d)}.  Our first result provides new, interesting factorizations for  $ G(A) \, $:

\medskip

\begin{proposition}  \label{prop:fact-Liesgrp_G}
 Let  $ \, G \, $  be a Lie supergroup as above, let  $ \, {\big\{ Y_i \big\}}_{i \in I} \, $  be a totally ordered  $ \, \K $--basis  of  $ \, \fg_\uno \, $  (for some total order in the set  $ I \, $)  and let  $ \, A \in \Wsalg_\K \, $  be any Weil superalgebra.  Then:
 \vskip5pt
   (a)\,  $ \; G^-(A) $  coincides with the subgroup  $ \, \big\langle G_-^{\scriptscriptstyle \,<}(A) \big\rangle \, $ of  $ \, G(A) $  generated by  $ \, G_-^{\scriptscriptstyle \,<}(A) \, $  and with the subgroup  $ \, \big\langle \exp\!\big( A_\uno \!\otimes_\K \fg_\uno \big) \big\rangle \, $  generated by  $ \, \exp\!\big( A_\uno \!\otimes_\K \fg_\uno \big) \; $;
 \vskip5pt
   (b)\,  there exist set-theoretic factorizations (with respect to the group product  ``$ \; \cdot \, $'')
 \vskip-5pt
\begin{equation}  \label{eq:factor-1_G^-}
  G^-(A)  \; = \;  N_G^{[2]}(A) \cdot G_-^{\scriptscriptstyle \,<}(A)  \;\;\quad ,  \qquad\;\;
  G^-(A) \; = \; G_-^{\scriptscriptstyle \,<}(A) \cdot N_G^{[2]}(A)
\end{equation}
\begin{equation}  \label{eq:factor-1_N_G}
   N_G(A) \; = \; N_{G_\zero}\big(A_\zero\big) \cdot G_-^{\scriptscriptstyle \,<}(A)  \;\quad ,
  \qquad\;  N_G(A) \; = \; G_-^{\scriptscriptstyle \,<}(A) \cdot N_{G_\zero}\big(A_\zero\big)
\end{equation}
\begin{equation}  \label{eq:factor-1_G}
   G(A) \; = \; G_\zero(A) \cdot G_-^{\scriptscriptstyle \,<}(A)  \;\;\;\;\quad ,  \qquad\;\;\;\;  G(A) \; = \; G_-^{\scriptscriptstyle \,<}(A) \cdot G_\zero(A)
\end{equation}
 \vskip5pt
   (c)\,  there exist set-theoretic factorizations (with respect to the group product  ``$ \; \cdot \, $'')
 \vskip-15pt
\begin{equation}  \label{eq:factor-2_G^-}
  G^-(A)  \; = \;  N_G^{[2]}(A) \cdot \exp\!\big( A_\uno \!\otimes_\K \fg_\uno \big)  \;\;\quad ,  \qquad\;\;
  G^-(A) \; = \; \exp\!\big( A_\uno \!\otimes_\K \fg_\uno \big) \cdot N_G^{[2]}(A)
\end{equation}
\begin{equation}  \label{eq:factor-2_N_G}
   N_G(A) \; = \; N_{G_\zero}\big(A_\zero\big) \cdot \exp\!\big( A_\uno \!\otimes_\K \fg_\uno \big)  \;\quad ,
  \qquad\;  N_G(A) \; = \; \exp\!\big( A_\uno \!\otimes_\K \fg_\uno \big) \cdot N_{G_\zero}\big(A_\zero\big)
\end{equation}
\begin{equation}  \label{eq:factor-2_G}
   G(A) \; = \; G_\zero(A) \cdot \exp\!\big( A_\uno \!\otimes_\K \fg_\uno \big)  \;\;\;\;\quad ,  \qquad\;\;\;\;  G(A) \; = \; \exp\!\big( A_\uno \!\otimes_\K \fg_\uno \big) \cdot G_\zero(A)
\end{equation}
\end{proposition}

\begin{proof}
 {\it (a)}\,  In  Remark \ref{expA1g1-in-G-}  above we proved that  $ \; \big\langle \exp\!\big( A_\uno \!\otimes_\K \fg_\uno \big) \big\rangle = \big\langle G_-^{\scriptscriptstyle \,<}(A) \big\rangle \; $,  and we are left to prove that  $ \, \big\langle G_-^{\scriptscriptstyle \,<}(A) \big\rangle = G^-(A) \, $.
 It is clear by definition that  $ G^-(A) $  is the subgroup in  $ G(A) $  generated by  $ \, \big\{\, (1+\eta\,Y) \,\big|\, \eta \in A_\uno \, , \, Y \!\in \fg_\uno \big\} \, $:  thus it is enough to prove that each of its generators  $ \, (1+\eta\,Y) \, $  actually belongs to  $ \, \big\langle G_-^{\scriptscriptstyle \,<}(A) \big\rangle \; $.
                                                      \par
   Now, given  $ \, Y \in \fg_\uno \, $  let  $ \; Y = \sum_{i \in I} c_i Y_i \; $  (with  $ \, c_i \in \K \, $)  be its  $ \K $--linear  expansion with respect to the  $ \K $--basis  $ \, {\big\{ Y_i \big\}}_{i \in I} \, $  of  $ \, \fg_\uno \, $.  Then repeated applications of the identity in  Lemma \ref{lemma:relations-in-G(A)}{\it (d)\/}  yield
 $ \;\; (1+\eta\,Y) \, = \, \big(\, 1 + \eta \, \sum_{i \in I} c_i Y_i \,\big) \, = \, \big(\, 1 + \sum_{i \in I} (c_i\,\eta) \, Y_i \,\big) \, = \, \overrightarrow{\prod}_{i \in I} \big(\, 1 + (c_i\,\eta) \, Y_i \,\big) \, \in \, G_-^{\scriptscriptstyle \,<}(A) \; $,
 \; q.e.d.
 \vskip5pt
   {\it (b)}\,  We begin with the proof of  \eqref{eq:factor-1_G^-}:  by left-right symmetry, it is enough to prove the left-hand side, that is  $ \; G^-(A) \, = \, N_G^{[2]}(A) \cdot G_-^{\scriptscriptstyle \,<}(A) \; $,  \; so we focus on that.  By claim  {\it (a)},  any element of  $ G^-(A) $  can be written as a (unordered) product of the form  $ \, {\textstyle \prod_{k=1}^N} \big( 1 + \eta_k \, Y_{i_k} \big) \, $  with $ \, \eta_k \in A_\uno \, $  and  $ \, i_k \in I \, $  for all  $ k \, $.  Our goal is to prove the following

\vskip5pt

   {\sl  $ \underline{\text{\it Claim}} $:  Any (unordered) product of the form  $ \, {\textstyle \prod_{k=1}^N} \big( 1 + \eta_k \, Y_{i_k} \big) \, $  can be ``re-ordered'', namely it can be re-written as an element of  $ \; N_G^{[2]}(A) \cdot G_-^{\scriptscriptstyle \,<}(A) \; $.}

\vskip5pt

   To prove this  {\it Claim},  let  $ \fa $  be the (two-sided) ideal of  $ A $  generated by the  $ \eta_k $'s,  and denote by  $ \, \fa^n \, $  its  $ n $--th  power  ($ \, n \in \N \, $);  as the  $ \eta_k $'s  are $ N $  elements and are  {\sl odd},  we have  $ \, \fa^n = \{0\} \, $  for  $ \, n > N \, $.
                                                     \par
   Let us denote by  $ \, \preceq \, $  our fixed total order in  $ I $.  Given the product  $ \, {\textstyle \prod_{k=1}^N} \big( 1 + \eta_k \, Y_{i_k} \big) \, $,  we define its  {\sl inversion number\/}  as being the number of occurrences of two consecutive indices  $ k_s $  and  $ k_{s+1} $  for which  $ \, i_{k_s} \npreceq i_{k_{s+1}} \, $:  then the product itself is  {\sl ordered\/}  if and only if its inversion number is zero.
                                                     \par
   Now assume the product  $ \, g := {\textstyle \prod_{k=1}^N} \big( 1 + \eta_k \, Y_{i_k} \big) \, $  is unordered: then there exists at least an inversion, say  $ \, i_{k_s} \npreceq i_{k_{s+1}} \, $,  i.e.~either  $ \, i_{k_s} \succ i_{k_{s+1}} \, $  or  $ \, i_{k_s} = i_{k_{s+1}} \, $.  Once again, using some of the relations considered in  Lemma \ref{lemma:relations-in-G(A)}  we can re-write the product of these two ``unordered factors''  $ \; \big( 1 + \eta_{k_s} Y_{i_{k_s}} \big) \, \big( 1 + \eta_{k_{s+1}} Y_{i_{k_{s+1}}} \big) \; $  in either form (depending on whether  $ \, i_{k_s} \succ i_{k_{s+1}} \, $  or  $ \, i_{k_s} = i_{k_{s+1}} \, $)
   $$  \displaylines{
    \big( 1 + \eta_{k_s} Y_{i_{k_s}} \big) \, \big( 1 + \eta_{k_{s+1}} Y_{i_{k_{s+1}}} \big)  \; = \;  \big( 1 + \eta_{k_{s+1}} \, \eta_{k_s} \, \big[Y_{i_{k_{s+1}}},Y_{i_{k_s}}\big] \big) \, \big( 1 + \eta_{k_{s+1}} Y_{i_{k_{s+1}}} \big) \, \big( 1 + \eta_{k_s} Y_{i_{k_s}} \big)  \cr
    \big( 1 + \eta_{k_s} Y_{i_{k_s}} \big) \, \big( 1 + \eta_{k_{s+1}} Y_{i_{k_s}} \big)  \,\; = \;  \Big( 1 + \eta_{k_{s+1}} \, \eta_{k_s} Y_{i_{k_s}}^{\,\langle 2 \rangle} \Big) \, \big( 1 + (\eta_{k_s} \! + \eta_{k_{s+1}}) \, Y_{i_{k_s}} \big)  }  $$
according to whether  $ \, i_{k_s} \succ i_{k_{s+1}} \, $  or  $ \, i_{k_s} = i_{k_{s+1}} \, $  respectively.  In either case, re-writing in this way the product of the  $ k_s $--th  and the  $ k_{s+1} $--th  factor in the original product  $ \, g := {\textstyle \prod_{k=1}^N} \big( 1 + \eta_k \, Y_{i_k} \big) \, $,  we end up with another product expression were we did eliminate one inversion, but we ``payed the price'' of inserting a  {\sl new factor}.  However, in both cases the newly added factor is of the form  $ \, \big( 1 + a \, X \big) \, $  for some $ \, X \in [\,\fg_\uno , \fg_\uno] \, $  and  $ \, a \in \fa^2 \, $,  so that  $ \, \big( 1 + a \, X \big) \in N_G^{[2]}(A) \, $.
                                                     \par
   By repeated use of relations of the form  $ \; \big( 1 + \eta \, Y \big) \cdot g_{{}_0} \, = \, g_{{}_0} \cdot \big( 1 + \eta \, \text{\sl Ad}\big(g_{{}_0}^{-1}\big)(Y) \big) \; $  we can shift the newly added factor  $ \, \big( 1 + a \, X \big) \, $  to the leftmost position in  $ g $   --- now re-written once more in yet a different product form ---   at the cost of inserting several  {\sl new factors of the form\/}  $ \, \big( 1 + b_t \, Z_t \big) \, $  for some $ \, Z_t \in \fg_\uno \, $  and  $ \, b_t \in \fa^3 \, $.  Moreover, by repeated use of relations of the form
 $ \; \big( 1 + \eta \, Y' \big) \cdot \big( 1 + \eta \, Y'' \big) \, = \, \big( 1 + \eta \, \big( Y' + Y'' \big) \big) \; $
 we can re-write each of these new factors as a product of factors of the form  $ \, \big( 1 + \eta'_h \, Y_{i'_h} \big) \, $  where  $ \, \eta'_h \in A_\uno \, $  is a multiple of some  $ \, b_t \, $,  so that  $ \, \eta'_h \in \fa^3 \, $  too.
                                                     \par
   Eventually, we find a new factorization of the original element  $ \, g := {\textstyle \prod_{k=1}^N} \big( 1 + \eta_k \, Y_{i_k} \big) \, $  in the new form  $ \, g := g'_0 \cdot {\textstyle \prod_{h=1}^{N'}} \big( 1 + \eta'_h \, Y_{i'_h} \big) \, $,  where now  $ \, g'_0 \in N_G^{[2]}(A) \, $  and the factors  $ \, \big( 1 + \eta'_h \, Y_{i'_h} \big) \, $  satisfy the following conditions:
 \vskip5pt
   {\it --- (a)}\,  each factor  $ \, \big( 1 + \eta'_h \, Y_{i'_h} \big) \, $  is either one of the old factors  $ \, \big( 1 + \eta_k \, Y_{i_k} \big) \, $  or a truly new factor;
 \vskip4pt
   {\it --- (b)}\,  in every (truly) new factor  $ \, \big( 1 + \eta'_h \, Y_{i'_h} \big) \, $  one has  $ \; \eta'_h \in \fa^3 \; $;
 \vskip4pt
   {\it --- (c)}\,  the number of inversions among factors  $ \, \big( 1 + \eta'_h \, Y_{i'_h} \big) = \big( 1 + \eta_k \, Y_{i_k} \big) \, $  of the old type is one less than before.
 \vskip7pt
   Iterating this procedure, after finitely many steps we obtain a new factorization of the initial element  $ \, g := {\textstyle \prod_{k=1}^N} \big( 1 + \eta_k \, Y_{i_k} \big) \, $  of the form  $ \, g = g''_{{}_0} \cdot {\textstyle \prod_{h=1}^{N''}} \big( 1 + \eta''_h \, Y_{i''_h} \big) \, $  where  $ \, g''_{{}_0} \in N_G^{[2]}(A) \, $  and the factors  $ \, \big( 1 + \eta''_h \, Y_{i''_h} \big) \, $  enjoy properties  {\it (a)\/}  and  {\it (b)\/}  above plus the ``optimal version'' of  {\it (c)},  namely
 \vskip4pt
   {\it --- (c+)}\,  the number of inversions among factors
%%
% $ \, \big( 1 + \eta'_h \, Y_{i'_h} \big) = \big( 1 + \eta_k \, Y_{i_k} \big) \, $
%%
 of the old type is zero.
 \vskip7pt
   Now we apply the same ``reordering operation'' to the product  $ \, {\textstyle \prod_{h=1}^{N''}} \big( 1 + \eta''_h \, Y_{i''_h} \big) \, $.  By construction, an inversion now can occur only among two factors of new type or among an old and a new factor; but then the two coefficients  $ \eta''_h $  involved by this inversion belong to  $ \fa $  and at least one of them is in  $ \fa^3 \, $.  It follows that when one performs the ``reordering operation'' onto the pair of factors involved in the inversion the new factor which pops up necessarily involves a coefficient in  $ \fa^4 \, $.  As this applies for any possible inversion, in the end we shall find a new factorization of  $ g $  of the form
 \vskip-9pt
  $$  g  \; = \;  g''_{{}_0} \cdot \widehat{g}_{{}_0} \cdot {\textstyle \prod_{t=1}^{\widehat{N}}} \big( 1 + \widehat{\eta}_t \, Y_{\widehat{i}_t} \big)  $$
 \vskip-3pt
\noindent
 in which  $ \, \widehat{g}_{{}_0} \in G_+(A) \, $  and the factors  $ \, \big( 1 + \widehat{\eta}_t \, Y_{\widehat{i}_t} \big) \, $  are either old factors  $ \, \big( 1 + \eta_k \, Y_{i_k} \big) \, $,  {\sl with no inversions among them},  or new factors such that  $ \, \widehat{\eta}_t \in \fa^5 \, $.
                                             \par
   In order to conclude, we can iterate at will this procedure: then   --- as  $ \, \fa^n = \{0\} \, $  for  $ \, n > N \, $  ---   after finitely many steps we shall no longer find any new factor coming in;
%
% therefore, we eventually find a last factorization of  $ g $  of the form
%
 thus, we eventually find
 \vskip-9pt
  $$  g  \,\; = \;\,  \widetilde{g}_{{}_0} \cdot {\textstyle \prod_{\ell=1}^{\widetilde{N}}} \big( 1 + \widetilde{\eta}_\ell \, Y_{\widetilde{i}_\ell} \big)  \,\; = \;\,  \widetilde{g}_{{}_0} \cdot {\textstyle \prod\limits^\rightarrow} {\textstyle {}_{\ell=1}^{\widetilde{N}}} \big( 1 + \widetilde{\eta}_\ell \, Y_{\widetilde{i}_\ell} \big)  $$
 \vskip-3pt
\noindent
 in which  $ \, \widetilde{g}_{{}_0} \in N_G^{[2]}(A) \, $  and  $ \,\; {\textstyle \prod_{\ell=1}^{\widetilde{N}}} \big( 1 + \widetilde{\eta}_\ell \, Y_{\widetilde{i}_\ell} \big) \, = \, {\textstyle \prod\limits^\rightarrow} {}_{\ell=1}^{\widetilde{N}} \big( 1 + \widetilde{\eta}_\ell \, Y_{\widetilde{i}_\ell} \big) \; \in \; G_-^{\scriptscriptstyle \,<}(A) \;\, $  {\sl  is an ordered product},  as required: this means exactly  $ \phantom{\Big|} g \, \in \, N_G^{[2]}(A) \cdot G_-^{\scriptscriptstyle \,<}(A) \; $,  \, thus the  {\it Claim\/}  is proved.
 \vskip2pt
   Our  {\it Claim\/}  above ensures that  $ \; G^-(A) \, \subseteq \, N_G^{[2]}(A) \cdot G_-^{\scriptscriptstyle \,<}(A) \; $.  Now we have to prove the converse inclusion.  To this end, recall that  $ N_G^{[2]}(A) $  is generated by elements of the form  $ \, (1\,+\,c\,X) \, $  with  $ \, c \in A_\uno^{[2]} \, $  and  $ \, X \in [\fg_\uno , \fg_\uno] \, $   --- still with notation of type  $ \, \big(\, 1 + c \, X \,\big) := \exp(c\,X) \, $,  as before ---   therefore  $ \, c = \sum_s \alpha'_s\,\alpha''_s \, $  and  $ \, X = \sum_r \big[ Y'_r , Y''_r \big] \, $  for some  $ \, \alpha'_s , \alpha''_s \in A_\uno \, $  and some  $ \, Y'_r , Y''_r \in \fg_\uno \, $.  Now, inside  $ G_\zero\big(A_\zero\big) $  we always have relations of the form
 \vskip-5pt
  $$  \big(\, 1 + a_1 \, Z \,\big) \cdot \big(\, 1 + a_2 \, Z \,\big)  \; = \;  \big(\, 1 + \big( a_1 + a_2 \big) \, Z \,\big)  \; = \;  \big(\, 1 + a_2 \, Z \,\big) \cdot \big(\, 1 + a_1 \, Z \,\big)  $$
\noindent
 for all  $ \, Z \in \fg_\zero \, $  and  $ \, a_1 , a_2 \in A_\zero \, $  such that  $ \, a_1^{\,2} = 0 = a_2^{\,2} \, $.  Applying this repeatedly to  $ \; \big(\, 1 + c \, X \,\big) \, = \, \big(\, 1 + \big( \sum_s \alpha'_s\,\alpha''_s \big) \, X \,\big) \; $  yields
\begin{equation}  \label{eq:expans_1+c.X}
  \big(\, 1 + c \, X \,\big)  \; = \;  \big(\, 1 + \big(\, \textstyle{\sum_s} \alpha'_s\,\alpha''_s \big) \, X \,\big)  \; = \;  \textstyle{\prod_s} \big(\, 1 + \alpha'_s\,\alpha''_s \, X \big)
\end{equation}
where the factors in the final product can be taken  {\sl in any order},  as they mutually commute.
                                                             \par
   Now recall instead that  $ \, X = \sum_r \big[ Y'_r , Y''_r \,\big] \, $,  \, hence each factor  $ \, \big( 1 + \alpha'_s\,\alpha''_s \, X \big) \, $  in  \eqref{eq:expans_1+c.X}  reads
\begin{equation}  \label{eq:1st-expans_1+a'a".X}
  \big( 1 + \alpha'_s\,\alpha''_s \, X \big)  \; = \;  \Big( 1 + \textstyle{\sum_r} \, \alpha'_s\,\alpha''_s \, \big[ Y'_r , Y''_r \big] \Big)
\end{equation}
In addition, inside  $ G_\zero\big(A_\zero\big) $  we also have, for all  $ \, Z_1, Z_2 \in \fg_\zero \, $  and  $ \, a \in A_\zero \, $  such that  $ \, a^2 = 0 \, $,  relations of the form
  $$  \big(\, 1 + a \, Z_1 \,\big) \cdot \big(\, 1 + a \, Z_2 \,\big)  \; = \;  \big(\, 1 + a \, \big( Z_1 + Z_2 \big) \,\big)  \; = \;  \big(\, 1 + a \, Z_2 \,\big) \cdot \big(\, 1 + a \, Z_1 \,\big)  $$
With repeated applications of this to the right-hand side term in  \eqref{eq:1st-expans_1+a'a".X}  above we eventually get
\begin{equation}  \label{eq:2nd-expans_1+a'a".X}
  \big( 1 + \alpha'_s\,\alpha''_s \, X \big)  \; = \;  \Big( 1 + \textstyle{\sum_r} \, \alpha'_s\,\alpha''_s \, \big[ Y'_r , Y''_r \big] \Big)  \; = \;  \textstyle{\prod_r} \Big( 1 + \, \alpha'_s\,\alpha''_s \, \big[ Y'_r , Y''_r \big] \Big)
\end{equation}
   \indent   As next step, recall that in  $ \, G(A) \, $  also hold relations of the form
  $$  \big( 1 + \eta'' \, Y'' \big) \cdot \big( 1 + \eta' \, Y' \big)  \; = \;  \Big( 1 + \, \eta' \, \eta'' \, \big[Y',Y''\big] \Big) \cdot \big( 1 + \eta' \, Y' \big) \cdot \big( 1 + \eta'' \, Y'' \big)  $$
that we can re-shape as
%
%%
%   $$  \Big( 1 + \eta' \eta'' \big[Y',Y''\big] \Big)  =  \big( 1 + \eta'' Y'' \big)
% \big( 1 + \eta' Y' \big) {\big( 1 + \eta'' Y'' \big)}^{-1} {\big( 1 + \eta' Y' \big)}^{-1}
% =  \Big(\! \big( 1 + \eta'' Y'' \big) \, , \big( 1 + \eta' Y' \big) \!\Big)
%%
%
\begin{equation}  \label{eq:_1+e'e".[Y',Y"]_=_commutator}
  \Big(\, 1 \, + \, \eta' \, \eta'' \, \big[Y',Y''\big] \Big)  \,\; = \;\,  \Big( \big( 1 + \eta'' \, Y'' \big) \, , \big( 1 + \eta' \, Y' \big) \Big)
\end{equation}
where in right-hand side we used standard group-theoretical notation  $ \; \big(\, a \, , b \,\big) := a\,b\,a^{-1}\,b^{-1} \; $  for the commutator of any two elements  $ a $  and  $ b $  in a given group. Now  \eqref{eq:_1+e'e".[Y',Y"]_=_commutator}  together with  \eqref{eq:2nd-expans_1+a'a".X}  gives
\begin{equation}  \label{eq:3rd-expans_1+a'a".X}
  \big( 1 + \alpha'_s\,\alpha''_s \, X \big)  \; = \;  \textstyle{\prod_r} \Big( 1 + \, \alpha'_s\,\alpha''_s \, \big[ Y'_r , Y''_r \big] \Big)  \; = \;  \textstyle{\prod_r}  \Big( \big( 1 + \alpha''_s \, Y''_r \big) \, , \big( 1 + \alpha'_s \, Y'_r \big) \Big)
\end{equation}
Eventually, matching  \eqref{eq:3rd-expans_1+a'a".X}  with  \eqref{eq:expans_1+c.X}  we get
  $$  \big(\, 1 + c \, X \,\big)  \,\; = \;\,  \textstyle{\prod_s} \big( 1 + \alpha'_s\,\alpha''_s \, X \big)  \,\; = \;\,  \textstyle{\prod_{s,r}} \, \Big( \big( 1 + \alpha''_s \, Y''_r \big) \, , \big( 1 + \alpha'_s \, Y'_r \big) \Big)  \;\; \in \;\;  \big\langle G_-^{\scriptscriptstyle \,<}(A) \big\rangle  \, = \, G^-(A)  $$
(where the factors in the last product can be taken in any order).  Thus the element  $ \, \big(\, 1 + \, c \, X \,\big) \, $  belongs to  $ \, G^-(A) \, $;  since  $ N_G^{[2]}(A) $  is generated by such elements, we get  $ \; N_G^{[2]}(A) \subseteq G^-(A) \; $,  whence clearly  $ \; N_G^{[2]}(A) \cdot G_-^{\scriptscriptstyle \,<}(A) \subseteq G^-(A) \; $  and we are done.
%
%%%
%
 \vskip7pt
   Now we go and prove  \eqref{eq:factor-1_N_G}:  like before, it is enough to prove the left-hand side, that is  $ \; N_G(A) \, = \, N_{G_\zero}\big(A_\zero\big) \cdot G_-^{\scriptscriptstyle \,<}(A) \; $,  \; by left-right symmetry.
                                            \par
 Thanks to  Proposition \ref{prop:2nd-descr_N_G(A)},  we can take as generators of  $ N_G(A) $  the elements of the set
\begin{equation}  \label{eq:gen-set-N_G(A)}
  \Big\{\, \exp\big(t_j X_j\big) \, , \; \exp\big(\eta_i Y_i\big) = \big( 1 + \eta_i Y_i \big) \;\Big|\;\, t_j \in \nil{A}_\zero \, , \, \eta_i \in \nil{A}_\uno = A_\uno \, , \; \forall \, j \in J, \, i \in I \,\Big\}
\end{equation}
where  $ \, {\big\{ X_j \big\}}_{j \in J} \, $  is any  $ \K $--basis  of  $ \, \fg_\zero \, $  and  $ \, {\big\{ Y_i \big\}}_{i \in I} \, $  is our fixed, totally ordered  $ \K $--basis  of  $ \fg_\uno \, $.  Therefore, our aim is to prove that any  $ \, n \in N_G(A) \, $,  originally expressed as an unordered product of factors taken from  \eqref{eq:gen-set-N_G(A)},  can be ``re-ordered'' so to read as an  {\sl ordered product\/}  of the form
 $ \; n_{{}_\zero} \, {\textstyle \prod\limits_{i \in I}^\rightarrow} \big( 1 + \widehat{\eta}_i \, Y_i \big) \; $,
 \, for some  $ \, n_{{}_\zero} \in N_{G_\zero}\big(A_\zero\big) \, $  and  $ \, \widehat{\eta}_i \in A_\uno \, $,  which does belong to  $ \; N_{G_\zero}\big(A_\zero\big) \cdot G_-^{\scriptscriptstyle \,<}(A) \; $.
 \vskip5pt
   First of all, whenever in our original product  $ n $  we have two consecutive factors  $ \, \big( 1 + \eta_s \, Y_{i_s} \big) \, $  and  $ \, n'_{{}_\zero} := \exp\big(t_j X_j\big) \in N_{G_\zero}\big(A_\zero\big) \, \big( \subseteq G_\zero\big(A_\zero\big) \,) \, $   --- that is, we have a ``sub-product'' of the form  $ \, \big( 1 + \eta_s \, Y_{i_s} \big) \cdot n'_{{}_\zero} \, $  ---   using relations  {\it (b)\/}  and  {\it (d)\/}  in  Lemma \ref{lemma:relations-in-G(A)}  we can re-write this subproduct as
\begin{equation}  \label{eq:(1+Y)n'=n'(1+Ad(Y))}
  \big( 1 + \eta_s \, Y_{i_s} \big) \, n'_{{}_\zero}  = \,  n'_{{}_\zero} \Big( 1 + \eta_s \, \Ad\big({(n'_{{}_\zero})}^{-1}\big)\big(Y_{i_s}\big) \Big)  = \,  n'_{{}_\zero} \Big( 1 + \eta_s {\textstyle \sum\limits_{j \in I}} c_{s,j} Y_j \Big)  = \,  n'_{{}_\zero} \, {\textstyle \prod\limits_{j \in I}^\rightarrow} \big( 1 + \eta_s \, c_{s,j} Y_j \,\big)
\end{equation}
for suitable  $ \, c_{s,j} \in \K \, $  ($ \, j \in I \, $);  note in particular that the rightmost term in the chain of identities  \eqref{eq:(1+Y)n'=n'(1+Ad(Y))}  does belong to  $ \; N_{G_\zero}\big(A_\zero\big) \cdot G_-^{\scriptscriptstyle \,<}(A) \, $,  \, i.e.\ it has the form we are looking form.
 Applying this procedure, whenever in our element  $ n $   --- written as a product as above ---   we have a factor of type  $ \, n'_{{}_\zero} \in N_{G_\zero}\big(A_\zero\big) \, $  that occurs on the right of any factor of type  $ \, \big( 1 + \eta_s \, Y_{i_s} \big) \, $,  we can ``move the former to the left of the latter'' in the sense that we apply  \eqref{eq:(1+Y)n'=n'(1+Ad(Y))}.  Then, after finitely many repetitions of this move we end up with a new factorization of  $ n $  of the form
\begin{equation}  \label{eq: n = n"Y"}
  n  \,\; = \;\,  n''_{{}_\zero} \cdot \, {\textstyle \prod\limits_{s=1}^N} \, {\textstyle \prod\limits_{j \in I}^\rightarrow} \big( 1 + \eta_s \, k_{s,j} \, Y_j \,\big)
\end{equation}
for some  $ \, n''_{{}_\zero} \in N_{G_\zero}\big(A_\zero\big) \, $  and  $ \, k_{s,j} \in \K \, $  (where  $ N $  is the number of factors of type  $ \, \big( 1 + \eta_s \, Y_{i_s} \big) \, $  occurring in the initial factorization of  $ n \, $.  Now, in this last factorization, the right-hand side gives
  $$  {\textstyle \prod\limits_{s=1}^N} \, {\textstyle \prod\limits_{j \in I}^\rightarrow} \big( 1 + \eta_s \, k_{s,j} \, Y_j \,\big) \; \in \;  G^-(A) \; = \;  N_G^{[2]}(A) \cdot G_-^{\scriptscriptstyle \,<}(A)  $$
thanks to  \eqref{eq:factor-1_G^-}.  This together with  \eqref{eq: n = n"Y"}  gives
  $$  n  \; \in \;  N_{G_\zero}\big(A_\zero\big) \cdot G^-(A) \, = \, N_{G_\zero}\big(A_\zero\big) \cdot N_G^{[2]}(A) \cdot G_-^{\scriptscriptstyle \,<}(A) \, = \, N_{G_\zero}\big(A_\zero\big) \cdot G_-^{\scriptscriptstyle \,<}(A)  $$
--- because  $ \, N_G^{[2]}(A) \, $  is a subgroup of  $ \, N_{G_\zero}\big(A_\zero\big) \, $,  by construction ---   and we are done.
 \vskip7pt
%
%%%
%
   Finally, as to  \eqref{eq:factor-1_G}  we prove it via the following chain of identities:
  $$  \displaylines{
   \qquad \quad   G(A)  \, = \,  G_\zero(\K) \cdot N_G(A)  \, = \,  G_\zero(\K) \cdot \big( N_{G_\zero}\big(A_\zero\big) \cdot G_-^{\scriptscriptstyle \,<}(A) \big)  \, =   \hfill  \cr
   \hfill   = \,  \big( G_\zero(\K) \cdot N_{G_\zero}\big(A_\zero\big) \big) \cdot G_-^{\scriptscriptstyle \,<}(A)  \, = \,  G_\zero\big(A_\zero\big) \cdot G_-^{\scriptscriptstyle \,<}(A)  \, = \,  G_\zero(A) \cdot G_-^{\scriptscriptstyle \,<}(A)   \quad \qquad  }  $$
where we first used Boseck's splitting for  $ G(A) $   --- cf.\  \eqref{Boseck-split_Lie-sgrp}  ---   and then  \eqref{eq:factor-1_N_G}.
 \vskip7pt
   {\it (c)}\,  We begin with the proof of  \eqref{eq:factor-2_G^-},  for which again it is enough to prove the left-hand side, that is  $ \; G^-(A) \, = \, N_G^{[2]}(A) \cdot \exp\!\big( A_\uno \!\otimes_\K \fg_\uno \big) \; $.
                                                              \par
   First of all, the inclusion  $ \; N_G^{[2]}(A) \cdot \exp\!\big( A_\uno \!\otimes_\K \fg_\uno \big) \subseteq G^-(A) \; $  follows at once from  \eqref{eq:factor-1_G^-}  together with claim  {\it (a)}\,.  Moreover, again by claim  {\it (a)\/}  we have  $ \; G^-(A) = \big\langle \exp\!\big( A_\uno \!\otimes_\K \fg_\uno \big) \big\rangle \; $  so it is enough to prove that the product of any two generators  $ \, \exp\!\big(\cY'\big) \, $  and  $ \, \exp\!\big(\cY''\big) \, $  of  $ \, G^-(A) = \big\langle \exp\!\big( A_\uno \!\otimes_\K \fg_\uno \big) \big\rangle \, $  lies in  $ \, N_G^{[2]}(A) \cdot \exp\!\big( A_\uno \!\otimes_\K \fg_\uno \big) \, $.  In fact, this follows at once from the identity
 \vskip4pt
   \centerline{ $ \exp\!\big( \cY' \big) \, \exp\!\big( \cY'' \big)  \,\; = \;\,  \exp\!\Big( P_\zero^{(d_1)}\big( \cY' , \cY'' \big) \Big) \, \exp\!\Big(\, \cY' + \cY'' + P_\uno^{(d_1)}\big( \cY' , \cY'' \big) \Big) $ }
\noindent
 in  Lemma \ref{lemma:relations-in-G(A)}{\it (h)},  just because it tells us exactly that  $ \; \exp\!\Big( P_\zero^{(d_1)}\big( \cY' , \cY'' \big) \Big) \in N_G^{[2]}(A) \; $  and  $ \; \exp\!\Big(\, \cY' + \cY'' + P_\uno^{(d_1)}\big( \cY' , \cY'' \big) \Big) \in \exp\!\big( A_\uno \!\otimes_\K \fg_\uno \big) \; $.
 \vskip7pt
   The same argument used above also applies to prove  \eqref{eq:factor-2_N_G}.
 \vskip7pt
   Finally, \eqref{eq:factor-2_G}  follows from the chain of identities
  $$  \displaylines{
   G(A)  \, = \,  G_\zero(\K) \cdot N_G(A)  \, = \,  G_\zero(\K) \cdot \big( N_{G_\zero}\big(A_\zero\big) \cdot \exp\!\big( A_\uno \!\otimes_\K \fg_\uno \big) \big)  \, =   \hfill  \cr
   \hfill   = \,  \big( G_\zero(\K) \cdot N_{G_\zero}\big(A_\zero\big) \big) \cdot \exp\!\big( A_\uno \!\otimes_\K \fg_\uno \big)  \, = \,  G_\zero\big(A_\zero\big) \cdot \exp\!\big( A_\uno \!\otimes_\K \fg_\uno \big)  \, = \,  G_\zero(A) \cdot \exp\!\big( A_\uno \!\otimes_\K \fg_\uno \big)  }  $$
using Boseck's splitting for  $ G(A) $   --- cf.\  \eqref{Boseck-split_Lie-sgrp}  ---   and then \eqref{eq:factor-2_N_G}.
\end{proof}

\vskip7pt

   The previous proposition provides remarkable factorizations for the  $ A $--points  of the Lie supergroup  $ G $  and their remarkable subgroups  $ N_G $  and  $ G^- $.  Our ultimate goal is to improve such a result, eventually achieving stronger factorization results: in case of  $ G $  itself, this will be what is (more or less) known as ``global splitting'' for Lie supergroups.  We still need
 a technical lemma:

\medskip

\begin{lemma}  \label{lemma-triv_prod}
  Given a Lie supergroup  $ G $  and  $ \, A \in \Wsalg_\bk \, $,  let  $ \, \zeta_i \in A_\uno \, $  ($ \, i \in I $).  Then:
 \vskip5pt
   {\it (a)}\,  if
 $ \;\; g \, := \, {\textstyle \prod\limits_{i \in I}^\rightarrow} \big( 1 + \zeta_i \, Y_i \,\big) \; \in \; G_\zero(A) \,\bigcap\, G_-^{\scriptscriptstyle \,<}(A) \; $,
 \; then  $ \; \zeta_i = 0 \; $  for all  $ \, i \in I \, $;
 \vskip5pt
   {\it (b)}\,  if
 $ \;\; g \, := \, \exp\big(\, {\textstyle \sum_{i \in I}}\, \zeta_i \, Y_i \,\big) \; \in \; G_\zero(A) \,\bigcap\, \exp\!\big( A_\uno \!\otimes_\K \fg_\uno \big) \; $,
 \; then  $ \; \zeta_i = 0 \; $  for all  $ \, i \in I \, $.
\end{lemma}

\begin{proof}
 {\it (a)}\,  Recall that, by definition, we have  $ \, G(A) := \! \coprod\limits_{x \in |G|} \! \Hom_{\salg_\K} \big( \cO_{|G|,x} \, , A \big) \, $;  \, therefore, it makes sense to formally expand the product defining  $ g $  as
\begin{equation}  \label{eq:expans-prod_g}
 g  \, := \,  {\textstyle \prod\limits_{i \in I}^\rightarrow} \big( 1 + \zeta_i \, Y_i \,\big)  \, = \,  1 + {\textstyle \sum_{i \in I}} \, \zeta_i \, Y_i \, + \, \cO(2)
\end{equation}
where  $ \, \cO(2) \, $  is a short-hand notation for  ``{\sl additional summands of higher order in the  $ \zeta_i $'s\,}''.  Let  $ \, \fa := \big( {\{\zeta_i\}}_{i \in I} \big) \, $  be the ideal of  $ A $  generated by the  $ \zeta_i $'s;  then  \eqref{eq:expans-prod_g}  yields
\begin{equation}  \label{eq:expans-prod_g_mod-2}
 {[g]}_2  \, := \,  1 + {\textstyle \sum_{i \in I}} \, {[\zeta_i]}_2 \, Y_i
\end{equation}
inside  $ \; G\big( A \big/ \fa^2 \big) := \! \coprod\limits_{x \in |G|} \! \Hom_{\salg_\K} \big( \cO_{|G|,x} \, , A \big/ \fa^2 \big) \; $.  On the other hand, the assumption that  $ \; g \, \in \, G_\zero(A) \bigcap G_-^{\scriptscriptstyle \,<}(A) \; $  implies  $ \; {[g]}_2 \, \in \, G_\zero\big( A \big/ \fa^2 \big) \bigcap G_-^{\scriptscriptstyle \,<}\big( A \big/ \fa^2 \big) \; $  as well, hence in particular   --- thinking of  $ {[g]}_2 $  as an  $ A \big/ \fa^2 \, $--valued  map ---  we have  $ \; \text{\sl Im}\big( {[g]}_2 \big) \subseteq {\big( A \big/ \fa^2 \big)}_\zero \; $.
                                                                       \par
   Now, as  $ \, {\big\{ Y_i \big\}}_{i \in I} \, $  is a  $ \K $--basis  of  $ \fg_\uno \, $,  there exists a local system of coordinates around the unit point  $ \, 1_{{}_G} \in |G| \, $,  say  $ \, {\{ y_i \}}_{i \in I} \, $,  such that  $ \, Y_i(y_j) = \delta_{i,j} \, $  for all  $ \, i, j \in I \, $.  Then  \eqref{eq:expans-prod_g_mod-2}  gives
 $ \; {[g]}_2(y_j) \, := \, 1 + {\textstyle \sum_{i \in I}} \, {[\zeta_i]}_2 \, Y_i(y_j) \, = \, {[\zeta_j]}_2 \; $,
 \, in particular  $ \, {[g]}_2(y_j) = {[\zeta_j]}_2 \in {\big( A \big/ \fa^2 \big)}_\uno \, $;
 this together with  $ \; \text{\sl Im}\big( {[g]}_2 \big) \subseteq {\big( A \big/ \fa^2 \big)}_\zero \; $  implies  $ \; {[\zeta_j]}_2 = {[0]}_2 \in A \big/ \fa^2 \; $,  \, i.e.\  $ \; \zeta_j \in \fa^2 = {\big( {\{ \zeta_i \}}_{i \in I} \big)}^2 \; $,  \, for all  $ \, j \in I \, $:  thus  $ \, \zeta_j \in \fa^n \, $  for  $ \, n \in \N \, $,  $ \, j \in I \, $.  But  $ \, \fa^n = 0 \, $  for  $ \, n \gg 0 \, $,  hence  $ \, \zeta_j = 0 \, $  for all  $ \, j \in I \, $,  q.e.d.
%%%
%
% but this together with  $ \; \text{\sl Im}\big( {[g]}_2 \big) \subseteq {\big( A \big/ \fa^2 \big)}_\zero % \; $  implies  $ \; {[\zeta_j]}_2 = {[0]}_2 \in A \big/ \fa^2 \; $,  \, that is  $ \; \zeta_j \in \fa^2
% = {\big( {\{ \zeta_i \}}_{i \in I} \big)}^2 \; $,  \, for all  $ \, j \in I \, $:  it follows that  $ \, % \zeta_j \in \fa^n \, $  for all  $ \, n \in \N \, $  (and  $ \, j \in I \, $).  By construction we have
% also  $ \, \fa^n = 0 \, $  for  $ \, n \gg 0 \, $,  thus we end up with  $ \, \zeta_j = 0 \, $  for all  % $ \, j \in I \, $,  q.e.d.
%
 \vskip5pt
 {\it (b)}\,  Acting like in  {\it (a)},  we can expand  $ \,\; g \, := \, \exp\big(\, {\textstyle \sum_{i \in I}}\, \zeta_i \, Y_i \,\big) \; \in \; G_\zero(A) \,\bigcap\, \exp\!\big( A_\uno \!\otimes_\K \fg_\uno \big) \;\, $  into a formal power series
%
%%%
%  (actually a finite sum indeed) of the form
%   $$  g  \, := \,  \exp\big(\, {\textstyle \sum_{i \in I}}\, \zeta_i \, Y_i \,\big)
% \, = \,  1 + {\textstyle \sum_{i \in I}} \, \zeta_i \, Y_i \, + \, \cO(2)  $$
%%%
%
 (indeed a finite sum) of the form
  $ \; g \, := \, \exp\big(\, {\textstyle \sum_{i \in I}}\, \zeta_i \, Y_i \,\big)  \, = \,  1 + {\textstyle \sum_{i \in I}} \, \zeta_i \, Y_i \, + \, \cO(2) \, $,  \,
 just like in  \eqref{eq:expans-prod_g}.  Then the same argument applies again, and leads to conclusion.
\end{proof}

\medskip

   Finally, we can now state and prove the main result of the present subsection:

\medskip

\begin{theorem}  \label{thm:dir-prod-fact-G - gen}
 {\sl (existence of Global Splittings for Lie supergroups)}
 \vskip3pt
   Let  $ \, G $  be a Lie supergroup, and  $ \, \fg $  its tangent Lie superalgebra.
 \vskip5pt
   {\it (a)} \,  The restriction of group multiplication in  $ G $  provides isomorphisms of (set-valued) functors
 \vskip-15pt
  $$  \displaylines{
   N_G^{[2]} \times G_-^{\scriptscriptstyle \,<}  \; \cong \;  G^-  \;\;\quad ,  \qquad\;\;
   N_{G_\zero} \times G_-^{\scriptscriptstyle \,<} \; \cong \; N_G  \;\quad ,  \qquad\;
   G_\zero \times G_-^{\scriptscriptstyle \,<} \; \cong \; G  \cr
   G_-^{\scriptscriptstyle \,<} \times N_G^{[2]} \; \cong \;  G^-  \;\;\quad ,  \qquad\;\;
   G_-^{\scriptscriptstyle \,<} \times N_{G_\zero} \; \cong \; N_G  \;\quad ,  \qquad\;
   G_-^{\scriptscriptstyle \,<} \times G_\zero \; \cong \; G  \cr
   N_G^{[2]} \times \exp\!\big( {(-)}_\uno \!\otimes_\K \fg_\uno \big) \; \cong \;  G^-  \,\; ,  \!\!\quad
   N_{G_\zero} \times \exp\!\big( {(-)}_\uno \!\otimes_\K \fg_\uno \big) \; \cong \; N_G  \,\; ,  \!\!\quad
   G_\zero \times \exp\!\big( {(-)}_\uno \!\otimes_\K \fg_\uno \big) \; \cong \; G  \cr
   \exp\!\big( {(-)}_\uno \!\otimes_\K \fg_\uno \big) \times N_G^{[2]} \; \cong \;  G^-  \,\; ,  \!\!\quad
   \exp\!\big( {(-)}_\uno \!\otimes_\K \fg_\uno \big) \times N_{G_\zero} \; \cong \; N_G  \,\; ,  \!\!\quad
   \exp\!\big( {(-)}_\uno \!\otimes_\K \fg_\uno \big) \times G_\zero \; \cong \; G  }  $$
 with  $ \, \exp\!\big( {(-)}_\uno \!\otimes_\K \fg_\uno \big) \, $  the set-valued functor  $ \; \Wsalg_\K \longrightarrow \text{\rm ($\cat{sets}$)} \; $  given by  $ \; A \mapsto \exp\!\big( A_\uno \!\otimes_\K \fg_\uno \big) \; $.
%
%%%
%%%%%
% %
% \begin{equation}  \label{eq:dir-factor_G^-}
%   N_G^{[2]} \times G_-^{\scriptscriptstyle \,<}  \; \cong \;  G^-  \;\;\quad ,
% \qquad\;\;
%   G_-^{\scriptscriptstyle \,<} \times N_G^{[2]} \; \cong \; G^-
% %
% \end{equation}
% %
% \begin{equation}  \label{eq:dir-factor_N_G}
%    N_{G_\zero} \times G_-^{\scriptscriptstyle \,<} \; \cong \; N_G  \;\quad ,
%   \qquad\;  G_-^{\scriptscriptstyle \,<} \times N_{G_\zero} \; \cong \; N_G
% %
% \end{equation}
% %
% \begin{equation}  \label{eq:dir-factor_G}
%    G_\zero \times G_-^{\scriptscriptstyle \,<} \; \cong \; G  \;\;\;\;\quad ,
% \qquad\;\;\;\;  G_\zero \times G \; = \; G_-^{\scriptscriptstyle \,<}
% %
% \end{equation}
% %
%%%
%%%%%
%
 \vskip7pt
   {\it (b)} \,  Setting notation  $ \, d_1 := \text{\it dim}_{\,\K}\big(\fg_\uno\big) = |I| \, $,  there exist isomorphisms of (set-valued) functors  $ \; \mathbb{A}_\K^{0|d_1} \! \cong G_-^{\scriptscriptstyle \,<} \; $  and  $ \; \mathbb{A}_\K^{0|d_1} \! \cong \exp\!\big( {( - )}_\uno \!\otimes_\K \fg_\uno \big) \; $,  \, given on  $ A $--points   --- for every  $ \, A \in \Wsalg_\K \, $  --- by
 \vskip-5pt
  $$  \mathbb{A}_\K^{0|d_1}\!(A)  \; = \;  A_\uno^{\,d_1} \relbar\joinrel\longrightarrow\, G_-^{\scriptscriptstyle \,<}(A) \;\;\; ,  \qquad  {\big(\eta_i\big)}_{i \in I} \,\mapsto\, {\textstyle \prod\limits_{i \in I}^\rightarrow} (1 + \eta_i \, Y_i)  $$
 \vskip-12pt
\noindent
 and by
 \vskip-10pt
  $$  \mathbb{A}_\K^{0|d_1}\!(A)  \; = \;  A_\uno^{\,d_1} \relbar\joinrel\longrightarrow\, \exp\!\big( A_\uno \!\otimes_\K \fg_\uno \big) \;\;\; ,  \qquad  {\big(\eta_i\big)}_{i \in I} \,\mapsto\, \exp\big(\, {\textstyle \sum_{i \in I}} \, \eta_i \, Y_i \,\big)  $$
 \vskip-5pt
   {\it (c)} \,  There exist isomorphisms of (set-valued) functors
 \vskip-15pt
  $$  \displaylines{
   N_G^{[2]} \times \mathbb{A}_\K^{0|d_1}  \; \cong \;  G^-  \;\;\quad ,  \qquad\;\;
   N_{G_\zero} \times \mathbb{A}_\K^{0|d_1} \; \cong \; N_G  \;\quad ,  \qquad\;
   G_\zero \times \mathbb{A}_\K^{0|d_1} \; \cong \; G  \cr
   \mathbb{A}_\K^{0|d_1} \times N_G^{[2]} \; \cong \;  G^-  \;\;\quad ,  \qquad\;\;
   \mathbb{A}_\K^{0|d_1} \times N_{G_\zero} \; \cong \; N_G  \;\quad ,  \qquad\;
   \mathbb{A}_\K^{0|d_1} \times G_\zero \; \cong \; G  }  $$
%
%  given on  $ A $--points   --- for every  $ \, A \in \Wsalg_\K \, $  ---   respectively by
% %
%  \vskip-11pt
% %
%   $$  \big(\, g_{{}_0} \, , \, {\big(\eta_i\big)}_{i \in I} \big) \, \mapsto \, g_{{}_0} \cdot
% {\textstyle \prod\limits_{i \in I}^\rightarrow} (1 + \eta_i \, Y_i)  \qquad  \text{and}  \qquad
% \big( {\big(\eta_i\big)}_{i \in I} \, , \, g_{{}_0} \big) \, \mapsto \, {\textstyle
% \prod\limits_{i \in I}^\rightarrow} (1 + \eta_i \, Y_i) \cdot g_{{}_0}  $$
%
 induced in the obvious way by those in  {\it (a)\/}  and  {\it (b)}.
\end{theorem}

\begin{proof}
 {\it (a)} \,  The claim yields a (strong) refinement of the factorization results of  Proposition \ref{prop:fact-Liesgrp_G}{\it (b)},  now stated in functorial terms.  Just like for that proposition, it is enough to prove half of the results, say those in first and in third line; moreover, we bound ourselves to prove the claim concerning  $ G \, $,  as the others go along similarly.
 \vskip5pt
   Assume that  $ \; \hat{g}_+ \, \hat{g}_- \, = \, \check{g}_+ \, \check{g}_- \; $  with  $ \; \hat{g}_+ \, , \check{g}_+ \in G_\zero(A) \, $,  $ \, \hat{g}_- \, , \check{g}_- \in G_-^{\scriptscriptstyle \,<}(A) \; $;  then  $ \; g \, := \, \hat{g}_- \, \check{g}_-^{-1} \, = \, \hat{g}_+^{-1} \, \check{g}_+ \, \in \, G_\zero(A) \, $,  as  $ G_\zero(A) $  is a subgroup in  $ G(A) \, $.  Now  $ \, \hat{g}_- \, , \check{g}_- \in \, G_-^{\scriptscriptstyle \,<}(A) \, $  have the form  $ \, \hat{g}_- = {\textstyle \prod\limits^\rightarrow}_{i \in I} \big(\, 1 + \hat{\eta}_i \, Y_i \,\big) \, $  and  $ \, \check{g}_- = {\textstyle \prod\limits^\rightarrow}_{i \in I} \big( 1 + \check{\eta}_i Y_i \big) \, $,  so that  $ \, \check{g}_-^{\,-1} = {\textstyle \prod\limits^\leftarrow}_{i \in I} \big( 1 - \check{\eta}_i Y_i \big) \, $,  where once more  $ {\textstyle \prod\limits^\rightarrow} $  and  $ {\textstyle \prod\limits^\leftarrow} $  respectively denote an ordered and a reversely-ordered product.  Therefore we have
 \vskip-7pt
\begin{equation}  \label{eq: g = prod.dorp}
  G_\zero(A)  \;\; \ni \;\;  g  \; := \;  \hat{g}_- \, \check{g}_-^{-1}  \; = \;  {\textstyle \prod\limits_{i \in I}^\rightarrow} \big( 1 + \hat{\eta}_i Y_i \big) \, {\textstyle \prod\limits_{i \in I}^\leftarrow} \big( 1 - \check{\eta}_i Y_i \big)
\end{equation}
 \vskip-4pt
   Let  $ \, \fb := \big( {\big\{\, \check{\eta}_i \, \alpha_j \,\big\}}_{i, j \in I} \big) \, $  be the ideal of  $ A $  generated by the products of the  $ \check{\eta}_i $'s  and the  $ \alpha_j := \big( \hat{\eta}_j - \check{\eta}_j \big) $'s;  also, write  $ \; \pi_\fb : A \relbar\joinrel\relbar\joinrel\twoheadrightarrow A \,\big/\, \fb \; $  for the quotient map,  $ \, {[a\,]}_\fb := \pi_\fb(a) \, $  for  $ \, a \in A \, $,  and  $ \, G(\pi_\fb) : G(A) \relbar\joinrel\relbar\joinrel\rightarrow G\big( A \,\big/\, \fb \big) \, $  for the associated group morphism, with  $ \, {[y]}_\fb := G(\pi_\fb)(y) \, $  for  $ \, y \in G(A) \, $.  Now applying to  \eqref{eq: g = prod.dorp}  the commutation relations in  Lemma \ref{lemma:relations-in-G(A)}  we get
%
% for  $ \, G\big({[A]}_2\big) $
%
 \vskip-7pt
  $$  {[g]}_\fb  \; = \;  {\textstyle \prod\limits_{i \in I}^\rightarrow} \big( 1 + {[\hat{\eta}_i]}_2 \, Y_i \big) \, {\textstyle \prod\limits_{i \in I}^\leftarrow} \big( 1 - {[\check{\eta}_i]}_2 \, Y_i \big)  \; = \;  {\textstyle \prod\limits_{i \in I}^\rightarrow} \big( 1 + {[\alpha_i]}_2 \, Y_i \,\big) \;\; \in \;\;  G_-^{\scriptscriptstyle \,<}\big( A \,\big/\, \fb \big)  $$
 \vskip-4pt
\noindent
 Since it is also  $ \; {[g]}_\fb \, \in \, G_\zero\big( A \,\big/\, \fb \big) \, $,  \, we can apply  Lemma \ref{lemma-triv_prod}{\it (a)},  with  $ \, A \,\big/\, \fb \, $  playing the r{\^o}le of  $ A \, $,  thus finding  $ \, {[\alpha_i]}_\fb = {[0]}_\fb \in A \,\big/\, \fb \, $,  that is  $ \, \alpha_i \in \fb \, $,  for all  $ \, i \in I \, $.  In turn, this implies at once that  $ \, \alpha_i \in \fa^{\,n} \, $  (for  $ \, i \in I \, $)  for all  $ \, n \in \N_+ \, $,  where  $ \, \fa := \big( {\big\{\, \hat{\eta}_i \, \check{\eta}_j \,\big\}}_{i, j \in I} \big) \, $  is the ideal of  $ A $  generated by all the  $ \hat{\eta}_i $'s  and  $ \check{\eta}_j $'s;  but  $ \, \fa^{\,n} = \{0\} \, $  for  $ \, n \gg 0 \, $,  so  $ \, \hat{\eta}_i - \check{\eta}_i =: \alpha_i = 0 \, $,  i.e.\  $ \, \hat{\eta}_i = \check{\eta}_i \, $,  for all  $ \, i \in I \, $.  This yields  $ \, \hat{g}_- = \check{g}_- \, $,  and from this we get also  $ \, \hat{g}_+ = \check{g}_+ \, $,  \, q.e.d.
 \vskip5pt
   Next we prove the result for  $ G $  in third line: again, acting pointwise this amounts to proving that, for any  $ \, A \in \Wsalg_\bk \, $,  the multiplication map  $ \; G_\zero(A) \times \exp\!\big( A_\uno \!\otimes_\K \fg_\uno \big) \relbar\joinrel\relbar\joinrel\longrightarrow G(A) \; $   --- which is surjective by  \eqref{eq:factor-2_G}  in  Proposition  \ref{prop:fact-Liesgrp_G}{\it (b)}  ---   is also  {\sl injective\/}:  that is, we must show that, if  $ \; \hat{g}_+ \, \hat{g}_e = \check{g}_+ \, \check{g}_e \; $  for  $ \, \hat{g}_+ \, , \check{g}_+ \in G_\zero(A) \, $  and  $ \, \hat{g}_- \, , \check{g}_e \in \exp\!\big( A_\uno \!\otimes_\K \fg_\uno \big) \, $,  then  $ \, \hat{g}_+ = \check{g}_+ \, $  and  $ \, \hat{g}_e = \check{g}_e \; $.
                                                                   \par
   From  $ \; \hat{g}_+ \, \hat{g}_e = \check{g}_+ \, \check{g}_e \; $  we get  $ \; \check{g}_e \, \hat{g}_e^{-1} \, = \, \check{g}_+^{-1} \hat{g}_+ \, \in \, G_\zero(A) \; $.  Writing  $ \, \hat{g}_e \, $  and  $ \, \check{g}_e \, $  explicitly as  $ \; \hat{g}_e = \exp\!\big( \sum_{i \in I} \hat{\eta}_i \, Y_i \big) \; $  and  $ \; \check{g}_e = \exp\!\big( \sum_{i \in I} \check{\eta}_i \, Y_i \big) \; $   --- with  $ \, \hat{\eta}_i , \check{\eta}_i \in A_\uno \, $  ---   we find
\begin{equation}  \label{eq: cGe-hGe= - 1}
   \check{g}_e \, \hat{g}_e^{-1}  \; = \;  \exp\!\big(\, {\textstyle \sum_{i \in I}} \, \hat{\eta}_i \, Y_i \,\big) \, \exp\!\big( -{\textstyle \sum_{i \in I}} \, \check{\eta}_i \, Y_i \,\big)  \; = \;  \exp\!\Big( \big(\, {\textstyle \sum_{i \in I}} \, \hat{\eta}_i \, Y_i \,\big) * \big( -{\textstyle \sum_{i \in I}} \, \check{\eta}_i \, Y_i \,\big) \Big)
\end{equation}
where  ``$ \, * \, $''  denotes again the formal product given by the Campbell-Baker-Hausdorff formula.  Now using again  Lemma \ref{lemma:relations-in-G(A)}{\it (h)\/}  we get
  $$  \check{g}_e \, \hat{g}_e^{-1}  \; = \;\,  \exp\!\Big(\, P_\zero^{(d_1\!)} \Big(\, {\textstyle \sum_i} \, \hat{\eta}_i \, Y_i \, , -{\textstyle \sum_i} \, \check{\eta}_i \, Y_i \,\Big) \Big) \cdot\, \exp\!\Big(\, {\textstyle \sum_i} \, \big( \hat{\eta}_i \! - \check{\eta}_i \big) \, Y_i \, + \, P_\uno^{(d_1\!)} \Big(\, {\textstyle \sum_i} \, \hat{\eta}_i \, Y_i \, , -{\textstyle \sum_i} \, \check{\eta}_i \, Y_i \,\Big) \Big)  $$
with  $ \; \exp\!\Big(\, P_\zero^{(d_1\!)} \Big(\, {\textstyle \sum_i} \, \hat{\eta}_i \, Y_i \, , -{\textstyle \sum_i} \, \check{\eta}_i \, Y_i \,\Big) \Big) \in G_\zero(A) \ni \check{g}_e \, \hat{g}_e^{-1} \, $;  \, this in turn implies
  $$  \exp\!\Big(\, {\textstyle \sum_{i \in I}} \, \big( \hat{\eta}_i \! - \check{\eta}_i \big) \, Y_i \, + \, P_\uno^{(d_1\!)} \Big(\, {\textstyle \sum_{i \in I}} \, \hat{\eta}_i \, Y_i \, , -{\textstyle \sum_{i \in I}} \, \check{\eta}_i \, Y_i \,\Big) \Big)  \; \in \;  G_\zero(A)  $$
as well, hence eventually
\begin{equation}  \label{eq: exp_in_G0-exp(A1xg1)}
   \exp\!\Big(\, {\textstyle \sum_{i \in I}} \, \big( \hat{\eta}_i \! - \check{\eta}_i \big) \, Y_i \, + \, P_\uno^{(d_1\!)} \Big(\, {\textstyle \sum_{i \in I}} \, \hat{\eta}_i \, Y_i \, , -{\textstyle \sum_{i \in I}} \, \check{\eta}_i \, Y_i \,\Big) \Big)  \; \in \;  G_\zero(A) \,\bigcap\, \exp\!\big( A_\uno \!\otimes_\K \fg_\uno \big)
\end{equation}
Now consider any Lie monomial of degree greater than 1, say  $ \mathcal{M} \, $:  denoting by  $ \, \mathcal{M}\big(\ell',\ell''\big) \, $  any arbitrary way of filling it with two Lie variables  $ \ell' $  and  $ \ell'' \, $, we always have  $ \, \mathcal{M}\big(\ell',\ell''-\ell'\big) = \mathcal{M}\big(\ell',\ell''\big) \, $.  As an application, for  $ \, \hat{\cY} := {\textstyle \sum_{i \in I}} \, \hat{\eta}_i \, Y_i \, $,  $ \, \check{\cY} := {\textstyle \sum_{i \in I}} \, \check{\eta}_i \, Y_i \, $  and  $ \, \cY_\alpha := {\textstyle \sum_{i \in I}} \, \alpha_i \, Y_i \, = \, \hat{\cY} - \check{\cY} \, $,  this gives
  $ \, \mathcal{M}\big( \hat{\cY} \, , -\check{\cY} \big) = \mathcal{M}\big( \hat{\cY} \, , \cY_\alpha - \hat{\cY} \big) = \mathcal{M}\big( \hat{\cY} \, , \cY_\alpha \big) \, $.
By  Lemma \ref{lemma:relations-in-G(A)}{\it (h)\/}  we know that  $ P_\zero^{(d_1\!)\!}(x,y) $  is a  $ \K $--linear  combination of Lie monomials of degree greater than 1, hence we can conclude that
  $$  P_\zero^{(d_1\!)\!}\big( \hat{\cY} \, , -\check{\cY} \big)  \; = \;  P_\zero^{(d_1\!)\!}\big( \hat{\cY} \, , \cY_\alpha \big)  \; = \;  P_\zero^{(d_1\!)\!}\Big(\, {\textstyle \sum_{i \in I}} \, \hat{\eta}_i \, Y_i \, , {\textstyle \sum_{i \in I}} \, \alpha_i \, Y_i \Big)  \; = \;  {\textstyle \sum_{i \in I}} \, \beta_i \, Y_i  $$
where in the right-hand side expansion of  $ \; P_\zero^{(d_1\!)\!}\Big(\, {\textstyle \sum_{i \in I}} \, \hat{\eta}_i \, Y_i \, , {\textstyle \sum_{i \in I}} \, -\check{\eta}_i \, Y_i \Big) \, = \, P_\zero^{(d_1\!)\!}\big( \hat{\cY} \, , -\check{\cY} \big) \in A_\uno \!\otimes \fg_\uno \; $  we have  $ \, \beta_i \in \fb \, $  for all  $ \, i \in I \, $,  with  $ \, \fb := \big( {\big\{\, \check{\eta}_i \, \alpha_j \big\}}_{i, j \in I} \big) \, $  the ideal of  $ A $  given above.  Then  \eqref{eq: exp_in_G0-exp(A1xg1)}  reads
  $$  \exp\!\Big(\, {\textstyle \sum_{i \in I}} \, \big( \alpha_i + \beta_i \big) \, Y_i \,\Big)  \; \in \;  G_\zero(A) \,\bigcap\, \exp\!\big( A_\uno \!\otimes_\K \fg_\uno \big)  $$
which by Lemma \ref{lemma-triv_prod}{\it (b)\/}  implies  $ \; \big( \alpha_i + \beta_i \big) = 0 \, $,  \, hence  $ \; \alpha_i = -\beta_i \in \fb \subseteq \fa^2 \, $,  \, for all  $ \, i \in I \, $,  where  $ \, \fa := \big( {\big\{\, \hat{\eta}_i \, \check{\eta}_j \,\big\}}_{i, j \in I} \big) \, $  is the ideal of  $ A $  introduced above.  By construction this implies  $ \; \alpha_i \in \fa^n \; $  for all  $ \, i \in I \, $  and  $ \, n \in \N \, $,  hence   --- since  $ \, \fa^{\,n} = \{0\} \, $  for  $ \, n \gg 0 \, $ ---   also  $ \, \hat{\eta}_i - \check{\eta}_i =: \alpha_i = 0 \, $,  i.e.\  $ \, \hat{\eta}_i = \check{\eta}_i \, $,  for all  $ \, i \in I \, $.  This means that  $ \, \hat{g}_e = \check{g}_e \, $,  whence  $ \, \hat{g}_+ = \check{g}_+ \, $  too,  \, q.e.d.
 \vskip7pt
   {\it (b)} \,  To begin with, by definition of  $ G_-^{\scriptscriptstyle \,<} $  there exists a functor epimorphism  $ \, \Theta^{\scriptscriptstyle \,<} : \mathbb{A}_\K^{0|d_1} \!\!\relbar\joinrel\longrightarrow G_-^{\scriptscriptstyle \,<} \, $  which is given on every single  $ \, A \in \Wsalg_\K \, $  by
 \vskip-7pt
  $$  \Theta_{\!A}^{\scriptscriptstyle \,<} \, : \, \mathbb{A}_\K^{0|d_1}(A) := A_\uno^{\,\times d_1} \!\! \relbar\joinrel\longrightarrow\, G_-^{\scriptscriptstyle \,<}(A) \;\; ,  \quad  {\big( \eta_i \big)}_{i \in I} \, \mapsto \, \Theta_{\!A}^{\scriptscriptstyle \,<}\big(\! {\big( \eta_i \big)}_{i \in I} \big) := {\textstyle \prod\limits_{i \in I}^\rightarrow} \big(\, 1 + \eta_i \, Y_i \,\big)  $$
 \vskip-5pt
We prove now that all these  $ \Theta_{\!A}^{\scriptscriptstyle \,<} $'s  are injective, so that  $ \Theta^{\scriptscriptstyle \,<} $  is indeed an isomorphism.
 \vskip5pt
   Let  $ \, {\big( \hat{\eta}_i \big)}_{i \in I} \, , {\big( \check{\eta}_i \big)}_{i \in I} \in A_\uno^{\,\times d_1} \, $  be such that  $ \, \Theta_{\!A}^{\scriptscriptstyle \,<} \big(\! {\big( \hat{\eta}_i \big)}_{i \in I} \big) = \Theta_{\!A}^{\scriptscriptstyle \,<} \big(\! {\big( \check{\eta}_i \big)}_{i \in I} \big) \; $,  in other words we have  $ \; {\textstyle \prod\limits_{i \in I}^\rightarrow} \big(\, 1 + \hat{\eta}_i \, Y_i \,\big) = {\textstyle \prod\limits_{i \in I}^\rightarrow} \big(\, 1 + \check{\eta}_i \, Y_i \,\big)  \; $.  Then  {\sl we can replay the proof of the first part of claim  {\it (a)}},  now with  $ \; \hat{g}_+ := 1 =: \check{g}_+ \; $;  the outcome is again  $ \, \hat{\eta}_i = \check{\eta}_i \, $  for all  $ \, i \in I \, $,  that is  $ \, {\big( \hat{\eta}_i \big)}_{i \in I} = {\big( \check{\eta}_i \big)}_{i \in I} \; $,  \, q.e.d.
 \vskip5pt
   As to the isomorphism  $ \; \mathbb{A}_\K^{0|d_1} \cong \exp\!\big( {(-)}_\uno \!\otimes_\K \fg_\uno \big) \; $,  \, definitions provide a functor epimorphism  $ \, \Theta^e : \mathbb{A}_\K^{0|d_1} \!\!\relbar\joinrel\longrightarrow \exp\!\big( {(-)}_\uno \!\otimes_\K \fg_\uno \big) \, $  given on each  $ \, A \in \Wsalg_\K \, $  by
 \vskip-13pt
  $$  \Theta_{\!A}^e \, : \, \mathbb{A}_\K^{0|d_1}(A) := A_\uno^{\,\times d_1} \!\! \relbar\joinrel\longrightarrow\, G_-^{\scriptscriptstyle \,<}(A) \;\; ,  \quad  {\big( \eta_i \big)}_{i \in I} \, \mapsto \, \Theta_{\!A}^e \big(\! {\big( \eta_i \big)}_{i \in I} \big) := \exp\big(\, {\textstyle \sum_{i \in I}}\, \eta_i \, Y_i \,\big)  $$
 \vskip-5pt
But all these  $ \Theta_{\!A}^e $'s  are indeed injective, so that overall  $ \Theta^e $  is indeed an isomorphism.  In fact, let  $ \, {\big( \hat{\eta}_i \big)}_{i \in I} \, , {\big( \check{\eta}_i \big)}_{i \in I} \in A_\uno^{\,\times d_1} \, $  give  $ \, \Theta_{\!A}^e \big(\! {\big( \hat{\eta}_i \big)}_{i \in I} \big) = \Theta_{\!A}^e \big(\! {\big( \check{\eta}_i \big)}_{i \in I} \big) \; $,  i.e.\  $ \; \exp\big(\, {\textstyle \sum_{i \in I}}\, \hat{\eta}_i \, Y_i \,\big) = \exp\big(\, {\textstyle \sum_{i \in I}}\, \check{\eta}_i \, Y_i \,\big) \; $.  Then  {\sl we can proceed again as in the proof of the second part of claim  {\it (a)}},  now with  $ \; \hat{g}_+ := 1 =: \check{g}_+ \; $;  this gives again  $ \, \hat{\eta}_i = \check{\eta}_i \, $  for all  $ \, i \in I \, $,  thus  $ \, {\big( \hat{\eta}_i \big)}_{i \in I} = {\big( \check{\eta}_i \big)}_{i \in I} \; $,  \, q.e.d.
 \vskip7pt
   {\it (c)} \,  It follows at once from claims  {\it (a)\/}  and  {\it (b)\/}  together.
\end{proof}

\medskip

\begin{remark}
  For every Lie supergroup  $ G \, $,  we shall refer to the isomorphisms in claim  {\it (a)\/}  and/or  {\it (c)\/}  of  Theorem \ref{thm:dir-prod-fact-G - gen}  as to  {\sl ``Global Splittings''\/}  of  $ G $   --- or of  $ N_G \, $,  or of  $ G^- \, $,  respectively.
%
% The existence of such splittings is more or less known among specialists, but usually stated (and
% proved), to the best of the author's knowledge, in a different manner, i.e.\ in sheaf-theoretic terms.
%
\end{remark}

\bigskip

\section{From super Harish-Chandra pairs to Lie supergroups}  \label{sHCp's->Lsgroups}

\smallskip

   {\ } \;\;   In this section we provide two different functors  $ \Psi $  that are quasi-inverse to the functor  $ \Phi $  of  Theorem \ref{thm_Lsgrps-->sHCp's}.  In both cases, for any super Harish-Chandra pair  $ \cP $,  we define as associated  $ \, \Psi(\cP) := \GP \, $  a suitable functor from Weil superalgebras to groups, and then prove that it has the ``right properties''.  Concretely, we follow the pattern provided by the  {\sl Global Splitting Theorem\/}  for Lie supergroups, which tells us two possible ways how our would-be Lie supergroup  $ \GP $  should look like, in terms of  $ \cP $ itself: this leads us, eventually, to provide two different recipes.
                                                          \par
   In both cases, we proceed along the following lines: we define our looked-for  $ \GP $  as a group-valued functor on the category of Weil  $ \K $--superalgebras,  such that each single group  $ \, \GP(A) \, $   --- with  $ \, A \in \Wsalg_\K \, $  ---   is given by generators and relations, in a uniform way (with respect to  $ A \, $).  The idea that dictates the desired presentation by generators and relations is somewhat simple: as part of our ultimate goal, we want  $ \, \Phi\big( \GP \big) = \cP \; \big( = (G_+ \,,\, \fg) \big) \, $,  so we must have  $ \, {\big( \GP \big)}_\zero = G_+ \, $  and  $ \, \Lie\big( \GP \big) = \fg \, $.  The former requirement gives us the reduced Lie subgroup;
%
% (which somehow holds all the ``truly geometrical'' content of the supergroup) of  $ \GP \, $;
%
 the latter instead prescribes what the Lie superalgebra of $ \GP $  must be: then we might think of using this to realize the ``missing part'' of  $ \GP $  as  ``$ \, \exp(\fg_\uno) \, $''.  So each  $ \, \GP(A) \, $  should be presented as generated by  $ \, G_+(A) = G_+(A_\zero) \, $  and  $ \, \exp\big( A_\uno \otimes \fg_\uno \big) \, $,  or at least enough of its elements, and suitable relations.
                                                               \par
   To realize all this, we follow the pattern provided by
%
% the analysis of the structure of a Lie supergroup presented in  Section \ref{interlude},  in particular  %
 the ``Global Splitting(s) Theorem''   --- cf.\  Theorem \ref{thm:dir-prod-fact-G - gen}  ---   which essentially prescribes, in  {\sl two\/}  ways, how  $ \GP $  {\sl must\/}  be done.

\medskip

\subsection{Supergroup functors out of super Harish-Chandra pairs: first recipe}  \label{sgroups-out-sHCp's-1}

\smallskip

   {\ } \;\;   In this subsection we construct a first Lie supergroup functor, denoted  $ \GPt \, $,  along the lines mentioned above.  As a matter of notation, hereafter we shall adopt the following: given  $ \, \cP = (G_+ \,,\, \fg) \in \sHCp_\K \, $,  $ \, A \in \Wsalg_\K \, $  and  $ \, c \in A_\zero \, $  such that  $ \, c^2 = 0 \, $,  for every  $ \, X \in \fg_\zero \, $  we set
\begin{equation}  \label{eq:exp_1+cX}
  \big( 1_{{}_{G_+}} \!\! + c \, X \big)  \; := \;  \exp\big( c\,X \big)  \,\; \in \;\,  G_+(A_\zero)
\end{equation}
 which is obviously inspired by the formal series expansion of the  ``$ \, \exp \, $''  function; when no confusion is possible we shall drop the subscript  $ G_+ $  and simply write  $ \, \big( 1 + c\,X \big) \, $  instead.  Similarly, we shall presently introduce new formal elements of type  ``$ \, \big( 1 + \eta \, Y \big) = \exp\big( \eta \, Y \big) \, $''  with  $ \, \eta \in A_\uno \, $,  $ \, Y \in \fg_\uno \, $.

\vskip15pt

\begin{definition}  \label{def G_- / G_Pt - gen}
 Let  $ \, \cP := \big( G_+ \, , \fg \big) \in \sHCp_\K \, $  be a super Harish-Chandra pair over  $ \K \, $.  %
 \vskip7pt
   {\it (a)}\,  We introduce a functor  $ \; \GPt : \Wsalg_\K \!\relbar\joinrel\longrightarrow \grp \; $  as follows.  For any Weil superalgebra  $ \, A \in \Wsalg_\K \, $,  we define  $ \, \GPt(A) \, $  as being the group with generators the elements of the set
 \vskip-13pt
  $$  \Gamma_{\!A}  \,\; := \;\,  \big\{\, g_+ \, , \big( 1 + \eta \, Y \big) \,\big|\, g_+ \in G_+(A) \, , \, (\eta,Y) \, \in \, A_\uno \!\times\! \fg_\uno \,\big\}  \,\; = \;\,  G_+(A) \,{\textstyle \bigcup}\, {\big\{ (1 + \eta \, Y) \big\}}_{(\eta\,,Y) \, \in \, A_\uno \!\times \fg_\uno}  $$
and relations (for  $ \, g'_+ \, , g''_+ \in G_+(A) \, $,  $ \, \eta \, , \eta' \, , \eta'' \in A_\uno \, $,  $ \, Y \, , Y' \, , Y'' \in \fg_\uno \, $,  $ \, c \in \K \, $)
 \vskip-13pt
  $$  \displaylines{
   g'_+ \cdot\, g''_+  \,\; = \;\,  g'_+ \,\cdot_{\!\!\!{}_{G_+}} g''_+  \quad  ,
 \qquad \qquad  \big( 1 + \eta \, Y \big) \cdot g_+  \,\; = \;\,  g_+ \cdot \big( 1 + \eta \, \text{\sl Ad}\big(g_+^{-1}\big)(Y) \big)  \cr
   \big( 1 + \eta'' \, Y \big) \cdot \big( 1 + \eta' \, Y \big)  \; = \;  \Big( 1_{{}_{G_+}} \!\! + \, \eta' \, \eta'' \, Y^{\langle 2 \rangle} \Big)_{\!\!{}_{G_+}} \!\!\cdot \big(\, 1 + \big( \eta' + \eta'' \big) \, Y \big) \phantom{{}_{\big|}}  \cr
   \big( 1 + \eta'' \, Y'' \big) \cdot \big( 1 + \eta' \, Y' \big)  \; = \;  \Big( 1_{{}_{G_+}} \!\! + \, \eta' \, \eta'' \, \big[Y',Y''\big] \Big)_{\!\!{}_{G_+}} \!\!\cdot \big( 1 + \eta' \, Y' \big) \cdot \big( 1 + \eta'' \, Y'' \big)  \cr
   \big( 1 + \eta \, Y' \big) \cdot \big( 1 + \eta \, Y'' \big)  \; = \;  \big( 1 + \eta \, \big( Y' + Y'' \big) \big)  \cr
   \big( 1 + (c\,\eta) \, Y \big)  \; = \;  \big( 1 + \eta \, (c Y) \big)   \quad ,
 \qquad   \big( 1 + \eta \; 0_{\fg_\uno} \big)  \; = \;  1  \quad ,
 \qquad   \big( 1 + 0_{\scriptscriptstyle A} \, Y \big)  \; = \;  1  }  $$
 \vskip-1pt
\noindent
 where the first line just means that for generators chosen in  $ G_+(A) $  their product, denoted with  ``$ \, \cdot \, $'',  inside  $ \GPt(A) $  is the same as in  $ G_+(A) \, $,  where it is denoted with  ``$ \,\; \cdot_{\!\!\!{}_{G_+}} $'';  moreover, notation like  $ \, \Big(\, 1_{{}_{G_+}} \! + \, \eta' \, \eta'' \, Y^{\langle 2 \rangle} \Big)_{\!\!{}_{G_+}} \, $  and  $ \, \Big(\, 1_{{}_{G_+}} \! + \, \eta' \, \eta'' \, \big[Y',Y''\big] \Big)_{\!\!{}_{G_+}} \, $  denotes two elements in  $ G_+(A) $  as in  \eqref{eq:exp_1+cX}.
                                                             \par
   This yields the functor  $ \GPt $  on objects, and one then defines it on morphisms in the obvious way.  Namely, for any morphism  $ \, f : A' \longrightarrow A'' \, $  in  $ \Wsalg_\K $  we let $ \, \GPt(f) : \GPt\big(A'\big) \longrightarrow \GPt\big(A''\big) \, $ be the group morphism uniquely defined on generators
%
% by
%   $$  \displaylines{
%    \hskip101pt   \GPt(f)\big(\,g'_+\big)  \, := \,  G_+(f)\big(g'_+\big)
% \qquad \qquad \quad \;  \forall \;\;\; g'_+ \in G_+\big(A'\big)   \hfill  \cr
%    \hskip75pt   \GPt(f)\big(\, 1 + \eta' \, Y \,\big)  \, := \,
% \big(\, 1 + f\big(\eta'\big) \, Y \,\big)  \qquad \qquad
% \forall \;\;\; \eta \in A'_\uno\, , \; Y \in \fg_\uno   \hfill  }  $$
%
 --- for all  $ \, g'_+ \in G_+\big(A'\big) \, $,  $ \, \eta \in A'_\uno \, $,  $ \, Y \in \fg_\uno \, $  ---   by
  $$  \GPt(f)\big(\,g'_+\big)  \, := \,  G_+(f)\big(g'_+\big)  \quad ,  \qquad  \GPt(f)\big(\, 1 + \eta' \, Y \,\big)  \, := \,  \big(\, 1 + f\big(\eta'\big) \, Y \,\big)  $$
%%
%%
%
%  As the defining relations of every group  $ \GPt(A) $  are independent of the chosen Weil
% superalgebra  $ A \, $,  it follows that such a  $ \GPt(f) $  is well defined indeed.
%
 As the defining relations of each  $ \GPt(A) $  are independent of  $ A \, $,  such a  $ \GPt(f) $  is well defined indeed.
 \vskip7pt
   {\it (b)}\,  We define a functor  $ \; \GPtm : \Wsalg_\K \!\relbar\joinrel\longrightarrow \set \; $  on any object  $ \, A \in \Wsalg_\K \, $  by
 \vskip-15pt
  $$  \GPtm(A)  \; := \;  \Big\{\, {\textstyle \prod_{s=1}^n} \big( 1 + \eta_s Y_s \big) \;\Big|\; n \in \N \, , \; (\eta_s , Y_s) \in A_\uno \times \fg_\uno \;\, \forall \; s \in \{1,\dots,n\} \,\Big\}  \qquad  \big(\, \subseteq \GPt(A) \,\big)  $$
 \vskip-5pt
\noindent
 and on morphism in the obvious way   --- just like for  $ G_{{}_{\!\cP}\,} $.
 \vskip7pt
   {\it (c)}\,  Let us fix in  $ \, \fg_\uno \, $  a  $ \K$--basis  $ \, {\big\{ Y_i \big\}}_{i \in I} \, $   --- for some index set  $ I \, $  ---   and a total order in  $ I \, $.  We define a functor  $ \; G_-^{\scriptscriptstyle \,<} : \Wsalg_\K \!\relbar\joinrel\longrightarrow \set \; $  as follows.  For  $ \, A \in \Wsalg_\K \, $  we set
 \vskip-5pt
  $$  G_-^{\scriptscriptstyle \,<}(A)  \; := \;  \bigg\{\, {\textstyle \prod\limits_{i \in I}^\rightarrow} \big( 1 + \eta_i Y_i \big) \;\bigg|\; \eta_i \in A_\uno \;\, \forall \; i \in I \,\bigg\}  \qquad  \big(\, \subseteq \GPtm(A) \subseteq \GPt(A) \,\big)  $$
 \vskip-9pt
\noindent
 where  $ \, \prod\limits_{i \in I}^\rightarrow \, $  denotes an  {\sl ordered product}   --- with respect to the fixed total order in  $ I \, $.  This defines the functor  $ G_-^{\scriptscriptstyle \,<} $  on objects, and its definition on morphism is the obvious one (like for  $ \GPt $)\,.
 \hfill   $ \diamondsuit $
\end{definition}

\smallskip

%
% \begin{remarks}  \label{remarks_post-def G_- / G_P - gen}  {\ }
% %
%  \vskip5pt
% %
%    {\it (a)}\,  By their very definition, both  $ \GPtm $  and  $ G_-^{\hskip0,5pt\circ, {\scriptscriptstyle<}} $
% can be thought of as subfunctors of  $ \GPt \, $.
% %
%  \vskip4pt
% %
%    {\it (b)}\,  It's easy to see that  $ \GPtm(A) $  is the subgroup of  $ \GPt(A) $
% generated by  $ \, {\big\{ (1 + \eta \, Y) \big\}}_{(\eta\,,Y) \in A_\uno \times \fg_\uno} \, $.
%
%  \vskip2pt
% %
%    {\it (c)}\,  By definition, the subfunctor  $ G_-^{\scriptscriptstyle \,<} $  depends on the
% choice of the ordered basis  $ \, {\big\{ Y_i \big\}}_{i \in I} \, $  of  $ \fg_\uno \, $;
% nevertheless, we shall presently see (in  Proposition \ref{prop:fact-G_Pt + sbgr-gener-G^[2]&G_-}
% {\it (a)\/}  later on) that this dependence is actually irrelevant for our purposes.  Indeed,
% although  $ G_-^{\scriptscriptstyle \,<} $  is definitely non-canonical, the supergroup subfunctor
% that it generates instead is  {\sl independent\/}  of the choice of the ordered  $ \K $--basis
% of  $ \fg_\uno \, $.
% %
% \end{remarks}
%

\begin{remark}  \label{remark_post-def G_- / G_P - gen}  {\ }
 By their very definition, both  $ \GPtm $  and  $ G_-^{\scriptscriptstyle \,<} $  can be thought of as subfunctors of  $ \GPt \, $.  Moreover, every group  $ \GPtm(A) $  is clearly the subgroup of  $ \GPt(A) $  generated by  $ G_-^{\scriptscriptstyle \,<}(A) \, $,  or even by  $ \, {\big\{ (1 + \eta \, Y) \big\}}_{(\eta\,,Y) \in A_\uno \times \fg_\uno} \, $.  In particular, although  $ G_-^{\scriptscriptstyle \,<} $  depends on the choice of  $ \, {\big\{ Y_i \big\}}_{i \in I} \, $,  the supergroup subfunctor that it generates (inside  $ \GPt \, $)  instead is  {\sl independent\/}  of any such choice.
\end{remark}

\vskip5pt

   Next result shows that  $ \GPt $  can also be described using a much smaller set of generators:

\vskip11pt

\begin{proposition}  \label{prop:G_P-gener-wt-basis}
 Let  $ \, \cP := \big( G_+ \, , \fg \big) \in \sHCp_\K \, $  be a super Harish-Chandra pair over  $ \K \, $;  also, we fix in  $ \, \fg_\uno $  a  $ \K $--basis  $ \, {\big\{ Y_i \big\}}_{i \in I} \, $   --- for some index set  $ I \, $  ---   and a total order in  $ I \, $.
                                                            \par
   Then for every Weil\/  $ \K $--superalgebra  $ \, A \in \Wsalg_\K \, $  the group  $ \GPt(A) $  is generated by the set
 \vskip-7pt
  $$  \Gamma_{\!A}^{\,\diamond}  \;\; := \;\,  G_+(A) \,\;{\textstyle \bigcup}\;\, \big\{ \big( 1 + \eta_i \, Y_i \big) \;\big|\; \eta_i \in A_\uno \, , \; \forall \; i \in I \,\big\}  $$
 \vskip1pt
\end{proposition}

\begin{proof}
 Given  $ \, A \in \Wsalg_\K \, $,  let  $ \, G^{\,\diamond}_{{}_{\!\cP}}(A) \, $  be the subgroup of  $ \GPt(A) $  generated by  $ \Gamma_{\!A}^{\,\diamond} \, $.  We shall prove that every generator of the (larger,  {\it a priori\/})  group  $ \GPt(A) $  of the form  $ \, (1 + \eta \, Y) \, $  with  $ \, (\eta\,,Y) \in A_\uno \times \fg_\uno \, $  also belongs to the subgroup  $ G^{\,\diamond}_{{}_{\!\cP}}(A) \, $:  this then will prove the claim.
 \vskip2pt
   So let  $ \, (\eta\,,Y) \in A_\uno \times \fg_\uno \, $;  then, in terms of the  $ \K $--basis  $ \, {\big\{ Y_i \big\}}_{i \in I} \, $  of  $ \fg_\uno \, $,  our  $ Y $  expands into  $ \, Y = \sum_{s=1}^k c_{j_s} Y_{j_s} \, $.  By repeated applications of relations of the form  $ \; \big( 1 + \eta \, Y' \big) \cdot \big( 1 + \eta \, Y'' \big) \, = \, \big( 1 + \eta \, \big( Y' + Y'' \big) \big) \; $,  we find that the generator  $ \, (1 + \eta \, Y) \, $  in  $ \GPt(A) $  factors as
 \vskip-7pt
\begin{equation}  \label{eq:expans_1+eta.Y}
   \big( 1 + \eta \, Y \big) \, = \, \Big( 1 + \eta \, \textstyle{\sum_{s=1}^k} c_{j_s} Y_{j_s} \Big) \, = \, \textstyle{\prod_{s=1}^k} \big(\, 1 + c_{j_s} \eta \, Y_{j_s} \big)
\end{equation}
 \vskip-1pt
\noindent
 where the product can be done in any order, as the factors in it mutually commute.  Now the product in right-hand side does belong to  $ G^{\,\diamond}_{{}_{\!\cP}}(A) \, $,  hence we are done.
\end{proof}

\smallskip

\begin{free text}  \label{constr-tGamma_P-t}
 {\bf Another realization of  $ \GPt \, $.}  Let  $ \, \cP = \big( G_+ \, , \fg \big) \in \sHCp_\K \, $  be a super Harish-Chandra pair; we present now yet another way of realizing the  $ \K $--supergroup  $  \GPt $  introduced in  Definition \ref{def G_- / G_Pt - gen}{\it (a)}.  In the following, if  $ K $  is any group presented by generators and relations, we write  $ \, K = \big\langle \varGamma \,\big\rangle \Big/ \big( \cR \big) \, $  if  $ \varGamma $  is a set of free generators (of  $ K \, $),  $ \cR $  is a set of relations among generators and  $ \big( \cR \big) $  is the normal subgroup in  $ K $  generated by  $ \cR \, $.  As a matter of notation, given a presentation  $ \, K = \big\langle \varGamma \,\big\rangle \Big/ \big( \cR \big) = \big\langle \varGamma \,\big\rangle \Big/ \big( \cR_1 \cup \cR_2 \big) \, $  with  $ \, \cR = \cR_1 \cup \cR_2 \; $,  the Double Quotient Theorem gives us
 \vskip-7pt
\begin{equation}  \label{eq:Double-Quot-Thm}
  K  \; = \;  \big\langle \varGamma \,\big\rangle \!\Big/\! \big( \cR \big)  \; = \;  \big\langle \varGamma \,\big\rangle \!\Big/\! \big( \cR_1 \cup \cR_2 \big)  \; = \;  \big\langle \varGamma \,\big\rangle \!\Big/\! \big( \cR_1 \big) \Bigg/\! \big( \cR_1 \cup \cR_2 \big) \Big/\! \big( \cR_1 \big)  \; = \;  \big\langle\, \overline{\varGamma} \,\big\rangle \Big/\! \big(\, \overline{\cR_2} \,\big)
\end{equation}
where  $ \overline{\varGamma} $  and  $ \overline{\cR_2} $  respectively denote the images of  $ \varGamma $  and of  $ \cR_2 $  in the quotient group  $ \, \big\langle \varGamma \,\big\rangle \Big/ \big( \cR_1 \big) \; $.
%%%%%
 %%%%%
   \eject
 %%%%%
%%%%%
%
 \vskip3pt
   For any fixed  $ \, A \in \Wsalg_\K \, $,  we denote by  $ \, G_+^{[2]}(A) \, $  the subgroup of  $ G_+(A) $  generated by the set
 $ \; \big\{\, (1 + c \, X ) \;\big|\; c \in A_\uno^{[2]} \, , \; X \in [\, \fg_\uno , \fg_\uno ] \,\big\} \; $
 ---  cf.\ \S \ref{bas-alg-sobjcts}  for notation  $ A_\uno^{[2]} $.  {\sl Note\/}  then that  $ G_+^{[2]}(A) $  is  {\sl normal\/}  in  $ G_+(A) \, $,  as one easily sees by construction (taking into account that, as  $ \, \cP := \big( G_+ \, , \fg \big) \, $  is a super Harish-Chandra pair, the ``adjoint'' action of  $ G_+ $  onto  $ \fg $  maps  $ [\,\fg_\uno,\fg_\uno] $  into itself).
                                                                 \par
  We consider also the three sets
%
%%%
%   $$  \displaylines{
%    \Gamma_{\!A,+}  \,\; := \;\,  \big\{\, g_+ \,\big|\; g_+ \! \in G_+(A) \big\}  \cr
%    \Gamma'_{\!A,-}  \,\; := \;\,  \big\{\, \big( 1 + \eta' \eta'' \, X \big) \,\big|\, \eta',
% \eta'' \in A_\uno \, , X \in [\fg_\uno,\fg_\uno] \cup \fg_\uno^{\langle 2 \rangle} \,\big\}  \cr
%    \Gamma_{\!A,-}  \; := \;  \Gamma'_{\!A,-} \,{\textstyle \bigcup}\, {\big\{ (1 + \eta \, Y)
% \big\}}_{\eta \in A_\uno}^{Y \in \fg_\uno}  }  $$
%%%
%
  $$  \Gamma_{\!A}^{\,+}  \; := \;  G_+(A)  \;\; ,  \quad
 \Gamma_{\!A}^{\,[2]} \; := \; G_+^{[2]}(A)  \;\; ,  \quad
 \Gamma_{\!A}^{\,-}  \; := \;  \Gamma_{\!A}^{\,[2]} \,{\textstyle \bigcup}\, {\big\{ (1 \! + \eta \, Y) \big\}}_{(\eta\,,Y) \, \in \, A_\uno \!\times \fg_\uno}  $$
and the sets of relations   --- for all  $ \, g_+ \, , g'_+ \, , g''_+ \in \Gamma_{\!A}^{\,+} \, $,  $ \, g_{[2]} \, , g'_{[2]} \, , g''_{[2]} \in \Gamma_{\!A}^{\,[2]} \, $,  $ \, \eta \, , \eta' , \eta'' \in A_\uno \, $,  $ \, X \in [\fg_\uno,\fg_\uno] \, $,  $ \, Y , Y' , Y'' \in \fg_\uno \, $,  with  $ \;\cdot_{\!\!\!{}_{G_+}} $  and  $ \;\cdot_{\!\!\!{}_{G_+^{[2]}}} $  being the product in  $ G_+(A) $  and in  $ G_+^{[2]}(A) \, $  ---   given by
 \vskip-13pt
  $$  \displaylines{
   \mathcal{R}_{\!A}^+ \, : \quad  g'_+ \cdot\, g''_+  \,\; = \;\,  g'_+ \,\cdot_{\!\!\!{}_{G_+}} g''_+  \cr
   \mathcal{R}_{\!A}^- \, : \,
 \begin{cases}
   \qquad \qquad \qquad \qquad \quad  g'_{[2]} \cdot\, g''_{[2]}  \,\; = \;\,  g'_{[2]} \,\cdot_{\!\!\!{}_{G_+^{[2]}}} g''_{[2]}  \cr
   \qquad \qquad \quad  \big( 1 + \eta \, Y \big) \cdot g_{[2]}  \; = \;  g_{[2]} \cdot \big( 1 + \eta \, \text{\sl Ad}\big(g_{[2]}^{-1}\big)(Y) \big)  \cr
    \quad  \big( 1 + \eta'' \, Y \big) \cdot \big( 1 + \eta' \, Y \big)  \; = \;  \Big( 1 \! + \, \eta' \, \eta'' \, Y^{\langle 2 \rangle} \Big) \cdot \big(\, 1 + \big( \eta' + \eta'' \big) \, Y \,\big) \phantom{{}_{\big|}}  \cr
   \,  \big( 1 + \eta'' \, Y'' \big) \cdot \big( 1 + \eta' \, Y' \big)  \, = \,  \Big( 1 \! + \, \eta' \, \eta'' \, \big[Y',Y''\big] \Big) \cdot \big( 1 + \eta' \, Y' \big) \cdot \big( 1 + \eta'' \, Y'' \big)  \cr
   \qquad \qquad \quad  \big( 1 + \eta \, Y' \big) \cdot \big( 1 + \eta \, Y'' \big)  \; = \;  \big( 1 + \eta \, \big( Y' + Y'' \big) \big)  \cr
   \qquad \qquad \qquad \;\;  \big( 1 + \eta \; 0_{\fg_\uno} \big)  \,\; = \;\,  1  \quad ,
 \qquad   \big( 1 + 0_{\scriptscriptstyle A} \, Y \big)  \,\; = \;\,  1
 \end{cases}   \cr
   \mathcal{R}_{\!A}^\ltimes \, : \quad  g_{[2]} \cdot g_+  \, = \;  g_+ \cdot \big(\, g_+^{-1} \,\cdot_{\!\!\!{}_{G_+}} g_{[2]} \,\cdot_{\!\!\!{}_{G_+}} g_+ \big) \;\; ,  \;\;\;
  \big( 1 + \eta \, Y \big) \cdot g_+  \, = \;  g_+ \cdot \big( 1 + \eta \, \text{\sl Ad}\big(g_+^{-1}\big)(Y) \big)   \phantom{\Big|^{\big|}}  \cr
   \mathcal{R}_{\!A}^{[2]} \, : \quad  \big( g_{[2]} \big)_{{\Gamma_A^{\,[2]}}}  = \;\,  \big( g_{[2]} \big)_{{\Gamma_A^{\,+}}}   \phantom{\big|^{|}}  \cr
   \mathcal{R}_{\!A}  \; := \;  \mathcal{R}_{\!A}^+ \,{\textstyle \bigcup}\, \mathcal{R}_{\!A}^- \,{\textstyle \bigcup}\, \mathcal{R}_{\!A}^\ltimes \,{\textstyle \bigcup}\, \mathcal{R}_{\!A}^{[2]}   \phantom{\big|^{\big|}}  }  $$
(in particular, note that the relations of type  $ \mathcal{R}_{\!A}^{[2]} $  in down-to-earth terms just identify each element in  $ \Gamma_A^{\,[2]} $  with its corresponding copy inside  $ \Gamma_A^+ \, $).  Then we  {\sl define\/}  a new group, by generators and relations, namely  $ \; \GPtm(A) := \big\langle\, \Gamma_{\!A}^- \,\big\rangle \Big/ \big(\, \mathcal{R}_{\!A}^- \,\big) \;\, $.
 \vskip5pt
   From the very definition of  $ \GPt(A) $   --- cf.\  Definition \ref{def G_- / G_Pt - gen} ---   it follows that
\begin{equation}   \label{eq:1st-present-G_Pt}
  \hskip3pt   \GPt(A)  \,\; \cong \;\,  \big\langle\, \Gamma_{\!A}^+ \,{\textstyle \bigcup}\; \Gamma_{\!A}^- \,\big\rangle \Big/ \big( \mathcal{R}_{\!A} \big)  \,\; = \;\,  \big\langle\, \Gamma_{\!A}^+ \,{\textstyle \bigcup}\; \Gamma_{\!A}^- \,\big\rangle \bigg/ \Big(\, \mathcal{R}_{\!A}^+ \,{\textstyle \bigcup}\; \mathcal{R}_{\!A}^- \,{\textstyle \bigcup}\; \mathcal{R}_{\!A}^\ltimes \,{\textstyle \bigcup}\; \mathcal{R}_{\!A}^{[2]} \,\Big)
\end{equation}
indeed, here above we are just taking larger sets of generators and of relations (w.r.t.\  Definition \ref{def G_- / G_Pt - gen}),  but with enough redundancies as to find a different presentation of  {\sl the same\/}  group.
 \vskip5pt
   From this we find a neat description of  $ \GPt(A) $  by achieving the presentation  \eqref{eq:1st-present-G_Pt}  in a series of intermediate steps, namely adding only one bunch of relations at a time.  As a first step, we have
\begin{equation}   \label{eq:2nd-present-G_Pt}
  \big\langle\, \Gamma_{\!A}^+ \,{\textstyle \bigcup}\; \Gamma_{\!A}^- \,\big\rangle \Big/ \big(\, \mathcal{R}_{\!A}^+ \,{\textstyle \bigcup}\; \mathcal{R}_{\!A}^- \,\big)  \,\; = \;\,  \big\langle\, \Gamma_{\!A}^+ \,\big\rangle \Big/ \big(\, \mathcal{R}_{\!A}^+ \,\big)  \,\; * \;\, \big\langle\, \Gamma_{\!A}^- \,\big\rangle \Big/ \big(\, \mathcal{R}_{\!A}^- \,\big)  \,\; \cong \;\,  G_+(A) \,*\, \GPtm(A)
\end{equation}
where  $ \; G_+(A) \, \cong \, \big\langle\, \Gamma_{\!A}^+ \,\big\rangle \Big/ \big(\, \mathcal{R}_{\!A}^+ \,\big) \; $  by construction and  $ \, * \, $  denotes the free product (of two groups).
 \vskip5pt
  For the next two steps we can follow two different lines of action.  On the one hand, one has
  $$  \big\langle\, \Gamma_{\!A}^+ \,{\textstyle \bigcup}\; \Gamma_{\!A}^- \,\big\rangle \Big/ \big(\, \mathcal{R}_{\!A}^+ \,{\textstyle \bigcup}\; \mathcal{R}_{\!A}^- \,{\textstyle \bigcup}\; \mathcal{R}_{\!A}^\ltimes \,\big)  \;\; \cong \;\;  \Big(\, G_+(A) * \GPtm(A) \Big) \bigg/ \Big(\, \overline{\mathcal{R}_{\!A}^\ltimes} \,\Big)  \;\; \cong \;\;  G_+(A) \ltimes \GPtm(A)  $$
because of  \eqref{eq:Double-Quot-Thm}  and  \eqref{eq:2nd-present-G_Pt}  together, where  $ \, G_+(A) \ltimes \GPtm(A) \, $  is the semidirect product of  $ \, G_+(A) \, $  with  $ \, \GPtm(A) \, $  with respect to the obvious (``adjoint'') action of the former on the latter.  Then
  $$  \displaylines{
   \qquad   \big\langle\, \Gamma_{\!A}^+ \,{\textstyle \bigcup}\; \Gamma_{\!A}^- \,\big\rangle \Big/ \big(\, \mathcal{R}_{\!A} \big)  \;\; \cong \;\;  \big\langle\, \Gamma_{\!A}^+ \,{\textstyle \bigcup}\; \Gamma_{\!A}^- \,\big\rangle \bigg/ \Big(\, \mathcal{R}_{\!A}^+ \,{\textstyle \bigcup}\; \mathcal{R}_{\!A}^- \,{\textstyle \bigcup}\; \mathcal{R}_{\!A}^\ltimes \,{\textstyle \bigcup}\; \mathcal{R}_{\!A}^{[2]} \,\Big)  \;\; \cong   \hfill  \cr
   \hfill   \cong \;\;  \Big( G_+(A) \ltimes \GPtm(A) \Big) \bigg/ \Big(\; \overline{\mathcal{R}_{\!A}^{[2]}} \;\Big)  \;\; \cong \;\;  \Big( G_+(A) \ltimes \GPtm(A) \Big) \bigg/ N_{[2]}(A)   \qquad  }  $$
where  $ \, N_{[2]}(A) \, $  is the normal subgroup of  $ \; G_+(A) \, \ltimes \, \GPtm(A) \; $  generated by
 $ \, {\Big\{\! \big(\, g_{[2]} \, , g_{[2]}^{-1} \,\big) \!\Big\}}_{g_{[2]} \in \Gamma_A^{[2]}} \; $.
%
%%
%   $$  \Big\{\, \big(\, g_{[2]} \, , g_{[2]}^{-1} \,\big) \;\Big|\;\,
% g_{[2]} \in \Gamma_A^{[2]} \,\Big\}  $$
%%
%
 This together with  \eqref{eq:1st-present-G_Pt}  eventually yields
  $$  \GPt(A)  \;\; = \;\;  \Big( G_+(A) \ltimes G_-(A) \Big) \bigg/ N_{[2]}(A)  $$
 \vskip3pt
   On the other hand, again from  \eqref{eq:Double-Quot-Thm}  and  \eqref{eq:2nd-present-G_Pt}  together we get
  $$  \big\langle\, \Gamma_{\!A}^+ \,{\textstyle \bigcup}\; \Gamma_{\!A}^- \,\big\rangle \bigg/ \Big(\, \mathcal{R}_{\!A}^+ \,{\textstyle \bigcup}\; \mathcal{R}_{\!A}^- \,{\textstyle \bigcup}\; \mathcal{R}_{\!A}^{[2]} \,\Big)  \,\;\; \cong \;\;  \Big( G_+(A) \,*\, \GPtm(A) \Big) \bigg/ \Big(\; \overline{\mathcal{R}_{\!A}^{[2]}} \,\Big)  \,\;\; \cong \;\;\,  G_+(A) \!\!\mathop{*}_{G_+^{[2]}(A)}\!\!\! \GPtm(A)  $$
where  $ \, G_+(A) \!\mathop{*}\limits_{G_+^{[2]}(A)}\! \GPtm(A) \, $  is the amalgamated product of  $ \, G_+(A) \, $  and  $ \, \GPtm(A) \, $  over  $ \, G_+^{[2]}(A) \, $  w.r.t.\  the natural monomorphisms  $ \; G_+^{[2]}(A) \lhook\joinrel\relbar\joinrel\longrightarrow G_+(A) \; $  and  $ \; G_+^{[2]}(A) \lhook\joinrel\relbar\joinrel\longrightarrow \GPtm(A) \; $.  Then
  $$  \displaylines{
   \quad   \big\langle\, \Gamma_{\!A}^+ \,{\textstyle \bigcup}\; \Gamma_{\!A}^- \,\big\rangle \Big/ \big(\, \mathcal{R}_{\!A} \big)  \,\;\; \cong \;\;\,  \big\langle\, \Gamma_{\!A}^+ \,{\textstyle \bigcup}\; \Gamma_{\!A}^- \,\big\rangle \bigg/ \Big(\, \mathcal{R}_{\!A}^+ \,{\textstyle \bigcup}\; \mathcal{R}_{\!A}^- \,{\textstyle \bigcup}\; \mathcal{R}_{\!A}^{[2]} \,{\textstyle \bigcup}\; \mathcal{R}_{\!A}^\ltimes \,\Big)  \,\;\; \cong   \hfill  \cr
   \hfill   \cong \;\;\,  \Big( G_+(A) \!\mathop{*}\limits_{G_+^{[2]}(A)}\! \GPtm(A) \Big) \bigg/ \Big(\, \overline{\mathcal{R}_{\!A}^\ltimes} \;\Big)  \,\;\; \cong \;\;\,  \Big( G_+(A) \!\mathop{*}\limits_{G_+^{[2]}(A)}\! \GPtm(A) \Big) \bigg/ N_\ltimes(A)   \quad  }  $$
where  $ \, N_\ltimes(A) \, $  is the normal subgroup of  $ \; G_+(A) \!\mathop{*}\limits_{G_+^{[2]}(A)}\! \GPtm(A) \; $  generated by
 \vskip5pt
%
%    \centerline{ \quad  $ {\Big\{\, g_+ \cdot \big( 1 + \eta \, Y \big) \cdot g_+^{-1} \cdot
% {\big( 1 + \eta \, \text{\sl Ad}(g_+)(Y) \big)}^{-1} \,\Big\}}_{(\eta,Y) \in A_\uno \times
% \fg_\uno}^{\,g_+ \in G_+(A)}  \;\; {\textstyle \bigcup} $   \hfill }
% %
%  \vskip4pt
% %
%    \centerline{ \hfill   $ {\textstyle \bigcup} \;\;  {\Big\{\, g_+ \cdot g_{[2]} \cdot g_+
% \cdot {\big( g_+ \,\cdot_{\!\!\!{}_{G_+}}\! g_{[2]} \;\cdot_{\!\!\!{}_{G_+}}\! g_+ \big)}^{-1}
% \,\Big\}}_{g_+ \in G_+(A)}^{g_{[2]} \in \Gamma_A^{[2]}} $  \quad }
%
   \centerline{ $ {\Big\{\, g_+ \big( 1 + \eta \, Y \big) \, g_+^{-1} {\big( 1 + \eta \, \text{\sl Ad}(g_+)(Y) \big)}^{-1} \Big\}}_{\!(\eta,Y) \in A_\uno \times \fg_\uno}^{g_+ \in G_+(A)}  {\textstyle \bigcup} \;\;  {\Big\{\, g_+ \, g_{[2]} \, g_+ {\big( g_+ \cdot_{\!\!\!{}_{G_+}}\!\! g_{[2]} \cdot_{\!\!\!{}_{G_+}}\!\! g_+ \big)}^{-1} \Big\}}_{g_+ \in G_+(A)}^{g_{[2]} \in \Gamma_A^{[2]}} $ }
 \vskip7pt
All this along with  \eqref{eq:1st-present-G_Pt}  eventually gives
  $$  \GPt(A)  \,\;\; = \;\;\,  \Big( G_+(A) \!\mathop{*}\limits_{G_+^{[2]}(A)}\! \GPtm(A) \Big) \bigg/ N_\ltimes(A)  $$
for all  $ \, A \in \Wsalg_\K \, $.  In functorial terms this yields
%
%   $$  \displaylines{
%    \GPt \; = \; \Big( G_+ \ltimes \GPtm \Big) \bigg/ N_{[2]}  \qquad  \text{and}  \qquad
% \GPt \; = \; \Big( G_+ \!\mathop{*}\limits_{G_+^{[2]}}\! \GPtm \Big) \bigg/ N_\ltimes  \cr
%    \text{or}  \quad \qquad\;  \GPt  \; = \;\,  G_+ \!\mathop{\ltimes}\limits_{G_+^{[2]}}\!
% \GPtm  \qquad  }  $$
%
  $$  \GPt \; = \; \Big( G_+ \ltimes \GPtm \Big) \bigg/ N_{[2]}  \!\qquad  \text{and}  \!\qquad \GPt \; = \; \Big( G_+ \!\mathop{*}\limits_{G_+^{[2]}}\! \GPtm \Big) \bigg/ N_\ltimes  \,\quad ,  \qquad
   \text{or}  \!\!\qquad  \GPt  = \;  G_+ \!\mathop{\ltimes}\limits_{G_+^{[2]}}\! \GPtm  $$
 where the last, (hopefully) more suggestive notation  $ \,\; \GPt \, = \; G_+ \!\mathop{\ltimes}\limits_{G_+^{[2]}}\! \GPtm \;\, $  tells us that  $ \, \GPt $  is the ``amalgamate semidirect product'' of  $ G_+ $  and  $ \GPtm $  over their common subgroup  $ G_+^{[2]} \; $.
\end{free text}

\medskip

\subsection{The supergroup functor  $ \GPt $  as a Lie supergroup}  \label{subsec:struct-G_Pt}

\smallskip

%
%    {\ } \;\;  We aim now to proving that the supergroup functor  $ \GPt $  is actually
% a Lie supergroup; for this, we need to investigate its structure in some detail.  We
% keep definitions and notations as given before: in particular, recall that for any
% $ \, A \in \Wsalg_\K \, $  we denote by  $ \, G_+^{[2]}(A) \, $  the subgroup of
% $ \, G_+(A) $  generated by  $ \; \big\{\, (1 + c \, X ) \;\big|\; c \in A_\uno^{[2]}
% \, , \; X \in [\, \fg_\uno , \fg_\uno ] \,\big\} \; $   ---  cf.\ \S \ref{bas-alg-sobjcts}
% for notation  $ A_\uno^{[2]} $.
%
   {\ } \;\;  We aim now to proving that the functor  $ \GPt $  is actually a Lie supergroup.  We keep definitions and notations as before: in particular, recall that for  $ \, A \in \Wsalg_\K \, $  we denote by  $ \, G_+^{[2]}(A) \, $  the subgroup of  $ \, G_+(A) $  generated by  $ \; \big\{ (1 + c \, X ) \,\big|\, c \in A_\uno^{[2]} , \, X \! \in [\, \fg_\uno , \fg_\uno ] \,\big\} \; $   ---  cf.\ \S \ref{bas-alg-sobjcts}  for notation  $ A_\uno^{[2]} $.

\vskip7pt

   Our first step is the following ``factorization result'' for  $ \GPt \, $:

%%%%%
 %%%%%
   \eject
 %%%%%
%%%%%

\medskip

\begin{proposition}  \label{prop:fact-G_Pt + sbgr-gener-G^[2]&G_-}
   Let  $ \, \cP := \big( G_+ \, , \fg \big) \in \sHCp_\K \, $  be a super Harish-Chandra pair over  $ \K \, $,  let  $ \, {\big\{ Y_i \big\}}_{i \in I} \, $  be a totally ordered  $ \, \K $--basis  of  $ \, \fg_\uno \, $  (for our fixed order in  $ I $)  and  $ \, A \in \Wsalg_\K \, $.  Then:
 \vskip5pt
   (a)\,  letting  $ \, \big\langle G_-^{\scriptscriptstyle \,<}(A) \big\rangle \, $  be the subgroup of  $ \GPt(A) $  generated by  $ \, G_-^{\scriptscriptstyle \,<}(A) \, $,  we have
 \vskip-7pt
  $$  \big\langle\, G_-^{\scriptscriptstyle \,<}(A) \,\big\rangle  \,\; = \;\,  \GPtm(A)  $$
 \vskip-2pt
\noindent
 and there exist set-theoretic factorizations (with respect to the group product  ``$ \; \cdot \, $'')
 \vskip-7pt
  $$  \GPtm(A)  \; = \;  G_+^{[2]}(A) \cdot G_-^{\scriptscriptstyle \,<}(A)  \quad ,  \qquad
  \GPtm(A) \; = \; G_-^{\scriptscriptstyle \,<}(A) \cdot G_+^{[2]}(A)  $$
 \vskip1pt
   (b)\,  there exist set-theoretic factorizations (with respect to the group product  ``$ \; \cdot \, $'')
 \vskip-7pt
  $$  \GPt(A) \; = \; G_+(A) \cdot G_-^{\scriptscriptstyle \,<}(A)  \quad ,  \qquad  \GPt(A) \; = \; G_-^{\scriptscriptstyle \,<}(A) \cdot G_+(A)  $$
\end{proposition}

\begin{proof}
 Claim  {\it(a)\/}  is the exact analogue of  Proposition  \eqref{eq:factor-1_G^-},  and claim  {\it(b)\/}  the analogue of  \eqref{eq:factor-1_G},  in  Proposition \ref{prop:fact-Liesgrp_G}{\it (b)}.   In both cases the proof (up to trivialities) is identical, so we can skip it.
 \end{proof}

\medskip

\begin{free text}  \label{G_Pt-module V}
 {\bf The representation  $ \, \GPt \!\relbar\joinrel\longrightarrow \rGL(V) \, $.}  When discussing the structure of a Lie supergroup  $ G $,  the factorization  $ \, G = G_\zero \cdot G_-^{\scriptscriptstyle \,<} \, $  was just a intermediate step;  Proposition \ref{prop:fact-G_Pt + sbgr-gener-G^[2]&G_-}  above gives us the parallel counterpart for  $ \GPt \, $.  This factorization result for  $ G $  is improved by the ``Global Splitting Theorem''   --- i.e.\  Theorem \ref{thm:dir-prod-fact-G - gen}  ---   that, roughly speaking, states that for any  $ \, g \in G(A) \, $  the factorization pertaining to  $ \, G_\zero(A) \cdot G_-^{\scriptscriptstyle \,<}(A) \, $  has uniquely determined factors, and similarly any element in  $ \, G_-^{\scriptscriptstyle \,<}(A) \, $  has a unique factorization into an ordered product of factors of the form  $ \, \big( 1 + \eta_i\,Y_i \big) \, $.  Both results are proved by showing that two factorizations of the same object necessarily have identical factors; in other words,  {\sl distinct\/}  choices of factors always give rise to  {\sl different\/}  elements in  $ G(A) $  or in  $ G_-^{\scriptscriptstyle \,<}(A) \, $.  This last fact was proved using the concrete realization of  $ G(A) $  as a special set of maps, namely  $ \, G(A) := \bigsqcup_{\,x \in |G|} \Hom_{\salg_\K} \big( \cO_{|G|,x} \, , A \big) \, $;  \, indeed, this algebra is rich enough to ``separate'' different elements of  $ G(A) $  itself just looking at their values as  $ A $--valued  maps.  When dealing with  $ \GPt(A) $  instead, that is defined abstractly,
%
% (by generators and relations),
%
 such a built-in realization is not available:
%
% from scratch:
%
 our strategy then is to replace it with a suitable ``partial linearization'',
%
% of  $ \GPt(A) \, $,  namely with a representation of that group
%
 namely a representation of  $ \GPt(A) $  that, although not being faithful, is still ``rich enough'' to (almost) separate elements.
%
% , so that uniqueness of factorizations can be proved again.
%
 \vskip5pt
   Let  $ \, \cP = \big( G_+ \, , \fg \big) \in \sHCp_\K \, $  be
%
% our given super Harish-Candra pair over  $ \K \, $;
%
 given;
 as before, we fix a  $ \K $--basis  $ \, {\big\{ Y_i \big\}}_{i \in I} $  of  $ \fg \, $,  where  $ I $  is an index set in which we fix some total order, hence the basis itself is totally ordered as well.
                                                             \par
   Recall that the  {\sl universal enveloping algebra\/}  $ U(\fg) $  is given by  $ \; U(\fg) \, := \, T(\fg) \Big/ J \; $  where  $ T(\fg) $  is the tensor algebra of  $ \fg $  and  $ J $  is the two-sided ideal in  $ T(\fg) $  generated by the set
  $$  \Big\{\, x \, y - {(-1)}^{|x|\,|y|} \, y \, x - [x,y] \; , \; z^2 - z^{\langle 2 \rangle} \,\;\Big|\; x, y \in \fg_\zero \,{\textstyle \bigcup}\, \fg_\uno \, , \, z \in \fg_\uno \,\Big\}  $$
--- where  $ \, z^{\langle 2 \rangle} := 2^{-1} \, [z,z] \, $,  see  Definition \ref{def:Lie-superalgebras}{\it (c)}.  It is known then   --- see for instance  \cite{vsv},  \S 7.2, which clearly adapt to the complex case too ---   that one has splitting(s) of  $ \K $--supercoalgebras
\begin{equation}  \label{eq:splitting_U(g)}
  U(\fg)  \,\; \cong \;\,  U(\fg_\zero) \otimes_\K {\textstyle \bigwedge} \, \fg_\uno  \,\; \cong \;\,  {\textstyle \bigwedge} \, \fg_\uno \otimes_\K U(\fg_\zero)
\end{equation}
   \indent   In addition,  $ {\textstyle \bigwedge} \, \fg_\uno $  has  $ \K $--basis  $ \, \big\{\, Y_{i_1} Y_{i_2} \cdots Y_{i_s} \,\big|\, s \leq |I| \, , \, i_1 \! < \! i_2 \! < \! \cdots \! < \! i_s \,\big\} \, $;  hereafter, we drop the sign  ``$ \wedge $''  to denote the product in  $ \; {\textstyle \bigwedge} \, \fg_\uno \; $.

\vskip5pt

   Now let  $ \Uuno $  be the (one-dimensional)  {\sl trivial representation\/}  of  $ \fg_\zero \, $.  By the standard process of  {\sl induction\/}  from  $ \fg_\zero $  to  $ \fg $   --- the former being thought of as a Lie subsuperalgebra of the latter ---   we can consider the  {\sl induced representation\/}  $ \, V := \text{\sl Ind}_{\fg_\zero}^{\,\fg}(\Uuno\,) \, $,  that is a  $ \fg $--module.  Looking at  $ \Uuno $  and  $ V $  respectively as a module over  $ U(\fg_\zero) $  and over  $ U(\fg) \, $,  taking  \eqref{eq:splitting_U(g)}  into account we get
\begin{equation}  \label{eq:splitting_V}
  V  \; := \;  \text{\sl Ind}_{\fg_\zero}^{\,\fg}(\Uuno\,)  \; = \;  U(\fg) \! \mathop{\otimes}\limits_{U(\fg_\zero)} \! \Uuno  \; \cong \;  {\textstyle \bigwedge}\, \fg_\uno \mathop{\otimes}\limits_\K \Uuno  \,\; \cong \;  {\textstyle \bigwedge}\, \fg_\uno
\end{equation}

\noindent
 The last one above is a natural isomorphism of  $ \K $--superspaces,  uniquely determined once a specific element  $ \, \underline{b} \in \Uuno \, $  is fixed to form a  $ \K $--basis  of  $ \Uuno $  itself: the isomorphism is  $ \, \omega \otimes \underline{b} \mapsto \omega \, $  for all  $ \, \omega \in \bigwedge \fg_\uno \, $.
%
%%%
% The last one above is a natural  $ \K$--module  isomorphism, that is uniquely determined
% once a specific element  $ \, \underline{b} \in \Uuno \, $  is fixed that forms a
% $ \K$--basis  of  $ \Uuno $  itself: the isomorphism is then simply given by  $ \, \omega
% \otimes \underline{b} \mapsto \omega \, $  for all  $ \, \omega \in \bigwedge \fg_\uno \, $.
%%%
%
%%%%%
 %%%%%
   \eject
 %%%%%
%%%%%
                                                                       \par
  This representation-theoretical construction and its outcome clearly give rise to similar functorial counterparts, for the Lie algebra valued  $ \K$--superfunctors  $ \cL_{\fg_\zero} $  and  $ \cL_\fg \, $,  as well as for the  $ \K$--superfunctors  associated with  $ U(\fg_\zero) $  and  $ U(\fg) \, $,  in the standard way, namely  $ \; A \, \mapsto \, A_\zero \otimes_\K U(\fg_\zero) \; $  and  $ \; A \, \mapsto \, {\big( A \otimes_\K U(\fg) \big)}_\zero \, = \, A_\zero \otimes_\K {U(\fg)}_\zero \, \bigoplus \, A_\uno \otimes_\K {U(\fg)}_\uno \; $  for all  $ \, A \in \Wsalg_\K \, $.
                                                                       \par
   On the other hand, recall that  $ \, \fg_\zero = \Lie\,(G_+) \, $,  and clearly  $ \Uuno $  is also the trivial representation for  $ G_+ \, $,  as a classical Lie group.
%
% --- of real smooth, real analytic or complex holomorphic type.
%
 Then, by construction,
%
% it is clear that
%
 the representation of  $ \fg $  on the space  $ V $  also induces a representation of the super Harish-Chandra pair  $ \, \cP = (G_+,\fg) \, $  on the same  $ V $,  in other words  $ V $  bears also a structure of  $ (G_+,\fg) $--module   --- in the obvious
%
% , natural
%
 sense: we have a morphism  $ \, (\boldsymbol{r}_{\!+},\rho) : (G_+\,\fg) \longrightarrow \big( \rGL(V),\rgl(V)\big) \, $  of super Harish-Chandra pairs.
%
% We shall also use again  $ \rho $  to denote
%
 We shall write again  $ \rho $  for
 the representation map  $ \, \rho : U(\fg) \longrightarrow \End_{\,\K}(V) \, $
   \hbox{giving the  $ U(\fg) $--module  structure on  $ V $.}
 \vskip7pt
   Our key step now is to remark that the above  $ (G_+,\fg) $--module  structure on  $ V $  actually ``integrates'' to a  $ \GPt $--module  structure, in a natural way.
\end{free text}

\smallskip

\begin{proposition}  \label{G_Pt-action on V}
 Retain notation as above for the  $ (G_+,\fg) $--module  $ V $.  There exists a unique structure of (left)  $ \GPt $--module  onto  $ V $  which satisfies the following conditions: for every  $ \, A \in \Wsalg_\K \, $,  the representation map  $ \, \boldsymbol{r}_{{}_{\!\cP,A}}^\circ \! : \GPt(A) \longrightarrow \rGL(V)(A) \, $  is given on generators of\/  $ \GPt(A) $   --- namely, all  $ \, g_+ \in G_+(A) \, $  and  $ \, (1 + \eta_i \, Y_i) \, $  for  $ \, i \in I \, $,  $ \, \eta_i \in A_\uno $  ---   by
  $$  \boldsymbol{r}_{{}_{\!\cP,A}}^\circ(g_+) \, := \, \boldsymbol{r}_{\!+}(g_+) \quad ,  \qquad  \boldsymbol{r}_{{}_{\!\cP,A}}^\circ(1 + \eta_i \, Y_i) \, := \, \rho(1 + \eta_i \, Y_i) \, = \, \text{\sl id}_{{}_V\!} + \eta_i \, \rho(Y_i)  $$
or, in other words,  $ \; g_+.v \, := \, \boldsymbol{r}_{\!+}(g_+)(v) \; $  and  $ \; (1 + \eta_i \, Y_i).v \, := \, \rho(1 + \eta_i \, Y_i)(v)  \, = \, v + \eta_i \, \rho(Y_i)(v) \; $  for all  $ \, v \in V(A) \, $.  Overall, this yields a morphism a  $ \K $--supergroup  functors  $ \; \boldsymbol{r}_{{}_{\!\cP}}^\circ \! : \GPt \!\longrightarrow \rGL(V) \; $.
\end{proposition}

\begin{proof}
 This is, essentially, a straightforward consequence of the whole construction, and of the very definition of  $ \GPt \, $.  Indeed, by definition of representation for the super Harish-Chandra pair  $ \cP $  we see that the operators  $ \boldsymbol{r}_{{}_{\!\cP,A}}^\circ(g_+) $  and  $ \boldsymbol{r}_{{}_{\!\cP,A}}^\circ(1 + \eta_i \, Y_i) $  on  $ V $   --- associated with the generators of  $ \GPt(A) $  ---   do satisfy all relations which, by  Definition \ref{def G_- / G_Pt - gen}{\it (a)},  are satisfied by the generators themselves.  Thus they uniquely provide a well-defined group morphism  $ \, \boldsymbol{r}_{{}_{\!\cP,A}}^\circ \! : \GPt(A) \longrightarrow \rGL(V)(A) \, $  as required.  The construction is clearly functorial in  $ A \, $,  whence the claim.
\end{proof}

\medskip

   The advantage of introducing the representation  $ \boldsymbol{r}_{{}_{\!\cP}} $  of  $ \GPt $  on  $ V $  is that it allows us to ``separate'', in a sense, the ``odd points of  $ \GPt(A) $  from each other and from the even ones'', i.e.\ we can separate the points in  $ G_-^{\scriptscriptstyle \,<}(A) $  from each other (in a ``very fine'' sense) and from those in  $ G_+(A) \, $.

\vskip5pt

   We are now ready to state and prove the main result of the present subsection, that is just the ``global splitting theorem'' for  $ \GPt $  (cf.\  Theorem \ref{thm:dir-prod-fact-G - gen}):

\medskip

\begin{proposition}  \label{dir-prod-fact-G_Pt - gen}  {\ }
 \vskip5pt
   {\it (a)} \,  The restriction of group multiplication in  $ \GPt $  provides isomorphisms of (set-valued) functors
 \vskip-15pt
  $$  G_+ \times G_-^{\scriptscriptstyle \,<}  \; \cong \;  \GPt  \;\; ,  \quad  G_-^{\scriptscriptstyle \,<} \times G_+  \; \cong \;  \GPt  \quad ,  \qquad
  G_+^{[2]} \times G_-^{\scriptscriptstyle \,<}  \; \cong \;  \GPtm  \;\; ,  \quad  G_-^{\scriptscriptstyle \,<} \times G_+^{[2]}  \; \cong \;  \GPtm  $$
 \vskip-3pt
   {\it (b)} \,  There exists an isomorphism of (set-valued) functors  $ \; \mathbb{A}_\K^{0|d_1} \! \cong G_-^{\scriptscriptstyle \,<} \; $,  \, with  $ \, d_1 := |I| = \text{\it dim}_{\,\K}\big(\fg_\uno\big) \, $,  given on  $ A $--points   --- for every  $ \, A \in \Wsalg_\K \, $  --- by
%%
%  $ \,\; \mathbb{A}_\K^{0|d_1}\!(A) = A_\uno^{\,d_1} \!\!\longrightarrow
% G_-^{\scriptscriptstyle \,<}(A) \, , \; {\big(\eta_i\big)}_{i \in I} \mapsto
% \prod\limits_{i \in I}^\rightarrow (1 + \eta_i \, Y_i) \; $.
%%
%
 \vskip-7pt
  $$  \mathbb{A}_\K^{0|d_1}\!(A)  \; = \;  A_\uno^{\,d_1} \!\!\longrightarrow G_-^{\scriptscriptstyle \,<}(A) \;\; ,  \quad  {\big(\eta_i\big)}_{i \in I} \mapsto {\textstyle \prod\limits_{i \in I}^\rightarrow} (1 + \eta_i \, Y_i)  $$
 \vskip-3pt
   {\it (c)} \,  There exist isomorphisms of (set-valued) functors
 \vskip-15pt
  $$  G_+ \times \mathbb{A}_\K^{0|d_1}  \cong \;  \GPt  \,\; ,  \quad  G_+^{[2]} \times \mathbb{A}_\K^{0|d_1}  \cong \;  \GPtm  \,\; ,  \!\qquad  \text{and}  \!\qquad
  \mathbb{A}_\K^{0|d_1} \!\times G_+  \; \cong \;  \GPt  \,\; ,  \quad  \mathbb{A}_\K^{0|d_1} \!\times G_+^{[2]}  \; \cong \;  \GPtm  $$
given on  $ A $--points   --- for every  $ \, A \in \Wsalg_\K \, $  ---   respectively by
 \vskip-7pt
  $$  \big(\, g_+ \, , \, {\big(\eta_i\big)}_{i \in I} \big) \, \mapsto \, g_+ \cdot {\textstyle \prod\limits_{i \in I}^\rightarrow} (1 + \eta_i \, Y_i)  \qquad  \text{and}  \qquad  \big( {\big(\eta_i\big)}_{i \in I} \, , \, g_+ \big) \, \mapsto \, {\textstyle \prod\limits_{i \in I}^\rightarrow} (1 + \eta_i \, Y_i) \cdot g_+  $$
\end{proposition}

\begin{proof}
 The proof is quite close to (half of) that of  Theorem \ref{thm:dir-prod-fact-G - gen},  with some technical differences, involving the use of the representation  $ V $  of  \S \ref{G_Pt-module V};  for completeness we present it explicitly.
 \vskip5pt
   {\it (a)}\, It is enough to prove the first identity concerning  $ \GPt \, $,  as all other are similar. Thus our goal amounts to showing the following: for any  $ \, A \in \Wsalg_\K \, $,  if  $ \; \hat{g}_+ \, \hat{g}_- = \check{g}_+ \, \check{g}_- \; $  for  $ \, \hat{g}_+ \, , \check{g}_+ \in G_+(A) \, $  and  $ \, \hat{g}_- \, , \check{g}_- \in G_-^{\scriptscriptstyle \,<}(A) \, $,  then  $ \, \hat{g}_+ = \check{g}_+ \, $  and  $ \, \hat{g}_- = \check{g}_- \; $.
 \vskip3pt
   The assumption  $ \; \hat{g}_+ \, \hat{g}_- \, = \, \check{g}_+ \, \check{g}_- \; $  implies  $ \; g \, := \, \hat{g}_- \, \check{g}_-^{-1} \, = \, \hat{g}_+^{-1} \, \check{g}_+ \, \in \, G_+(A) \, $,  as  $ G_+(A) $  is a subgroup in  $ \GPt(A) \, $.  Now  $ \, \hat{g}_- \in \, G_-^{\scriptscriptstyle \,<}(A) \, $  has the form  $ \, \hat{g}_- = {\textstyle \prod\limits^\rightarrow}_{i \in I} \big(\, 1 + \hat{\eta}_i \, Y_i \,\big) \, $  and similarly  $ \, \check{g}_- = {\textstyle \prod\limits^\rightarrow}_{i \in I} \big( 1 + \check{\eta}_i Y_i \big) \, $  so that  $ \, \check{g}_-^{\,-1} = {\textstyle \prod\limits^\leftarrow}_{i \in I} \big( 1 - \check{\eta}_i Y_i \big) \, $;  therefore we have
\begin{equation}  \label{eq:exp_hatg-checkg}
  g  \,\; := \;\,  \hat{g}_- \, \check{g}_-^{-1}  \,\; = \;\,  {\textstyle \prod\limits_{i \in I}^\rightarrow} \big( 1 + \hat{\eta}_i Y_i \big) \, {\textstyle \prod\limits_{i \in I}^\leftarrow} \big( 1 - \check{\eta}_i Y_i \big)  \;\; \in \;\;  G_+(A) \; \subseteq \; \GPt(A)
\end{equation}
 \vskip-5pt
   Let  $ \, \fa := \big( {\big\{\, \hat{\eta}_{{}_{\,\scriptstyle i}} , \check{\eta}_{{}_{\,\scriptstyle i}} \big\}}_{i \in I} \big) \, $  be the ideal of  $ A $  generated by the  $ \hat{\eta}_{{}_{\,\scriptstyle i}} $'s  and the  $ \check{\eta}_{{}_{\,\scriptstyle i}} $'s,  set  $ \; A \,{\buildrel {\pi_n} \over {\relbar\joinrel\relbar\joinrel\twoheadrightarrow}}\, A \big/ \fa^{\,n} \; $  for the quotient map and  $ \, {[a\,]}_n := \pi_n(a) \, $  for  $ \, a \in A \, $,  then  $ \, \GPt(A) \,{\buildrel {G(\pi_n)} \over {\relbar\joinrel\relbar\joinrel\relbar\joinrel\relbar\joinrel\rightarrow}}\, \GPt\big( A \big/ \fa^{\,n} \big) \, $  for the associated group morphism and  $ \, {[y]}_n := \GPt(\,\pi_n)(y) \, $  for every  $ \, y \in \GPt(A) \, $.
%
%%%
% Now  \eqref{eq:exp_hatg-checkg}  along with  Lemma \ref{lemma-prod-G_P}  for  $ \, n := 1 \, $,
% letting  $ \; \alpha_i := \hat{\eta}_i - \check{\eta}_i \, \in \fa \; $  for all  $ i \in I \, $,
% gives
%%%
%
 Now, the defining relations for  $ G_-^{\scriptscriptstyle \,<}\big( A \big/ \fa^{\,2} \big) $   --- taking into account that  $ \; \hat{\eta}_h \, \check{\eta}_k \, , \check{\eta}_k \, \hat{\eta}_h \, \in \, \fa^2 \; $  (for all  $ \, h, k \in I \, $)  ---   yield
 \vskip-7pt
\begin{equation}  \label{eq:exp_hatg-checkg mod 2}
  {[g]}_2  \; = \;  {\textstyle \prod\limits_{i \in I}^\rightarrow} \big( 1 + {[\hat{\eta}_i]}_2 \, Y_i \big) \cdot {\textstyle \prod\limits_{i \in I}^\leftarrow} \big( 1 - {[\check{\eta}_i]}_2 \, Y_i \big)  \; = \;  {\textstyle \prod\limits_{i \in I}^\rightarrow} \big( 1 + {[\alpha_i]}_2 \, Y_i \big)  \;\; \in \;\;  G_-^{\scriptscriptstyle \,<}\big( A \big/ \fa^{\,2} \,\big)
\end{equation}
 \vskip-3pt
   Next step then is to let  $ \, {[g]}_2 \, $  act onto  $ \, \underline{b} \in V\big( A \big/ \fa^{\,2} \,\big) \, $.
%
% In order to avoid
%
 To avoid
 confusion, when we describe  $ V $  as  $ \, V = \bigwedge \fg_\uno.\,\underline{b} \, \cong \bigwedge \fg_\uno \, $,  we write the elements of
%
% the  $ \K $--basis
%
 $ {\{Y_i\}}_{i \in I} $  of  $ \fg_\uno $  as  $ \bar{Y}_i $  instead of  $ Y_i \, $:  thus the
%
% $ \K $--linear
%
 isomorphism  $ \, \bigwedge \fg_\uno.\,\underline{b} \, \cong \bigwedge \fg_\uno \, $  is given by  $ \, (Y_{i_1} Y_{i_2} \cdots Y_{i_s}).\,\underline{b} \mapsto \bar{Y}_{i_1} \bar{Y}_{i_2} \cdots \bar{Y}_{i_s} \, $   --- for all  $ \, i_1 < i_2 < \cdots < i_s \, $.
                                                                  \par
   Taking into account that  $ \; {[\alpha_h]}_2 \, {[\alpha_k]}_2 \, = \, {[0]}_2 \, \in \, A \big/ \fa^{\,2} \; $  (for all  $ \, h, k \in I \, $)  from  \eqref{eq:exp_hatg-checkg mod 2}  we get that the action of  $ \, {[g]}_2 \, $  onto  $ \, \underline{b} \in V\big( A \big/ \fa^{\,2} \,\big) \, $  is given by %
\begin{equation}  \label{eq:hatg-checkg onto b}
  {[g]}_2\,.\,\underline{b}  \; = \;  {\textstyle \prod\limits_{i \in I}^\rightarrow} \big( 1 + {[\alpha_i]}_2 \, Y_i \big)\,.\,\underline{b}  \; = \;  1\,.\,\underline{b} \, + \sum_{i \in I} {[\alpha_i]}_2 \, Y_i\,.\,\underline{b}  \; = \;  \underline{b} \, + \sum_{i \in I} {[\alpha_i]}_2 \, \bar{Y}_i  \;\; \in \;\;  V\big( A \big/ \fa^{\,2} \,\big)
\end{equation}
On the other hand, we have also  $ \; {[g]}_2\,.\,\underline{b} \, = \, \underline{b} \; $  because  $ \, {[g]}_2 \in G_+\big( A \big/ \fa^{\,2} \,\big) \, $  and  $ G_+ $  acts trivially on  $ V $.  This compared with  \eqref{eq:hatg-checkg onto b},  taking into account that  $ \, \big\{\, \underline{b} \,\big\} \cup {\big\{\, \bar{Y}_i \,\big\}}_{i \in I} \, $  is part of the chosen basis of  $ V $,  implies that  $ \, {[\alpha_i]}_2 = {[0]}_2 \in A \big/ \fa^{\,2} \, $,  i.e.\  $ \, \alpha_i \in \fa^2 \, $,  for all  $ \, i \in I \, $.  But then we can repeat the previous argument, slightly improved: indeed, again, the defining relations of  $ \GPt\big( A \big/ \fa^{\,3} \,\big) $  give
\begin{equation}  \label{eq:exp_hatg-checkg mod 3}
  {[g]}_3  \; = \;  {\textstyle \prod\limits_{i \in I}^\rightarrow} \big( 1 + {[\hat{\eta}_i]}_3 \, Y_i \big) \, {\textstyle \prod\limits_{i \in I}^\leftarrow} \big( 1 - {[\check{\eta}_i]}_3 \, Y_i \big)  \; = \;  {\textstyle \prod\limits_{i \in I}^\rightarrow} \big( 1 + {[\alpha_i]}_3 \, Y_i \big)  \;\; \in \;\;  \GPt\big( A \big/ \fa^{\,3} \,\big)
\end{equation}
--- now because the new occurring factors depend on coefficients of the form  $ \; \alpha_h \, \check{\eta}_k \; $  and  $ \; \alpha_h \, \check{\eta}_k \, $,  that all belong to  $ \fa^{\,3} \, $.  Then we repeat the second step, namely we let  $ \, {[g]}_3 \, $  act onto  $ \, \underline{b} \in V\big( A \big/ \fa^{\,3} \,\big) \, $,  for which  \eqref{eq:exp_hatg-checkg mod 3}  gives the analogue of  \eqref{eq:hatg-checkg onto b},  namely
  $$  {[g]}_3\,.\,\underline{b}  \; = \;
%
% {\textstyle \prod\limits_{i \in I}^\rightarrow} \big( 1 + {[\alpha_i]}_3 \, Y_i \big)\,.\,\underline{b}
% \; = \;  1\,.\,\underline{b} \, + \sum_{i \in I} {[\alpha_i]}_3 \, Y_i\,.\,\underline{b}  \; = \;
%
 \underline{b} \, + \sum_{i \in I} {[\alpha_i]}_3 \, \bar{Y}_i  \;\; \in \;\;  V\big( A \big/ \fa^{\,3} \,\big)  $$
which in turn implies  $ \, \alpha_i \in \fa^3 \, $,  for all  $ \, i \in I \, $.  Clearly, we can iterate this process, and find  $ \, \alpha_i \in \fa^n \, $  for all  $ \, n \in \N \, $,  $ \, i \in I \, $;  as  $ \, \fa^n = \{0\} \, $  for  $ \, n \gg 0 \, $  (since  $ \fa $  is generated by finitely many odd elements)  we end up with  $ \, \alpha_i = 0 \, $,  i.e.\  $ \, \hat{\eta}_i = \check{\eta}_i \, $,  for all  $ \, i \in I \, $.  This means  $ \, \hat{g}_- = \check{g}_- \, $,  whence  $ \, \hat{g}_+ = \check{g}_+ $  as well.
 \vskip5pt
   {\it (b)}\,  By construction there exists a morphism  $ \, \Theta^{\scriptscriptstyle \,<} : \mathbb{A}_\K^{0|d_1} \!\!\relbar\joinrel\longrightarrow G_-^{\scriptscriptstyle \,<} \, $  of set-valued functors that is given on  $ A $--points   --- for every single  $ \, A \in \Wsalg_\K \, $  ---   by the map
 \vskip-7pt
  $$  \Theta_{\!A}^{\scriptscriptstyle \,<} \, : \, \mathbb{A}_\K^{0|d_1}(A) := A_\uno^{\,\times d_1} \!\! \relbar\joinrel\longrightarrow\, G_-^{\scriptscriptstyle \,<}(A) \;\; ,  \quad  {\big( \eta_i \big)}_{i \in I} \, \mapsto \, \Theta_{\!A}^{\scriptscriptstyle \,<} \big( {\big( \eta_i \big)}_{i \in I\,} \big) := {\textstyle \prod\limits_{i \in I}^\rightarrow} \big(\, 1 + \eta_i \, Y_i \,\big)  $$
 \vskip-5pt
\noindent
 that is actually  {\sl surjective}.  We need to prove that all these maps  $ \Theta_{\!A}^{\scriptscriptstyle \,<} $  are also injective, so that overall  $ \Theta^{\scriptscriptstyle \,<} $  is indeed an isomorphism.
                                                                \par
%%%%%
 %%%%%
   \eject
 %%%%%
%%%%%
   Let  $ \, {\big( \hat{\eta}_i \big)}_{i \in I} \, , {\big( \check{\eta}_i \big)}_{i \in I} \in A_\uno^{\,\times d_1} \, $  be such that  $ \, \Theta_{\!A}^{\scriptscriptstyle \,<} \big(\! {\big( \hat{\eta}_i \big)}_{i \in I} \big) = \Theta_{\!A}^{\scriptscriptstyle \,<} \big(\! {\big( \check{\eta}_i \big)}_{i \in I} \big) \; $,  that is  $ \; {\textstyle \prod\limits_{i \in I}^\rightarrow} \big(\, 1 + \hat{\eta}_i \, Y_i \,\big) = {\textstyle \prod\limits_{i \in I}^\rightarrow} \big(\, 1 + \check{\eta}_i \, Y_i \,\big)  \; $.  Then we can replay the proof of claim  {\it (a)},  now with  $ \; \hat{g}_+ := 1 =: \check{g}_+ \; $:  the outcome will be again  $ \, \hat{\eta}_i = \check{\eta}_i \, $  for all  $ \, i \in I \, $,  i.e.~$ \, {\big( \hat{\eta}_i \big)}_{i \in I} = {\big( \check{\eta}_i \big)}_{i \in I} \; $.  Thus  $ \Theta_{\!A}^{\scriptscriptstyle \,<} $  is injective, as desired.
 \vskip7pt
   {\it (c)}\,  This is a direct consequence of  {\it (a)\/}  and  {\it (b)\/}  together.
\end{proof}

\medskip

\begin{free text}  \label{Lie sgroup on G_Pt}
 {\bf The Lie supergroup structure of  $ \GPt \, $.}  For any given  $ \, \cP \in \sHCp_\K \, $  and  $ \, A \in \Wsalg_\K \, $,  consider the group  $ \GPt(A) \, $.  Thanks to  Proposition \ref{dir-prod-fact-G_Pt - gen}{\it (c)},  we have a particular bijection
\begin{equation}  \label{eq:bij G+_x_A- >-->> G_Pt}
  \phi_A^{\,\circ} : \, G_+(A) \times \mathbb{A}_\K^{0|d_1\!}(A) \,\;{\buildrel \cong \over {\lhook\joinrel\relbar\joinrel\relbar\joinrel\relbar\joinrel\relbar\joinrel\relbar\joinrel\twoheadrightarrow}}\;\, \GPt(A)
\end{equation}
whose restriction to  $ \, G_+(A) \, $,  identified with  $ \; G_+(A) \times \big\{{(0)}_{i \in I}\big\} \, \subseteq G_+(A) \times \mathbb{A}_\K^{0|d_1}(A) \; $,  \,is the identity    --- onto the copy of  $ \, G_+(A) \, $  naturally sitting inside  $ \GPt(A) \, $.
                                                                           \par
   Now,  $ G_+(A) $  is by definition an  $ A_\zero $--manifold  (cf.\ \S \ref{Shvarts_embedding}),  of the same type (real smooth, etc.) as the super Harish-Chandra pair  $ \, \cP := (G_+ , \fg ) \, $  it pertains to; on the other hand,  $ \mathbb{A}_\K^{0|d_1\!}(A) $  carries natural, canonical structures of  $ A_\zero $--manifold  of any possible type (real smooth or real analytic if  $ \, \K = \R \, $,  complex holomorphic if  $ \, \K = \C \, $), in particular then also of the type of  $ G_+(A) \, $.  Then we know that there is also a canonical ``product structure'' of  $ A_\zero $--manifold   --- of the same type of  $ G_+(A) \, $,  i.e.\ of  $ \cP \, $  ---   onto the Cartesian product  $ \, G_+(A) \times \mathbb{A}_\K^{0|d_1}(A) \, $.  Using the bijection  $ \phi_A^{\,\circ} $  in  \eqref{eq:bij G+_x_A- >-->> G_Pt}  we push-forward this canonical  $ A_\zero $--manifold  structure of  $ \, G_+(A) \times \mathbb{A}_\K^{0|d_1}(A) \, $  onto  $ \GPt(A) \, $,  which then is turned into an  $ A_\zero $--manifold  on its own, still of the same type as  $ \cP \, $.
                                                    \par
   Strictly speaking, the structure of  $ A_\zero $--manifold  defined on  $ \GPt(A) $  formally depends on the choice of  $ G_-^{\scriptscriptstyle \,<} \, $,  hence of a totally ordered  $ \K $--basis  of  $ \fg_\uno \, $,  as this choice enters in the construction of  $ \phi_A^{\,\circ} $  in  \eqref{eq:bij G+_x_A- >-->> G_Pt}  above.  However, thanks to the special form of the defining relations of  $ \GPt(A) $  it is straightforward to show that  {\it changing such a basis amounts to changing local charts for the same, unique  $ A_\zero $--manifold  structure\/};  so in the end the structure is actually independent of such a choice.
 \vskip3pt
   Now, using the above mentioned structure of  $ A_\zero $--manifold  on  $ \GPt(A) $  for each  $ \, A \in \Wsalg_\K \, $,  given a morphism  $ \, f : A' \longrightarrow A'' \, $  in  $ \Wsalg_\K $  it is straightforward to check that the corresponding group morphism  $ \, \GPt(f) : \GPt(A') \longrightarrow \GPt(A'') \, $  is a morphism of  $ A'_\zero $--manifolds,  hence it is a morphism of  $ \cA_\zero $--manifolds  (cf.\ \S \ref{Shvarts_embedding}).  Thus  {\it  $ \GPt $  can also be seen as a functor from Weil  $ \K $--superalgebras  to  $ \cA_\zero $--manifolds  (real smooth, real analytic or complex holomorphic as  $ \cP $  is)}.
 \vskip3pt
   At last, again looking at the commutation relations in  $ \GPt(A) \, $,  we see that the group multiplication and the inverse map are ``regular'' (that is to say, ``real smooth'', ``real analytic'' or  ``complex holomorphic'' depending on the type of  $ \cP \, $);  indeed, this is explicitly proved by calculations like those needed in the proof of  Proposition \ref{prop:fact-G_Pt + sbgr-gener-G^[2]&G_-}{\it (b)}   --- that we skipped, so refer instead to the proof of  Proposition \ref{prop:fact-Liesgrp_G}{\it (b)}.  Thus they are morphisms of  $ \cA_\zero $--manifolds,  so  {\it  $ \GPt(A) $  is a group element among  $ \cA_\zero $--manifolds,  i.e.\ it is a Lie $ \cA_\zero $--group\/};  hence (cf.\  Proposition \ref{funct-char_Lie-supergrps}),  overall  {\it the functor  $ \GPt $  {\sl is a Lie supergroup},  of real smooth, real analytic or complex holomorphic type as  $ \cP $  is}.
 \vskip5pt
   Eventually, the outcome of this discussion   --- and core result of the present section ---   is the following statement, which provides a ``backward functor'' from sHCp's to Lie supergroups:
\end{free text}

\vskip15pt

\begin{theorem}  \label{thm_sHCp's-->Lsgrps - 1}
 The recipe in  Definition \ref{def G_- / G_Pt - gen}  provides functors
%
%   $$  \Psi^\circ : \sHCp_\R^\infty \longrightarrow \Lsgrp_\R^\infty  \quad ,  \!\!\qquad
% \Psi^\circ : \sHCp_\R^\omega \longrightarrow \Lsgrp_\R^\omega  \quad ,  \!\!\qquad
% \Psi^\circ : \sHCp_\C^\omega \longrightarrow \Lsgrp_\C^\omega  $$
%
 \vskip5pt
   \centerline{ $  \Psi^\circ : \sHCp_\R^\infty \longrightarrow \Lsgrp_\R^\infty  \quad ,  \!\!\!\qquad  \Psi^\circ : \sHCp_\R^\omega \longrightarrow \Lsgrp_\R^\omega  \quad ,  \!\!\!\qquad  \Psi^\circ : \sHCp_\C^\omega \longrightarrow \Lsgrp_\C^\omega  $ }
 \vskip5pt
\noindent
 given on objects by  $ \; \cP \mapsto \Psi^\circ(\cP) := \GPt \; $  and on morphisms by
%
%   $$  \Big(\, (\phi_+ \, , \varphi) : \, \cP' \relbar\joinrel\longrightarrow \cP'' \;\Big)
% \quad \mapsto \quad  \Big(\; \Psi^\circ\big( (\phi_+ \, , \varphi) \big) : \,
% \Psi^\circ\big(\cP'\big) := G_{{}_{\!\cP'}}^{\,\circ} \!\relbar\joinrel\longrightarrow
% G_{{}_{\cP''}}^{\,\circ} \! =: \Psi^\circ\big(\cP''\big) \,\Big)  $$
%
 \vskip3pt
    \centerline{ $  \Big(\, (\phi_+ \, , \varphi) : \, \cP' \relbar\joinrel\longrightarrow \cP'' \;\Big)  \quad \mapsto \quad  \Big(\; \Psi^\circ\big( (\phi_+ \, , \varphi) \big) : \, \Psi^\circ\big(\cP'\big) := G_{{}_{\!\cP'}}^{\,\circ} \!\relbar\joinrel\longrightarrow G_{{}_{\cP''}}^{\,\circ} \! =: \Psi^\circ\big(\cP''\big) \,\Big)  $ }
%%%
%%%%%
  \eject
%%%%%
%%%
%
 \vskip3pt
\noindent
 where the functor morphism  $ \;  \Psi^\circ\big( (\phi_+ \, , \varphi) \big) : \Psi^\circ\big(\cP'\big) := G_{{}_{\!\cP'}}^{\,\circ} \!\relbar\joinrel\longrightarrow G_{{}_{\cP''}}^{\,\circ} \! =: \Psi^\circ\big(\cP''\big) \; $  is defined by
\begin{equation}  \label{eq:def Psi on morphisms - 1}
  {\Psi^\circ\big( (\phi_+ \, , \varphi) \big)}_{\!A}  \quad :  \qquad  g'_+ \, \mapsto \, \phi_+\big(\,g'_+\big) \;\; ,
 \quad  \big( 1 + \eta \, Y' \,\big) \, \mapsto \, \big( 1 + \eta \, \varphi\big(Y'\big) \big)   \qquad \qquad
\end{equation}
for all  $ \, A \in \Wsalg_\K \, $,  $ \, g'_+ \in G'_+(A) \, $,  $ \, \eta \in A_\uno \, $,  $ \, Y' \in \fg'\, $,  with  $ \, \cP' = \big( G'_+ \, , \fg' \,\big) \, $  and  $ \, \cP'' = \big( G''_+ \, , \fg'' \,\big) \; $.
\end{theorem}

\begin{proof}
 What is still left to prove is only that the given definition for  $ \Psi^\circ\big( (\phi_+ \, , \varphi) \big) $  actually makes sense, as all the rest is already proved by our previous analysis   --- in particular, by \S \ref{Lie sgroup on G_Pt}  above ---   or is elementary.  Now,  \eqref{eq:def Psi on morphisms - 1}  above fixes the values of our would-be morphism  $ {\Psi^\circ\big( (\phi_+ \, , \varphi) \big)}_{\!A} $  {\sl on generators\/}  of  $ \, \Psi^\circ\big(\cP'\big)(A) := G_{{}_{\!\cP'}}^{\,\circ}(A) \, $:  then a straightforward check shows that all defining  {\sl relations\/}  among such generators   --- inside  $ G_{{}_{\!\cP'}}^{\,\circ}(A) $   ---   are mapped to corresponding (defining) relations in  $ G_{{}_{\!\cP''}}^{\,\circ}(A) \, $,  thus providing a unique, well-defined group morphism as required.  However, we must show that  {\sl this is a morphism of  $ \cA_\zero $--manifolds  too},  which needs some extra work.
                                                              \par
   Let  $ \, {\big\{ Y'_i \big\}}_{\! i \in I} \, $  and  $ \, {\big\{ Y''_j \big\}}_{\! j \in J} \, $  be  $ \K $--bases  of  $ \fg'_\uno $  and  $ \fg''_\uno $  respectively, both endowed with some fixed total order. Accordingly, both  $ G_{{}_{\!\cP'}}^{\,\circ}(A) $  and  $ G_{{}_{\!\cP''}}^{\,\circ}(A) $  admit factorizations as in  Proposition \ref{dir-prod-fact-G_Pt - gen}{\it (a)}   --- say of type  $ \, G_+ \times G_-^< \, $.  In particular, any given  $ \, g' \in G_{{}_{\!\cP'}}^{\,\circ}(A) \, $  uniquely factors into  $ \; g' \, = \, g'_+ \cdot {\overrightarrow{\prod}}_{i \in I} \big( 1 + \eta_i \, Y'_i \big) \; $;  then  $ {\Psi^\circ\big( (\phi_+ \, , \varphi) \big)}_{\!A} \, $,  being a group morphism, maps  $ g' $  onto
  $$  {\Psi^\circ\big( (\phi_+ \, , \varphi) \big)}_{\!A} \big(\,g'\big)  \; = \;  \phi_+\big(\,g'_+\big) \cdot {\textstyle {\overrightarrow{\prod}}_{i \in I}} \big( 1 + \eta_i \, \varphi\big(Y'_i\big) \big)  $$
and from this, letting  $ \; \varphi\big(Y'_i\big) \, = \, \sum_{j \in J} c_{i,j}\,Z_j \; $   --- with  $ \, c_{i,j} \in \K \, $  ---   we get
\begin{equation}  \label{eq:fact Psit_phi(g')}
  {\Psi^\circ\big( (\phi_+ \, , \varphi) \big)}_{\!A} \big(\,g'\big)  \, = \,  \phi_+\big(\,g'_+\big) \cdot {\textstyle \mathop{\overrightarrow{\prod}}\limits_{i \in I}} \Big( 1 \, + \, \eta_i \, \big(\, {\textstyle \sum_{j \in J\,}} c_{i,j}\,Z_j \big) \!\Big)  \, = \,  \phi_+\big(\,g'_+\big) \cdot {\textstyle \mathop{\overrightarrow{\prod}}\limits_{i \in I}} {\textstyle \prod\limits_{j \in J}} \Big( 1 \, + \, \eta_i \, c_{i,j}\,Z_j \Big)   \hskip5pt
\end{equation}
where in the second product in the rightmost term the order of factors is irrelevant, as they do commute with each other.  Now we must re-order the result according to the factorization of  $ \, G_{{}_{\!\cP''\!}}^{\,\circ}(A) $  of the form  $ \, G_+ \times G_-^< \; $;  in doing this, when we reorder the second factor  $ \; {\textstyle \mathop{\overrightarrow{\prod}}\limits_{i \in I}} {\textstyle \prod\limits_{j \in J}} \big( 1 \, + \, \eta_i \, c_{i,j}\,Z_j \big) \; $  in  \eqref{eq:fact Psit_phi(g')}  above we find   --- via calculations as in the proof of  Proposition \ref{prop:fact-G_Pt + sbgr-gener-G^[2]&G_-}  (which means like those for  Proposition \ref{prop:fact-Liesgrp_G}{\it (b)\/})  ---   an outcome of the form  $ \; \prod\limits_{r=1}^n \!\! {\big( 1 + a_r\,X_r \big)}_{\!{}_{G''_+}} \cdot \mathop{\overrightarrow{\prod}}\limits_{j \in J} \!\! \big( 1 + \alpha_j\,Z_j \big) \; $  where
 \vskip-5pt
   \quad   --- \;{\it (a)} \;  the  $ X_r $'s  belong to  $ \fg''_\zero \, $,
 \vskip3pt
   \quad   --- \;{\it (b)} \;  the  $ a_r $'s  are (even) polynomial expressions in the  $ \eta_i $'s,
 \vskip3pt
   \quad   --- \;{\it (c)} \;  the  $ \alpha_j $'s  are (odd) polynomial expressions in the  $ \eta_i $'s,
 \vskip5pt
\noindent
 Overall, this implies that the map
 \vskip-9pt
  $$  {\textstyle \mathop{\overrightarrow{\prod}}_{i \in I}} \big( 1 + \eta_i \, Y'_i \,\big)  \;\; \mapsto \;  {\textstyle \prod\limits_{r=1}^n} {\big( 1 + a_r\,X_r \big)}_{\!{}_{G''_+}} \cdot {\textstyle \mathop{\overrightarrow{\prod}}_{j \in J}} \big( 1 + \alpha_j\,Z_j \big)  $$
 \vskip-7pt
\noindent
 is a  {\sl map of $ \cA_\zero $--manifolds\/}  from  $ \, {\big( G_{{}_{\!\cP'}}^{\,\circ} \!\big)}_-^<(A) \, $  to  $ \, G_{{}_{\!\cP''\!}}^{\,\circ}(A) \, $.  But  $ \; \phi_+ : G'_+(A) \relbar\joinrel\longrightarrow G''_+(A) \; $  is a map of  $ \cA_\zero $--manifolds  too, by assumptions; this along with the previous remark  {\sl and\/}  \eqref{eq:fact Psit_phi(g')}  above eventually implies that  $ \, {\Psi^\circ\big( (\phi_+ \, , \varphi) \big)}_{\!A} \, $  is a map of  $ \cA_\zero $--manifolds  as claimed.
\end{proof}

\medskip

\subsection{Supergroup functors out of super Harish-Chandra pairs: second recipe}  \label{sgroups-out-sHCp's-2}

\smallskip

   {\ } \;\;   In this subsection we construct a second Lie supergroup functor, denoted  $ \GPe \, $,  which we later prove is a Lie supergroup: this is in fact a ``sibling alternative'' to the functor  $ \GPt $  considered in  \S \ref{sgroups-out-sHCp's-1}  above.  As a matter of notation, recall that for any  $ \, \cP = (G_+ \,,\, \fg) \in \sHCp_\K \, $,  $ \, A \in \Wsalg_\K \, $  and  $ \, \cX \in A_\zero \otimes_\K \fg_\zero \, $  there exists a well-defined  $ \, \exp_{G_+}\!(\cX) \in G_+(A_\zero) \, $;  furthermore, if in particular  $ \, \cX \in A_\uno^{\,2} \otimes_\K \fg_\zero \, $,  then the formal series expansion of  $ \exp(\cX) $  can be actually realized as a finite sum.
%%%%%
 %%%%%
  \eject
 %%%%%
%%%%%

\noindent
 When no confusion is possible we shall drop the subscript  $ G_+ $  and simply write  $ \, \exp(\cX) \, $  instead.  Similarly, we shall presently introduce new formal elements of type  ``$ \, \exp(\cY) \, $''  with  $ \, \mathcal{Y} \in A_\uno \otimes_K \fg_\uno \, $.  Finally, we extend the built-in  $ G_+ $--action  onto  $ \fg_\uno $  to a (same-name)  $ G_+ $--action  onto  $ \, A_\uno \otimes_K \fg_\uno \, $ by  $ \; \text{\sl Ad}(g)\big( \sum_{s=1}^n \eta_s Y_s \big) \, := \, \sum_{s=1}^n \eta_s \, \text{\sl Ad}(g)(Y_s) \; $  for all  $ \,\; \sum_{s=1}^n \eta_s Y_s := \sum_{s=1}^n \eta_s \otimes Y_s \, \in \, A_\uno \otimes_K \fg_\uno \; $.

\vskip15pt

\begin{definition}  \label{def G_- / G_Pe - gen}
 Let  $ \, \cP := \big( G_+ \, , \fg \big) \in \sHCp_\K \, $  be a super Harish-Chandra pair over  $ \K \, $.  %
 \vskip7pt
   {\it (a)}\,  We introduce a functor  $ \; \GPe : \Wsalg_\K \!\relbar\joinrel\longrightarrow \grp \; $  as follows.  For any Weil superalgebra  $ \, A \in \Wsalg_\K \, $,  we define  $ \, \GPe(A) \, $  as being the group with generators the elements of the set
 \vskip-13pt
  $$  \varGamma_{\!A}  \,\; := \;\,  \big\{\, g_+ \, , \exp(\cY) \,\big|\, g_+ \in G_+(A) \, , \, \mathcal{Y} \in A_\uno \otimes_\K \fg_\uno \,\big\}  \,\; = \;\,  G_+(A) \;{\textstyle \bigcup}\; {\big\{ \exp(\cY) \big\}}_{\cY \in A_\uno \otimes_\K \fg_\uno}  $$
and relations (for  $ \, g_+ \, , g'_+ \, , g''_+ \in G_+(A) \, $,  $ \, \cY \, , \cY' \, , \cY'' \in A_\uno \!\otimes_\K \fg_\uno \, $)
 \vskip-11pt
  $$  \displaylines{
   g'_+ \cdot\, g''_+  \,\; = \;\,  g'_+ \,\cdot_{\!\!\!{}_{G_+}} g''_+  \quad  ,   \qquad  \exp(0)  \,\; = \;\,  1  \quad ,   \qquad  \exp\big(\cY) \cdot g_+  \,\; = \;\,  g_+ \cdot\, \exp\big( \text{\sl Ad}\big(g_+^{-1}\big)(\cY) \big)  \cr
   \exp\!\big( \cY' \big) \cdot \exp\!\big( \cY'' \big)  \,\; = \;\,  \exp\!\Big( P_\zero^{(d_1)}\big( \cY' , \cY'' \big) \Big) \cdot \exp\!\Big(\, \cY' + \cY'' + P_\uno^{(d_1)}\big( \cY' , \cY'' \big) \Big)  }  $$
with  $ \, P_\zero^{(d_1)} \, $  and  $ \, P_\uno^{(d_1)} \, $  as given in  Lemma \ref{lemma:relations-in-G(A)}{\it (h)}.  This yields the functor  $ \GPe $  on objects.
 \vskip3pt
   To define  $ \GPe $  on morphisms, for any morphism  $ \, f : A' \longrightarrow A'' \, $  in  $ \Wsalg_\K $  we define the group morphism  $ \, \GPe(f) : \GPe\big(A'\big) \longrightarrow \GPe\big(A''\big) \, $  as being the unique one given on generators   --- for all  $ \, g'_+ \in G_+\big(A'\big) \, $,  $ \, \eta \in A'_\uno \, $,  $ \, \cY' \in A'_\uno \!\otimes_\K \fg_\uno \, $  ---   by
  $$  \GPe(f)\big(\,g'_+\big)  \, := \,  G_+(f)\big(g'_+\big)  \quad ,  \qquad  \GPe(f)\big( \exp\big(\cY) \big)  \, := \,  \exp\big( f(\cY) \big)  $$
where  $ \; f\big(\cY'\big) \,:= \, \sum_{s=1}^n f\big(\eta'_s\big) Y_s \; $  for all  $ \; \cY' := \sum_{s=1}^n \eta'_s \, Y_s \, \in \, A'_\uno \!\otimes_K \fg_\uno \; $.
 \vskip7pt
   {\it (b)}\,  We define a functor  $ \; \GPem : \Wsalg_\K \!\relbar\joinrel\longrightarrow \set \; $  on any object  $ \, A \in \Wsalg_\K \, $  by
%
% \vskip-15pt
%
  $$  \GPem(A)  \; := \;  \Big\langle \exp\big( A_\uno \!\otimes_K \fg_\uno \big) \Big\rangle  \qquad  \big(\, \subseteq \GPe(A) \big)   \qquad  \big(\, \subseteq \GPe(A) \big)  $$
--- the subgroup of  $ \GPe(A) $  generated by  $ \; \exp\big( A_\uno \!\otimes_K \fg_\uno \big) \, := \, {\big\{ \exp(\cY) \big\}}_{\cY \in A_\uno \otimes_K \fg_\uno} $  ---   and on morphisms in the obvious way. By definition,  $ \GPem $  can be thought of as subfunctor of  $ \GPe \, $.
 \hfill   $ \diamondsuit $
\end{definition}

\smallskip

\begin{free text}  \label{constr-tGamma_P-e}
 {\bf Another realization of  $ \GPe \, $.}  Given a super Harish-Chandra pair  $ \, \cP = \big( G_+ \, , \fg \big) \in \sHCp_\K \, $,  we present now another way of realizing the  $ \K $--supergroup  $  \GPe $  introduced in  Definition \ref{def G_- / G_Pe - gen}{\it (a)\/}:  this mimics what we did in  \S \ref{constr-tGamma_P-t},  so we keep the same kind of notation and are a bit shorter.
 \vskip3pt
   For any fixed  $ \, A \in \Wsalg_\K \, $,  we denote by  $ \, G_+^{\langle 2 \rangle}(A) \, $  the subgroup of  $ G_+(A) $  generated by the set
 $ \; \big\{\, \exp(\cX) \;\big|\; \cX \in A_\uno^{[2]} \!\otimes_\K [\, \fg_\uno , \fg_\uno ] \,\big\} \; $.
 Then one easily sees that  $ G_+^{\langle 2 \rangle}(A) $  is  {\sl normal\/}  in  $ G_+(A) \, $
%
% ,  as one easily sees by construction (taking into account that, as  $ \, \cP :=
% \big( G_+ \, , \fg \big) \, $  is a super Harish-Chandra pair, the ``adjoint''
% action of  $ G_+ $  onto  $ \fg $  maps  $ [\,\fg_\uno,\fg_\uno] $  into itself).
%
                                                                 \par
  We consider also the three sets
%
%%%
%   $$  \displaylines{
%    \Gamma_{\!A,+}  \,\; := \;\,  \big\{\, g_+ \,\big|\; g_+ \! \in G_+(A) \big\}  \cr
%    \Gamma'_{\!A,-}  \,\; := \;\,  \big\{\, \big( 1 + \eta' \eta'' \, X \big) \,\big|\, \eta',
% \eta'' \in A_\uno \, , X \in [\fg_\uno,\fg_\uno] \cup \fg_\uno^{\langle 2 \rangle} \,\big\}  \cr
%    \Gamma_{\!A,-}  \; := \;  \Gamma'_{\!A,-} \,{\textstyle \bigcup}\, {\big\{ (1 + \eta \, Y)
% \big\}}_{\eta \in A_\uno}^{Y \in \fg_\uno}  }  $$
%%%
%
  $$  \varGamma_{\!A}^{\,+}  \; := \;  G_+(A)  \;\; ,  \quad
 \varGamma_{\!A}^{\,\langle 2 \rangle} \; := \; G_+^{\langle 2 \rangle}(A)  \;\; ,  \quad
 \varGamma_{\!A}^{\,-}  \; := \;  \varGamma_{\!A}^{\,\langle 2 \rangle} \;{\textstyle \bigcup}\; \exp\!\big( A_\uno \!\otimes_K \fg_\uno \big)  $$
and the five sets of relations   --- for all  $ \, g_+ \, , g'_+ \, , g''_+ \in \varGamma_{\!A}^{\,+} $,  $ \, g_{\langle 2 \rangle} \, , g'_{\langle 2 \rangle} \, , g''_{\langle 2 \rangle} \in \varGamma_{\!A}^{\,\langle 2 \rangle} $,  $ \, \cY , \cY' , \cY'' \in A_\uno \otimes_\K \fg_\uno \, $
%
% ,  with  $ \;\cdot_{\!\!\!{}_{G_+}} $  and  $ \;\cdot_{\!\!\!{}_{G_+^{\langle 2 \rangle}}} $
% being the product in  $ G_+(A) $  and in  $ G_+^{\langle 2 \rangle}(A) \, $
%
 ---   given by
 \vskip-13pt
  $$  \displaylines{
   \mathcal{R}_{\!A}^+ \, : \quad  g'_+ \cdot\, g''_+  \,\; = \;\,  g'_+ \,\cdot_{\!\!\!{}_{G_+}} g''_+  \cr
%%%
   \mathcal{R}_{\!A}^- \, : \,
 \begin{cases}
   \;\;  g'_{\langle 2 \rangle} \cdot\, g''_{\langle 2 \rangle}  \, = \;  g'_{\langle 2 \rangle} \,\cdot_{\!\!\!{}_{G_+^{\langle 2 \rangle}}} g''_{\langle 2 \rangle}  \;\;  ,  \;\quad  \exp(\cY) \cdot g_{\langle 2 \rangle}  \, = \;  g_{\langle 2 \rangle} \cdot \exp\!\big( \text{\sl Ad}\big(g_{\langle 2 \rangle}^{-1}\big)(\cY) \big)  \;\;  ,  \;\quad  \exp(0)  \, = \,  1  \cr
   \qquad \qquad   \exp\!\big( \cY' \big) \, \exp\!\big( \cY'' \big)  \,\; = \;\,  \exp\!\Big( P_\zero^{(d_1)}\big( \cY' , \cY'' \big) \Big) \, \exp\!\Big(\, \cY' + \cY'' + P_\uno^{(d_1)}\big( \cY' , \cY'' \big) \Big)  \cr
   \qquad \qquad \qquad \qquad  \text{with  $ \, P_\zero^{(d_1)} \, $  and  $ \, P_\uno^{(d_1)} \, $  as given in  Lemma \ref{lemma:relations-in-G(A)}{\it (h)}}
 \end{cases}   \cr
%%%
   \mathcal{R}_{\!A}^\ltimes \, : \quad  g_{\langle 2 \rangle} \cdot g_+  \, = \;  g_+ \cdot \big(\, g_+^{-1} \,\cdot_{\!\!\!{}_{G_+}} g_{\langle 2 \rangle} \,\cdot_{\!\!\!{}_{G_+}} g_+ \big) \;\; ,  \;\;\;
  \exp(\cY) \cdot g_+  \, = \;  g_+ \cdot \exp\!\big( \text{\sl Ad}\big(g_+^{-1}\big)(\cY) \big)   \phantom{\Big|^{\big|}}  \cr
%%%
   \mathcal{R}_{\!A}^{\langle 2 \rangle} \, : \quad  \big( g_{\langle 2 \rangle} \big)_{{\varGamma_A^{\,\langle 2 \rangle}}}  = \;\,  \big( g_{\langle 2 \rangle} \big)_{{\varGamma_A^{\,+}}}   \phantom{\big|^{|}}  \cr
%%%
   \mathcal{R}_{\!A}  \; := \;  \mathcal{R}_{\!A}^+ \,{\textstyle \bigcup}\, \mathcal{R}_{\!A}^- \,{\textstyle \bigcup}\, \mathcal{R}_{\!A}^\ltimes \,{\textstyle \bigcup}\, \mathcal{R}_{\!A}^{\langle 2 \rangle}   \phantom{\big|^{\big|}}  }  $$
%
% (in particular, note that the relations of type  $ \mathcal{R}_{\!A}^{\langle 2 \rangle} $
% in down-to-earth terms just identify each element in  $ \varGamma_A^{\,\langle 2 \rangle} $
% with its corresponding copy inside  $ \varGamma_A^+ \, $).
%
 Then we  {\sl define\/}  a new group, by generators and relations, namely  $ \; \GPem(A) := \big\langle\, \varGamma_{\!A}^- \,\big\rangle \Big/ \big(\, \mathcal{R}_{\!A}^- \,\big) \;\, $.
 \vskip5pt
   Directly from  Definition \ref{def G_- / G_Pe - gen}  it follows that
\begin{equation}   \label{eq:1st-present-G_Pe}
  \hskip3pt   \GPe(A)  \,\; \cong \;\,  \big\langle\, \varGamma_{\!A}^+ \,{\textstyle \bigcup}\; \varGamma_{\!A}^- \,\big\rangle \Big/ \big( \mathcal{R}_{\!A} \big)  \,\; = \;\,  \big\langle\, \varGamma_{\!A}^+ \,{\textstyle \bigcup}\; \varGamma_{\!A}^- \,\big\rangle \bigg/ \Big(\, \mathcal{R}_{\!A}^+ \,{\textstyle \bigcup}\; \mathcal{R}_{\!A}^- \,{\textstyle \bigcup}\; \mathcal{R}_{\!A}^\ltimes \,{\textstyle \bigcup}\; \mathcal{R}_{\!A}^{\langle 2 \rangle} \,\Big)
\end{equation}
%
% indeed, here above we are just taking larger sets of generators and of relations
% (w.r.t.\  Definition \ref{def G_- / G_Pe - gen}),  but with enough redundancies
% as to find a different presentation of  {\sl the same\/}  group.
% %
%  \vskip5pt
% %
%    From this we find a neat description of  $ \GPe(A) $  by achieving
%
 but we can also achieve
 the presentation  \eqref{eq:1st-present-G_Pe}  in a series of intermediate steps, just adding one set of relations at a time.  As a first step, we have
\begin{equation}   \label{eq:2nd-present-G_Pe}
  \big\langle\, \varGamma_{\!A}^+ \,{\textstyle \bigcup}\; \varGamma_{\!A}^- \,\big\rangle \Big/ \big(\, \mathcal{R}_{\!A}^+ \,{\textstyle \bigcup}\; \mathcal{R}_{\!A}^- \,\big)  \,\; = \;\,  \big\langle\, \varGamma_{\!A}^+ \,\big\rangle \Big/ \big(\, \mathcal{R}_{\!A}^+ \,\big)  \,\; * \;\, \big\langle\, \varGamma_{\!A}^- \,\big\rangle \Big/ \big(\, \mathcal{R}_{\!A}^- \,\big)  \,\; \cong \;\,  G_+(A) \,*\, \GPem(A)
\end{equation}
where  $ \; G_+(A) \, \cong \, \big\langle\, \varGamma_{\!A}^+ \,\big\rangle \Big/ \big(\, \mathcal{R}_{\!A}^+ \,\big) \; $  by construction and  $ \, * \, $  denotes the free product (of two groups).
 \vskip5pt
  For the next two steps we can follow two different lines of action.  The first one gives
  $$  \big\langle\, \varGamma_{\!A}^+ \,{\textstyle \bigcup}\; \varGamma_{\!A}^- \,\big\rangle \Big/ \big(\, \mathcal{R}_{\!A}^+ \,{\textstyle \bigcup}\; \mathcal{R}_{\!A}^- \,{\textstyle \bigcup}\; \mathcal{R}_{\!A}^\ltimes \,\big)  \;\; \cong \;\;  \Big(\, G_+(A) * \GPem(A) \Big) \bigg/ \Big(\, \overline{\mathcal{R}_{\!A}^\ltimes} \,\Big)  \;\; \cong \;\;  G_+(A) \ltimes \GPem(A)  $$
because of  \eqref{eq:Double-Quot-Thm}  and  \eqref{eq:2nd-present-G_Pe}  together, where  $ \, G_+(A) \ltimes \GPem(A) \, $  is the semidirect product of  $ \, G_+(A) \, $  with  $ \, \GPem(A) \, $.
%
%   with respect to the obvious (``adjoint'') action of the former on the latter.
%
 Then
  $$  \displaylines{
   \qquad   \big\langle\, \varGamma_{\!A}^+ \,{\textstyle \bigcup}\; \varGamma_{\!A}^- \,\big\rangle \Big/ \big(\, \mathcal{R}_{\!A} \big)  \;\; \cong \;\;  \big\langle\, \varGamma_{\!A}^+ \,{\textstyle \bigcup}\; \Gamma_{\!A}^- \,\big\rangle \bigg/ \Big(\, \mathcal{R}_{\!A}^+ \,{\textstyle \bigcup}\; \mathcal{R}_{\!A}^- \,{\textstyle \bigcup}\; \mathcal{R}_{\!A}^\ltimes \,{\textstyle \bigcup}\; \mathcal{R}_{\!A}^{\langle 2 \rangle} \,\Big)  \;\; \cong   \hfill  \cr
   \hfill   \cong \;\;  \Big( G_+(A) \ltimes \GPtm(A) \Big) \bigg/ \Big(\; \overline{\mathcal{R}_{\!A}^{\langle 2 \rangle}} \;\Big)  \;\; \cong \;\;  \Big( G_+(A) \ltimes \GPtm(A) \Big) \bigg/ N_{\langle 2 \rangle}(A)   \qquad  }  $$
where  $ \, N_{\langle 2 \rangle}(A) \, $  is the normal subgroup of  $ \; G_+(A) \, \ltimes \, \GPem(A) \; $  generated by
 $ \, {\Big\{\! \big(\, g_{\langle 2 \rangle} \, , g_{\langle 2 \rangle}^{-1} \,\big) \!\Big\}}_{g_{\langle 2 \rangle} \in \varGamma_A^{\langle 2 \rangle}} \; $.
 This together with  \eqref{eq:1st-present-G_Pe}  eventually yields
  $$  \GPe(A)  \;\; = \;\;  \Big( G_+(A) \ltimes G_-(A) \Big) \bigg/ N_{\langle 2 \rangle}(A)  $$
 \vskip3pt
   On the other hand,  \eqref{eq:Double-Quot-Thm}  and  \eqref{eq:2nd-present-G_Pe}  jointly give
  $$  \big\langle\, \varGamma_{\!A}^+ \,{\textstyle \bigcup}\; \varGamma_{\!A}^- \,\big\rangle \bigg/ \Big(\, \mathcal{R}_{\!A}^+ \,{\textstyle \bigcup}\; \mathcal{R}_{\!A}^- \,{\textstyle \bigcup}\; \mathcal{R}_{\!A}^{\langle 2 \rangle} \,\Big)  \;\; \cong \;\;  \Big( G_+(A) \,*\, \GPem(A) \Big) \bigg/ \Big(\; \overline{\mathcal{R}_{\!A}^{\langle 2 \rangle}} \,\Big)  \;\; \cong \;\;  G_+(A) \!\!\!\mathop{*}_{G_+^{\langle 2 \rangle}(A)}\!\!\! \GPem(A)  $$
with  $ \, G_+(A) \!\!\mathop{*}\limits_{G_+^{\langle 2 \rangle}(A)}\!\! \GPem(A) \, $  the amalgamated product of  $ \, G_+(A) \, $  and  $ \, \GPem(A) \, $  over  $ \, G_+^{\langle 2 \rangle}(A) \, $.
%
% w.r.t.\  the natural monomorphisms  $ \; G_+^{\langle 2 \rangle}(A)
% \lhook\joinrel\relbar\joinrel\longrightarrow G_+(A) \; $  and  $ \;
% G_+^{\langle 2 \rangle}(A) \lhook\joinrel\relbar\joinrel\longrightarrow \GPem(A) \; $.
%
 Then
  $$  \displaylines{
   \quad   \big\langle\, \varGamma_{\!A}^+ \,{\textstyle \bigcup}\; \varGamma_{\!A}^- \,\big\rangle \Big/ \big(\, \mathcal{R}_{\!A} \big)  \,\;\; \cong \;\;\,  \big\langle\, \varGamma_{\!A}^+ \,{\textstyle \bigcup}\; \varGamma_{\!A}^- \,\big\rangle \bigg/ \Big(\, \mathcal{R}_{\!A}^+ \,{\textstyle \bigcup}\; \mathcal{R}_{\!A}^- \,{\textstyle \bigcup}\; \mathcal{R}_{\!A}^{\langle 2 \rangle} \,{\textstyle \bigcup}\; \mathcal{R}_{\!A}^\ltimes \,\Big)  \,\;\; \cong   \hfill  \cr
   \hfill   \cong \;\;\,  \Big( G_+(A) \!\mathop{*}\limits_{G_+^{\langle 2 \rangle}(A)}\! \GPem(A) \Big) \bigg/ \Big(\, \overline{\mathcal{R}_{\!A}^\ltimes} \;\Big)  \,\;\; \cong \;\;\,  \Big( G_+(A) \!\mathop{*}\limits_{G_+^{\langle 2 \rangle}(A)}\! \GPem(A) \Big) \bigg/ N_\ltimes(A)   \quad  }  $$
where  $ \, N_\ltimes(A) \, $  is the normal subgroup of  $ \; G_+(A) \!\mathop{*}\limits_{G_+^{\langle 2 \rangle}(A)}\! \GPem(A) \; $  generated by
 \vskip5pt
   \centerline{ $ {\Big\{\, g_+ \exp(\cY) \, g_+^{-1} {\exp\!\big( \text{\sl Ad}(g_+)(\cY) \big)}^{-1} \Big\}}_{\cY \in A_\uno \otimes_\K \fg_\uno}^{g_+ \in G_+(A)}  \;\;{\textstyle \bigcup} \;\;  {\Big\{\, g_+ \, g_{\langle 2 \rangle} \, g_+ {\big( g_+ \cdot_{\!\!\!{}_{G_+}}\!\! g_{\langle 2 \rangle} \cdot_{\!\!\!{}_{G_+}}\!\! g_+ \big)}^{-1} \Big\}}_{g_+ \in G_+(A)}^{g_{\langle 2 \rangle} \in \varGamma_A^{\langle 2 \rangle}} $ }
%%%%%
 %%%%%
   \eject
 %%%%%
%%%%%
%
 \vskip7pt

   All this along with  \eqref{eq:1st-present-G_Pe}  eventually gives
  $$  \GPe(A)  \,\;\; = \;\;\,  \Big( G_+(A) \!\mathop{*}\limits_{G_+^{\langle 2 \rangle}(A)}\! \GPem(A) \Big) \bigg/ N_\ltimes(A)  $$
for all  $ \, A \in \Wsalg_\K \, $.  In functorial terms this yields
  $$  \GPe \; = \; \Big( G_+ \ltimes \GPem \Big) \bigg/ N_{\langle 2 \rangle}  \!\qquad  \text{and}  \!\qquad \GPe \; = \; \Big( G_+ \!\mathop{*}\limits_{G_+^{\langle 2 \rangle}}\! \GPem \Big) \bigg/ N_\ltimes  \!\quad ,  \,\qquad
   \text{or}  \quad  \GPe  = \;  G_+ \!\mathop{\ltimes}\limits_{G_+^{\langle 2 \rangle}}\! \GPem  $$
 so that
%
% where the last, (hopefully) more suggestive notation  $ \,\; \GPe \, = \; G_+
% \!\mathop{\ltimes}\limits_{G_+^{\langle 2 \rangle}}\! \GPem \;\, $  tells us that
%
 $ \, \GPe $  is the ``amalgamate semidirect product'' of  $ G_+ $  and  $ \GPem $  over
%
% their common subgroup
%
 $ G_+^{\langle 2 \rangle} \; $.
\end{free text}

\medskip

\subsection{The supergroup functor  $ \GPe $  as a Lie supergroup}  \label{subsec:struct-G_Pe}

\smallskip

   {\ } \;\;  We aim now to proving that the functor  $ \GPe $  is actually a Lie supergroup.  Keeping notations as before, we follow closely the same line of arguing as in  \S \ref{subsec:struct-G_Pt},  so we can be somewhat shorter.

\vskip3pt

   We begin with the following ``factorization result'' for  $ \GPe \, $:

\medskip

\begin{proposition}  \label{prop:fact-G_Pe + sbgr-gener-G^<2>&G_-}
   Let  $ \, \cP := \big( G_+ \, , \fg \big) \in \sHCp_\K \, $  be a super Harish-Chandra pair over  $ \K \, $.  Then:
 \vskip5pt
   (a)\,  there exist group-theoretic factorizations
 \vskip-7pt
  $$  \GPem(A)  \; = \;  G_+^{\langle 2 \rangle}(A) \cdot \exp\!\big( A_\uno \!\otimes_K \fg_\uno \big)  \quad ,  \qquad
  \GPem(A) \; = \; \exp\!\big( A_\uno \!\otimes_K \fg_\uno \big) \cdot G_+^{\langle 2 \rangle}(A)  $$
 \vskip1pt
   (b)\,  there exist group-theoretic factorizations
 \vskip-7pt
  $$  \GPe(A) \; = \; G_+(A) \cdot \exp\!\big( A_\uno \!\otimes_K \fg_\uno \big)  \quad ,  \qquad  \GPe(A) \; = \; \exp\!\big( A_\uno \!\otimes_K \fg_\uno \big) \cdot G_+(A)  $$
\end{proposition}

\begin{proof}
 Claim  {\it(a)\/}  is the exact analogue of  \eqref{eq:factor-2_G^-},  and claim  {\it(b)\/}  the analogue of  \eqref{eq:factor-2_G},  in  Proposition \ref{prop:fact-Liesgrp_G}{\it (c)}.   In both cases the proof (up to details) is the same, so we can skip it.
\end{proof}

\medskip

\begin{free text}  \label{G_Pe-module V}
 {\bf The representation  $ \, \GPe \!\relbar\joinrel\longrightarrow \rGL(V) \, $.}  Let  $ \, \cP = \big( G_+ \, , \fg \big) \in \sHCp_\K \, $  be any given super Harish-Candra pair over  $ \K \, $.  Just like we did for $ \GPt $  in  \S \ref{G_Pt-module V},  we need for  $ \GPe $  as well a suitable linear representation  $ V $,  which we now define along the same lines, keeping the same notation.
 \vskip5pt
   Let  $ U(\fg) $  be the universal enveloping algebra of  $ \fg \, $,  and let
  $$  V  \; := \;  \text{\sl Ind}_{\fg_\zero}^{\,\fg}(\Uuno\,)  \; = \;  U(\fg) \! \mathop{\otimes}\limits_{U(\fg_\zero)} \! \Uuno  \; \cong \;  {\textstyle \bigwedge}\, \fg_\uno \mathop{\otimes}\limits_\K \Uuno  \,\; \cong \;  {\textstyle \bigwedge}\, \fg_\uno  $$
be the  $ \fg $--representation  induced from the trivial representation  $ \Uuno $  of  $ \fg_\zero $  (as in  \S \ref{G_Pt-module V}).  As we saw in  \S \ref{G_Pt-module V},  there exists a morphism  $ \, (\boldsymbol{r}_{\!+},\rho) : (G_+\,\fg) \longrightarrow \big( \rGL(V),\rgl(V)\big) \, $  of super Harish-Chandra pairs  that gives to  $ V $  a structure  $ (G_+,\fg) $--module;  by  $ \rho $  again we denote also the representation map  $ \, \rho : U(\fg) \longrightarrow \End_{\,\K}(V) \, $  giving the  $ U(\fg) $--module  structure on  $ V \, $,  and similarly   --- in functorial way ---   for the representation maps of  $ {\big( A \otimes_\K \fg \big)}_\zero $  and  $ {\big( A \otimes_\K U(\fg) \big)}_\zero $  onto  $ {\big( A \otimes_\K V \big)}_\zero \, $.
 \vskip3pt
   We will now show that the  $ (G_+,\fg) $--module  structure on  $ V $  can be naturally ``integrated'' to a  $ \GPe $--module  structure too.
\end{free text}

\smallskip

\begin{proposition}  \label{G_Pe-action on V}
 Retain notation as above for the  $ (G_+,\fg) $--module  $ V $.  There exists a unique structure of (left)  $ \GPe $--module  onto  $ V $  which satisfies the following conditions: for every  $ \, A \in \Wsalg_\K \, $,  the representation map  $ \, \boldsymbol{r}_{{}_{\!\cP,A}}^{\,e} \! : \GPe(A) \longrightarrow \rGL(V)(A) \, $  is given on generators of\/  $ \GPe(A) $   --- namely, all  $ \, g_+ \in G_+(A) \, $  and  $ \, \exp(\cY) \, $  for  $ \, i \in I \, $,  $ \, \cY \in A_\uno \!\otimes_\K \fg_\uno $  ---   by
  $$  \boldsymbol{r}_{{}_{\!\cP,A}}^{\,e}(g_+) \, := \, \boldsymbol{r}_{\!+}(g_+) \quad ,  \qquad  \boldsymbol{r}_{{}_{\!\cP,A}}^{\,e}\big(\exp(\cY)\big) \, := \, \rho\big(\exp(\cY)\big) \, = \, \exp\!\big(\rho(\cY)\big)  $$
that is,  $ \; g_+.v \, := \, \boldsymbol{r}_{\!+}(g_+)(v) \; $  and  $ \; \exp(\cY).v \, := \, \exp\!\big(\rho(\cY)\big)(v) \; $   --- with  $ \exp\!\big(\rho(\cY)\big) $  being a finite sum ---   for all  $ \, v \in V(A) \, $.  Overall, this yields a morphism of  $ \K $--supergroup  functors  $ \; \boldsymbol{r}_{{}_{\!\cP}}^{\,e} \! : \GPe \!\longrightarrow \rGL(V) \; $.
\end{proposition}

\begin{proof}
 This follows from the whole construction: indeed, by definition of representation for the super Harish-Chandra pair  $ \cP \, $,  the operators  $ \boldsymbol{r}_{{}_{\!\cP,A}}^{\,e}(g_+) $  and  $ \boldsymbol{r}_{{}_{\!\cP,A}}^{\,e}\!\big(\exp(\cY)\big) $  on  $ V $
%
% --- associated with the generators of  $ \GPe(A) $  ---
%
 satisfy all relations which, by  Definition \ref{def G_- / G_Pe - gen}{\it (a)},  are satisfied by the generators of  $ \GPe(A) \, $.
%
% Thus they define a unique group morphism  $ \, \boldsymbol{r}_{{}_{\!\cP,A}}^{\,e} \! :
% \GPe(A) \longrightarrow \rGL(V)(A) \, $  as required.  The whole construction is clearly
% functorial in  $ A \, $,  whence the claim.
%
 So they define a group morphism  $ \, \boldsymbol{r}_{{}_{\!\cP,A}}^{\,e} \! : \GPe(A) \longrightarrow \rGL(V)(A) \, $,  functorial in  $ A $  by construction, whence the claim.
\end{proof}

\medskip

   We are now ready to state the ``global splitting theorem'' for  $ \GPe $  (cf.\  Theorem \ref{thm:dir-prod-fact-G - gen}):

\medskip

\begin{proposition}  \label{dir-prod-fact-G_Pe - gen}  {\ }
 \vskip5pt
   {\it (a)} \,  The restriction of group multiplication in  $ \GPe $  provides isomorphisms of (set-valued) functors
 \vskip-15pt
  $$  \displaylines{
   G_+ \times \exp\!\big( {(-)}_\uno \!\otimes_\K \fg_\uno \big)  \; \cong \;  \GPe  \quad ,   \qquad  \exp\!\big( {(-)}_\uno \!\otimes_\K \fg_\uno \big) \times G_+  \; \cong \;  \GPe  \cr
   G_+^{\langle 2 \rangle} \times \exp\!\big( {(-)}_\uno \!\otimes_\K \fg_\uno \big)  \; \cong \;  \GPem  \quad ,   \qquad  \exp\!\big( {(-)}_\uno \!\otimes_\K \fg_\uno \big) \times G_+^{\langle 2 \rangle}  \; \cong \;  \GPem  }  $$
 \vskip3pt
   {\it (b)} \,  For any fixed  $ \K $--basis  $ {\big\{ Y_i \big\}}_{i \in I} $  of  $ \fg_\uno \, $,  there exists an isomorphism of (set-valued) functors  $ \; \mathbb{A}_\K^{0|d_1} \! \cong \exp\!\big( {(-)}_\uno \!\otimes_\K \fg_\uno \big) \; $,  with  $ \, d_1 := \text{\it dim}_{\,\K}\big(\fg_\uno\big) = |I| \, $,  given on  $ A $--points,  for  $ \, A \in \Wsalg_\K \, $,  by
 \vskip-3pt
  $$  \mathbb{A}_\K^{0|d_1}\!(A)  \; = \;  A_\uno^{\,d_1} \relbar\joinrel\relbar\joinrel\longrightarrow \exp\!\big( A_\uno \!\otimes_\K \fg_\uno \big) \;\; ,  \quad  {\big(\eta_i\big)}_{i \in I} \mapsto \, \exp\!\big(\, {\textstyle \sum_{i \in I}} \, \eta_i \, Y_i \big)  $$
 \vskip5pt
   {\it (c)} \,  There exist isomorphisms of (set-valued) functors
 \vskip-15pt
  $$  G_+ \times \mathbb{A}_\K^{0|d_1}  \cong \;  \GPe  \,\; ,  \quad  G_+^{\langle 2 \rangle} \times \mathbb{A}_\K^{0|d_1}  \cong \;  \GPem  \,\; ,  \!\qquad  \text{and}  \!\qquad
  \mathbb{A}_\K^{0|d_1} \!\times G_+  \; \cong \;  \GPe  \,\; ,  \quad  \mathbb{A}_\K^{0|d_1} \!\times G_+^{\langle 2 \rangle}  \; \cong \;  \GPem  $$
given on  $ A $--points   --- for every  $ \, A \in \Wsalg_\K \, $  ---   respectively by
 \vskip-11pt
  $$  \big(\, g_+ \, , \, {\big(\eta_i\big)}_{i \in I} \big) \, \mapsto \, g_+ \cdot \exp\!\big(\, {\textstyle \sum_{i \in I}} \, \eta_i \, Y_i \big)  \qquad  \text{and}  \qquad  \big( {\big(\eta_i\big)}_{i \in I} \, , \, g_+ \big) \, \mapsto \, \exp\!\big(\, {\textstyle \sum_{i \in I}} \, \eta_i \, Y_i \big) \cdot g_+  $$
\end{proposition}

\begin{proof}
 Like for  Theorem \ref{dir-prod-fact-G_Pt - gen},  the proof is very close to (half of) that of  Theorem \ref{thm:dir-prod-fact-G - gen},  with a few, technical differences that involve the representation  $ V $  of  \S \ref{G_Pe-module V};  the necessary changes can easily be dealt with much like in the proof of  Theorem \ref{dir-prod-fact-G_Pt - gen}.  Details are left to the reader.
\end{proof}

\medskip

\begin{free text}  \label{Lie sgroup on G_Pe}
 {\bf The Lie supergroup structure of  $ \GPe \, $.}  For any given  $ \, \cP \in \sHCp_\K \, $  and  $ \, A \in \Wsalg_\K \, $,  by  Proposition \ref{dir-prod-fact-G_Pe - gen}{\it (c)},  we have a particular bijection for the corresponding group  $ \GPe(A) \, $,  that is
\begin{equation}  \label{eq:bij G+_x_A- >-->> G_Pe}
  \phi_A^{\,e} : \, G_+(A) \times \mathbb{A}_\K^{0|d_1\!}(A) \,\;{\buildrel \cong \over {\lhook\joinrel\relbar\joinrel\relbar\joinrel\relbar\joinrel\relbar\joinrel\relbar\joinrel\twoheadrightarrow}}\;\, \GPe(A)
\end{equation}
whose restriction to  $ \, G_+(A) \, $,  identified with  $ \; G_+(A) \times \big\{{(0)}_{i \in I}\big\} \, \subseteq G_+(A) \times \mathbb{A}_\K^{0|d_1}(A) \; $,  \,is the identity    --- onto the natural copy of  $ \, G_+(A) \, $  inside  $ \GPe(A) \, $.
                                                                           \par
   Now,  $ G_+(A) $  is by definition an  $ A_\zero $--manifold  (cf.\ \S \ref{Shvarts_embedding}),  of the same type (real smooth, etc.) as the super Harish-Chandra pair  $ \, \cP := (G_+ , \fg ) \, $;  on the other hand,  $ \mathbb{A}_\K^{0|d_1\!}(A) $  carries canonical structures of  $ A_\zero $--manifold  of any type
%
% (real smooth or real analytic if  $ \, \K = \R \, $,  complex holomorphic if  $ \, \K = \C \, $),
%
 (real smooth, etc.),
 then also of the type of  $ G_+(A) \, $.  So there exists also a canonical ``product structure'' of  $ A_\zero $--manifold   --- of the same type of  $ G_+(A) \, $,  i.e.\ of  $ \cP \, $  ---   onto
%
% the Cartesian product
%
 $ \, G_+(A) \times \mathbb{A}_\K^{0|d_1}(A) \, $.  Then we push-forward this canonical  $ A_\zero $--manifold  structure of  $ \, G_+(A) \times \mathbb{A}_\K^{0|d_1}(A) \, $   --- through the bijection  $ \phi_A^{\,e} $  in  \eqref{eq:bij G+_x_A- >-->> G_Pe}  ---   onto  $ \GPe(A) \, $,  which then is turned into an  $ A_\zero $--manifold  on its own, still of the same type as  $ \cP \, $.
 \vskip1pt
   Using the above mentioned structure of  $ A_\zero $--manifold  on  $ \GPe(A) $  for each  $ \, A \in \Wsalg_\K \, $,  given a morphism  $ \, f : A' \longrightarrow A'' \, $  in  $ \Wsalg_\K $  one can easily check that the corresponding group morphism  $ \, \GPe(f) : \GPe(A') \longrightarrow \GPe(A'') \, $  is a morphism of  $ A'_\zero $--manifolds,  hence it is a morphism of  $ \cA_\zero $--manifolds  (cf.\ \S \ref{Shvarts_embedding}).  Thus  {\it  $ \GPe $  can also be seen as a functor from Weil  $ \K $--superalgebras  to  $ \cA_\zero $--manifolds  (real smooth, real analytic or complex holomorphic as  $ \cP $  is)}.
 \vskip1pt
   Finally, looking again at the commutation relations in  $ \GPe(A) $  one finds that the group multiplication and the inverse map are ``regular'' (i.e., ``real smooth'', ``real analytic'' or  ``complex holomorphic'', according to the type of  $ \cP \, $);  namely, this is proved via calculations like those used to prove  Proposition \ref{prop:fact-G_Pe + sbgr-gener-G^<2>&G_-}{\it (b)}   --- that we skipped ---   or to prove  Proposition \ref{prop:fact-Liesgrp_G}{\it (c)}.  Thus they are morphisms of  $ \cA_\zero $--manifolds,  so  {\it  $ \GPe(A) $  is a group element among  $ \cA_\zero $--manifolds,  i.e.\ it is a Lie $ \cA_\zero $--group\/};  therefore (cf.\  Proposition \ref{funct-char_Lie-supergrps}),  overall  {\it the functor  $ \GPe $  {\sl is a Lie supergroup}   --- of real smooth, real analytic or complex holomorphic type as  $ \cP $  is}.
 \vskip1pt
   Eventually, the outcome of this discussion   --- the key result of the present section ---   is the following statement, which provides a second ``backward functor'' from sHCp's to Lie supergroups:
\end{free text}

\vskip11pt

\begin{theorem}  \label{thm_sHCp's-->Lsgrps - 2}
 The recipe in  Definition \ref{def G_- / G_Pe - gen}  provides functors
%
%   $$  \Psi^e : \sHCp_\R^\infty \longrightarrow \Lsgrp_\R^\infty  \quad ,  \!\!\qquad
% \Psi^e : \sHCp_\R^\omega \longrightarrow \Lsgrp_\R^\omega  \quad ,  \!\!\qquad
% \Psi^e : \sHCp_\C^\omega \longrightarrow \Lsgrp_\C^\omega  $$
%
 \vskip5pt
   \centerline{ $  \Psi^e : \sHCp_\R^\infty \longrightarrow \Lsgrp_\R^\infty  \quad ,  \!\!\!\qquad  \Psi^e : \sHCp_\R^\omega \longrightarrow \Lsgrp_\R^\omega  \quad ,  \!\!\!\qquad  \Psi^e : \sHCp_\C^\omega \longrightarrow \Lsgrp_\C^\omega  $ }
 \vskip5pt
\noindent
 given on objects by  $ \; \cP \mapsto \Psi^e(\cP) := \GPe \; $  and on morphisms by
 \vskip5pt
   \centerline{ $  \Big(\, (\phi_+ \, , \varphi) : \, \cP' \relbar\joinrel\longrightarrow \cP'' \;\Big)  \quad \mapsto \quad  \Big(\; \Psi^e\big( (\phi_+ \, , \varphi) \big) : \, \Psi^e\big(\cP'\big) := G_{{}_{\!\cP'}}^{\,e} \!\relbar\joinrel\longrightarrow G_{{}_{\cP''}}^{\,e} \! =: \Psi^\circ\big(\cP''\big) \,\Big)  $ }
 \vskip5pt
\noindent
 where the functor morphism  $ \;  \Psi^\circ\big( (\phi_+ \, , \varphi) \big) : \Psi^\circ\big(\cP'\big) := G_{{}_{\!\cP'}}^{\,e} \!\relbar\joinrel\longrightarrow G_{{}_{\cP''}}^{\,e} \! =: \Psi^e\big(\cP''\big) \; $  is defined by
\begin{equation}  \label{eq:def Psi on morphisms - 2}
  {\Psi^e\big( (\phi_+ \, , \varphi) \big)}_{\!A}  \quad :  \qquad  g'_+ \, \mapsto \, \phi_+\big(\,g'_+\big) \;\; ,
 \quad  \exp\big(\cY'\,\big) \, \mapsto \, \exp\big(\varphi\big(\cY'\,\big)\big)   \qquad \qquad
\end{equation}
for all  $ \, A \in \Wsalg_\K \, $,  $ \, g'_+ \in G'_+(A) \, $,  $ \, \cY' \in A_\uno \!\otimes_\K \fg'_\uno \, $,  with  $ \, \cP' = \big( G'_+ \, , \fg' \,\big) \, $  and  $ \, \cP'' = \big( G''_+ \, , \fg'' \,\big) \; $.
\end{theorem}

\begin{proof}
 We are only left to prove that the given definition for  $ \Psi^e\big( (\phi_+ \, , \varphi) \big) $  really makes sense: indeed, all the rest is proved by our previous analysis (in particular, by \S \ref{Lie sgroup on G_Pe}  above) or is trivial.  Now,  \eqref{eq:def Psi on morphisms - 2}  above fixes the values of our would-be morphism  $ {\Psi^e\big( (\phi_+ \, , \varphi) \big)}_{\!A} $  {\sl on generators\/}  of  $ \, \Psi^e\big(\cP'\big)(A) := G_{{}_{\!\cP'}}^{\,e}(A) \, $:  a direct check shows that all defining  {\sl relations\/}  among such generators   --- in  $ G_{{}_{\!\cP'}}^{\,e}(A) $   ---   are mapped to corresponding (defining) relations in  $ G_{{}_{\!\cP''}}^{\,e}(A) \, $,  thus yielding a well-defined group morphism as required.
%%%
 In particular, this follows from the special properties of the Lie polynomials  $ P_\zero^{(d_1)} $  and  $ P_\uno^{(d_1)} $   --- for  $ \, d_1 \in \big\{\, d'_1 := \text{\sl dim}\big(\fg'_\uno\big) \, , \, d''_1 := \text{\sl dim}\big(\fg''_\uno\big) \,\big\} \, $  with possibly  $ \, d'_1 \not= d''_2 \, $  ---   and of their factors/summands  $ T_\zero^{(s)} $  and  $ T_\uno^{(s)} $  mentioned in  Lemma \ref{lemma:relations-in-G(A)}{\it (h)}.
%%%

  However, we must still show that  {\sl each such  $ {\Psi^e\big( (\phi_+ \, , \varphi) \big)}_{\!A} $  is a morphism of  $ \cA_\zero $--manifolds  too}.
%
% ,  which needs some extra work.
%
                                                              \par
   Both groups  $ G_{{}_{\!\cP'}}^{\,e}(A) $  and  $ G_{{}_{\!\cP''}}^{\,e}(A) $  admit factorizations of type  $ \; G_+ \! \times \exp\!\big( A_\uno \!\otimes_\K \fg_\uno \big) \, $,  \, as in  Proposition \ref{dir-prod-fact-G_Pe - gen}{\it (a)\/}:  in particular, any  $ \, g' \in G_{{}_{\!\cP'}}^{\,e}(A) \, $  uniquely factors into  $ \; g' \, = \, g'_+ \cdot\, \exp\!\big( \cY' \,\big) \; $;  \, then  $ {\Psi^e\big( (\phi_+ \, , \varphi) \big)}_{\!A} \, $,  being a group morphism, maps  $ g' $  onto
\begin{equation}  \label{eq:fact Psie_phi(g')}
  {\Psi^e\big( (\phi_+ \, , \varphi) \big)}_{\!A} \big(\,g'\,\big)  \,\; = \;\,  {\Psi^e\big( (\phi_+ \, , \varphi) \big)}_{\!A} \big(\, g'_+ \cdot\, \exp\!\big(\cY'\,\big) \big)  \,\; = \;\,  \phi_+\big(\,g'_+\big) \cdot \exp\!\big( \varphi_{{}_A}\!\big(\cY'\,\big) \big)
\end{equation}
where  $ \, \varphi_{{}_A\!}\big(\cY'\,\big) \, $  stands for the image of  $ \cY' $  for the map  $ \; \varphi_{{}_A} \! : A_\uno \!\otimes_\K \fg'_\uno \!\relbar\joinrel\longrightarrow A_\uno \!\otimes_\K \fg''_\uno \; $  obtained by scalar extension from  $ \; \varphi{\big|}_{\fg'_\uno} : \fg'_\uno \!\relbar\joinrel\longrightarrow \fg''_\uno \; $.  As both  $ \, \phi_+ \, $  and  $ \; \exp \circ \, \varphi_{{}_A\!} \circ {\big( \exp{\big|}_{A_\uno \otimes\, \fg_\uno} \big)}^{-1} \, $  are {\sl maps of $ \cA_\zero $--manifolds},  from  \eqref{eq:fact Psie_phi(g')}  we deduce that  $ \, {\Psi^e\big( (\phi_+ \, , \varphi) \big)}_{\!A} \, $  is a map of  $ \cA_\zero $--manifolds  too, q.e.d.
\end{proof}

\bigskip

\section{The new equivalences  $ \, \sHCp \cong \Lsgrp \, $.}  \label{sec-equivalences}

\smallskip

   {\ } \;\;   In  Section \ref{sHCp's->Lsgroups}  we introduced two functors, denoted  $ \Psi^\circ $  and  $ \Psi^e \, $,  from sHCp's (of any type: real smooth, real analytic or complex holomorphic) to Lie supergroups (of the same type); in particular, they go the other way round with respect to the ``natural'' functor  $ \Phi $  considered in  Section \ref{sgroups-to-sHCp's}.  We shall now show that both these two functors are quasi-inverse to  $ \Phi $,  so that (together with  $ \Phi \, $)  they provide equivalences between the categories of super Harish-Chandra pairs and of Lie supergroups.

\medskip

\subsection{The functor  $ \Psi $  as a quasi-inverse to  $ \Phi \, $:  proof of  $ \,\; \Phi \circ \Psi \, \cong \, \text{\sl id}_{\,\sHCp} $}  \label{Psi-inv-Phi - 1}

\smallskip

   {\ } \;\;   In this subsection we cope with the first half of our task, namely proving that  $ \; \Phi \circ \Psi \cong \text{\sl id}_{\sHCp} \; $  for  $ \, \Psi \in \big\{\, \Psi^\circ , \Psi^e \,\big\} \, $;  indeed, this is the easy part of our task.

\vskip11pt

\begin{proposition}  \label{Phi_inverse_Psi}
 Let  $ \, \Phi $  and  $ \, \Psi \in \big\{\, \Psi^\circ , \Psi^e \,\big\} \, $  be as in  Theorems \ref{thm_Lsgrps-->sHCp's}, \ref{thm_sHCp's-->Lsgrps - 1}  and  \ref{thm_sHCp's-->Lsgrps - 2}.  Then  $ \,\; \Phi \circ \Psi^\circ \cong \, \text{\sl id}_{\,\sHCp} \; $,  \; where  ``$ \, \sHCp \, $''  must be read as either  $ \, \sHCp_\R^\infty \, $,  or  $ \, \sHCp_\R^\omega \, $,  or  $ \, \sHCp_\C^\omega \, $,  and  $ \, \Phi $  and  $ \, \Psi $  must be taken as working onto the corresponding types of Lie supergroups or sHCp's.
\end{proposition}

\begin{proof}
 This follows almost directly from definitions.  We begin with the case of  $ \Psi^\circ \, $.  Let us consider a super Harish-Chandra pair (of either smooth, or analytic, or holomorphic type)  $ \, \cP := \big( G_+ \, , \, \fg \,\big) \, $,  and let  $ \, \GPt = \Psi^\circ(\cP) \, $,  so that  $ \, \big( \Phi \circ \Psi^\circ \big)(\cP) = \Phi\big(\GPt\big) = \big( {\big( \GPt \big)}_\zero \, , \, \Lie\big( \GPt \big) \big) \; $.  Then, by the very construction of  $ \GPt $  we have  $ \, {\big( \GPt \big)}_\zero = G_+ \;\, $.  In addition, from the definition of  $ \Lie\,(G) $  in  Definition \ref{tangent_Lie_superalgebra}  and exploiting the factorization  $ \; \GPt \, \cong \, G_+ \times G_-^< \; $  from Proposition \ref{dir-prod-fact-G_Pt - gen},  one finds
%
% --- by straightforward, bare hands computation, ---
%
 that
  $$  \Lie\,\big( \GPt \big)  \;\; = \;\;  \Lie\,\big(\, G_+ \times G_-^{\scriptscriptstyle \,<} \,\big)  \;\; = \;\;  \Lie\,\big( G_+ \big) \oplus T_e\big( G_-^{\scriptscriptstyle \,<} \,\big)  \;\; = \;\;  \cL_{\fg_\zero} \oplus \cL_{\fg_\uno}  \;\; = \;\;  \cL_\fg  $$
(cf.\ \S \ref{Lie-salg_funct}  and  Proposition \ref{Lie-funct_Lie(G)}  for notation), which means   --- identifying  $ \cL_\fg $  with  $ \fg $  as usual ---   simply  $ \, \Lie\,\big( \GPt \big) = \fg \, $,  this being an identification as Lie  $ \K $--superalgebras.  Therefore, in the end
  $$  \big( \Phi \circ \Psi^\circ \big)(\cP)  \,\; = \;\,  \Phi\big(\GPt\big)  \,\; = \;\,  \big( {\big( \GPt \big)}_\zero \, , \Lie\big( \GPt \big) \big)  \,\; \cong \;\,  \big( G_+ \, , \, \fg \,\big)  \,\; = \;\,  \cP  $$
which means that  $ \, \Phi \circ \Psi^\circ \, $  acts on objects   --- up to natural isomorphisms ---   as the identity, q.e.d.
 \vskip5pt
   As to morphisms, let  $ \; (\phi_+ , \varphi) : \cP' \!= \big( G'_+ \, , \, \fg' \,\big) \relbar\joinrel\longrightarrow \! \big( G''_+ \, , \, \fg'' \,\big) = \cP'' \; $  be a morphism of sHCp's and  $ \; \phi := \Psi^\circ\big( (\phi_+ , \varphi) \big) : \Psi^\circ\big(\cP'\big) = G_{{}_{\!\cP'}} \!\relbar\joinrel\longrightarrow G_{{}_{\!\cP''}} \! = \Psi^\circ\big(\cP''\big) \; $  the corresponding (via  $ \Psi^\circ $)  morphism of supergroups; we aim to prove that  $ \, \big( \Phi \circ \Psi^\circ \big)\big((\phi_+ , \varphi)\big) = \Phi(\phi) \, $  does coincide (up to the natural isomorphisms  $ \, \big( \Phi \circ \Psi^\circ \big)\big(\cP'\big) \cong \cP' \, $  and  $ \, \big( \Phi \circ \Psi^\circ \big)\big(\cP''\big) \cong \cP'' \, $  mentioned above) with  $ \, (\phi_+ , \varphi) \, $  itself.
 \vskip5pt
   By definition  $ \, \Phi(\phi) := \big(\, \phi{\big|}_{{(G_{{}_{\!\cP'}})}_\zero} , \, d\phi \big) \, $.  Now, on the one hand by the very construction of  $ \phi $  we have  $ \; \phi{\big|}_{{(G_{{}_{\!\cP'}})}_\zero} \! = \, \Psi^\circ\big( (\phi_+ \, , \varphi) \big){\big|}_{G'_+} = \, \phi_+ \; $.  On the other hand, like in the proof of  Theorem \ref{thm_sHCp's-->Lsgrps - 1}  we consider factorizations  $ \, G_{{}_{\!\cP'}} \! = G'_+ \times G_-^{\prime,{\scriptscriptstyle <}} \, $  and  $ \, G_{{}_{\!\cP''}} \! = G''_+ \times G_-^{\prime\prime, {\scriptscriptstyle <}} \; $;  using these and the very constructions we find that the action of $ \, d\phi \, $  onto  $ \; T_e\big( G_{{}_{\!\cP'}} \big) \, = \, T_e\big( G'_+ \big) \,\oplus\, T_e\big( G_-^{\prime,{\scriptscriptstyle <}} \big) \, = \, \fg_\zero \,\oplus\, \fg_\uno \, = \, \fg \; $  is given by  $ \; d\phi{\big|}_{\fg_\zero} \! = \, d\phi_+ \, = \varphi{\big|}_{\fg_\zero} \; $  onto  $ \, \fg_\zero \, $  and by  $ \; d\phi{\big|}_{\fg_\uno} \! = \, \varphi{\big|}_{\fg_\uno} \; $  onto  $ \, \fg_\uno \, $;  eventually, this gives  $ \; d\phi = d\phi{\big|}_{\fg_\zero} \!\oplus d\phi{\big|}_{\fg_\uno} \! = \varphi_\zero \oplus \varphi_\uno = \varphi \; $  as expected, so that  $ \;  \Phi(\phi) := \big(\, \phi{\big|}_{{(G_{{}_{\!\cP'}})}_\zero} , \, d\phi \big) = \big( \phi_+ \, , \, \varphi \,\big) \; $.
 \vskip5pt
   Finally, the case of  $ \Psi^e $  is dealt with in an entirely similar way.  One simply has to exploit the parallel results for  $ \GPe $  to those for  $ \GPt \, $,  like  Proposition \ref{dir-prod-fact-G_Pe - gen}  (instead of  Proposition \ref{dir-prod-fact-G_Pt - gen}),  Theorem \ref{thm_sHCp's-->Lsgrps - 2}  (instead of  Theorem \ref{thm_sHCp's-->Lsgrps - 1}),  the factorizations  $ \, G_{{}_{\!\cP'}} \! = G'_+ \times \exp\!\big( {(-)}_\uno \!\otimes_\K \fg'_\uno \big) \, $  and  $ \, G_{{}_{\!\cP''}} \! = G''_+ \times \exp\!\big( {(-)}_\uno \!\otimes_\K \fg''_\uno \big) \; $,  and so on and so forth.  Details are left to the reader.
\end{proof}

\medskip

\subsection{The functor  $ \Psi $  as a quasi-inverse to  $ \Phi \, $:  proof of  $ \,\; \Psi \circ \Phi \, \cong \, \text{\sl id}_{\,\Lsgrp} $}  \label{Psi-inv-Phi - 2}

\smallskip

   {\ } \;\;   We can now add up the last missing piece, proving that  $ \, \Psi \circ \Phi \, $  is isomorphic to the identity functor on Lie supergroups, for all  $ \, \Psi \in \big\{\, \Psi^\circ , \Psi^e \,\big\} \, $.  This last step, together with  Proposition \ref{Phi_inverse_Psi},  will prove that  $ \Phi $  and  $ \Psi $  are quasi-inverse to each other, in particular they are equivalences between Lie supergroups an super Harish-Chandra pairs (and viceversa).  The result reads as follows:

\smallskip

\begin{proposition}  \label{Psi_inverse_Phi}
 Let  $ \, \Phi $  and  $ \, \Psi \in \big\{\, \Psi^\circ , \Psi^e \,\big\} \, $  be as in  Theorems \ref{thm_Lsgrps-->sHCp's}, \ref{thm_sHCp's-->Lsgrps - 1}  and  \ref{thm_sHCp's-->Lsgrps - 2}.  Then  $ \,\; \Psi^\circ \circ \Phi \, \cong \, \text{\sl id}_{\,\Lsgrp} \; $,  \, where  ``$ \, \Lsgrp \, $''  must be read as either  $ \, \Lsgrp_\R^\infty \, $,  or  $ \, \Lsgrp_\R^\omega \, $,  or  $ \, \Lsgrp_\C^\omega \, $,  and  $ \, \Phi $  and  $ \, \Psi $  must be taken as working onto the corresponding types of Lie supergroups or sHCp's.
\end{proposition}

\begin{proof}
 We begin again by looking at the case of  $ \Psi^\circ \, $.
                                                                          \par
   Given a Lie supergroup  $ G \, $,  set  $ \, \fg := \Lie\,(G) \, $  and  $ \, \cP := \Phi_g(G) = (G_\zero,\fg) \, $.  We look at the supergroup  $ \; \Psi^\circ\big(\Phi(G)\big) = \Psi^\circ(\cP) := \GPt \; $,  \, and prove that the latter is naturally isomorphic to  $ G \, $.
                                                                          \par
   For any  $ \, A \in \salg_\bk \, $,  by abuse of notation we denote with the same symbol any element  $ \, g_{{}_0} \in G_\zero(A) \, $  as belonging to  $ G(A) $   --- via the embedding of  $ G_\zero(A) $  into  $ G(A) $  ---   and as an element of  $ \GPt(A) $   --- actually, one of the distinguished generators of  $ \GPt(A) $  given from scratch.
                                                                          \par
   With this convention, it is immediate to see that  Lemma \ref{lemma:relations-in-G(A)}  yields the following:  {\sl there exists a unique group morphism  $ \; \phi_{{}_A} \! : \GPt(A) \!\relbar\joinrel\relbar\joinrel\longrightarrow \! G(A) \; $  such that  $ \; \phi_{{}_A}(\,g_{{}_0}) = g_{{}_0} \; $  for all  $ \, g_{{}_0} \in G_\zero(A) \, $  and  $ \; \phi_{{}_A}\big( (1 + \eta \, Y) \big) = (1 + \eta \, Y) \; $  for all  $ \, \eta \in A_\uno \, $,  $ \, Y \in \fg_\uno \, $}.  Indeed, thanks to that lemma we know that the defining relations among generators of  $ \GPt(A) $  are also satisfied by their images in  $ G(A) $ through  $ \phi_{{}_A} $  under the above prescription.
                                                                          \par
   Due to the factorization  \eqref{eq:factor-1_G}  in  Proposition \ref{prop:fact-Liesgrp_G},  we have also that  {\sl the morphism  $ \phi_{{}_A} $  is actually surjective}.  Even more, the  {\sl Global Splitting Theorem\/}  for  $ G $  (namely,  Theorem \ref{thm:dir-prod-fact-G - gen})  and for  $ \GPt $  (that is,  Proposition \ref{dir-prod-fact-G_Pt - gen})  together easily imply that  {\sl the morphism  $ \phi_{{}_A} $  is also injective},  hence  {\sl it is a group isomorphism}.  Finally, it is clear that all these  $ \phi_{{}_A} $'s  are natural in  $ A \, $,  thus altogether they provide an isomorphism between  $ \, \GPt = \Psi^\circ\big(\Phi(G)\big) \, $  and  $ G \, $,  which ends the proof.
 \vskip3pt
   The case of  $ \Psi^e $  is dealt with again similarly, just using the parallel results for  $ \GPe $  to those applied for  $ \GPt $  in the above arguments: namely,  Lemma \ref{lemma:relations-in-G(A)}  (the last part), formula  \eqref{eq:factor-2_G}  instead of  \eqref{eq:factor-1_G},  Proposition \ref{dir-prod-fact-G_Pe - gen}  instead of  Proposition \ref{dir-prod-fact-G_Pt - gen},  etc.  Details are left to the reader.
\end{proof}

\bigskip

\section{Linear case and representations.}  \label{sec_lin-cae/repr's}

\smallskip

%
%    {\ } \;\;   In this subsection we spend, somewhat shortly, a few words about the
% fallout of the existence of a category equivalence between Lie supergroups and super
% Harish-Chandra pairs, in particular the one that we realize via the functor  $ \Psi $
% presented in  \S \ref{sgroups-out-sHCp's-1},  in representation theory.
%
   {\ } \;\;   We shall now briefly discuss the fallout, in representation theory, of the existence of a an equivalence between Lie supergroups and sHCp's, in particular when realized
   \hbox{via the functors  $ \Psi $  of  \S \ref{sHCp's->Lsgroups}.}

\medskip

\subsection{The linear case}  \label{subsec:linear-case}

\smallskip

   {\ } \;\;   The very construction of our functors  $ \Psi^\circ $  and  $ \Psi^e $  gains a much more concrete meaning when the super Harish-Chandra pairs (and Lie supergroups) we deal with are  {\sl linear}.
 \vskip5pt
   Indeed, assume first that the Lie supergroup  $ G $  is linear, i.e.\ it embeds into some  $ \rGL(V) \, $, where  $ V $  is a suitable superspace (in other words, there exists a faithful  $ G $--module  $ V \, $).  Then both  $ G_\zero $  and  $ \, \fg := \Lie(G) \, $  embed into  $ \End(V) \, $,  and the relations linking them (formalized by saying that  ``$ \, (G_\zero , \fg) $  is a super Harish-Chandra pair'')
%
% actually
%
 are relations among elements of the unital, associative superalgebra  $ \End(V) \, $.  Conversely,
%
% basing on this,
%
 one can formally
%
% make up the notion of
%
 define a
 ``{\sl linear\/}  super Harish-Chandra pair'' as being any sHCp $ (G_\zero , \fg) $  such that both  $ G_\zero $  and  $ \fg $  embed into some  $ \End(V) \, $,  and the compatibility relations linking  $ G_\zero $  and  $ \fg $
%
% actually
%
 hold true as relations inside
%
% the superalgebra
%
 $ \End(V) $  itself   ---  cf.\ \cite{ga4},  Definition 4.2.1{\it (b)\/}  for such a formalization in the
%
% setup of algebraic supergroups and sHCp's
%
 algebraic setup
 (which easily adapts to the present Lie setup).  The above then tells us, in short, that if  $ G $  is linear then its associated super Harish-Chandra pair  $ \, \Phi(G) =: \cP \, $  is linear too   --- both being linearized through their faithful representation onto  $ V $.
 \vskip5pt
   On the other hand, let us start with a linear sHCp, say  $ \, \cP = (G_+,\fg) \, $:  so the latter is embedded (in the obvious sense) into the sHCp  $ \, \big( \rGL_\zero(V) , \rgl(V) \big) \, $  for some representation superspace  $ V $.  Thus for  $ \, A \in \Wsalg_\K \, $  both  $ G_+(A) $  and  $ A \otimes_\K \fg $  are embedded into  $ \, \big(\End(V)\big)(A) \, $,  with relations among them   --- inside the algebra  $ \, \big(\End(V)\big)(A) \, $  ---   induced by the very notion of linear sHCp.  Now, one can consider in  $ \, \big(\End(V)\big)(A) \, $  all elements of the form  $ \, \exp(\eta \, Y) = (1 + \eta \, Y) \, $   --- with  $ \, \eta \in A_\uno \, $,  $ \, Y \in \fg_\uno \, $  ---   that actually belong to  $ \, \big(\rGL(V)\big)(A) \, $:  and clearly  $ \, G_+(A) \subseteq \big(\rGL(V)\big)(A) \, $  too.  Therefore,  {\sl one can take inside  $ \big(\rGL(V)\big)(A) $  the subgroup  $ \GP^{\circ, \scriptscriptstyle V}(A) $  generated by  $ G_+(A) $  and by all the  $ (1 + \eta \, Y) $'s}.
 \vskip4pt
   A trivial check shows that the elements from  $ G_+(A) $  and the  $ (1 + \eta_i Y_i) $'s  enjoy all relations that enter in the very definition  $ \GPt(A) $  for their parallel counterparts: thus,
  {\it there exists a (unique) group morphism  $ \; \phi^\circ_{{}_A} \! : \GPt(A) \!\relbar\joinrel\relbar\joinrel\longrightarrow \! \GP^{\circ , \scriptscriptstyle V}(A) \; $  such that  $ \; \phi^\circ_{{}_A}(\,g_{{}_+}) = g_{{}_+} \; $  and  $ \; \phi^\circ_{{}_A}\big( (1 + \eta \, Y) \big) = (1 + \eta \, Y) \; $  for all  $ \, g_{{}_+} \in G_+(A) \, $,  $ \, \eta \in A_\uno \, $,  $ \, Y \in \fg_\uno \, $;  in addition, by construction this  $ \phi^\circ_{{}_A} $  is clearly onto}.
                                                                      \par
   On the other hand,  $ \GP^{\circ, \scriptscriptstyle V}(A) $  acts faithfully on  $ V $,  by definition: indeed, it is linear, namely ``linearized by  $ V \, $''.  A key fact then is that this linearization allows one to show that  {\it the like of the  {\sl Global Splitting Theorem}   ---  {\rm cf.\ Theorem \ref{thm:dir-prod-fact-G - gen}  and  Proposition \ref{dir-prod-fact-G_Pe - gen}}  ---   does hold true for  $ \GP^{\circ, \scriptscriptstyle V}(A) \, $};  indeed, one can apply the same analysis and arguments used in the proofs of either  Theorem \ref{thm:dir-prod-fact-G - gen}  or  Proposition \ref{dir-prod-fact-G_Pe - gen},  {\sl but\/}  for one single change: the linearization of  $  \GP^{\circ, \scriptscriptstyle V}(A) $   (induced by the initial linearization of the sHCp  $ \cP $  we started with) has to replace the following key ingredients:
 \vskip4pt
%
%%%
 \vbox{
%%%
   --- in the proof of  Theorem \ref{thm:dir-prod-fact-G - gen},  one has that  $ \, G(A) := \coprod_{x \in |G|} \Hom_{\salg_\K} \big( \cO_{|G|,x} \, , A \big) \, $,
 \vskip3pt
   --- in the proof of  Proposition \ref{dir-prod-fact-G_Pe - gen}  (and of the lemmas before it, mostly), one has that
                                                                     \par
   \ \phantom{--}   $ \GPt(A) $  is acting onto $ \; V := \text{\sl Ind}_{\fg_\zero}^{\,\fg}(\Uuno) \; $   --- see  \eqref{eq:splitting_V}.
%%%
 }
%%%
%%%
%
 \vskip4pt
\noindent
 In fact, in both cases   --- of either  $ G(A) $  or  $ \GPt(A) $  ---   the group under exam is realized as a group of maps (linear operators, in the second case), and these are rich enough to ``separate (enough) points'' so to guarantee the uniqueness of factorization(s) that is the core part of the Global Splitting Theorem.  In the case of  $ \GP^{\circ, \scriptscriptstyle V}(A) $  instead, its built-in linearization provides a similar realization as ``group of maps'', and this again allows to separate enough points to get global splitting(s) for  $ \GP^{\circ, \scriptscriptstyle V}(A) $  too.  Finally, thanks to the Global Splitting Theorem for  $ \GPt(A) $  {\sl and\/}  for  $ \GP^{\circ, \scriptscriptstyle V}(A) $  one can apply again the arguments used in the proof of  Proposition \ref{Psi_inverse_Phi}  and successfully prove that  {\it the above group (epi)morphism  $ \; \phi^\circ_{{}_A} \! : \GPt(A) \!\relbar\joinrel\relbar\joinrel\longrightarrow \! \GP^{\circ, \scriptscriptstyle V}(A) \; $  is also injective, hence it is an}  {\sl isomorphism}.
                                                                     \par
   By construction all these isomorphisms  $ \phi^\circ_{{}_A} $  are natural in  $ A \, $,  hence they give altogether a functor isomorphism  $ \, \phi^\circ : \GPt {\buildrel \cong \over {\relbar\joinrel\longrightarrow}}\, \GP^{\circ, \scriptscriptstyle V} \, $.  Therefore  $ \, \GPt \cong \GP^{\circ, \scriptscriptstyle V} \, $,  which means that we found a
%%%
 different,\allowbreak{\ }
%%%
 concrete realization of  $ \GPt \, $,  that is now constructed explicitly as the  {\sl linear\/}  Lie supergroup  $ \GP^{\circ, \scriptscriptstyle V} \, $.
 \vskip5pt
   In a parallel way, still starting with a linear sHCp  $ \, \cP = (G_+,\fg) \, $  embedded into  $ \, \big( \rGL_\zero(V) , \rgl(V) \big) \, $,  one can consider in  $ \, \big(\End(V)\big)(A) \, $  all elements of the form  $ \, \exp(\cY) \, $  with  $ \, \cY \in A_\uno \!\otimes_\K \fg_\uno \, $,
%
%   ---   that actually belong to  $ \, \big(\rGL(V)\big)(A) \, $:
% and clearly  $ \, G_+(A) \subseteq \big(\rGL(V)\big)(A) \, $  too.
% Therefore,  {\sl one can take inside  $ \big(\rGL(V)\big)(A) $
%
 and then  {\sl take the subgroup  $ \GP^{e, \scriptscriptstyle V}(A) $  of  $ \big(\rGL(V)\big)(A) $  generated by  $ G_+(A) $  and by all the  $ \exp(\cY) $'s}.  Acting like above we find (by parallel arguments) that
  {\it there exists a (unique) group epimorphism  $ \; \phi^e_{{}_A} \! : \GPe(A) \!\relbar\joinrel\relbar\joinrel\longrightarrow \! \GP^{e , \scriptscriptstyle V}(A) \; $  such that  $ \; \phi^e_{{}_A}(\,g_{{}_+}) = g_{{}_+} \, $,  $ \; \phi^e_{{}_A}\big( \exp(\cY) \big) = \exp(\cY) \, $,  for all  $ \, g_{{}_+} \! \in G_+(A) \, $  and  $ \, \cY \in A_\uno \!\otimes_\K \fg_\uno \, $}.
 Even more, still by the same method as above (up to minimal changes) we eventually find that all these  $ \phi^e_{{}_A} $'s  are in fact  {\sl isomorphisms},  natural in  $ A \, $,  hence they define a functor isomorphism  $ \, \phi^e : \GPe {\buildrel \cong \over {\relbar\joinrel\longrightarrow}}\, \GP^{e, \scriptscriptstyle V} \, $.  This gives yet another concrete realization of  $ \GPe \, $,  now explicitly realized as the  {\sl linear\/}  Lie supergroup  $ \GP^{e, \scriptscriptstyle V} \, $.

\medskip

%%%%%
 %%%%%
  \vfill
    \eject
 %%%%%
%%%%%

\subsection{Representations 1: supergroup modules vs.\ sHCp modules}  \label{subsec:reps-1_sgrp-mod's/sHCp-mod's}

\smallskip

%
%    {\ } \;\;   An important byproduct of the equivalence between Lie supergroups and
% super Harish-Chandra pairs comes as application to representation theory.  Indeed, let
% $ G $  and  $ \cP $  respectively be Lie a supergroup and a super Harish-Chandra pair
% --- of either type (smooth, etc.) as usual ---   over  $ \K $  that correspond to each
% other through the equivalence presented in  \S \ref{sec-equivalences}   --- namely,
% $ \, G = \Psi(\cP) \, $  and  $ \, \cP = \Phi(G) \, $.  Then we let  $ G $--Mod and
% $ \cP $--Mod  respectively be the category of  $ G $--modules  and of  $ \cP $--modules
% (again of type smooth, etc.); in short, here we mean that a  $ G $--module  is the
% datum of a finite dimensional supermodule  $ M' $  with a morphism  $ \, \phi :
% G \longrightarrow \rGL(M') \, $  of Lie supergroups (in the proper category), whereas
% a  $ \cP $--module  is the datum of a finite dimensional supermodule  $ M'' $  with a
% morphism  $ \, (\phi_+,\varphi) : \cP \longrightarrow \big( {\rGL(M'')}_\zero \, ,
% \rgl(M'') \big) \, $  of super Harish-Chandra pairs.
%
   {\ } \;\;   An important application of the equivalence between Lie supergroups and super Harish-Chandra pairs occurs in representation theory.  Indeed, let  $ G $  and  $ \cP $  be a Lie supergroup and a super Harish-Chandra pair that correspond to each other through one of the equivalences  $ \, \Psi \in \big\{\, \Psi^\circ , \Psi^e \,\big\} \, $  presented in  \S \ref{sec-equivalences},  i.e.  $ \, G = \Psi(\cP) \, $  and  $ \, \cP = \Phi(G) \, $.  Let also  $ G $--Mod and  $ \cP $--Mod  be the category of  $ G $--modules  and of  $ \cP $--modules  (of the correct type: smooth, etc.); in short, we mean that a  $ G $--module  is the datum of a finite dimensional supermodule  $ M' $  with a morphism  $ \, \phi : G \longrightarrow \rGL(M') \, $  of Lie supergroups (in the proper category), whereas a  $ \cP $--module  is the datum of a finite dimensional super\-module  $ M'' $  with a morphism
%%%%%
   \hbox{$ \, (\phi_+,\varphi) : \cP \longrightarrow \!\big( {\rGL\big(M''\big)}_\zero \, , \rgl\big(M''\big) \!\big) \, $  of super Harish-Chandra pairs.}
%%%%%
%%%%%
%
 \vskip4pt
   Just to fix notation, we assume to be in the real smooth case, the other cases being entirely similar.  Let's assume  $ M' $  is a  $ G $--module;  applying  $ \, \Phi : \Lsgrp_\R^\infty \!\longrightarrow \sHCp_\R \, $  to the morphism  $ \, \phi : G \longrightarrow \rGL(M') \, $  we find a morphism  $ \, \Phi(\phi) : \Phi(G) \longrightarrow \Phi\big(\rGL(M')\big) \, $  between the corresponding objects in  $ \sHCp_\R \, $.  But  $ \, \Phi(G) = \cP \, $  by assumption and  $ \, \Phi\big(\rGL(M')\big) = \big( {\rGL(M')}_\zero \, , \rgl(M') \big) \, $,  so what we have is a morphism  $ \, \Phi(\phi) \! : \cP \! \longrightarrow \!\big( {\rGL(M')}_\zero \, , \rgl(M') \big) \, $  making  $ M' $ into a  $ \cP $--module.
                                                                               \par
   Conversely, let  $ M'' $  be a  $ \cP $--module.  Applying the functor  $ \, \Psi \! : \! \sHCp \!\longrightarrow \! \Lsgrp_\R^\infty \, $  to the corresponding morphism  $ \, (\phi_+,\varphi) : \cP \longrightarrow \big( {\rGL(M'')}_\zero \, , \rgl(M'') \big) \, $  we get a morphism between the corresponding supergroups, namely  $ \, \Psi\big((\phi_+,\varphi)\big) : \Psi(\cP) \longrightarrow \Psi\big(\big( {\rGL(M'')}_\zero \, , \rgl(M'') \big)\big) \, $.  As  $ \, \Psi(\cP) = G \, $  by and  $ \, \Psi\big(\big( {\rGL(M'')}_\zero \, , \rgl(M'') \big)\big) = \rGL(M'') \, $,  we find a morphism  $ \, \Psi\big((\phi_+,\varphi)\big) \! : G \! \longrightarrow \! \rGL(M'') \, $  in  $ \Lsgrp_\R^\infty $  which makes  $ M'' $ into a  $ G $--module.
 In fact, in this way the  $ G $--action  that one gets on  $ M'' $  is just what one obtains by direct application of the recipe in  \S \ref{subsec:linear-case}  to the  {\sl linear\/}  super Harish-Chandra pair  $ \, (\phi_+,\varphi)(\cP) \, $ inside  $ \, \big( {\rGL(M'')}_\zero \, , \rgl(M'') \big) \, $  that is the image of  $ \, \cP $  through  $ (\phi_+,\varphi) \, $.
                                                                               \par
   The reader can easily check that the previous discussion   --- extended to the real analytic case and to the complex holomorphic case as well ---   has the following outcome:

\vskip7pt

\begin{theorem}  \label{main-thm_sHCp-equiv}
 Let  $ \, G \, $  and  $ \, \cP\, $  be a Lie supergroup and a super Harish-Chandra pair (of smooth, analytic or holomorphic type) over\/  $ \K $  corresponding to each other as above.  Then:
 \vskip2pt
 (a)\,  for any fixed finite dimensional  $ \K $--supermodule  $ M $,  the above constructions provide two bijections, inverse to each other, between  $ G $--module  structures and  $ \cP $--module  structures on  $ M \, $;
 \vskip2pt
 (b)\,  the whole construction above is natural in  $ M $,  in that the above bijections between  $ G $--module  structures and  $ \cP $--module  structures over two finite dimensional  $ \K $--supermodules  $ \widehat{M} $  and  $ \widetilde{M} $  are compatible with  $ \K$--supermodule  morphisms from  $ \widehat{M} $  to  $ \widetilde{M} \, $.  Thus, all the bijections mentioned in (a), for all different  $ M $'s,  do provide equivalences, quasi-inverse to each other, between the category of all (finite dimensional)  $ G $--modules  and the category of all (finite dimensional)  $ \cP $--modules.
\end{theorem}

\medskip

\subsection{Representations 2: induction from  $ \bG_\zero $  to  $ \bG $}  \label{subsec:reps-2_induction-G_0/G}

\smallskip

   {\ } \;\;   Let  $ \, G \, $  be a supergroup (of any type), with associated classical subsupergroup  $ \, G_\zero \, $.  Let  $ V $  be any  $ G_\zero $--module:  we shall now present an explicit construction of the  {\it induced  $ G $--module}  $ \text{\sl Ind}_{G_\zero}^{\,G}\!(V) \, $.

\vskip3pt

   Being a  $ G_\zero $--module,  $ V $  is also, automatically, a  $ \fg_\zero $--module.  Then one does have the induced  $ \fg $--module  $ \text{\sl Ind}_{\fg_\zero}^{\,\fg}\!(V) \, $,  which can be realized as
 \vskip3pt
 \centerline{ $ \text{\sl Ind}_{\fg_\zero}^{\,\fg}\!(V)  \,\; = \;\,  \text{\sl Ind}_{U(\fg_\zero)}^{\,U(\fg)}(V)  \,\; = \;\,  U(\fg) \otimes_{U(\fg_\zero)} \! V $ }

\vskip5pt

\noindent
 By construction, it is clear that this bears also a unique structure of  $ G_\zero $--module  which is compatible with the  $ \fg $--action  and coincides with the original  $ G_\zero $--action  on  $ \, \K \otimes_{U(\fg_\zero)} \! V \cong V \, $  given from scratch.  Indeed, we can describe explicitly this  $ G_\zero $--action,  as follows.  First, by construction we have
 \vskip4pt
 \centerline{ $ \text{\sl Ind}_{\fg_\zero}^{\,\fg}\!(V)  \,\; = \;\, U(\fg) \otimes_{U(\fg_\zero)} \! V  \,\; = \;\,  {\textstyle \bigwedge} \, \fg_\uno \! \otimes_\K \! V $ }
 \vskip6pt
\noindent
--- because  $ \, U(\fg) \cong {\textstyle \bigwedge} \, \fg_\uno \! \otimes_\K \! U(\fg_\zero) \, $  as a  $ \K$-module,  by the PBW theorem for Lie superalgebras, see  \eqref{eq:splitting_U(g)}  ---   with the  $ \fg_\zero $--action  given by  $ \; x.(y \otimes v) = \text{\sl ad}(x)(y) \otimes v + y \otimes (x.v) \; $  for  $ \, x \in \fg_\zero \, $,  $ \, y \in {\textstyle \bigwedge} \, \fg_\uno \, $,  $ \, v \in V \, $,  where by  $ \, \text{\sl ad} \, $  we denote the unique  $ \fg_\zero $--action  on  $ \, {\textstyle \bigwedge} \, \fg_\uno \, $  by algebra derivations induced by the adjoint  $ \fg_\zero $--action  on  $ \fg_\uno \, $.  Second, this action clearly integrates to a (unique)  $ G_\zero $--action  given by  $ \; g_0.(y \otimes v) \, := \, \text{\sl Ad}(g_0)(y) \otimes (g_0.v) \; $  for  $ \, g_0 \in G_\zero \, $,  $ \, y \in {\textstyle \bigwedge} \, \fg_\uno \, $,  $ \, v \in V \, $,  where we write  $ \, \text{\sl Ad} \, $  for the unique  $ G_\zero $--action  on  $ \, {\textstyle \bigwedge} \, \fg_\uno \, $  by algebra automorphisms induced by the adjoint  $ G_\zero $--action  on  $ \fg_\uno \, $.
                                                                      \par
   The key point is that the above  $ G_\zero $--action  and the built-in  $ \fg $--action on  $ \text{\sl Ind}_{\fg_\zero}^{\,\fg}\!(V) $  are compatible, in that they make  $ \text{\sl Ind}_{\fg_\zero}^{\,\fg}\!(V) $  into a  $ (G_\zero\,,\fg) $--module,  i.e.~a module for the sHCp  $ \, \cP := (G_\zero\,,\fg) = \Phi(G) \, $.  Then for  $ \, \Psi \in \big\{\, \Psi^\circ , \Psi^e \,\big\} \, $,  since  $ \, \Psi\big((G_\zero\,,\fg)\big) \cong G \, $,  by  \S \ref{subsec:reps-1_sgrp-mod's/sHCp-mod's}  we have that  $ \text{\sl Ind}_{\fg_\zero}^{\,\fg}\!(V) $  bears a unique structure of  $ G $--module  which correspond to the previous  $ \cP $--action   --- i.e., it yields (by restriction and ``differentiation'') the previously found  $ G_\zero $--action  and  $ \fg $--action.  In down-to-earth terms what happens is the following.  The action of  $ \, \cP := (G_\zero\,,\fg) = \Phi(G) \, $  onto  $ \, W := \text{\sl Ind}_{\fg_\zero}^{\,\fg}\!(V) \, $  is given by the  $ G_\zero $--action  (induced by the original action on $ V \, $)  and a compatible  $ \fg $--action.  Then for any  $ \, A \in \Wsalg_\K \, $  we have also that all of  $ A_\uno \!\otimes_\K \fg_\uno $  ``acts'' onto  $ \, A \otimes_\K W \, $:  thus well-defined operators  $ \, (1 + \, \eta \, Y) \, $  and  $ \, \exp(\cY) \, $   --- with  $ \, (\eta\,,Y) \in A_\uno \times \fg_\uno \, $,  $ \, \cY \in A_\uno \!\otimes_\K \fg_\uno \, $ ---   exist in
%
% $ \big(\End(W)\big)(A) $  that do belong to
%
 $ \big(\rGL(W)\big)(A) \, $.  One easily checks that these  $ (1 + \eta \, Y) $'s,  or the  $ \exp(\cY) $'s, altogether enjoy among themselves and with the operators given by the  $ G_\zero $--action  the very relations that enter in the definition of either  $ \, \Psi^\circ(\cP) := \GPt \, $  or  $ \, \Psi^e(\cP) = \GPe \, $,  respectively   --- therefore in both cases we get a well-defined action of  $ \Psi(\cP) $  on  $ W \, $,  extending the initial one by  $ G_\zero \, $:  but  $ \, \Psi(\cP) = G \, $,  so we are done.
                                                                      \par
%
%    Therefore,  {\sl we define as  $ \text{\sl Ind}_{G_\zero}^{\,G}\!(V) $  the
% space  $ \, W := \text{\sl Ind}_{\fg_\zero}^{\,\fg}\!(V) \, $  endowed with this
% $ G $--action\/}:  one easily checks that this construction is functorial in
% $ V  $  and has the universal property which makes it into the adjoint of
% ``restriction'' (from  $ G $--modules  to  $ G_\zero $--modules),  so it has
% all rights to be called ``induction'' functor (from  $ G_\zero $--modules
% to  $ G $--modules).
%
   Thus  {\sl we define as  $ \text{\sl Ind}_{G_\zero}^{\,G}\!(V) $  the space  $ \, W := \text{\sl Ind}_{\fg_\zero}^{\,\fg}\!(V) \, $  endowed with this  $ G $--action\/}:  one checks that this construction is functorial in  $ V  $  and has the universal property making it adjoint of ``restriction'' (from  $ G $--modules  to  $ G_\zero $--modules),  so it can be correctly called ``induction'' functor.
                                                                      \par
   In addition, if the original  $ G_\zero $--module  $ V $  is faithful then the induced  $ G $--module  $ \text{\sl Ind}_{G_\zero}^{\,G}\!(V) $  is faithful too: in particular, this means that if  $ G_\zero $  is linearizable, then  $ G $  is linearizable too; more precisely, from a linearization of  $ G_\zero $  one can construct (via induction) a linearization of  $ G $  as well.


\vskip45pt


\begin{thebibliography}{vsv}

\vskip3pt

\bibitem{all-lau}  A.\ Alldridge, M.\ Laubinger,  {\it Infinite-dimensional supermanifolds over arbitrary base fields},  Forum Math.\  {\bf 24}  (2012), no.\ 3, 565--608.

\bibitem{bcf}  L.\ Balduzzi, C.\ Carmeli, R.\ Fioresi, {\it A comparison of the functors of points of supermanifolds},  J.\ Algebra Appl.\  {\bf 12}  (2013),  no.\ 3, 1250152, pp.\ 41.

\bibitem{be}  F.~A.~Berezin,  {\it Introduction to superanalysis},  Edited by A.~A.~Kirillov. D.~Reidel Publishing Company, Dordrecht (Holland), 1987.
%
% With an Appendix by V.~I.~Ogievetsky. Translated from the Russian
% by J.~Niederle and R.~Koteck\'y. Translation edited by Dimitri Le\u\i tes.
%

\bibitem{be-so}  W.\ Bertram, A.\ Souvay,  {\it A general construction of Weil functors},  Cah.\ Topol.\ G{\'e}om.\ Diff{\'e}r.\ Cat{\'e}g.\  {\bf 55} (2014), no.\ 4, 267--313.

\bibitem{bos}  H.~Boseck,  {\it Affine Lie supergroups},  Math.~Nachr.~{\bf 143}  (1989), 303--327.

\bibitem{ccf}  C.~Carmeli, L.~Caston, R.~Fioresi, {\it  Mathematical Foundations of Supersymmetry\/}  (with an appendix by I.~Dimitrov),  EMS Series of Lectures in Mathematics  {\bf 15},  European Mathematical Society, Z{\"u}rich, 2011.

\bibitem{cf}  C.~Carmeli, R.~Fioresi, {\it  Super Distributions, Analytic and Algebraic Super Harish-Chandra pairs},  Pacific J.\ Math.\  {\bf 263}  (2013),  no.\ 1, 29--51.

\bibitem{dm}  P.~Deligne, J.~Morgan,  {\it Notes on supersymmetry (following J.~Bernstein)},  in:  {\sl ``Quantum fields and strings: a course for mathematicians''},  Vol.~1, 2 (Princeton, NJ, 1996/1997), 41--97, Amer.~Math.~Soc., Providence, RI, 1999.

\bibitem{eh}  D.\ Eisenbud, J.\ Harris,  {\it The geometry of schemes},  Graduate Texts in Mathematics  {\bf 197},  Springer-Verlag, New York, 2000.

\bibitem{fg1}  R.~Fioresi, F.~Gavarini,  {\it Chevalley Supergroups},  Memoirs AMS  {\bf 215},  no.~1014 (2012).

\bibitem{fg2}  R.~Fioresi, F.~Gavarini,  {\it On the construction of Chevalley supergroups},  in: S.~Ferrara, R.~Fio\-resi, V.~S.~Varadarajan (eds.),  {\sl ``Supersymmetry in Mathematics \& Physics''},  UCLA Los Angeles, U.S.A.~2010, Lecture Notes in Math.~{\bf 2027}, Springer-Verlag, Berlin-Heidelberg, 2011, pp.~101--123.

\bibitem{ga1}  F.~Gavarini,  {\it Chevalley Supergroups of type  $ D(2,1;a) $},  Proc.\ Edinb.\ Math.\ Soc.\ (2)  {\bf 57}  (2014), no.\ 2, 465--491.

\bibitem{ga2}  F.~Gavarini,  {\it Algebraic supergroups of Cartan type},  Forum Math.\  {\bf 26}  (2014), no.\ 5, 1473--1564.

\bibitem{ga3}  F.~Gavarini,  {\it Corrigendum to ``Algebraic supergroups of Cartan type''},  Forum Math.\ {\bf 28}  (2016), no.\ 5, 1005--1009.

\bibitem{ga4}  F.~Gavarini,  {\it Global splittings and super Harish-Chandra pairs for affine supergroups},  Trans.\ Amer.\ Math.\ Soc.\  {\bf 368}  (2016), no.\ 6, 3973--4026.

\bibitem{jac}  N.\ Jacobson,  {\it Lie algebras}, Republication of the 1962 original, Dover Publications, Inc., New York, 1979.

\bibitem{kostant}  B.~Kostant,  {\it Graded manifolds, graded Lie theory, and prequantization},
in: {\sl Differential geometrical methods in mathematical physics\/}
(Proc.~Sympos., Univ.~Bonn, Bonn, 1975), 177--306, Lecture Notes in Math.  {\bf 570},  Springer, Berlin, 1977.

\bibitem{koszul}  J.-L.~Koszul,  {\it Graded manifolds and graded Lie algebras},  Proceedings of the international meeting on geometry and physics (Florence, 1982), 71--84, Pitagora, Bologna, 1982.

\bibitem{leites}  D.\ A.\ Leites,  {\it Introduction to the theory of supermanifolds},  Russian Math.\ Surveys  {\bf 35}  (1980), no.\ 1, 1--64.

\bibitem{maclane}  S.\ MacLane,  {\it Categories for the working mathematician},  Graduate Texts in Mathematics  {\bf 5},  Springer-Verlag, New York, 1971.

\bibitem{mas}  A.~Masuoka,  {\it Harish-Chandra pairs for algebraic affine supergroup schemes over an arbitrary field},  Transform.~Groups  {\bf 17}  (2012), no.~4, 1085--1121.

\bibitem{mas-shi}  A.\ Masuoka, T.\ Shibata,  {\it Algebraic supergroups and Harish-Chandra pairs over a commutative ring},  Trans.\ Amer.\ Math.\ Soc.\  {\bf 369}  (2017), no.~5, 3443--3481.

\bibitem{mol}  V.\ Molotkov,  {\it Infinite-dimensional and colored supermanifolds},  J.\ Nonlinear Math.\ Phys.\  {\bf 17}  (2010), suppl.\ 1, 375--446.

%
% \bibitem{sac}  C.\ Sachse,  {\it A categorical formulation of superalgebra and supergeometry},
% {\tt arXiv:0802.4067}  (2008).
%

\bibitem{vis}  E.~G.~Vishnyakova,  {\it On complex Lie supergroups and homogeneous split
supermanifolds},  Transform.~Groups  {\bf 16}  (2011), no.~1, 265--285.

\bibitem{vsv}  V.~S.~Varadarajan,  {\it Supersymmetry for mathematicians: an introduction},  Courant Lecture Notes  {\bf 1},  AMS, 2004.

\end{thebibliography}
\end{document}